\documentclass{article}
\usepackage[a4paper,width=150mm,top=25mm,bottom=25mm]{geometry}
\usepackage[english]{babel}
\usepackage[utf8]{inputenc}
\usepackage[T1]{fontenc}
\usepackage{graphicx,color}
\usepackage[final]{pdfpages}
\usepackage[extra]{tipa}
\usepackage{titlesec}
\usepackage{xspace}		
\usepackage[svgnames]{xcolor}
 \usepackage{lmodern}
\usepackage{hyperref}
\usepackage{bm,soul}

\usepackage{enumitem}
\usepackage{mathtools}
\usepackage{graphicx}
\usepackage[ruled,vlined]{algorithm2e}
\usepackage{amsmath,amsfonts,amsthm}
\usepackage{dsfont}
\usepackage{amssymb}
\usepackage{nicefrac}
\usepackage{float}
\usepackage{theoremref}

\usepackage{caption,subcaption}
\usepackage{color}
\usepackage{comment}
\usepackage{amsfonts}
\usepackage{bm}

\usepackage{minitoc}
\usepackage{cite}
\mathtoolsset{showmanualtags}

\numberwithin{counter}{subsection}

\setstcolor{red}

\usepackage{aliascnt}

\newtheorem{theorem}{Theorem}[section]

\newtheorem{lemma}[theorem]{Lemma}
\newtheorem{proposition}[theorem]{Proposition}
\newtheorem{counter-e}{Counterexample}

\newcommand{\D}{\partial}
\newcommand{\eps}{\varepsilon}
\newcommand{\R}{\mathbb{R}}
\newcommand{\N}{\mathbb{N}}

\newcommand{\SU}{$\mathtt{SU}$}
\newcommand{\ST}{$\mathtt{ST}$}

\definecolor{Question}{HTML}{F77409}
\definecolor{Problem}{HTML}{F01A1A}
\definecolor{Comment}{HTML}{011AF8}
\definecolor{Check}{HTML}{08AB06}

\usepackage{authblk}
\title{Modeling the propagation of riots, collective behaviors and epidemics}
\author[1,2]{Henri Berestycki}
\author[3]{Samuel Nordmann}
\author[1]{Luca Rossi}

\affil[1]{\small 
	Centre d'analyse et de mathématique sociales,
	EHESS - CNRS, 54, boulevard Raspail, Paris, France}
\affil[2]{\small Senior Visiting
Fellow, Institute for Advanced Study, Hong Kong
University of Science and Technology}
\affil[3]{\small Department of Applied Mathematics, Tel Aviv University, Tel Aviv 6997801, Israel}

\newenvironment{dedication}
        {\vspace{0ex}\begin{quotation}\begin{center}\begin{em}}
        {\par\end{em}\end{center}\end{quotation}}

\begin{document}

\numberwithin{equation}{section}

\maketitle
\begin{dedication}
{\`A Italo Capuzzo Dolcetta, \\
en signe d'affection, de profonde estime et d'amiti\'e
}
\end{dedication}

\bigskip

\begin{abstract}

This paper is concerned with a family of Reaction-Diffusion systems that we introduced in~\cite{Berestycki2019b}, and that generalizes the {\em SIR} type models from epidemiology.
 Such systems are now also used to describe collective behaviors.
In this paper, we propose a modeling approach for these apparently diverse phenomena through the example of the dynamics of social unrest.
The model involves two quantities: the level of social unrest, or more generally activity,
 $u$, and a field of social tension $v$, which play asymmetric roles. We think of $u$ as the actually observed or explicit quantity while $v$ is an ambiant, sometimes implicit, field of susceptibility that modulates the dynamics of~$u$. 

In this article, we explore this class of model and prove several theoretical results based on the framework developed in~\cite{Berestycki2019b}, of which the present work is a companion paper. We particularly emphasize here two subclasses of systems: {\em tension inhibiting} and {\em tension enhancing}. These are characterized by  respectively a negative or a positive
feedback of the unrest on social tension. We establish several properties for these classes and also study some extensions. 
In particular, we describe the behavior of the system following an initial surge of activity. We show that the model can give rise to many diverse qualitative dynamics. 
We also provide a variety of numerical simulations to illustrate our results and to reveal further properties and open questions.

\end{abstract}
\textbf{Key words: }Epidemiology models $\cdot$ SIR model 
$\cdot$ Threshold phenomenon $\cdot$ Social systems $\cdot$ Reaction-diffusion systems
$\cdot$ Contagion $\cdot$ Traveling waves $\cdot$ Speed of propagation

\newpage
\tableofcontents

\newpage

\section{Motivation and framework}



 
Introduced by Kermack and McKendrick~\cite{Kermacka} as one particular instance of a family of models, the $SIR$ compartmental type model and its host of variants are basic tools of epidemiology. They have given rise to a vast literature.
The {\em SIR} system features two populations: the \emph{Susceptible}, represented by $S(t,x)$, and the \emph{Infected}, represented by $I(t,x)$. These are supplemented by the compartment of \emph{removed} $R(t,x)$.  In its simplest form (with spatial component $x$) and Brownian diffusion, the model is written as the following system of reaction-diffusion equations:
\begin{equation}\label{SIModel}
\left\{\begin{aligned}
&\D_t I -d_1 \Delta I= \beta IS -\gamma I,\\
&\D_t  S-d_2\Delta S=-\beta IS,
\end{aligned}\right.
\end{equation}
together with the equation for the \emph{recovered} $R(t,x)$
$$
\D_t R -d_3\Delta R=\gamma I.
$$
Since \eqref{SIModel} does not involve $R$ (which is derived from $I$), we can overlook it.

This system is essential in epidemiology, both from the point of view of theory~\cite{Hethcote1989,H2000,Vynnycky2010,Miller2017,Siettos2013,Hoppensteadt1982a,Ruan} and applications~\cite{Nepomuceno2018,Bailey1975,Anderson1982}.
Most of the available mathematical approaches are quite specific to this system and do not lend themselves to  be generalized to a broader class of systems. 
Given the variety of extensions and variants of this model, it is natural to seek a unified mathematical framework and to identify the core general properties of this class of system.
This is one of the aims of this paper.



Similar models have been used since long to describe the spread of riots and, more generally, of collective behaviors in various social contexts (see e.g. the survey of Dietz~\cite{Dietz1967}). 
One may trace this use of epidemiology models in the context of social phenomena to the analogy between the mechanisms of contagion and social imitation. 

The mechanism of social imitation implies that the inclination of an individual to join a social movement is largely influenced by the intensity of the movement itself. Once the movement has reached a certain size, owing to several mechanisms such as imitation or social influence, more people are prone to join it and the movement grows. The celebrated work of Granovetter \textit{Threshold Models of Collective Behavior}~\cite{Granovetter1978} described the formal analogy between epidemics and collective behavior in the following visionary terms\footnote{we left out the references}: 
\begin{quotation}
``There is also some similarity between the present models and models used in epidemiology, the diffusion of information and innovations, and the evolution of behavior in groups over time.''
\end{quotation}
This approach is also developed by Burbeck et al.~\cite{Burbeck1978} in their pioneering paper about the dynamics of riots: 
\begin{quotation}
``Patterns within three major riots suggest that the dynamics of the spread of riot behavior during a riot can be fruitfully compared to those operative in classical epidemics. We therefore conceptualize riots as behavioral epidemics, and apply the mathematical theory of epidemics [...].''
\end{quotation}
It is only relatively recently that several works have developed the actual epidemiology approach for the modeling of riots ~\cite{Berestycki2015,Berestycki2016a,Berestycki2018a,Petrovskii,Yang2020}. This approach proved very effective in fitting data from observations~\cite{Bonnasse-Gahot2018,Caroca2020}.

Similar ideas have also been applied to other instances of collective behaviors: the propagation of ideas (see the pioneering works of Goffman \& Newill~\cite{Goffman1964} and Daley \& Kendall~\cite{DALEY1964}, and also the more recent contributions on the propagation of scientific knowledge~\cite{Kaiser2006,Huo2016,Cao2016,Kiss2010,MoritzMarutschke2014}, rumors~\cite{Kawachi2008,Afassinou2014,Zhao2012,Zhao2013}, and extreme ideology in a society~\cite{Santonja2008}),
the diffusion of a new product in a market (as originally studied by Bass~\cite{Bass1969} and more recently included in numerous works dealing both with the theory~\cite{Fibich2017,Fibich2016,Bhattacharya2019,Rodrigues2015} and the applications~\cite{Goncalves2017,Gurley2017}), the growth of political parties~\cite{Jeffs2016} and the propagation of memes and hashtags~\cite{Wang2011a,Skaza2017,Woo2016,Gaumont2018}.

Conversely, the analogy between epidemics and collective behaviors shows us that to a large extent epidemics are  a social phenomenon.
The impacts and challenges of the current COVID-19 epidemic remind us of this fact. 
The website~\cite{MacGregor2020} of the Institute of Development Studies formulates it explicitly: 
\begin{quotation}
``As the COVID-19 pandemic rages across the world, one thing is clear: this epidemic, like all others, is a social phenomenon. The dynamics of the virus, infection and immunity, not to mention on-going efforts to revise and improve clinical care, and endeavors to develop medical treatments and vaccines, are a critical part of the unfolding story. So, too, are peoples' social responses to the disease and interactions with each other.''
\end{quotation}

\paragraph*{}
Since epidemics and collective behaviors, although very different in nature, have structural similarities, they can be modeled and studied within a unified mathematical paradigm.
A question of particular importance is to understand whether a \emph{triggering event}, i.e. a small initial social movement, can result in a significant movement by means of social imitation and other self-reinforcement mechanisms.
The answer to this question typically involves a threshold phenomenon on an ambient level of \emph{susceptibility}. In a context of \emph{low susceptibility}, the \emph{triggering event} fades out and the system promptly returns to a calm situation, whereas in a context of \emph{high susceptibility}, the \emph{triggering event} leads to a significant movement. 

Our main goal in~\cite{Berestycki2019b} and in the present article is to develop a unified mathematical framework to deal with this general setting. 
We aim to unify, generalize, and open new fields of application for epidemiology and collective behavior models.
This article can also be seen as a contribution to the program proposed by Granovetter: 
\begin{quotation}
``To develop these analogies in more detail would require that (1) my models, expressed below as difference equations in discrete time, be translated into differential equations in continuous time, and that (2) some way be found to introduce the "threshold" concept into these other models, which generally do not stress individual differences.
While some work in this direction has been accomplished, it is incomplete''
\end{quotation}

Let us emphasize that 
spatial diffusion plays a key role in the dynamics of many collective behaviors~\cite{Smith2012,Bonnasse-Gahot2018,Yang2020,Braha2012}. 
Therefore, in~\cite{Berestycki2019b} and here 
we include spatial dependence and we are especially interested in spatial propagation.

\paragraph*{}

In our approach, we consider the coupled dynamics of a level of \emph{activity}, denoted by $u$, representing the intensity of activity (e.g. rioting activity, fraction of population having adopted a belief or a technology, etc.), and an underlying level of \emph{susceptibility}, denoted by $v$, representing the ambient context. We emphasize that these two quantities play an asymmetric role: $u$ is thought of as the actual observed or explicit quantity while $v$ is a potential field that modulates the growth of $u$. From a modeling perspective, the level of activity $u$ often represents an explicit quantity that is tractable empirically, whereas the level of susceptibility $v$ is an implicit field. In a sense, we postulate the existence of such a field which is a lumped variable that results from several complex social interactions. Then, these two quantities interact: the field $v$ modulates the dynamics of the activity level, and there is also a feedback mechanism whereby the level of activity influences the field of susceptibility. This is why we call this general class of models the \emph{activity/susceptibility} systems.

Assuming that $u$ and $v$ are subject to diffusion and coupled reaction, the resulting mathematical model takes the following general form: 
\begin{subequations}\label{ActivityModulatorSystem}
\begin{equation}\label{ActivityModulatorSystem_eq}
\left\{\begin{aligned}
&\D_t  u -d_1 \Delta u=\Phi(u,v):=uF(u,v), \\
&\D_t  v-d_2\Delta v=\Psi(u,v):= uG(u,v)+(v_b-v)H(u,v),
\end{aligned}\right.
\end{equation}
\begin{equation}\label{ActivityModulatorSystem_init}
u(t=0,x)=u_0(x)\gneqq 0\quad;\quad v(t=0,x)\equiv v_b.
\end{equation}
\end{subequations}
We aim at keeping the assumptions on the terms in the system as general as possible.
In fact, the form of $\Phi$ and $\Psi$ given in~\eqref{ActivityModulatorSystem} is only suggested 
in this introduction to give an insight of our approach, but later on we deal with more general nonlinear terms, see  Section~\ref{sec:TheModel} below.

Let us briefly justify the structural form of the above system.  In a normal situation, i.e., in the absence of any exogenous event, we consider the system at equilibrium at some steady state $(u,v)\equiv (0,v_b)$, where the level of activity $u$ is null and the social tension $v$ is at its base value~$v_b$. Note that the special
form of $\Phi$, $\Psi$ in~\eqref{ActivityModulatorSystem} ensures that $(0,v_b)$ is indeed a steady state.
We also assume that the steady-state $(0,v_b)$ is, in a sense, weakly stable in a situation where $u= 0$, that is, $H(0,v)\geq0$.

We consider that an exogenous event occurs at $t=0$ affecting the activity and propose to study its effect on the system. This exogenous event is encoded in the initial condition $u_0(\cdot)\gneqq0$.

As we will see, the class of systems~\eqref{ActivityModulatorSystem} gives rise to
many diverse qualitative behaviors. 
This variety is illustrated by two subclasses that we will investigate in more details, namely, the \emph{tension inhibiting} systems (where $\Psi\leq 0$) which give rise to ephemeral episodes of activity, and the \emph{tension enhancing} systems (where $\Psi\geq 0$) which give rise to time-persisting episodes of activity.



%
%
%
%
%
%
%

\paragraph*{}
The $SI$ epidemiology model~\eqref{SIModel} is recovered from~\eqref{ActivityModulatorSystem} by taking
$\Phi(u,v)=\beta uv-\gamma u$ and
$\Psi(u,v)=-\beta uv,$
In this context, the \emph{Susceptible} are represented by $v$, and the \emph{Infected} are represented by $u$: this exeplifies the role of potential field assumed by $v$ --~here the susceptibles. Actually, we will see that the $SI$ belongs to the subclass of \emph{tension inhibiting} systems. The terms $\pm\beta uv$ (which could be described as law of mass action type terms) account for the contagion mechanism and derive from the assumption that the contagion is proportional to the rate of encounter between susceptible and infected individuals in a evenly mixed population. Many papers consider variants of the model where this term is replaced by a more elaborate contagion term. For example, the
Michaelis-Menten interaction assumes the existence of a saturation effect on the contagion term
and amounts to substitute $\pm\beta SI$ by $\pm\beta\frac{SI}{1+b(S+I)}$. Most of these variant fit our general framework~\eqref{ActivityModulatorSystem} and are therefore included in our study.

As further discussed in our other paper~\cite{Berestycki2019b}, a number of systems used in other modeling areas (such as propagation in excitable media, population dynamics, etc.) also fall into the setting defined by~\eqref{ActivityModulatorSystem_eq}.
When dealing with solid combustion (which is a typical example of propagation in an excitable media), one can choose $u$ to represent the temperature, $v$ to represent the chemical fuel, and assume that their dynamics is governed by~\eqref{ActivityModulatorSystem_eq} with
 $\Phi(u,v)=qF(u)v$ and $\Psi(u,v)=-F(u)v$, where $q>0$ is a constant, see~\cite{Bages2010,Larrouturou1991,Berestycki1985} and references therein. The function $F$ derives from the Arrhenius law and is typically taken to be of the form $F(u)= (u-\theta)_+ g(u) $, where $\theta\in(0,1)$ is the ignition temperature and $g$ is some postive  function involving the activation energy.
 
Another famous system that fit our framework is the classical Lotka-Volterra predator-prey model, obtained by taking (overlooking various parameters) $\Phi(u,v)=u(v-\omega)$ and $\Psi(u,v)=-uv+v(1-v)$ (with $v_b=1$) in~\eqref{ActivityModulatorSystem_eq}.
In this context, $u$ represents the density of \emph{predators} and $v$ the density of \emph{preys}. The term $\pm uv$ represents the transfer between prey and predators (through a law of mass action type term) ;  $-\omega u$ represents the natural death rate of predators ; $v(1-v)$ represents the natural birth rate and saturation effect (due, for instance, to the limitations of ressources) in the prey population.




\paragraph*{}
In our paper~\cite{Berestycki2019b}, we propose a theoretical study of~\eqref{ActivityModulatorSystem} in a general framework. In the present article, we discuss the significance of the results for modeling purposes while assuming a slightly more specific structure to the system. We also prove some new theoretical results concerning the long-time behavior of solutions in the \emph{tension inhibiting} and the \emph{tension enhancing} cases, dealing with the behaviors of solutions far from the leading edge of the front. Those results deal with both the traveling wave problem and the Cauchy problem.
We accompany our analysis with several numerical simulations which also reveal a number of interesting open questions.

To fix ideas, we place ourselves in the context of modeling of social unrest, which is a historical example and a textbook case of a propagating collective behavior. The epidemiology approach is particularly relevant in this context, as highlighted by the pioneering work of Burbeck et al.~\cite{Burbeck1978} already mentioned.
However, our approach can be envisioned to model other sociological phenomena in social sciences and population dynamics. Let us also mention that the literature on the modeling of social unrest often considers models with very particular forms, even though the quantities at stake (especially the \emph{susceptibility}, or \emph{social tension}) are not directly accessible from data. It thus seems important to us to develop a unified approach with mild assumptions on the parameters.


In this paper, our aim is to illustrate the richness of the framework and to discuss its qualitative relevance regarding the topic, while keeping the mathematical approach quite general.


\paragraph*{Outline.} 
We start with presenting, in Section~\ref{sec:FormalFeatures}, the social phenomena that we aim at describing, 
pointing out the basic sociological assumptions that lead, in Section~\ref{sec:ConstructTheModel},
to the mathematical derivation of our model.
The model and the assumptions are then stated in Sections~\ref{sec:TheModel}-\ref{sec:TravelingWaves}. In Section~\ref{sec:PreviousModels}, we discuss the existing literature on this and related topics.
Section~\ref{sec:Analysis} is devoted to a general analysis of the model. Applying the results of~\cite{Berestycki2019b}, we enlighten a threshold phenomenon on the initial level of social tension for the ignition of a social movement. We present some estimates on the speed of propagation of the movement and comment on the interpretation of the mathematical results in terms of modeling.
Next, we focus on two important classes of models: the \emph{tension inhibiting} systems (Section~\ref{sec:inhibiting}), which generate ephemeral movements of social unrest,  and the \emph{tension enhancing} systems (Section~\ref{sec:enhancing}), which give rise to time-persisting movements of social unrest.
In both cases, we present some new results about the behavior of solutions far from the 
leading edge of the propagating front, as well as for traveling wave solutions.
These results are corroborated by numerous numerical simulations.
In Section~\ref{sec:MixedCase}, we examine several mixed cases which are neither \emph{tension inhibiting} nor \emph{tension enhancing}, and that exhibit more complex dynamics.
Section~\ref{sec:SpaceHeterogeneous} deals with extensions of the model including spatial heterogeneity.
Section~\ref{sec:Proofs} contains all the proofs of our results.
 Finally, Section~\ref{sec:conclusion} is devoted to concluding remarks and perspectives.

\paragraph*{Remark on the numerical simulations.}
Numerical simulations are performed with a standard explicit Euler finite-difference scheme, with time-step $d t=0.05$, and space-step $d x= 1$. In the caption of each figure, we give a clickable URL link and the reference to a video of the simulation available online\footnote{Temporary address:
\href{https://sites.google.com/view/samuelnordmann/research/modeling-social-unrest-videos}{\color{blue}\underline{https://sites.google.com/view/samuelnordmann/research/modeling-social-unrest-videos}}\newline
Definite address to be specified in the published paper.}.

\section{The model}
\subsection{The dynamics of Social Unrest}\label{sec:FormalFeatures}
In this section, we introduce our modeling assumptions on the dynamics of social unrest in society and other collective behaviors. We do not aim at discussing the sociological origins of social unrest, which is the topic of an abundant literature and continues to be studied. Instead, we propose a model built from simple ingredients to account for recurrent patterns observed in these phenomena~\cite{Davies2013a,Bonnasse-Gahot2018}.
Our purpose here is to identify some possible features and mechanisms that land themselves to mathematical analysis. Of particular importance in this respect is the dynamical unfolding and spatial spreading of social unrest.

Our approach consists of using epidemiology models to account for the coupled dynamics of social unrest and social tension. This approach, introduced by Burbeck et al.~\cite{Burbeck1978} and further developped in~\cite{Berestycki2015,Berestycki2016a,Berestycki2018a}, turns out to be remarkably efficient to account for data from the field. In particular, a model~\cite{Bonnasse-Gahot2018} of the class we consider here reproduces rather precisely the dynamics and spreading of the French riots of 2005, which was triggered by the death of two young men trying to escape the police in Clichy-sous-Bois, a poor suburb of
Paris. This event occurred in a context of high social tension and was the spark for the riots that spread throughout the country and lasted over three weeks. 


 For literature on the modeling of social unrest, riots, and related topics, we refer the reader to~\cite{Granovetter1978,Bouchaud2013,Berestycki2015,Abudu1974} and references therein. We return in Section~\ref{sec:PreviousModels} to the existing literature and give a more detailed comparison between our model and several others.

We define the level of \emph{social unrest}, abbreviated to {\SU}, as the number of rioting activities or civil disobedience. 
We can think of {\SU} as the level of illegal actions resulting from rioting, measured in some homogeneous way (e.g. number of rioting incidents reported by the police).
Our model also features a level of Social Tension, abbreviated to {\ST}, accounting for the resentment of a population towards society, be it for political, economic, or for social reasons. This implicit quantity can be seen as the underlying (or potential) field of susceptibility for an individual to join a social movement.
The guiding principle of our approach rests on the hypothesis that {\SU} and {\ST} follow coupled dynamics.


\paragraph*{}
Let us now review the most common characteristics of the dynamics of {\SU} and define some vocabulary. 
To begin with, even if social movements can take many different forms, it appears that they often occur as episodic \emph{bursts}. 
A first simplistic classification would be to distinguish an ephemeral movement of social unrest, that we call here a ``riot'', which lasts at most a couple of weeks and then fades~(e.g. the London riots of 2011~\cite{Davies2013a,Baudains2013} or the French riot of 2005~\cite{Bonnasse-Gahot2018}), from a long-duration or persisting movement of social unrest, that we call here a ``lasting upheaval'',
which lasts longer and can result in significant political or sociological changes (e.g. the Yellow Vest Movement~\cite{Wikipedia,Morozov2019}, the Arab Spring~\cite{Lynch2013,Lang2014}, the Russian revolution of 1905–1907~\cite{Past2020}, or the French Revolution. See also~\cite{Arendt1972}).

However, most social protests are commonly considered to
have been ignited by a single \emph{triggering event}~\cite{Edmonds2011}. 

Accordingly, we assume that in a normal situation, the level of {\SU} is null and that {\ST} is at equilibrium at its base value. To account for the \emph{triggering event}, we assume that the system is perturbed at $t=0$. 

We therefore expect that whether the \emph{triggering event} ignites a \emph{burst} of {\SU} depends on the level of {\ST}.
If {\ST} is high enough, a small \emph{triggering event} triggers a \emph{burst} of {\SU}; whereas if {\ST} is low, the same event is followed by a prompt \emph{return to calm}. This threshold phenomenon is studied in the famous work of Granovetter~\cite{Granovetter1978}.

These observations suggest that, in a context of low {\ST}, an intrinsic mechanism of \emph{relaxation} occurs on {\SU}. The \emph{relaxation} rate accounts for various sociological features after a burst, such as fatigue, police repression, incarceration, etc.

On the other hand, a high {\ST} activates an \emph{endogeneous growth} of {\SU}. In other words, if {\ST} is above a threshold level, then a mechanism of {self-reinforcement} occurs on {\SU}. This is analogous to a flame propagation (an \emph{endogenous} growth is activated when the temperature is high enough) and pertains, more generally to ``excitable media''.
One can think of this \emph{endogeneous} feature as the gregarious dimension of social movements:
the larger the movement, the more prone an individual is to join it~\cite{Raafat2009,Schussman2005}.

Naturally, this self-reinforcement mechanism can be counterbalanced by a \emph{saturation} effect, accounting for the limited number of individuals, resources, goods to be damaged, etc.

Another important feature usually observed during movements of {\SU} is the \emph{geographical spread}~\cite{Braha2012,Yang2020}.
A striking example is the case of the 2005 riots in France~\cite{Bonnasse-Gahot2018,Snow2007}. This phenomenon is either caused by the rioters movement as in London~2011, or by a diffusion of the riot as in France~2005. However, the role played by the geography in the dynamics of {\ST} is less clear. For example, one could consider that {\ST} is, or is not, affected by diffusion, or even that it is affected by a non-local diffusion (see Section~\ref{sec:Ext_NonLocal}) since information nowadays is often available instantaneously through global media.

\paragraph*{}

With this vocabulary at hand, a \emph{riot} (i.e., an ephemeral movement of social unrest) 
will typically be observed in a case where the \emph{burst} of {\SU} results in a \emph{decrease} of {\ST}. 
Once {\ST} falls below a threshold value, {\SU} fades and eventually stops. 
We call this case \emph{tension inhibiting}. It is qualitatively comparable to the outburst of a disease, which propagates until the number of susceptible individuals falls below a certain threshold. This behavior is well captured by the famous $SI$ epidemiology model, with $S=$ {\ST} and $I=$ {\SU}.

On the contrary, a \emph{lasting upheaval} (i.e., a time-persisting movement of social unrest) will typically be observed in a case where the \emph{burst} of {\SU} results in an \emph{increases} of {\ST}. In this case, the dynamics escalates towards a sustainable state of high {\SU}. We call this case \emph{tension enhancing}. From a modeling point of view, it points to a cooperative system.

These two model classes give a first good idea of the variety of behaviors generated by the model. They suggest different classes of systems: epidemiology models on the one hand and monotone systems on the other hand. Those two classes of model are studied quite separately in the literature. Our aim here is to take advantage of the unified framework of~\cite{Berestycki2019b} to propose a single model able to encompass both behaviors.

 Of course, one can also consider more complex scenarios where the feedback of {\SU} on {\ST} is neither positive nor negative. This situation is included in our framework and illustrated with numerous examples later on.

\subsection{Construction of the model}\label{sec:ConstructTheModel}
We propose a mathematical model inspired from~\cite{Berestycki2015,Berestycki2016a,Berestycki2018a} to account for the dynamics of {\SU} and {\ST}.
We let $u(t,x)$ denote the level of {\SU} at given time $t$ and position $x$, 
and let $v(t,x)$ be the level of {\ST}.
We consider the general form of systems of Reaction-Diffusion equations
\begin{equation}\label{ActivityModulatorSystem_construction}
\left\{\begin{aligned}
&\D_t  u -d_1 \Delta u=\Phi(u,v),\\
&\D_t  v-d_2\Delta v=\Psi(u,v).
\end{aligned}\right.
\end{equation}
The diffusion terms $d_1 \Delta u(t,x)$ and $d_2\Delta v(t,x)$ describe the influence that one location has on its \emph{geographical neighbors}. The reaction terms $\Phi$ and $\Psi$ model the endogenous growths and feebacks of $u$ and $v$.

The level $u=0$ represents the absence of social unrest, or activity. 
The value $v_b$ stands for the {\em base value} of the social tension in the normal (quiet) regime.
We assume that the system is at equilibrium in the quiet regime, that is,
$(u\equiv0,v\equiv v_b)$ is a steady state for~\eqref{ActivityModulatorSystem_construction} 
($\Phi(0,v_b)=\Psi(0,v_b)=0$). For example, we can choose, as in~\eqref{ActivityModulatorSystem},
\begin{equation}\label{ParticularForm}
\Phi(u,v)=uF(u,v)\quad; \quad \Psi(u,v)= u G(u,v)+(v_b-v)H(u,v).
\end{equation}
We further assume that $v_b$ is a weakly stable state for the second equation when $u\equiv 0$, i.e., 
$\Psi(0,v)\geq 0$ if $v\leq v_b$ and $\Psi(0,v)\leq0$ if $v\geq v_b$. Under the particular form~\eqref{ParticularForm}, it amounts to saying that $H(0,\cdot)\geq0$ (which is the case of the $SI$ system, where $H\equiv0$).


We then introduce a \emph{triggering event}.
This corresponds to a small perturbation of the steady-state $(u=0,v=v_b)$ and is encoded in the initial condition. For clarity, we suppose that the initial perturbation only occurs on the $u$ component, that is, we take
$v_0(\cdot)\equiv v_b$. The case of more general initial conditions $v_0$ can be adapted from the results of~\cite{Berestycki2019b}). 

We choose the term $\Phi$ in the first equation of~\eqref{ActivityModulatorSystem_construction} to be of the form
\begin{equation*}
\Phi(u,v):= u\big[r(v)f(u)-\omega\big].
\end{equation*}
The term $f(u)$ represents the \emph{endogenous} factor (or self-reinforcement/saturation mechanism).
We take $f$ nonincreasing (to account for a saturation effect) and positive at $u=0$; for example, $f(u)=1-u$ or $f(u)=1$. 

The parameter $\omega>0$ is the natural rate of \emph{relaxation} of {\SU} in absence of self-reinforcement.

The endogenous factor is regulated by $r(v)$, which models the role of \emph{activator} played by {\ST}. We choose $r(\cdot)$ to be nonnegative and increasing.
We can think of this term as an on-off switch of the endogenous growth. For example, $r(\cdot)$ can be linear $r(v)=v$, or take the form of a sigmoïd
\begin{equation*}
 r(v)=\frac{1}{1+e^{(\alpha-v)\beta}},
\end{equation*}
where, $\alpha\geq 0$ is a threshold value while $\beta>0$ measures the stiffness of the transition between the relaxed state and the excited state (if we formally take $\beta=+\infty$, then $r$ is a step function which equals $0$ if $v<\alpha$ and equals $1$ if $v>\alpha$).


\paragraph*{}


Let us now describe the reaction term $\Psi$ in the second equation of~\eqref{ActivityModulatorSystem_construction}.
For the sake of clarity, we want to normalize $v$ such that it ranges in $(0,1)$; thus we will assume that $v_0\in(0,1)$,
$\Psi(u,0)\geq 0$ and $\Psi(u,1)\leq 0$.
We devote a particular attention to the  following two classes of models. 
Each case illustrates a typical qualitative behavior.
\begin{enumerate}
\item \emph{The tension inhibiting case:} $\Psi(u,v)<0$ for $u> 0$, $v\in(0,1)$. In this case, a \emph{burst} of $u$ causes a \emph{decrease} of $v$. We expect this case to give rise to an ephemeral riot and to behave comparably to the $SI$ epidemiology model~\eqref{SIModel}, in which
$$\Psi(u,v)=-\beta uv,$$
with $\beta>0$.

\item \emph{The tension enhancing case:} $\Psi(u,v)>0$ for $u> 0$, $v\in(0,1)$. In this case, a \emph{burst} of $u$ causes an \emph{increase} of $v$. We expect this case to give rise to a lasting upheaval and to behave comparably to a cooperative system (although we do not assume that $\Psi$ is monotonic). 
As an example, we can take
$$ \Psi(u,v)=uv(1-v).$$
\end{enumerate}


\subsection{The model: assumptions and notations}\label{sec:TheModel}
We consider $u(t,x)$, which stands for the level of social unrest at time $t\geq0$ and location $x\in\R^n$, and $v(t,x)$, which stands for the level of social tension, solution of
\begin{equation}\label{GeneralEquationMotivationFinal}
\left\{\begin{aligned}
&\D_t  u -d_1 \Delta u=\Phi(u,v):=u\big[r(v)f(u)-\omega\big],\\
&\D_t  v-d_2\Delta v=\Psi(u,v),\\
&u(0,x):=u_0(x),\quad v(0,x):=v_0(x).
\end{aligned}\right.
\end{equation}
Here are our standing assumptions, that will be understood throughout the paper:
\begin{enumerate}[label=\textbf{\alph*}) \,]
\item $d_1>0$, $d_2\geq0$, $\omega>0$.
\item $f(u)$ is smooth and nonincreasing on $[0,+\infty)$, with $f(0)>0$;\\
for example, $f(u)=1$ or $f(u)=1-u$.
\item $r(v)$ is smooth, nonnegative and increasing on $(0,1)$;\\
for example, $r(v)=v$.
\item $r(0)<\frac\omega{f(0)}<r(1)$, and we define
\begin{equation}\label{Def_v_star}
v_\star:=r^{-1}\left(\frac{\omega}{f(0)}\right)\in(0,1).
\end{equation}
\item $\Psi(0,\cdot)$ has a (weakly) stable zero $v_b\in(0,1)$, i.e.,
\begin{equation}\label{v-stability}
\Psi(0,v)\geq 0,\quad \forall v\in(0,v_b)\quad;\quad \Psi(0,v)\leq 0,\quad \forall v\in(v_b,1);
\end{equation}
for example, $\Psi(0,\cdot)\equiv0$.
\item
$\Psi(u,v)$ is smooth and satisfies the saturation conditions at $v=0,1$ 
 \begin{equation}\label{AssumptionSaturationV}
\Psi(u,0)\geq 0\quad \text{ and }
\quad\Psi(u,1)\leq 0\qquad\forall u\geq0.
\end{equation}
\item $u_0(x)\gneqq0$ is bounded and $v_0\equiv v_b$, where $v_b$ is the constant in~\eqref{v-stability}.
\end{enumerate}
The structure assumed on the system~\eqref{GeneralEquationMotivationFinal} is slightly more specific here than in the general framework developped in~\cite{Berestycki2019b} where no monotony is assumed and $\Phi$ can have a more generic form;
yet, our set of assumptions encompasses many diverse systems, which may be highly non-monotone and exhibit quite different qualitative behaviors, as illustrated in the sequel.

Note that if $\Psi(0,\cdot)\equiv0$  then~\eqref{v-stability} is automatically satisfied and so any $v_b\in(0,1)$ is suitable for our set of assumptions.


The following property is an immediate consequence of our assumptions.
\begin{lemma}\label{Lemma_uv}
Any solution of~\eqref{GeneralEquationMotivationFinal} satisfies
\begin{equation*}
u(t,x)>0\quad \text{and}\quad 0<v(t,x)<1,\qquad\text{ $\forall t>0$, $x\in\R^n$.}
\end{equation*}
\end{lemma}
\begin{proof}
Recall that $u_0\gneqq0$ and $v_0\in(0,1)$. 
For the range of $v$, we notice that in view of~\eqref{AssumptionSaturationV}, $v$ satisfies the inequalities
\begin{equation*}
\D_t v-d_2\Delta v-\big(\Psi(u,v)-\Psi(u,0)\big)=\Psi(u,0)\geq0,
\end{equation*}
and 
\begin{equation*}
\D_t (1-v)-d_2\Delta (1-v)-\big(\Psi(u,1)-\Psi(u,v)\big)=-\Psi(u,1)\geq0.
\end{equation*}
This implies that $v(t,x)$ has range in $(0,1)$, thanks 
to the parabolic strong comparison principle if $d_2>0$, or simple ODE considerations if $d_2=0$.
Analogously, since the constant $0$ is a solution of the first equation in~\eqref{GeneralEquationMotivationFinal}, we have that $u(t,x)>0$ for $t>0$, $x\in\R^n$.
\end{proof}

\paragraph*{}

The quantity $v_\star$ defined by~\eqref{Def_v_star} coincides with 
the value of $v$ where $\D_u\Phi(0,v)$ changes sign, i.e., 
\begin{equation*}
v_\star:= \sup\left\{v\in(0,1): r(v)\leq \frac{\omega}{f(0)}\right\}=\sup\left\{v\in(0,1):\D_u\Phi(0,v)\leq0\right\}.
\end{equation*}
We will see in the sequel that $v_\star$ is the threshold value on $v_0\equiv v_b$ which determines the regime of dynamics:
\begin{itemize}
\item if $v_b<v_\star$, a small \emph{triggering event} is followed by a \emph{return to calm},
\item if $v_b>v_\star$, a small \emph{triggering event} ignites a \emph{burst} of social unrest. This is akin to the \emph{Hair-trigger} effect in KPP equations.
\end{itemize}
The assumption \eqref{Def_v_star} allows us to 
cover both possibilities of a \emph{burst} or a \emph{return to calm} depending on the choice of $v_b\in(0,1)$.

The above dichotomy is readily revealed by the analysis of the constant steady states of~\eqref{GeneralEquationMotivationFinal}.
Consider the scalar equation (with unknown $u$)
\begin{equation}\label{SteadyStateUApplication}
\Phi(u,v_b):=u\big[r(v_b)f(u)-\omega \big]=0,\qquad u\geq0,
\end{equation}
under the assumption that
$f$ is strictly decreasing and $f(1)=0$.
Call
\begin{equation}\label{defK0}
K_b:=\D_u\Phi(0,v_b)=r(v_b)f(0)-\omega.
\end{equation}
Note that, since $r(\cdot)$ is increasing, the sign of $K_b$ coincides with that of $v_b-v_\star$. We have the following dichotomy:
\begin{itemize}
\item If $v_b< v_\star$, then $K_b<0$, and \eqref{SteadyStateUApplication} has exactly one solution $u=0$ (stable). See~\autoref{fig:SteadyState_1}.
\item If $v_b> v_\star$, then $K_b>0$, and \eqref{SteadyStateUApplication} has exactly two solutions, $u=0$ (unstable) and $u=u_\star(v_b)$ (stable), defined by
\begin{equation}\label{def:u_star}
u_\star(v):=f^{-1}\left(\frac{\omega}{r(v)}\right).
\end{equation}
See~\autoref{fig:SteadyState_2}.
Note that $v\mapsto u_\star(v)$ is continuous increasing and that 
$u_\star(v)\searrow 0$ as $v\searrow v_\star$. For example, if $f(u)=1-u$ and $r(v)=v$, then 
$v_\star=\omega$ and $u_\star(v):=1-\frac{\omega}{v}$.
\end{itemize}

\begin{figure}
\center
\begin{subfigure}[b]{0.4\linewidth}
\includegraphics[width=1\linewidth]{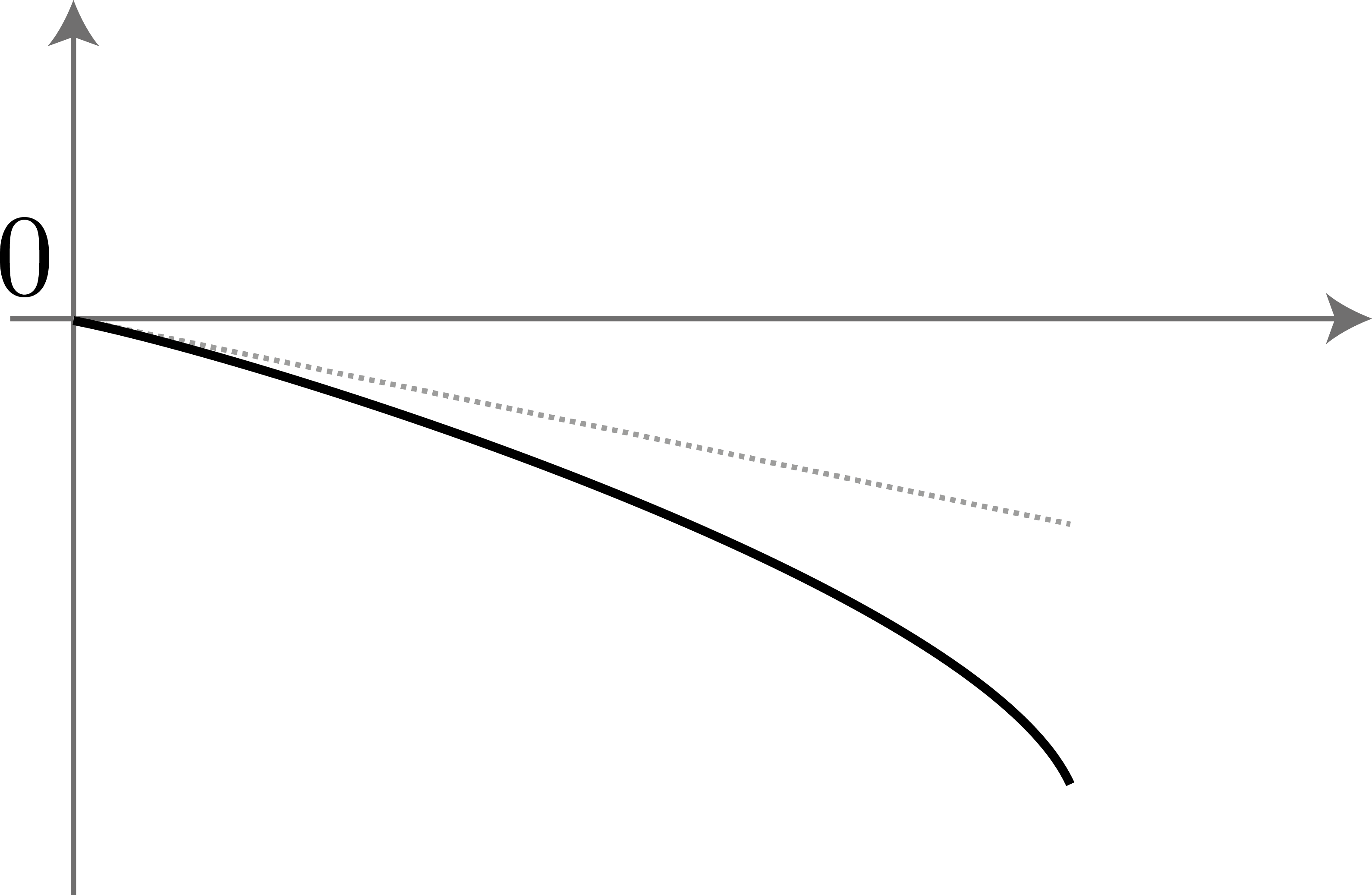}
\caption{Case $v_b<v_\star$}\label{fig:SteadyState_1}
  \end{subfigure}
  \hfill
\begin{subfigure}[b]{0.4\linewidth}
\includegraphics[width=1\linewidth]{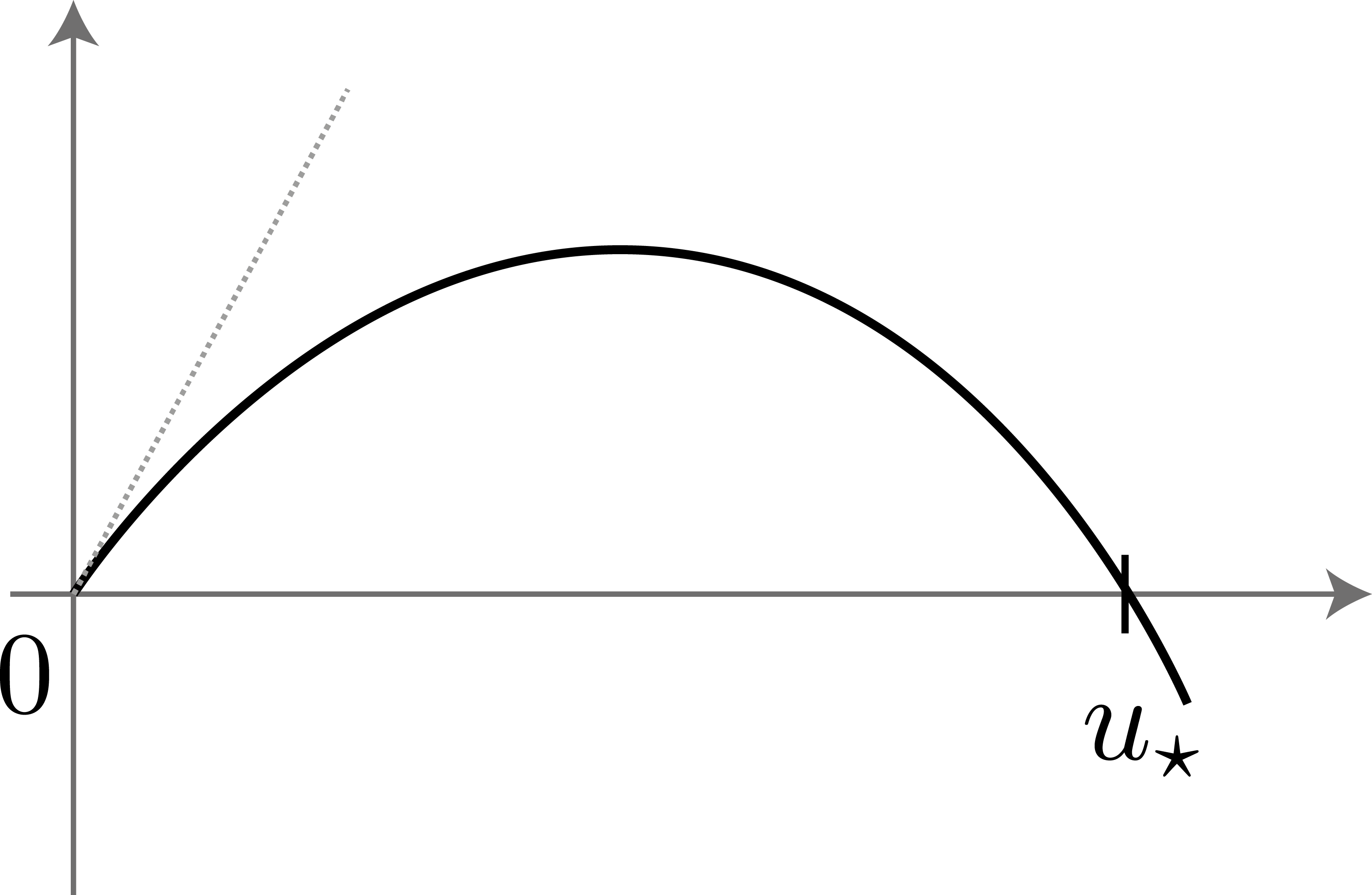}
\caption{Case $v_b>v_\star$}
  \label{fig:SteadyState_2}
  \end{subfigure}
  \caption{Graph of $u\mapsto \Phi(u,v_b)$ depending on the base level of tension $v_b$. 
  	The dashed line represents the slope at the origin (i.e.~$K_b$).}
\end{figure}

\subsection{Traveling waves}\label{sec:TravelingWaves}

It is reasonable to expect that, when a burst of social unrest occurs, the solution of~\eqref{GeneralEquationMotivationFinal} converges to a {\em traveling wave}, that is, an identical
profile moving at a constant speed. Although we do not 
prove such a result, we corroborate it with numerical evidence in the sequel. It is thus interesting to study the existence, non-existence, and the shape of traveling waves.

A traveling wave is defined as a solution of~\eqref{GeneralEquationMotivationFinal} of the form $ u(t,x)= U(x\cdot e+ct)$, $v(t,x)= V(x\cdot e+ct)$, with $c>0$, $e\in\mathbb{S}^{n-1}$ and prescribed values at $-\infty$. The profiles
 $U(\xi)$ and $V(\xi)$ thus satisfy the elliptic problem
\begin{equation}\label{SystemRiot}
\left\{\begin{aligned}
&-d_1 U''+cU'=\Phi(U,V):=U\big[r(V)f(U)-\omega\big],\\
&-d_2 V''+c V'=\Psi(U,V),\\
&c>0\quad ;\quad U>0\text{ is bounded}\quad ;\quad 0<V<1.
\end{aligned}\right.
\end{equation}
We complete it with the \emph{semi-boundary conditions} at $-\infty$
\begin{equation}\label{SystemBorderGeneralRiot}
\left\{\begin{aligned}
&U(-\infty)=0,\\
&V(-\infty)=v_b.
\end{aligned}\right.
\end{equation}
The term \emph{semi-boundary conditions} comes from the fact that we do not impose a prescribed value at $+\infty$. The traveling waves under consideration might not be unique and may take many diverse forms, such as monotone waves or bumps. This will be illustrated in the following sections. 

\subsection{Comparison with previous models and remarks}\label{sec:PreviousModels}

Our model is directly inspired by a series of papers~\cite{Berestycki2015,Berestycki2016a,Berestycki2018a} which introduce a system of Reaction-Diffusion equation, comparable to~\eqref{GeneralEquationMotivationFinal}, to model the dynamics of riots. 
As before, the quantity $u(t,x)$, depending on time $t\geq0$ and location $x\in\R^n$, represents the level of social unrest, and $v(t,x)$ represents the level of social tension.
In most cases, the model reduces to the following system
\begin{equation}\label{sys:homo}
\begin{cases}
\D_t u= d\Delta u+ ur(v)f(u)-\omega (u-u_b(x)),\\
\D_t v=  d\Delta v+S(t,x)-\theta \left(\frac{1}{(1+u)^p}v- v_b(x)\right).
\end{cases}
\end{equation}
This model has been first introduced~\cite{Berestycki2015}.

Let us describe the model~\eqref{sys:homo} and discuss the main differences with our model~\eqref{GeneralEquationMotivationFinal}.
The parameters $r(\cdot)$, $f(\cdot)$ and $\omega$ are the same as described in the previous section.
The quantity $\theta>0$ stands for the natural relaxation rate on the level of social tension to the base rate $v_b$ in absence of any rioting activity. 

The function $u_b(x)$ stands for the low recurrent rioting activity in the absence of any unusual factors. Accordingly, $v_b(x)$ denotes the base level of social tension in absence of any rioting activity. Our model~\eqref{GeneralEquationMotivationFinal} corresponds to the case where $u_b$ and $v_b$ are constant (using the change of variable $\tilde u:= u-u_b$). 

With non-constant $u_b(x)$ and $v_b(x)$, the model~\eqref{sys:homo} is spatially heterogenous. On the contrary, our model~\eqref{GeneralEquationMotivationFinal} is spatially homogeneous. The non-homogeneous setting is however an interesting perspective and is discussed in Section~\ref{sec:SpaceHeterogeneous}.

The parameter $p\in\R$ models the feedback of $u$ on $v$. If $p>0$, then a \emph{burst} of $u$ will \emph{slow down} the relaxation of $v$; if $p<0$, then a \emph{burst} of $u$ will \emph{speed up} the relaxation of $v$. If $p=0$, the system is decoupled. In~\cite{Berestycki2016a,Berestycki2018a}, the cases $p<0$ and $p>0$ are called respectively \emph{tension enhancing} and \emph{tension inhibiting}, however it does not correspond to what we call \emph{tension enhancing} and \emph{tension inhibiting} in the present work. 
Let us be more precise. Assume for simplicity that $v_b\equiv0$ and $S\equiv0$. In~\eqref{sys:homo}, the cases $p>0$ and $p<0$ model respectively a negative and a positive feedback of $u$ on $v$. However, since $-\theta \frac{1}{(1+u)^p}v$ is negative we deduce that $t\mapsto v(t,x)$ is decreasing and decays to $0$, regardless of the sign of $p$. This implies that any burst of social unrest eventually vanishes. Therefore, both the case $p>0$ and the case $p<0$ are contained in what we call \emph{tension inhibiting} in the present work (Section~\ref{sec:inhibiting}).
System~\eqref{sys:homo} does not model a situation where a burst of social unrest results in an increase of the social tension. Yet, this case is reasonable from the modeling perspective and allows to account for time-persisting movement of social unrest (see Section~\ref{sec:enhancing}).


In~\cite{Berestycki2015,Berestycki2016a,Berestycki2018a}, the source term $S(t,x)$ accounts for \emph{exogenous events}. In the present paper, we consider that a \emph{single exogenous event} occurs at time $t=0$, and so we encode it in the initial conditions.

 In~\cite{Berestycki2018a}, the authors focus on the effect of a restriction of information, which is modeled by substituting the KPP term $uf(u)$ with the combustion term $(u-\alpha)_+f(u)$, $\alpha\in(0,1)$, where the subscript $+$ denotes the positive part.
The paper~\cite{Berestycki2016a} considers~\eqref{sys:homo} without space (i.e. $d=0$) and studies the dynamics of the system for a periodic source term
\begin{equation*}
S(t):=A\sum\limits_{i\geq0}\delta_{t=iT}.
\end{equation*}
We also mention~\cite{Yang2020} in which a numerical analysis is conducted to investigate the influence of the parameters on the shape and speed of traveling waves.
The article~\cite{Berestycki2013} also proposes a model comparable to~\eqref{sys:homo} for criminal activity.

\paragraph*{}
A recent work~\cite{Petrovskii} proposes an other reaction-diffusion model to account for the dynamics of social unrest quite different in spirit from the model we discuss here. In~\cite{Petrovskii}, $u(t,x)$ represents the number of individuals that take part in the social movement. It is assumed to satisfy the equation
\begin{equation}\label{EquationPetrovskiiU}
\D_t u-d_1\D^2_{xx} u=\eps_0+\eps u+\frac{au^2}{h^2+u^2}-m(t,x)u,
\end{equation}
where $\eps_0$, $\eps$, $a$ and $h$ are positive constants.
The term $\eps_0+\eps u$ stands for the rate at which people are willing to join the social movement.
The nonlinear term $\frac{au^2}{h^2+u^2}$ accounts for a saturation effect.
The quantity $m(t,x)$ represents the rate at which individuals exit the movement. It can be thought as a field of \emph{non-susceptibility}, and is somehow opposite to $v$ in our model. It is assumed in~\cite{Petrovskii} that $m$ is either constant, or given by the explicit formula
\begin{equation*}
 m(t)= m_1+(m_0-m_1)e^{-bt},
 \end{equation*} 
or by the equation
\begin{equation*}
\D_t m-d_2\D^2_{xx} m=\beta u\qquad;\qquad m(t=0,x)=m_0.
\end{equation*}
Let us emphasize that, for some range of parameters, the equation on $u$ is of the bistable type. For example, if $\eps_0=0$ and $m$ is a constant lying in the interval $(\eps, \eps+\frac{a}{2h})$, then equation~\eqref{EquationPetrovskiiU} admits three constant steady states (two of them are stable and the third one is unstable). This differs from our model in which $u$ satisfies a monostable equation.



\paragraph*{}

Finally, we mention that many other mathematical approaches have been taken to model the dynamics of riots, protests and social unrest, such as individual-based models~\cite{Epstein2002,Societies2015}, cost/benefits analysis~\cite{Davies2013a}, or diffusion on networks~\cite{Yurevich2018}.

\section{General properties}\label{sec:Analysis}
In this section, we state general properties on the system~\eqref{GeneralEquationMotivationFinal} under the assumptions presented in Section~\ref{sec:TheModel}. These results are established in~\cite{Berestycki2019b}. Here, we apply the results of~\cite{Berestycki2019b} in our context and discuss the implications in terms of modeling. In particular, we highlight a threshold phenomenon on the initial level of social tension for a small \emph{triggering event} to ignite a movement of social unrest.

\subsection{Return to calm}\label{sec:ResumptionCalm}
First, let us observe that in a context of low social tension, a \emph{small triggering event} is followed by a \emph{return to calm}.
Mathematically speaking, if $v_b<v_\star$,
with $v_\star$ defined in~\eqref{Def_v_star}, the steady 
state $(u=0,v=v_b)$ is stable with respect to a small perturbation on $u$.
Indeed, from \cite[Theorem~1]{Berestycki2019b}, 
if $d_2>0$ and $v_0\equiv v_b<v_\star$, then any solution with $u_0$ sufficiently small and compactly supported,
satisfies
\begin{equation}\label{ResumptionOfCalm}
\lim_{t\to+\infty}\big(u(t,x),v(t,x)\big)=(0,v_b),\quad\text{uniformly in }
x\in\R^n.
\end{equation}
This has two implications from the modeling point of view. First, it means that a \emph{triggering event} with small intensity has no effect in the long run. 
Furthermore, since the convergence is uniform in space, it means that a \emph{localized triggering event} has a localized effect.

Under the same conditions, but with $d_2=0$, and $u_0<\eps$ 
small but not necessarily compactly supported, there holds that 
\begin{equation*}
\left\{\begin{aligned}
&\lim_{t\to+\infty}u(t,x)=0,\quad\text{uniformly in }
x\in\R^n,\\
&\sup\limits_{x\in\R^n}\vert v(t,x)-v_b\vert\leq C\eps,\quad\forall t\geq0,
\end{aligned}\right.
\end{equation*}
for some constant $C$ independent of $u_0$ and $\eps$. This expresses the fact that a \emph{triggering event with small intensity}, even if spread out, will have a small effect on the system.

We point out that, in the \emph{tension inhibiting case} 
(c.f. assumption~\eqref{hyp:inhibiting} below), the above results hold true for 
$u_0$ not necessarily small, see \cite[Proposition~6]{Berestycki2019b}.

\subsection{Burst of Social Unrest}\label{sec:Burst}
In contrast with the above \emph{return to calm}, when the initial level of social tension is 
sufficiently large,
an arbitrarily small \emph{triggering event} ignites a movement of social unrest. This feature is usually called a \emph{Hair-trigger effect}.

In other words, if $v_0\equiv v_b>v_\star$ 
defined by~\eqref{Def_v_star}, the steady state $(0,v_b)$ is unstable.
Namely, from \cite[Theorem~1]{Berestycki2019b}, we know that
for any $x_0\in\R^n$ and $r>0$, there holds that
\begin{equation}\label{Burst}
\limsup_{t\to+\infty}\left(
 u(t,x_0) +\sup\limits_{x\in B_r(x_0)}\vert v(t,x)-v_b\vert \right)>0.
\end{equation}
This means that even a small event is sufficient to trigger a \emph{burst} of social unrest, which will 
drive the system away from the initial condition. This can be put in contrast with the 
scenario when~$v_0\equiv v_b<v_\star$ for which \eqref{ResumptionOfCalm} occurs.

Aside from  property~\eqref{Burst}, the asymptotic behavior of the solution can be diverse, as revealed by numerical simulations presented later on in this paper.
Nevertheless, we are able to detail the picture for two important and general classes of systems: the \emph{tension inhibiting} systems or \emph{tension enhancing} systems, see Sections~\ref{sec:inhibiting}
and~\ref{sec:enhancing} respectively.

\subsection{Spatial propagation}\label{sec:SpatialPropagation}
Next, we investigate the long-range effect and the geographical spreading of a \emph{burst}.
When $v_0\equiv v_b>v_\star$ given by~\eqref{Def_v_star}, we define
\begin{equation}\label{Def_c}
c_b:=2\sqrt{d_1\left(r(v_b)f(0)-\omega\right)},\quad c_1:=2\sqrt{d_1\left(r(1)f(0)-\omega\right)}.
\end{equation}
Note that $c_b<c_1$ because $r(\cdot)$ is increasing and $v_b<1$.
Then \cite[Theorem~2]{Berestycki2019b} states that, for any $x_0\in\R^n$ and any direction $e\in\mathbb{S}^{n-1}$,
the following hold:
\begin{equation}\label{SpatialPropagation_BIS}
\forall c\in(0,c_b),\quad \exists r>0,\qquad \limsup\limits_{t\to+\infty}\left(
u(t,x_0+cte) +\sup\limits_{x\in B_r(x_0)}\vert v(t,x+cte)-v_b\vert \right)>0,\\
\end{equation}
\begin{equation}\label{SpatialPropagation}
\forall c>c_1,\qquad \lim\limits_{t\to+\infty}\left(\sup_{\vert x\vert\geq ct}\Big\vert
\big(u(t,x),v(t,x)\big)-\big(0,v_b\big)\Big\vert \right)=0.
\end{equation}
From the modeling point of view, it means that a localized triggering event 
leads to a movement of social unrest that propagates through space. 
More precisely, an observer moving at a speed 
$c$ eventually outruns the propagation if $c>c_1$, whereas, if $c<c_b$,
he will face situations away from the quiet state $(0,v_b)$.




Properties \eqref{SpatialPropagation_BIS}-\eqref{SpatialPropagation}
do not allow us to assert the existence of an  {\em asymptotic speed of propagation},
as usually intended,
because of the gap between $c_b$ and $c_1$ and also of the fact that we only have a 
``$\limsup$'' in~\eqref{SpatialPropagation_BIS} and not a ``$\liminf$''.
The gap between $c_b$ and $c_1$ is filled in the inhibiting case, because we show in~\cite[Theorem~8]{Berestycki2019b} that in this case
property~\eqref{SpatialPropagation} holds with $c_1$ replaced by $c_b$. That is, in this case, $c_b$ is the asymptotic speed of propagation. 
The numerical simulation in Section~\ref{sec:Enhancing_Numerics_Speed} shows that in the general case the 
{\em asymptotic speed of propagation} may be different from both $c_b$ and $c_1$. 
We also show in~\cite[Theorem~11]{Berestycki2019b} that, in the enhancing case, we can replace the ``$\limsup$'' with a ``$\liminf$'' in~\eqref{SpatialPropagation_BIS}. This means that an observer moving at speed $c<c_b$ 
faces an \emph{excited scenario} for all sufficiently large times.

%
%
%
%
%
%
%
%
%

\section{Tension Inhibiting - dynamics of a riot}\label{sec:inhibiting}
We consider our main equation~\eqref{GeneralEquationMotivationFinal}.
The \emph{tension inhibiting} structure relies on
 following negativity assumptions on $\Psi$:

\begin{equation}\label{hyp:inhibiting}
\begin{cases}
\Psi(u,v)<0 & \text{for all $u>0,\ v\in(0,1)$},\\
\displaystyle\limsup_{u\to+\infty}\frac{\Psi(u,v)}{u}<0 & \text{locally uniformly in $v\in(0,1]$}.
\end{cases} 	
 \end{equation}
In particular, in view of the saturation assumption~\eqref{AssumptionSaturationV},
there holds that $\Psi(u,0)=0$ for all $u\geq0$.

Assumptions~\eqref{hyp:inhibiting} essentially mean
that $u$ has a negative feedback on $v$ (however, we do not assume any monotonicity on $u\mapsto \Psi(u,v)$ as in the $SI$ model~\eqref{SIModel}).
As a direct consequence of this assumption and the parabolic comparison principle
(or ODE considerations if $d_2=0$), we have that
\begin{equation}\label{lemma:v<v0}
v(t,x)< v_0=v_b,\qquad \forall\ t>0,\ x\in\R^n.
\end{equation}
%
%
%
A typical example of a \emph{tension inhibiting} system is the $SI$ epidemiology model~\eqref{SIModel}.
This system enters the general class~\eqref{GeneralEquationMotivationFinal} by taking $r(v)=v$, $f(u)=u$, and $\Psi(u,v)=-uv$ in~\eqref{GeneralEquationMotivationFinal}. Note that the $SI$ model is inhibiting for any $v_b\in(0,1)$.



\paragraph*{}
As we will see, the \emph{tension inhibiting} case typically grasps the dynamics of a limited-duration movement of social unrest, that we call a \emph{riot}.
A good heuristics of the behavior of the model is given by a formal analysis of the underlying $ODE$ system.
Consider a constant initial datum $u_0$, hence the solutions $u(t,x)$ and $v(t,x)$ of~\eqref{GeneralEquationMotivationFinal} do not depend on $x$. 
In this case, we easily see that if the initial level of social tension $v_b$
is above the threshold $v_\star$ then the system features a \emph{Hair-trigger effect}, that is, any \emph{triggering event} (i.e., $u_0>0$) ignites a movement of social unrest. Then, from assumption~\eqref{hyp:inhibiting}, the level of social tension decreases, until it goes below than the threshold value $v_\star$. At this point, 
the level of social unrest begins to fade and eventually goes to $0$, while the level of social tension converges to some final state
smaller than the initial one.

The same qualitative behavior is observed in the context of epidemiology, where $v$ represents the number of susceptible and the number of infected. When an epidemic spreads out, the susceptible population decreases until it goes below a threshold value, after what the infection dies out.

The goal of this section is to recover, at least partially, the above properties for the PDE general system~\eqref{GeneralEquationMotivationFinal}. We begin with the 
study of traveling waves, next we present partial results for solutions of the Cauchy problem. 
We conclude this section with detailed numerical simulations.



\subsection{Traveling waves}\label{sec:Inhibiting_TravelingWave}

Let us first reclaim from~\cite{Berestycki2019b} 
the existence and non-existence results for traveling waves, i.e., solutions 
to~\eqref{SystemRiot}-\eqref{SystemBorderGeneralRiot}.
Recalling that the sign of $v_b-v_\star$ (from definition~\eqref{Def_v_star}) coincides with the one of $K_b$ (from~\eqref{defK0}), \cite[Theorem~4]{Berestycki2019b} implies that the following hold
under the inhibiting assumption~\eqref{hyp:inhibiting}:
\begin{itemize}
\item if $v_b<v_\star$, 
there exists no traveling wave;
\item if $v_b>v_\star$, there exists no traveling wave with speed $c<c_b$ and there exists a traveling wave for any speed $c>c_b$, where $c_b$ is defined in~\eqref{Def_c}.
\end{itemize}

Then, we quote from~\cite[Theorem~5]{Berestycki2019b} further qualitative properties of traveling waves, concerning
 monotonicity and identification of their limit at $+\infty$. Namely, under the inhibiting assumption~\eqref{hyp:inhibiting}, any traveling wave $(U,V)$ 
satisfies 
\begin{equation}\label{PropositionQualitativeInhibiting}
U(+\infty)=0, \qquad V(+\infty)=V_\infty,
\end{equation}
where $V_\infty\in[0, v_b)$ is a root of $\Psi(0,\cdot)$ and moreover $V_\infty\leq v_\star$. In addition,
\begin{equation}\label{PropositionQualitativeInhibitingTW}
V'<0\text{ in }\R\qquad ;\qquad \exists\xi_0\in\R \text{ such that }U'>0\text{ in }(-\infty,\xi_0)\text{ and }U'<0\text{ in }(\xi_0,+\infty).
\end{equation}
 Roughly speaking, $U$ has the shape of a bump and $V$ has the shape of a monotone wave.


Recall that assumptions~\eqref{AssumptionSaturationV} and~\eqref{hyp:inhibiting}
yield $\Psi(0,0)=0$, hence $\Psi(0,\cdot)$ has at least one root.
If $\Psi(0,\cdot)$ has many roots (e.g., if $\Psi(0,\cdot)\equiv0$ as in the $SI$ model~\eqref{SIModel}),
an important issue is the identification of the limiting state $V_\infty$ among them. 
From a modeling perspective, the quantity $v_b-V_\infty$ measures the amount of social tension 
dissipated by the movement of social unrest. In the context of epidemiology, $v_b-V_\infty$ represents the total number of individuals that has been infected during the epidemic.
%
As stated previously, we have the a priori bound~$V_\infty\leq v_\star$ from~\cite[Theorem~5]{Berestycki2019b}.
To derive a more precise estimate on $V_\infty$, one can integrate the equation on $V$ (and use $V'(\pm\infty)=0$) to find 
\begin{equation*}
c(V_\infty-v_b)=\int_{-\infty}^{+\infty}\Psi(U,V).
\end{equation*}
Analogously, from $U(\pm\infty)=U'(\pm\infty)=0$, integrating the equation on $U$ gives $$\int_{-\infty}^{+\infty}U\big[r(V)f(U)-\omega \big]=0.$$
However, this formula is not explicit. 
In some particular cases (including the $SI$ model~\eqref{SIModel} with $d_2=0$), one can obtain
a rather explicit expression for $V_\infty$.
\begin{proposition}\label{PropVInftyLeqVStar}
Let $(U,V)$ be a solution~\eqref{SystemRiot}-\eqref{SystemBorderGeneralRiot} with $d_2=0$, $f\equiv 1$ and $\Psi(U,V)=UG(V)$, that is
\begin{equation*}
\left\{\begin{aligned}
&-d_1 U''(\xi)+cU'(\xi)=U\big(r(V)-\omega\big),\\
&c V'(\xi)=UG(V),\\
\end{aligned}\right.
\end{equation*}
and assume that $G<0$ on $(0,1)$.
Then, letting $Q$ be a primitive of $v\mapsto\frac{\omega-r(v)}{G(v)}$, there holds that
\begin{equation*}
Q(V_\infty)=Q(v_b).
\end{equation*}
In particular, $V_\infty$ does not depend on $c$, nor on $d_1$.
\end{proposition}

In the particular case of the $SI$ model~\eqref{SIModel} with $d_2=0$, we have 
\begin{equation*}
Q(v)= \frac{1}{\beta}(v-\omega\ln(v)).
\end{equation*}
See \autoref{fig:PlotQ}.
In this case, $V_\infty$ is uniquely determined by the conditions
\begin{equation}\label{V_infty_exact}
\left\{\begin{aligned}
&Q(V_\infty)=Q(v_b),\\
&V_\infty\leq v_\star.
\end{aligned}\right.
\end{equation}
The numerical values of $V_\infty$ as a function of $v_b$ are plotted in \autoref{fig:PlotV_Infty_exact}
for $\beta=\frac{1}{2}$.
We see that $V_\infty$ is a decreasing convex function of $v_b\in(v_\star,+\infty)$. From a modeling perspective, it means that the higher the initial level of social tension, the lower the final level of social tension after the burst of a riot.
\begin{figure}
\centering
\begin{subfigure}[b]{0.5\linewidth}
\includegraphics[width=\textwidth]{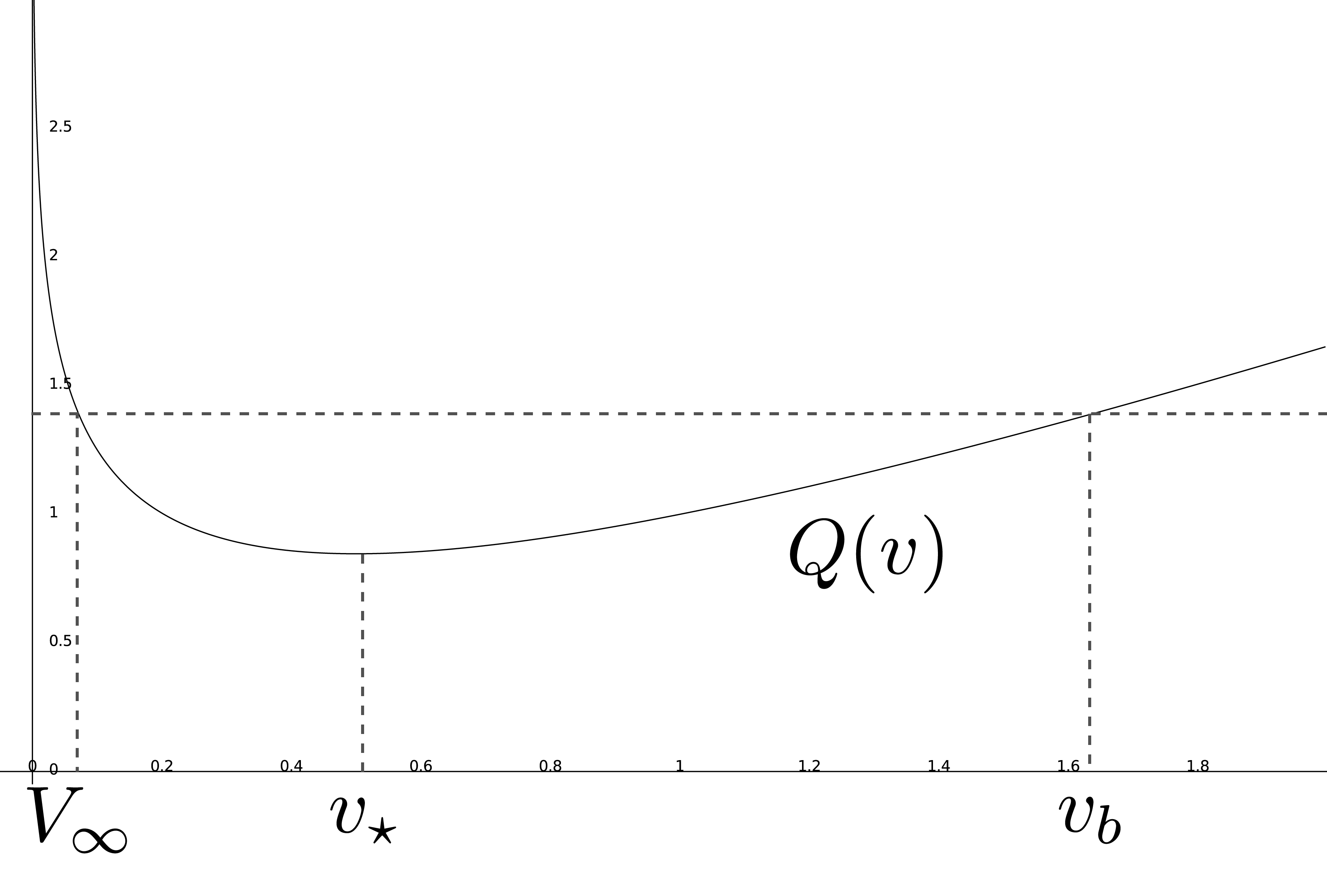}
\caption{Graph of $Q(v)=v-\frac{1}{2}\ln(v)$ (solid line) and graphic construction of the solution of~\eqref{V_infty_exact} (dashed lines) \label{fig:PlotQ}}
\end{subfigure}
\hfill
\begin{subfigure}[b]{0.45\linewidth}
\includegraphics[width=\textwidth]{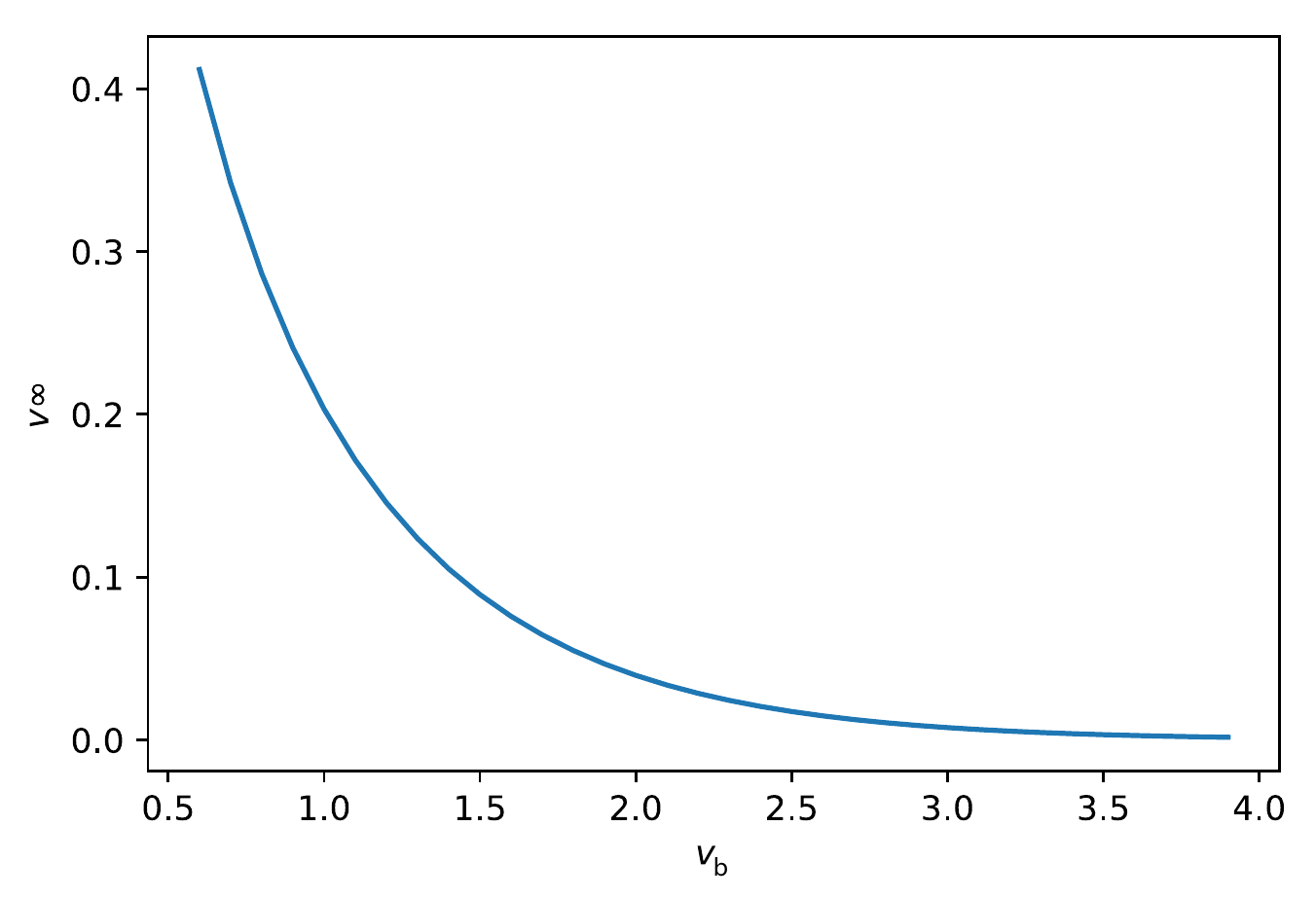}
\caption{$V_\infty$ solution of~\eqref{V_infty_exact} as a function of $v_b$, for $Q(v)=v-\frac{1}{2}\ln(v)$.}\label{fig:PlotV_Infty_exact}
\end{subfigure}
\caption{}
\end{figure}

\subsection{Large time behavior for the Cauchy problem}

We now turn to the more intricate question of studying the 
large time behavior of solutions to~\eqref{GeneralEquationMotivationFinal}. Compared with the previous section on traveling waves, we are only able to establish partial results. We 
have seen in Section~\ref{sec:ResumptionCalm} that when $v_0\equiv v_b<v_\star$,
perturbations of the steady state 
$(0,v_b)$ (not necessarily small) eventually disappear as $t\to+\infty$.
On the contrary, if $v_0\equiv v_b>v_\star$, perturbations do not tend to $0$ as $t\to+\infty$,
at least in one of the $(u,v)$ components, see Section~\ref{sec:Burst}.
However, solutions can exhibit many  qualitatively diverse behaviors in general.
We are able to obtain some informations in the inhibiting case.

\begin{theorem}\label{PropositionAsymptoticsInhibiting}
Assume that the inhibiting hypothesis~\eqref{hyp:inhibiting} holds, that $d_2>0$, and that the dimension is
$n=1$ or $2$.
Then any solution of~\eqref{GeneralEquationMotivationFinal} satisifes
\begin{equation}\label{limifuto0}
\liminf_{t\to+\infty}u(t,x)=0\quad\text{locally uniformly in $x\in\R^n$}.
\end{equation}

If in addition $v(t,x)$ converges pointwise to some $v_\infty(x)$ as $t\to+\infty$,
then $v_\infty$ is a constant in $[0,v_\star]$ and
	$u(t,x)\to0$ as $t\to+\infty$ locally uniformly in $x\in\R^n$.
\end{theorem}
The first part of \autoref{PropositionAsymptoticsInhibiting} is a consequence of~\cite[Proposition~7]{Berestycki2019b} and states that the movement of social unrest vanishes along some sequences of time. We establish in the second part of \autoref{PropositionAsymptoticsInhibiting} that, if we assume that $v$ converges as $t\to+\infty$, $u$ converges to $0$. This means that 
only the $v$ component matters in the estimate~\eqref{Burst} for the case $v_b>v_\star$.

As discussed in the previous section about the traveling waves, an interesting question is to estimate the final state $v_\infty$.
The following result concerns the homogeneous case, i.e., when $u_0(\cdot)$
is constant.
\begin{proposition}\label{Propo_Moyenne_Inhib}
Assume that \eqref{hyp:inhibiting} holds.	
Let $(u,v)$ and $(\tilde u,\tilde v)$ be the solutions of~\eqref{GeneralEquationMotivationFinal}
with initial conditions $(u_0,v_0)$ and $(u_0,\tilde v_0)$ respectively, where 
$u_0$ is constant and $v_0\equiv v_b$, $\tilde v_0\equiv \tilde v_b$. If $v_b>\tilde v_b>v_\star$,
the corresponding
final states $v_\infty$ and $\tilde v_\infty$ satisfy
\begin{equation*}
v_\infty\leq \tilde v_\infty<v_\star.
\end{equation*}
Moreover,
\begin{equation*}
\max_{t>0} u(t)>\max_{t>0}\tilde u(t).
\end{equation*}
\end{proposition}
The above proposition states that the final level
$v_\infty$ is decreasing with respect to $v_b$. From the modeling point of view, 
it implies that a higher initial level of social tension will lead to a lower final level of social tension. 
On the contrary, the higher the initial level of social tension, the higher the maximal level of social unrest.
This enlightens very well the non monotone structure of the inhibiting case.

%

\subsection{Numerical simulations}\label{sec:Inhibiting_numerics}
Let us illustrate the dynamics of~\eqref{GeneralEquationMotivationFinal} in the inhibiting case~\eqref{hyp:inhibiting} with numerical simulations.

\subsubsection{Threshold between calm and riot}
Consider the following particular instance of~\eqref{GeneralEquationMotivationFinal}:
\begin{equation}\label{ExampleInhibiting}
\left\{\begin{aligned}
&\D_t  u -\D_{xx} u=u\left[5v(1-u)-\frac{1}{2}\right],\\
&\D_t  v-\D_{xx} v=-uv.
\end{aligned}\right.
\end{equation}
As for the initial datum, we take $u_0(x)=0.2(1-\frac{x^2}{100})_+$, where $(\cdot)_+$ denotes the positive
 part, and $v_0\equiv v_b$. Observe that any $v_b\in(0,1)$ is allowed by the stability 
 hypothesis~\eqref{v-stability}.
 System~\eqref{ExampleInhibiting} 
 satisfies the inhibiting assumption~\eqref{hyp:inhibiting}.  
The quantities defined in~\eqref{Def_v_star} and~\eqref{defK0} 
are given by $v_\star= \frac{1}{10}$ and~$K_b= 5v_b-\frac{1}{2}$.

First, if $v_b<v_\star$ we observe in \autoref{fig:Inhib_Calm}
that a triggering event is promptly followed by a return to calm, i.e., $u$ rapidly vanishes. Next, 
$v$ converges in long time to its initial value $v_b$. This is in agreement with the
property~\eqref{ResumptionOfCalm}.

\begin{figure}[p]
\center
\begin{subfigure}[p]{0.4\linewidth}
\includegraphics[scale=0.5]{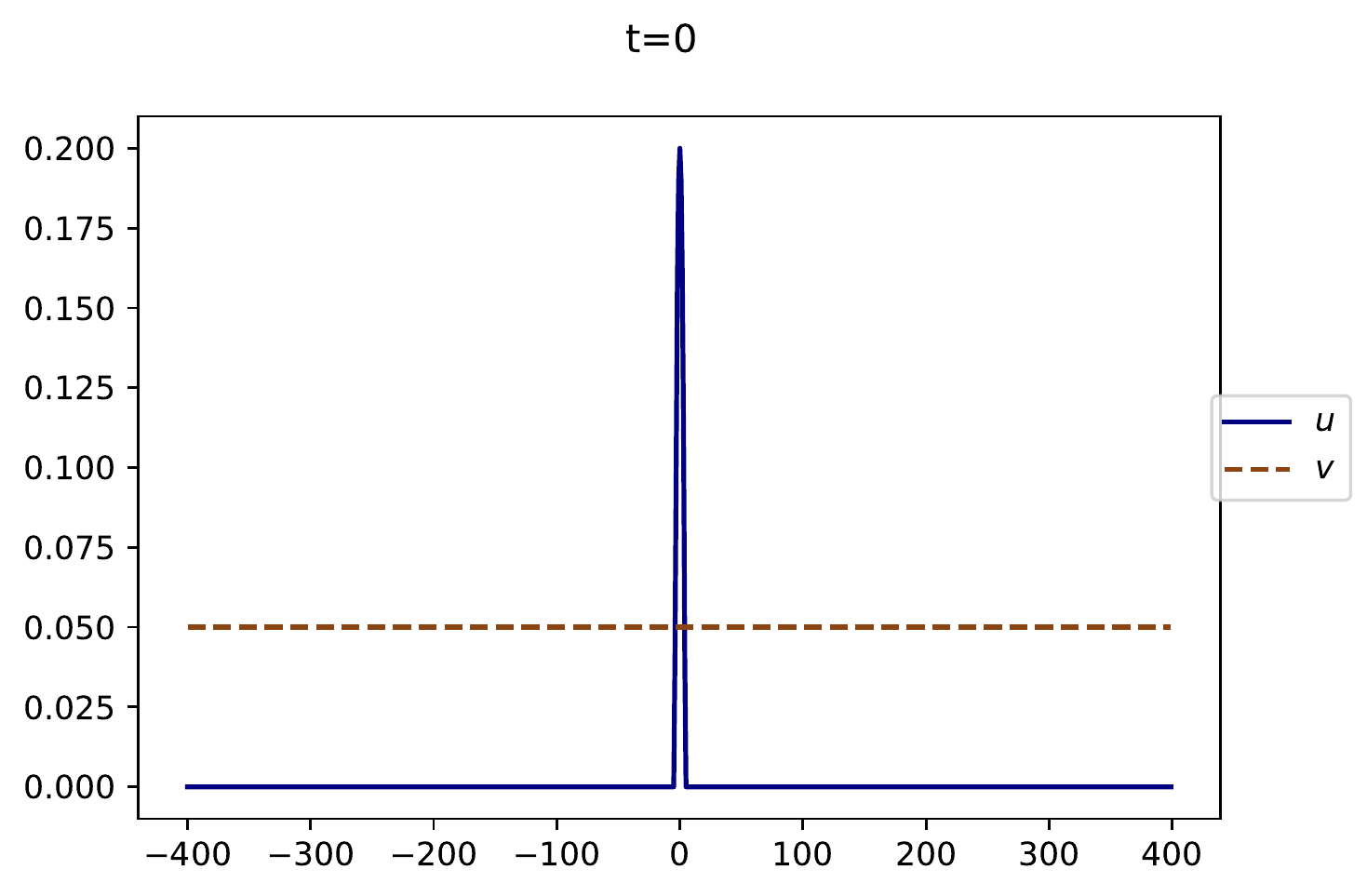}
  \end{subfigure}
  \hfill
\begin{subfigure}[p]{0.4\linewidth}
\includegraphics[scale=0.5]{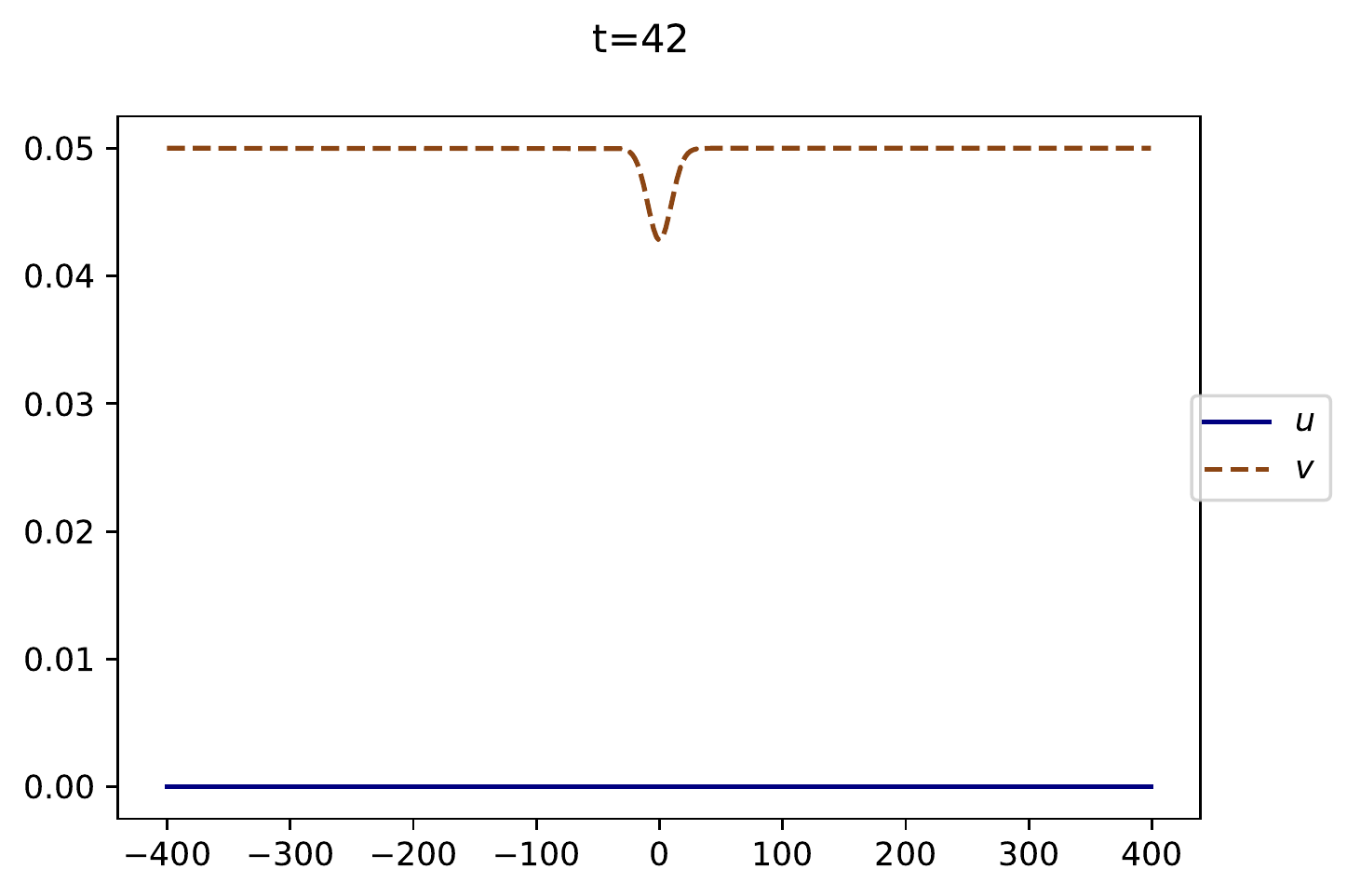}
  \end{subfigure}
  \\
  \begin{subfigure}[p]{0.4\linewidth}
\includegraphics[scale=0.5]{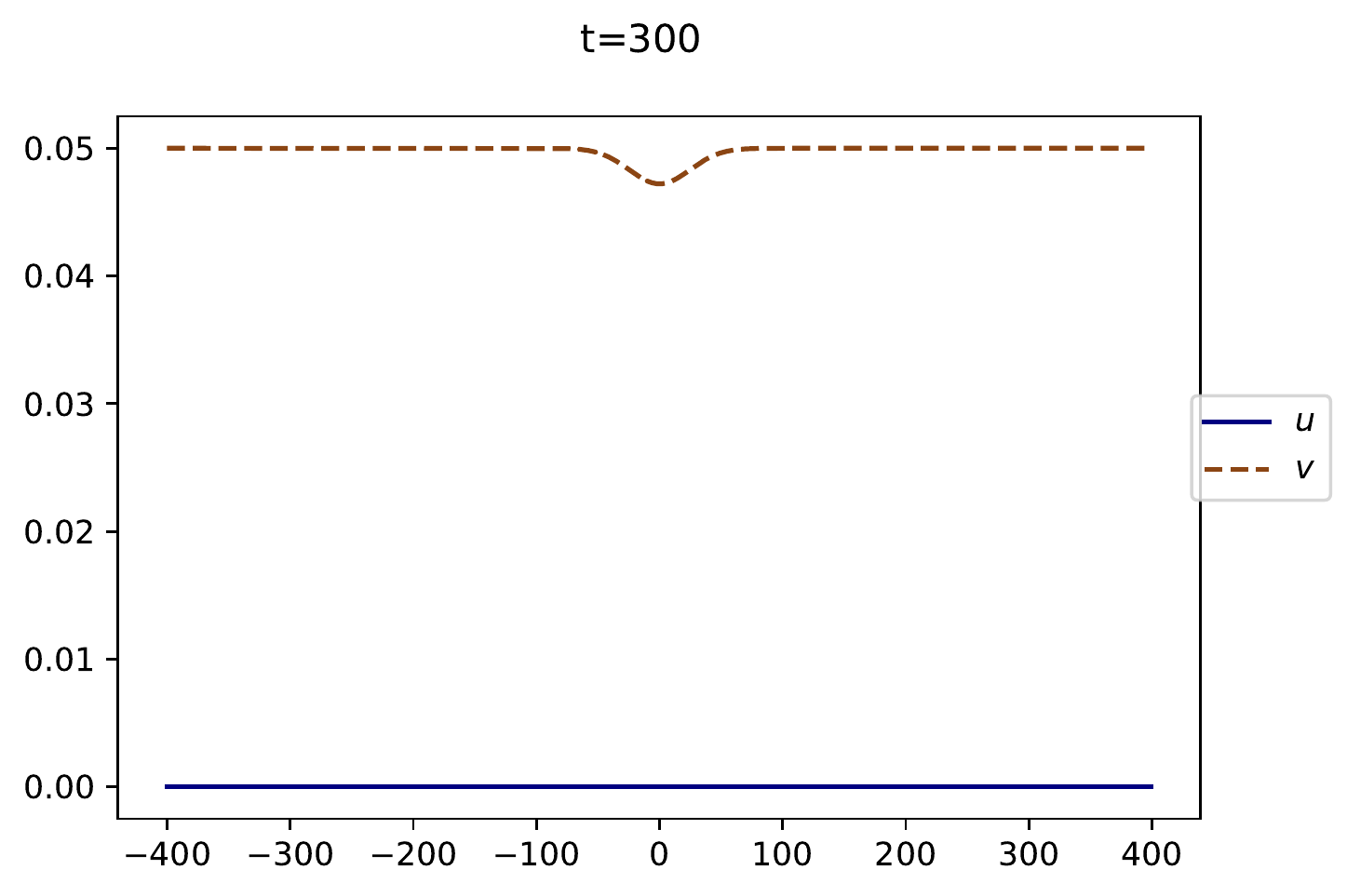}
  \end{subfigure}
  \caption{Inhibiting case -- return to calm. Snapshots at different times of the solution of~\eqref{ExampleInhibiting} with $v_b=0.05<v_\star$. Horizontal axis: space. Blue solid line: $u(t,\cdot)$. Brown dashed line: $v(t,\cdot)$.
\href{https://drive.google.com/file/d/1VHDMQr8X2sSdSES1CO3KRrI3CU7sGBD4/view?usp=sharing}{\color{blue}\underline{Video: Inhibiting\_v0=005.mp4}}  
  }
  \label{fig:Inhib_Calm}
\end{figure}

On the contrary, when $v_b=0.15>v_\star$, \autoref{fig:Inhib_Riot} shows that the solution does not vanish uniformly in space, but rather develops two traveling waves --one leftward and the other rightward--
i.e., two fixed profiles propagating at constant speed. This is in agreement with the discussion in Sections~\ref{sec:Burst}, \ref{sec:SpatialPropagation}. 
The profile of $u$ has the shape of a bump, while the profile of $v$ is a decreasing wave, as stated in~\eqref{PropositionQualitativeInhibiting}-\eqref{PropositionQualitativeInhibitingTW}. 
The fact that the level of {\SU} decays as $t\to+\infty$, which was expected from~\autoref{PropositionAsymptoticsInhibiting},
means that social movements get extinct after some times. Thus, the dynamics describe a limited-duration movement of social unrest, that we call a \emph{riot}.

\begin{figure}[p]
\center
\begin{subfigure}[p]{0.4\linewidth}
\includegraphics[scale=0.5]{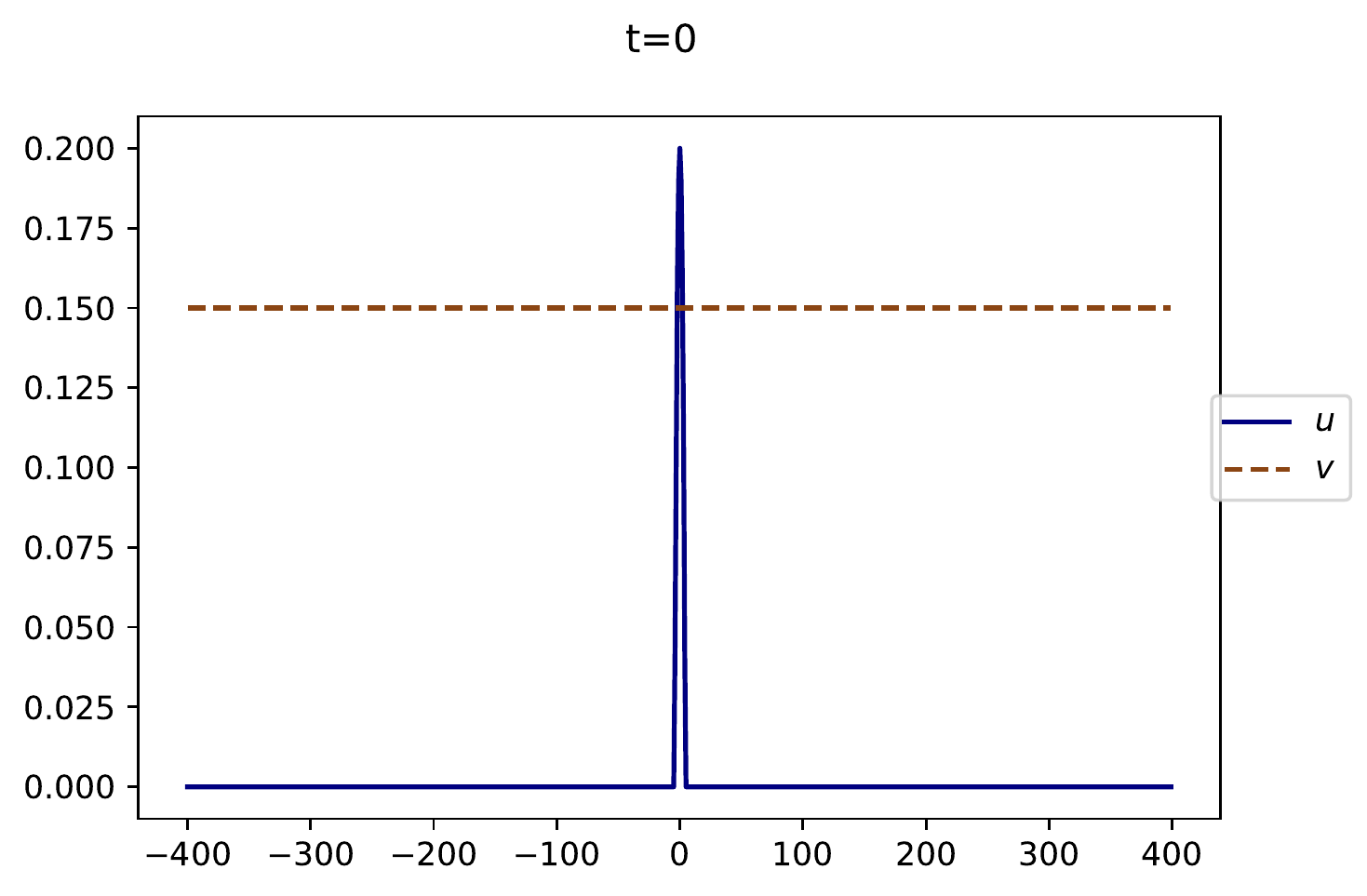}
\label{fig:Inhib_Riot_1}
  \end{subfigure}
  \hfill
\begin{subfigure}[p]{0.4\linewidth}
\includegraphics[scale=0.5]{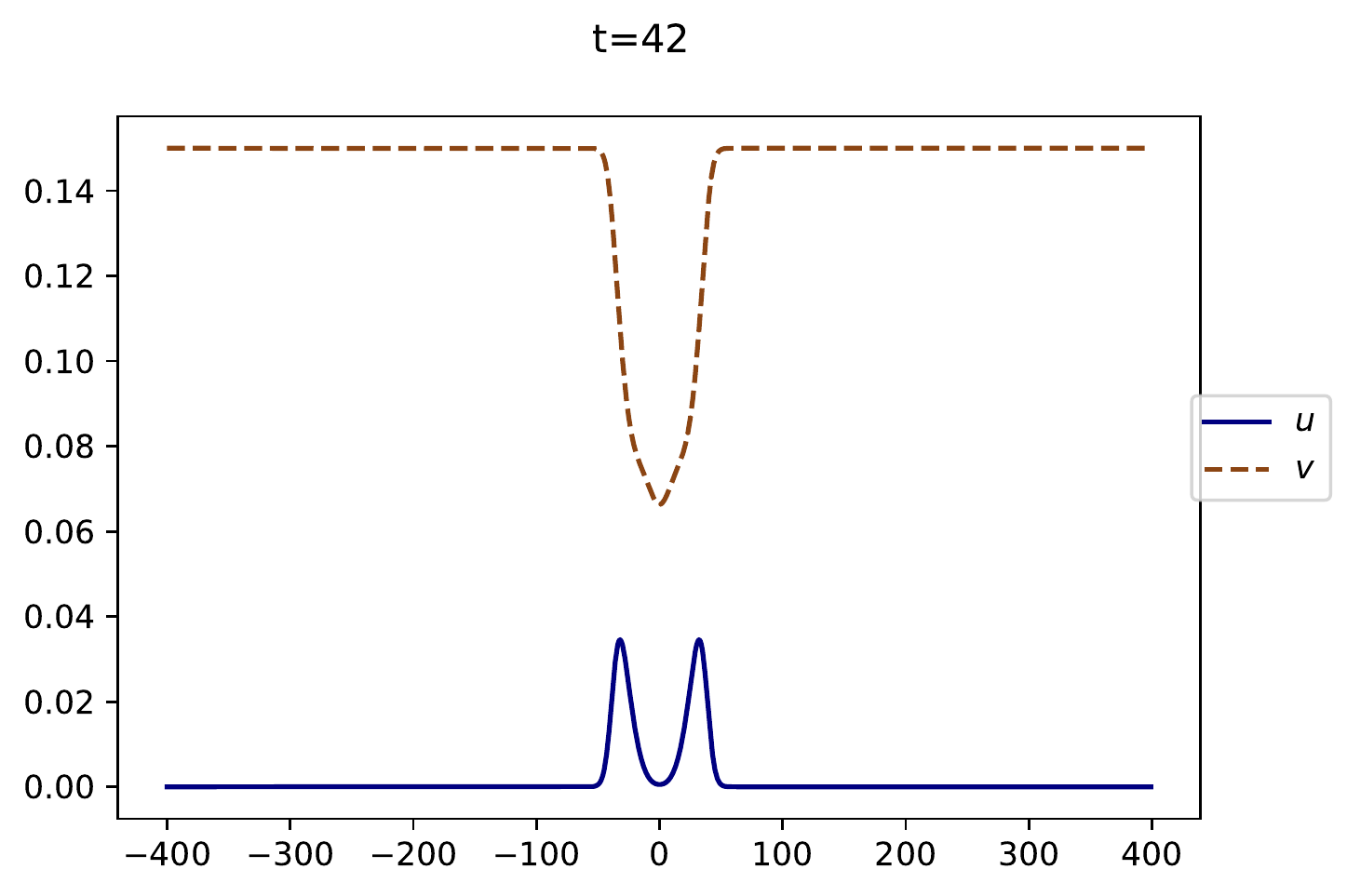}
\label{fig:Inhib_Riot_2}
  \end{subfigure}
  \\
  \begin{subfigure}[p]{0.4\linewidth}
\includegraphics[scale=0.5]{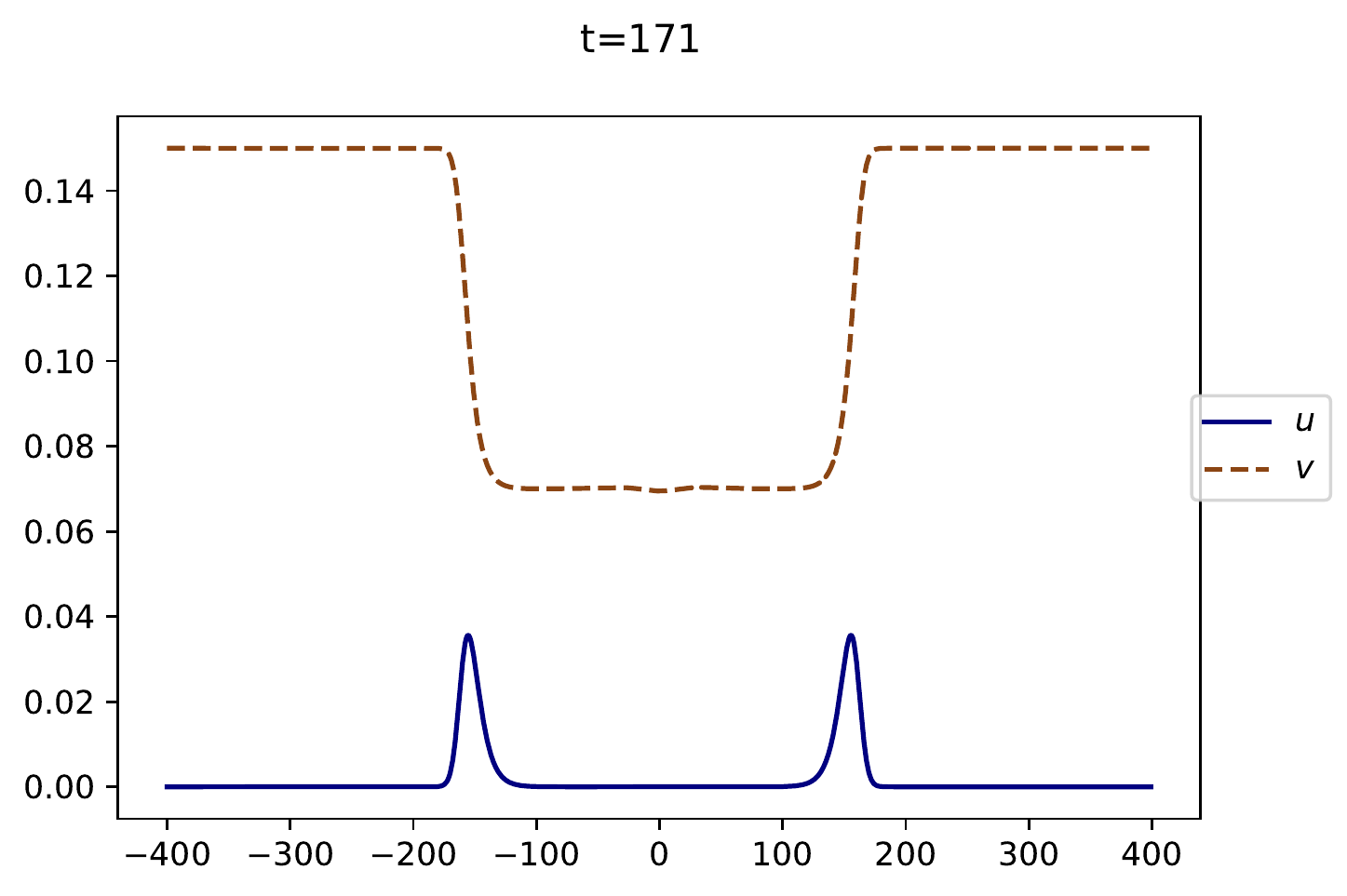}
\label{fig:Inhib_Riot_3}
  \end{subfigure}
    \hfill
\begin{subfigure}[p]{0.4\linewidth}
\includegraphics[scale=0.5]{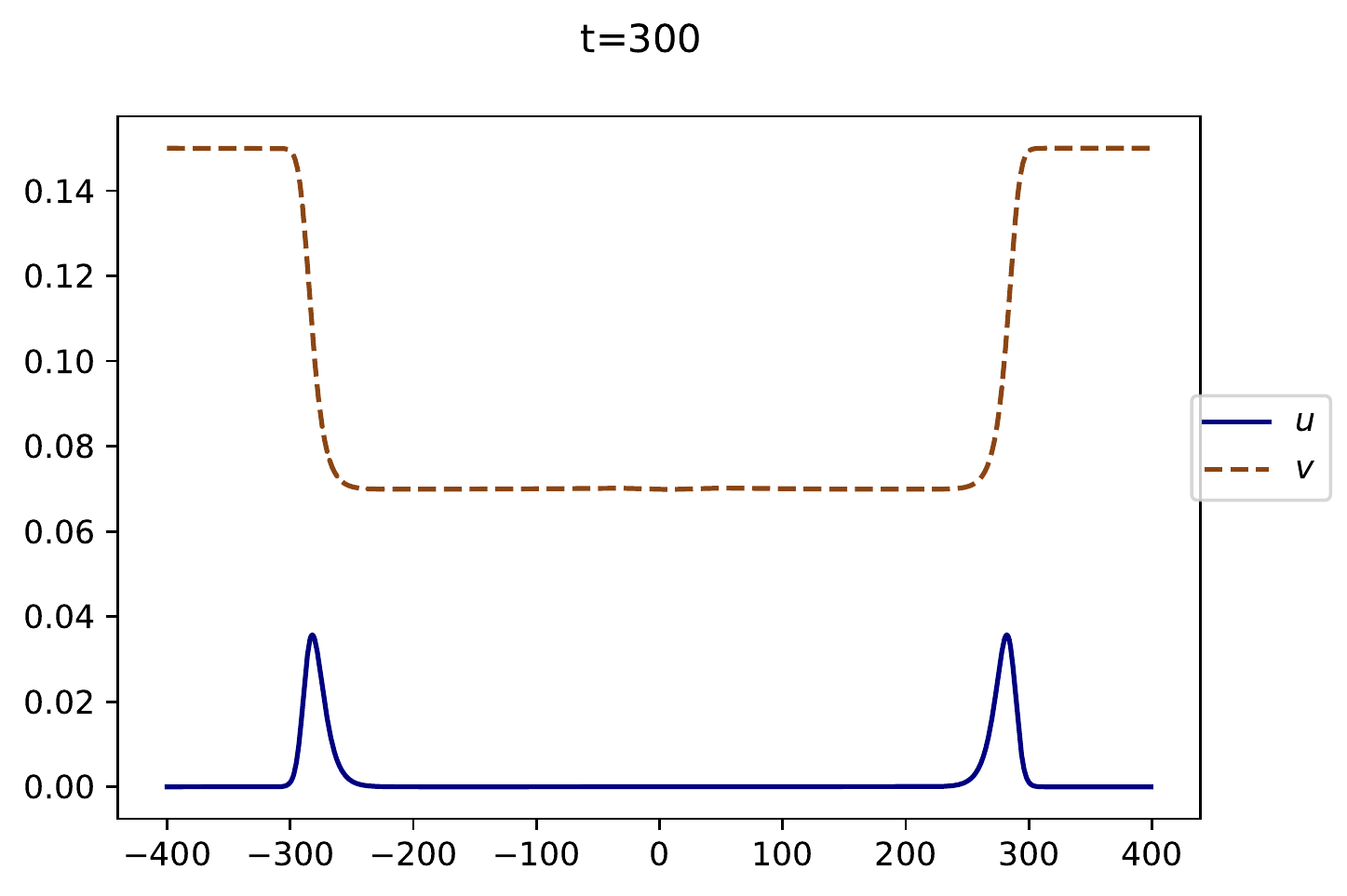}
\label{fig:Inhib_Riot_4}
  \end{subfigure}
  \caption{Inhibiting case -- riot. Snapshots at different times of the solution of~\eqref{ExampleInhibiting} with $v_b=0.15>v_\star$. Horizontal axis: space. Blue solid line: $u(t,\cdot)$. Brown dashed line: $v(t,\cdot)$. 
\href{https://drive.google.com/file/d/1XgImBPjwFPVgXAY87n9aR1M5XFt938e7/view?usp=sharing}{\color{blue}\underline{Video: Inhibiting\_v0=015.mp4}}  
  }  
  \label{fig:Inhib_Riot}
\end{figure}

\subsubsection{Speed of propagation}\label{sec:InhibSpeedNumerics}
Let us investigate in more details the dependence of the speed of propagation with respect to $v_0\equiv v_b$ on the system
\begin{equation}\label{ExampleInhibiting2}
\left\{\begin{aligned}
&\D_t  u -\D_{xx} u=u\left[v(1-u)-\frac{1}{3}\right],\\
&\D_t  v-\D_{xx} v=-uv.
\end{aligned}\right.
\end{equation}
We consider $u_0(x)=0.2(1-x^2)_+$ as an initial datum for $u$, then we will study the 
asymptotic speed of propagation as the initial datum for $v$, 
$v_0\equiv v_b$, varies in~$(0,1)$.

Let us briefly explain how we numerically compute the speed of propagation.
For each simulation and two given times $t_1=99$ and $t_2=399$, we track the leftmost locations $x_i$ where the solution $u(t_i,\cdot)$ reaches half its supremum, i.e., the value $\frac{1}{2}\sup_{x\in\R}u(t_i,x)$. See \autoref{fig:Inhib_Speed_Explanation}.
\begin{figure}[p]
\center
\includegraphics[width=0.6\linewidth]{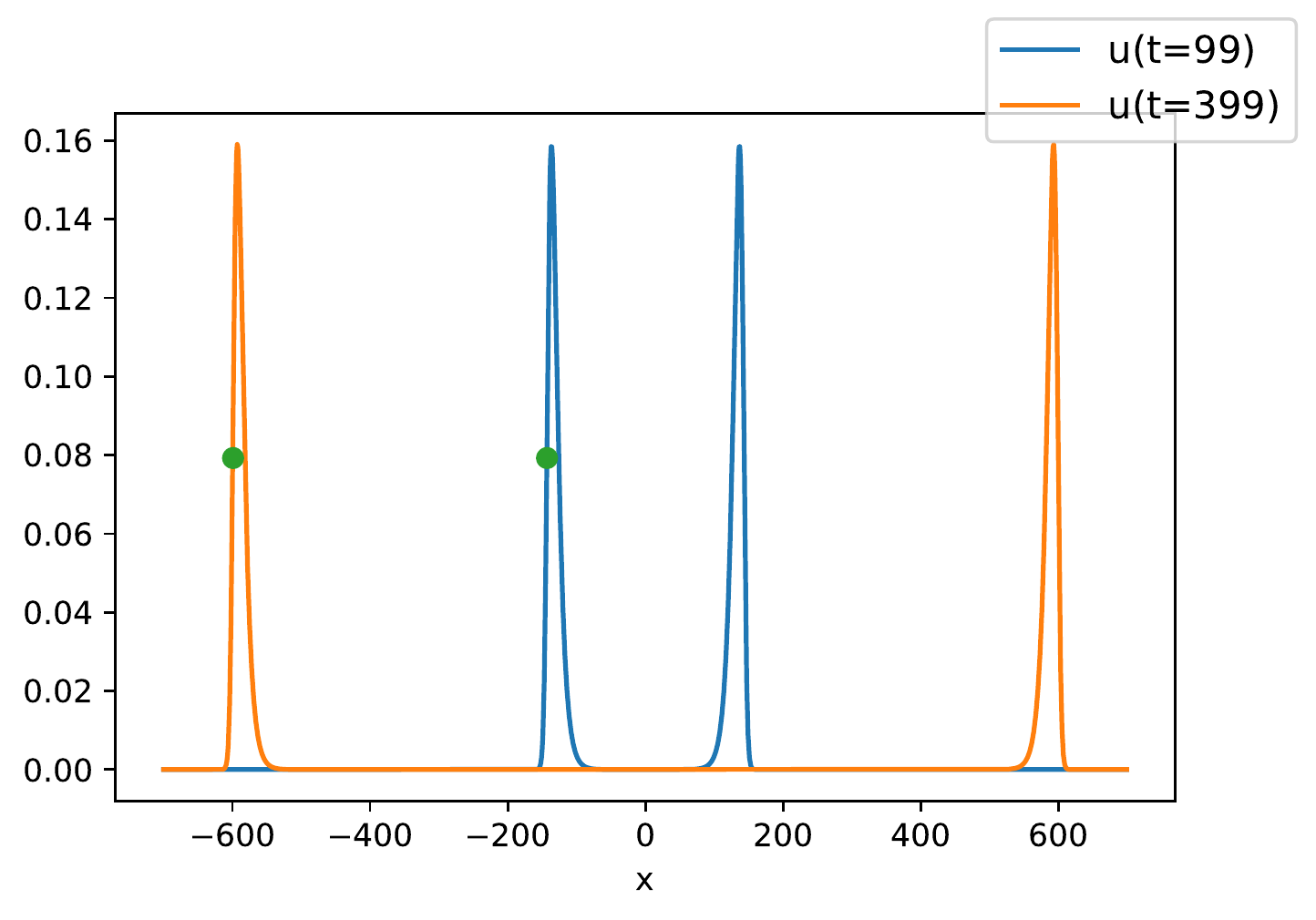}
\caption{Illustration of the method to compute the speed numerically. The horizontal coordinates of the green circles represent respectively the leftmost locations $x_1$ and $x_2$ where $u(t_1,\cdot)$ and $u(t_2,\cdot)$ reach half their supremum. The simulation is performed on~\eqref{ExampleInhibiting2} with $v_b=0.9$.\\ 
\href{https://drive.google.com/file/d/1oCFOPZqdSGBF0MQox1R4ElEGj1Vdajkm/view?usp=sharing}{\color{blue}\underline{Video: Speed\_v0=09.mp4}}  
}
\label{fig:Inhib_Speed_Explanation}
 \end{figure}
Then, we compute the speed as
\begin{equation*}
c=\frac{x_2-x_1}{t_2-t_1}.
\end{equation*}
We recall that, according to Section~\ref{sec:SpatialPropagation}, the speed of propagation of the solution is equal to 
$$c_b:=2\sqrt{v_b-\frac{1}{3}}.$$ 
We plot in~\autoref{fig:Inhib_Speed_Graph} the computed speed $c$ and the theoretical speed $c_b$, and we see that the two speeds are indeed equal.
\begin{figure}[p]
\center
\includegraphics[width=0.6\linewidth]{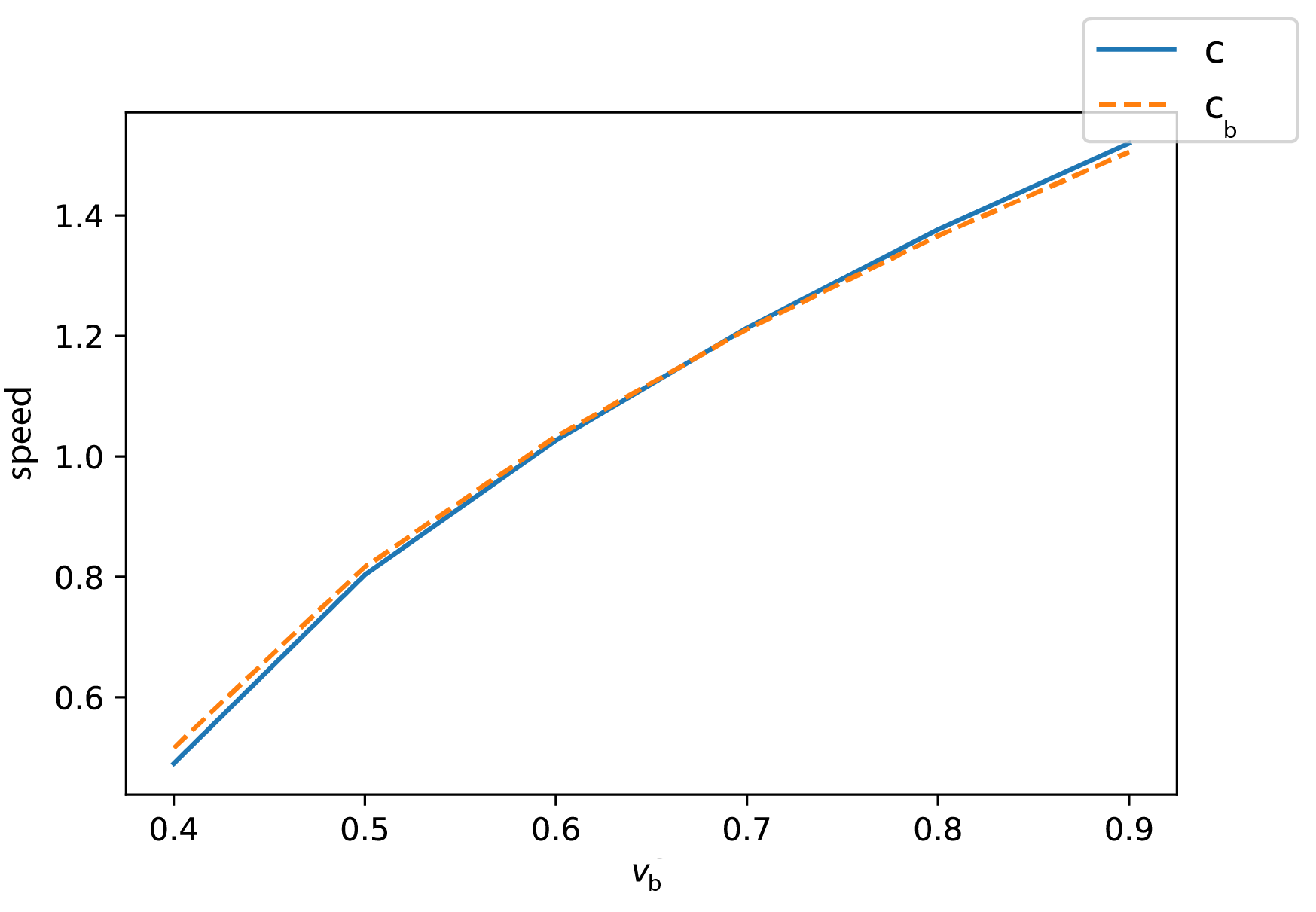}
\caption{Inhibiting case -- speed of propagation as a function of $v_b$. Blue solid line: empirical speed $c$ computed numerically via the solution of~\eqref{ExampleInhibiting2}. Dashed orange line: theoretical speed $c_b=2\sqrt{v_b-\frac{1}{3}}$.}
\label{fig:Inhib_Speed_Graph}
 \end{figure}

\subsubsection{Eventual level of social tension}

As we can see in~\autoref{fig:Inhib_Riot} and \autoref{fig:Inhib_RiotBIS}, the solution $(u,v)$ converges to some $(0,v_\infty)$ as $t\to+\infty$. We can verify numerically that $v_\infty\leq v_\star$, 
which is in agreement with \autoref{PropositionAsymptoticsInhibiting}.

It is an interesting question to estimate more precisely the final level of social tension $v_\infty$. 
Theoretical results in this direction are given by Propositions~\ref{PropVInftyLeqVStar} and~\ref{Propo_Moyenne_Inhib}.
Let us now study this question with numerical simulations.

We plot in figure~\autoref{fig:Inhib_RiotBIS} the simulation of 
equation~\eqref{ExampleInhibiting} with a high initial level of social tension $v_b=0.9$.
Compared with the simulation for $v_b=0.15$, in~\autoref{fig:Inhib_Riot}, we observe that $u$ reaches higher values, that its shape is sharper, and that the speed of propagation is larger.

\begin{figure}[p]
\center
\begin{subfigure}[p]{0.4\linewidth}
\includegraphics[scale=0.5]{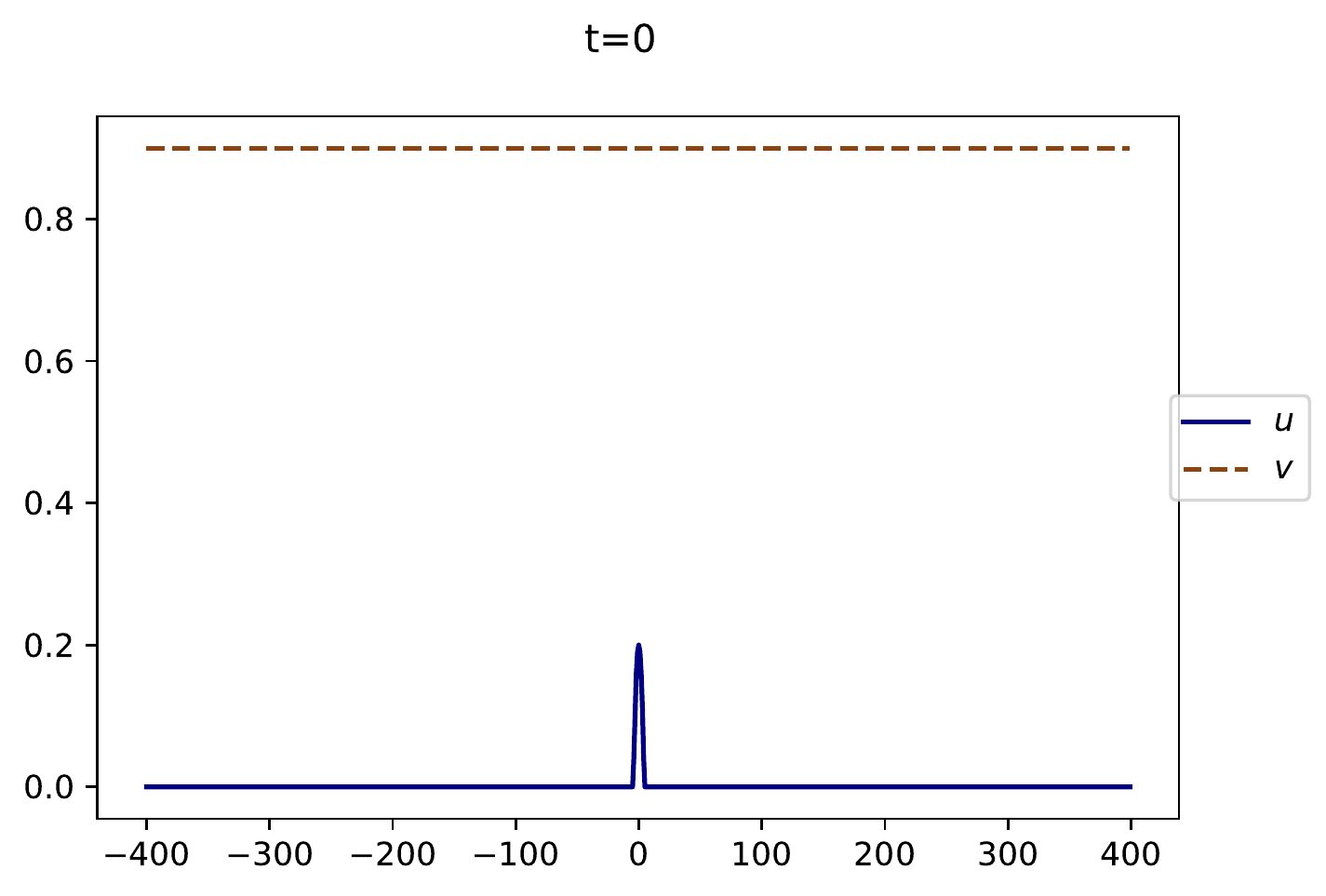}
\label{fig:Inhib_RiotBIS_1}
  \end{subfigure}
  \hfill
\begin{subfigure}[p]{0.4\linewidth}
\includegraphics[scale=0.5]{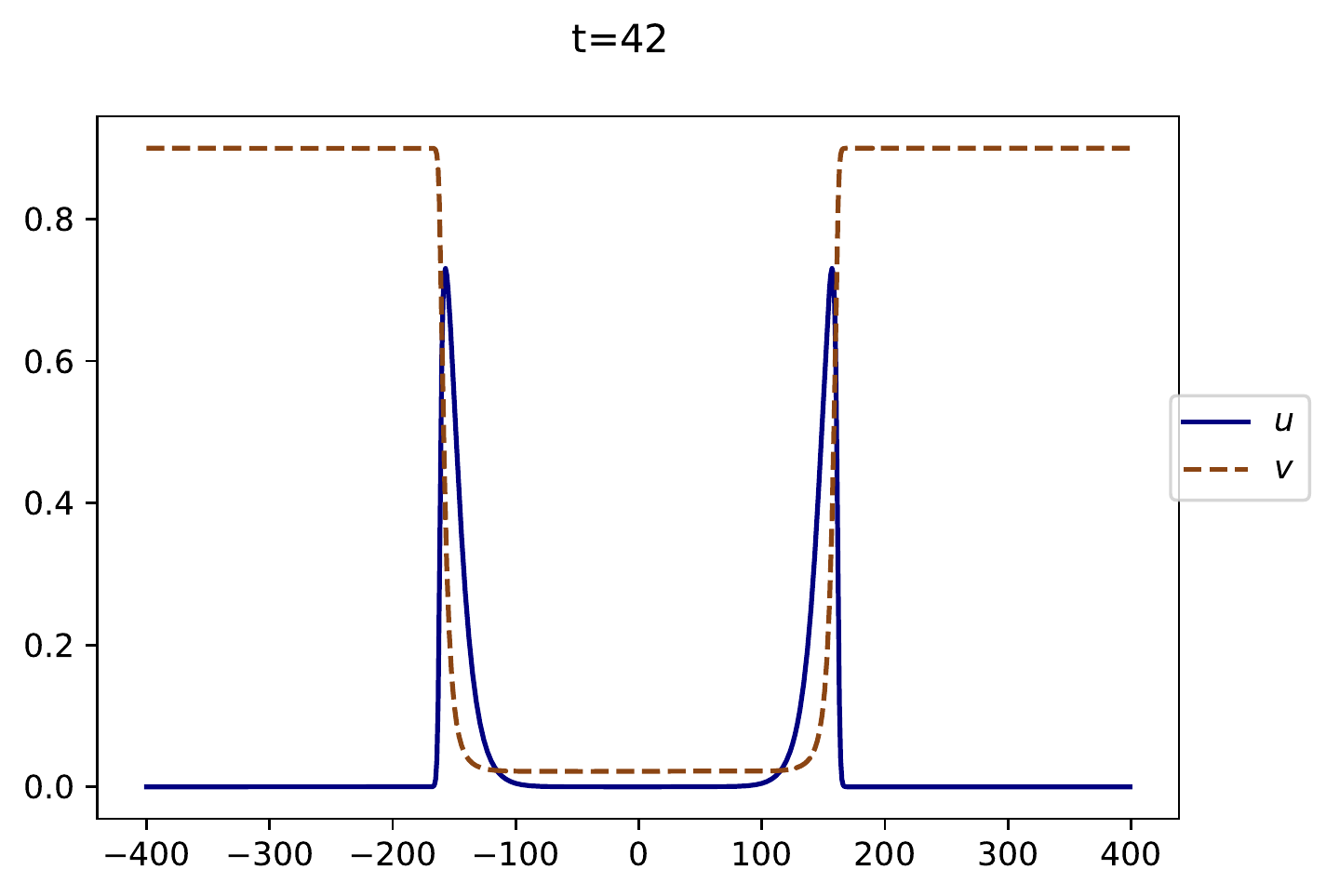}
\label{fig:Inhib_RiotBIS_2}
  \end{subfigure}
  \\
  \begin{subfigure}[p]{0.4\linewidth}
\includegraphics[scale=0.5]{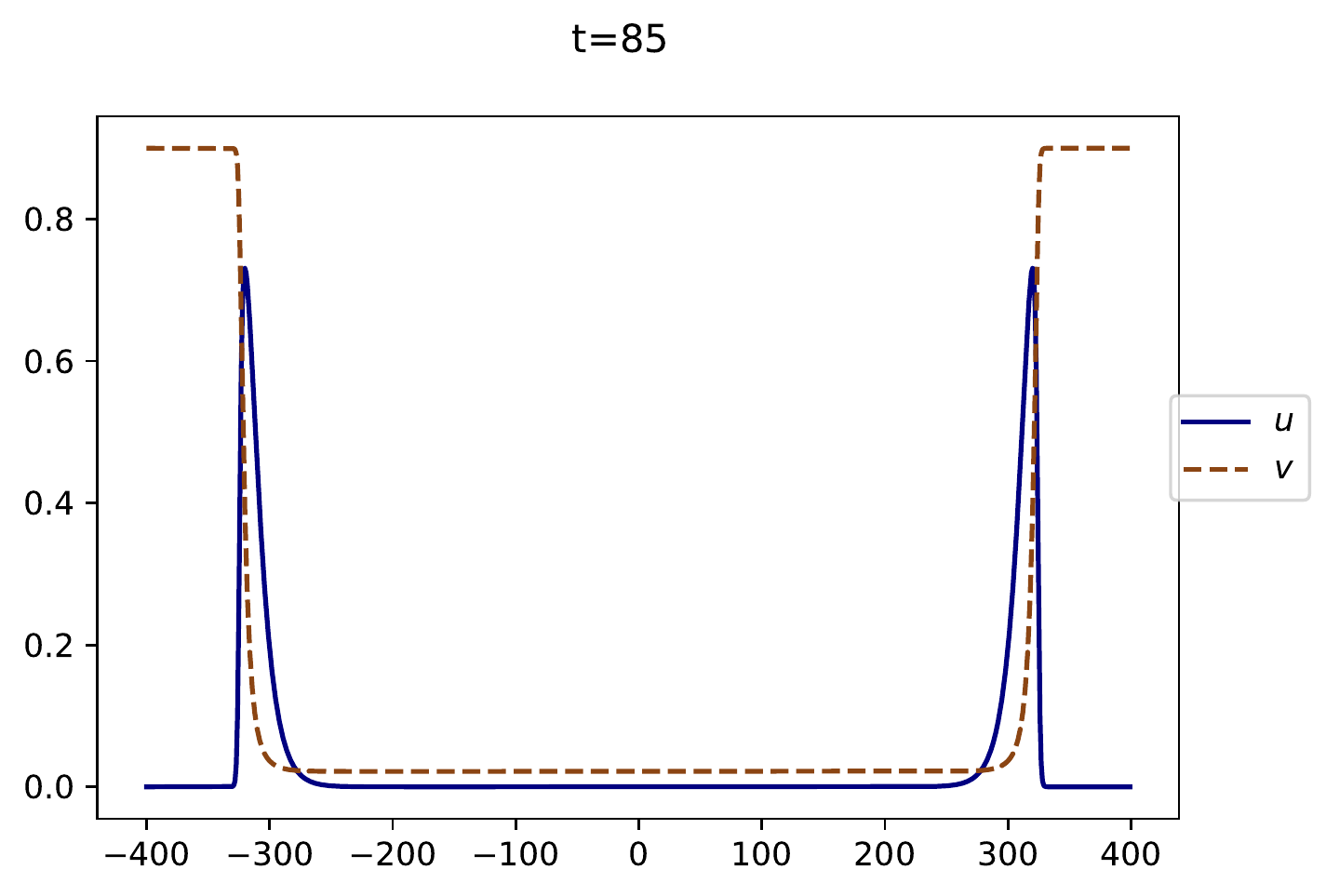}
\label{fig:Inhib_RiotBIS_3}
  \end{subfigure}
  \caption{Inhibiting case -- strong riot. 
  Snapshot at different times of the solution of~\eqref{ExampleInhibiting} with $v_b=0.9>v_\star$
  (compare with Figure~\ref{fig:Inhib_Riot} where $v_b=0.15$). 
  Horizontal axis: space. Blue solid line: $u(t,\cdot)$. Brown dashed line: $v(t,\cdot)$. 
\href{https://drive.google.com/file/d/1B-QNJXrc-0HWj6HcxC__QvvBoHTVNn1j/view?usp=sharing}{\color{blue}\underline{Video: Inhibiting\_v0=09.mp4}}  
  }
  \label{fig:Inhib_RiotBIS}
\end{figure}

We also observe in~\autoref{fig:Inhib_RiotBIS} that the eventual level of social tension $v(t=+\infty)$ is very low. To see this more precisely, we plot in~\autoref{fig:Inhib_Vinfty_V0} the value of $v_\infty$ as a function of~$v_b$ on the system~\eqref{ExampleInhibiting}. We observe that $v_\infty$ is indeed a decreasing function of $v_b$. From a modeling perspective, it means that the higher the initial level of social tension, the lower the final level of social tension. This phenomenon is in agreement with what was observed in~\autoref{fig:PlotV_Infty_exact} for the $SI$ model.

Now, let us focus on the dependence of $v_\infty$ with respect to the diffusion on social tension $d_2$. We consider \eqref{ExampleInhibiting2} with a varying diffusion on the second equation, that is,
\begin{equation}\label{ExampleInhibiting3}
\left\{\begin{aligned}
&\D_t  u -\D_{xx} u=u\left[v(1-u)-\frac{1}{3}\right],\\
&\D_t  v-d_2\D_{xx} v=-uv.
\end{aligned}\right.
\end{equation}
We fix $v_0\equiv v_b=0.5$ and $u_0(x)=0.2(1-\frac{x^2}{10})_+$. We plot in~\autoref{fig:Inhib_Vinfty_D2} the value of $v_\infty$ as a function of $d_2$. We observe that $d_2\mapsto v_\infty$ is increasing. From a modeling point of view, it means that the higher the diffusion on the social tension, the higher the final level of social tension. Heuristically, this is explained by the fact that $\lim_{\vert x\vert\to+\infty}v(t,x)=v_b>v_\infty$ and that, the larger $d_2$, the more $v(t,x)$ is influenced by its value at far distances.

\begin{figure}[p]
\center
\begin{subfigure}[p]{0.4\linewidth}
\includegraphics[scale=0.5]{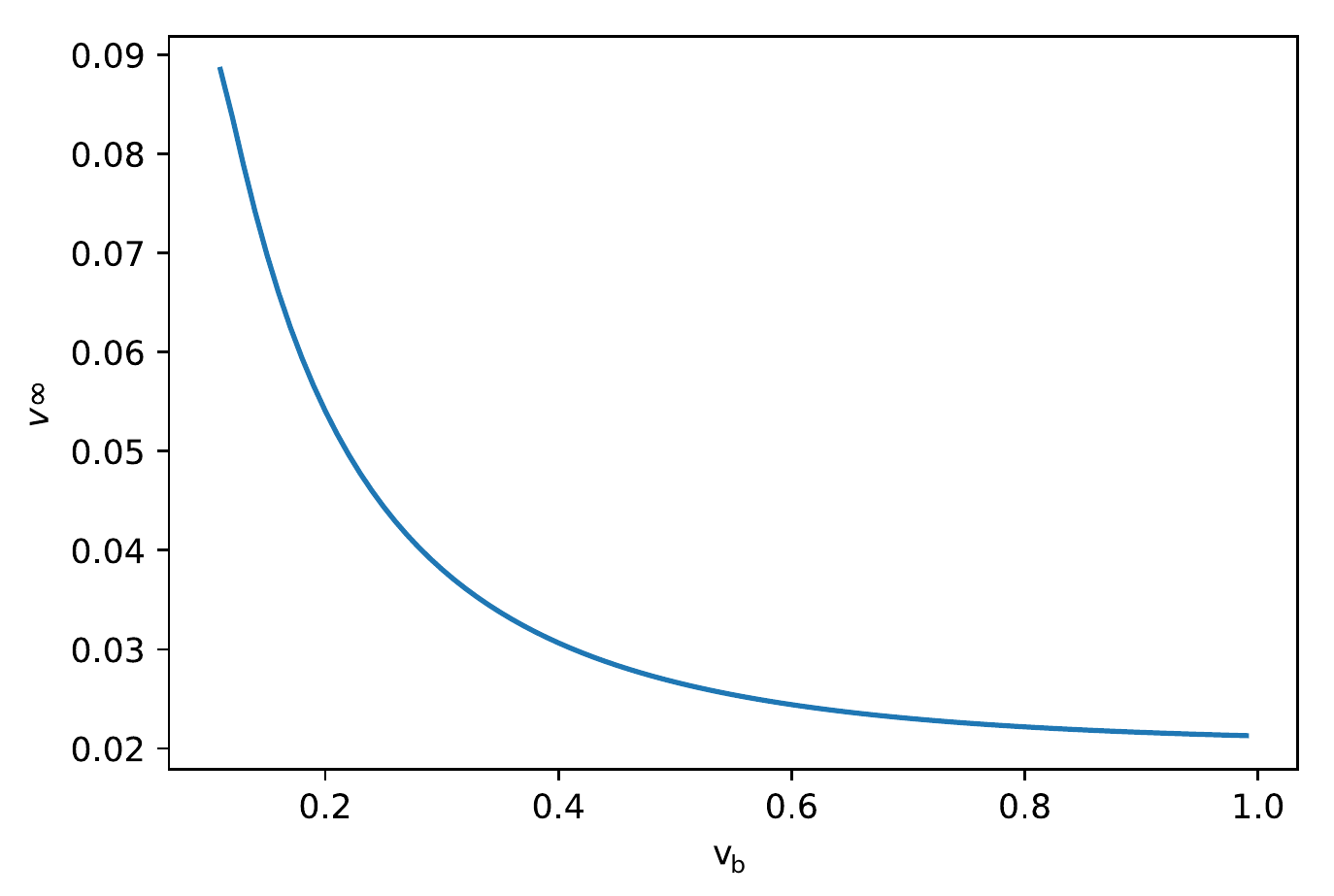}
\caption{$v_\infty$ as a function of $v_b$ in~\eqref{ExampleInhibiting}.}
\label{fig:Inhib_Vinfty_V0}
  \end{subfigure}
  \hfill
\begin{subfigure}[p]{0.4\linewidth}
\includegraphics[scale=0.5]{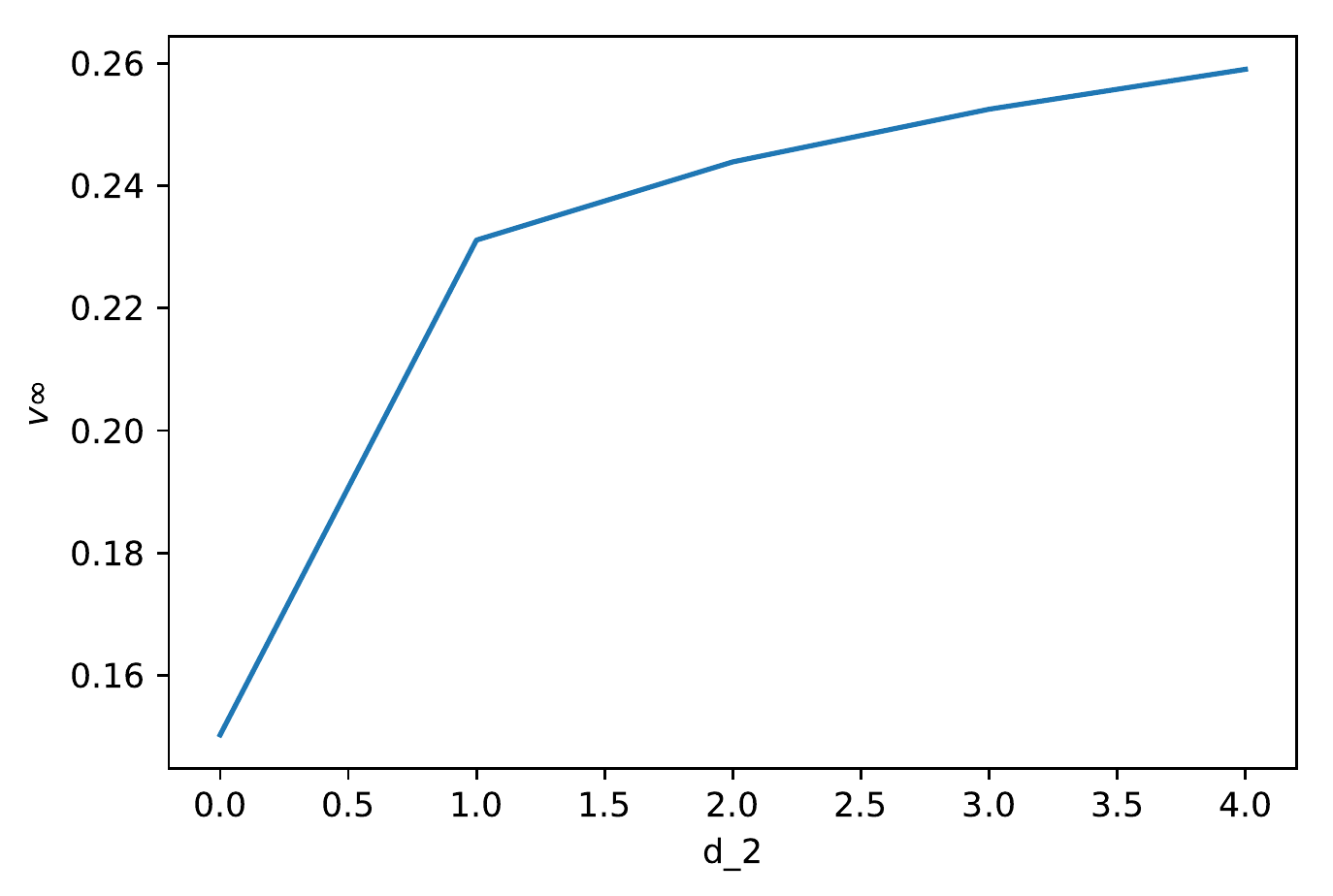}
\caption{$v_\infty$ as a function of $d_2$ in~\eqref{ExampleInhibiting3}.}
\label{fig:Inhib_Vinfty_D2}
  \end{subfigure}
  \caption{Inhibiting case --  eventual level of social tension $v_\infty=\lim_{t\to+\infty}v(t,\cdot)$ as a function of the parameters for the solution of~\eqref{ExampleInhibiting3}.}
\end{figure}

\section{Tension Enhancing - dynamics of a lasting upheaval}\label{sec:enhancing}
The \emph{tension enhancing} structure relies on the
 following saturation assumptions on~$f$ and positivity on~$\Psi$:
\begin{equation}\label{AssumptionTensionEnhancing}
\begin{cases}
f\text{ is strictly decreasing and }f(M)\leq 0 & \text{for some $M>0$},\\
\Psi(u,v)>0 & \text{$\forall u\in(0,M),\ v\in(0,1)$}.
\end{cases} 	
 \end{equation}


The assumption on $f$ accounts for the saturation effect on the level of social unrest. A typical example is to take $f(u)=M-u$. As a direct consequence, the parabolic comparison principle yields 
\begin{equation}\label{lemma:u<M}
\limsup_{t\to+\infty} u(t,x)\leq M,\quad\forall x\in\R^n.
\end{equation}

The assumption on $\Psi$ essentially means that $u$ has a positive feedback on $v$ (we do not assume, however, that $u\mapsto \Psi(u,v)$ is increasing). 
Again, as a direct consequence of the positivity of $\Psi$ we have that
\begin{equation}\label{lemma:v>v0}
v(t,x)> v_0\equiv v_b\qquad  \forall\ t>0,\ x\in\R^n.
\end{equation}

A typical example of the \emph{tension enhancing} case is the cooperative system
\begin{equation*}
\left\{\begin{aligned}
&\D_t  u -d_1 \Delta u=u\big[v(1-u)-\omega \big],\\
&\D_t  v-d_2\Delta v=uv(1-v),
\end{aligned}\right.
\end{equation*}
obtained by taking $r(u,v)=v$, $f(u,v)\equiv1-u$, $\Psi(u,v)=uv(1-v)$ 
in~\eqref{GeneralEquationMotivationFinal}.

The \emph{tension enhancing} assumption typically grasps the dynamics of a persisting movement of social unrest, 
that we refer to here as a \emph{lasting upheaval}.
A good heuristic of the behavior of the model is given by a formal analysis of the underlying $ODE$ system. 
Taking $u_0$ constant implies that the solutions $u$ and $v$ of~\eqref{GeneralEquationMotivationFinal} do not depend on $x$. If the initial level of social tension is above the threshold $v_\star=\omega$
from~\eqref{Def_v_star},
then any \emph{triggering event} ignites a \emph{burst} of social unrest. The level of social tension then \emph{increases}, which enhances the growth of social unrest. Eventually, both the level of social tension and social unrest converge to an excited state, $v\to1$ and $u\to u_\star(1)=1-\omega$ 
from definition~\eqref{def:u_star}.

We begin by stating qualitative properties for the traveling waves, then we investigate the 
large time behavior of solutions to~\eqref{GeneralEquationMotivationFinal}. 
On this latter question, we obtain more complete results than in the inhibiting case (Section~\ref{sec:inhibiting}), since we prove the  convergence of the solution to an excited state as $t\to+\infty$. We then present some numerical simulations to illustrate the results and to investigate further properties. 

%

\subsection{Traveling waves}\label{sec:enhancing_TW}

Let us first discuss the existence and non-existence of traveling solutions
of \eqref{SystemRiot}-\eqref{SystemBorderGeneralRiot} using the results of~\cite{Berestycki2019b}.
Consider the function $\gamma:[c_1,+\infty)\to\R$ defined~by
\begin{equation*}
\gamma(c):=\frac{\sqrt{c^2-c_b^2}-\sqrt{c^2- c_1^2}}{2c},
\end{equation*}
where $c_b$ and $c_1$ are defined in~\eqref{Def_c}. Because $c_b<c_1$, we have that $\gamma$ is positive and decreasing, with $\gamma(c_1)= \frac{1}{2}\sqrt{1-\frac{c_b^2}{ c_1^2}}$ and $\gamma(+\infty)=0$. Let us define
\begin{equation}\label{def_c_bar}
 c_\gamma:=\inf\left\{c\geq c_1: d_2\gamma(c)\leq d_1\right\}\in[c_1,+\infty).
\end{equation}
We point out that
\begin{equation*}
c_\gamma=\left\{\begin{aligned}
&c_1, &&\text{if }d_2\gamma(c_1)\leq d_1,\\
&\gamma^{-1}\left(\frac{d_1}{d_2}\right), &&\text{otherwise}.
\end{aligned}\right.
\end{equation*}
Recall that the sign of $v_b-v_\star$ (from definition~\eqref{Def_v_star}) coincides with the one of $K_b$ (from~\eqref{defK0}).

Under assumptions~\eqref{AssumptionTensionEnhancing}, we know from~\cite[Theorem 4]{Berestycki2019b} that if $v_b>v_\star$ then the following hold:
\begin{itemize}
\item there exists no traveling wave with speed $c<c_b$;
\item there exists a traveling wave with any speed $c>c_\gamma$.
\end{itemize}
In particular, if $d_2$ is sufficiently small, then $c_\gamma=c_1$, and so there exists a traveling wave for any $c>c_1$.

We then reclaim~\cite[Theorem 9]{Berestycki2019b} which eastablish properties on the shape of the traveling waves. Namely, under the enhancing assumption~\eqref{AssumptionTensionEnhancing}, any traveling wave
satisfies
\begin{equation}\label{PropositionQualitativeEnhancingTW}
U',V'>0,\qquad U(+\infty)= u_\star(1):=f^{-1}\left(\frac{\omega}{r(1)}\right), \qquad V(+\infty)=1.
\end{equation}
We also recall that $u_\star(1)$ is defined in~\eqref{def:u_star} as the only positive root of $r(1)f(\cdot)-\omega$.
Note that, contrarily to the analogous result in the inhibiting case in~\eqref{PropositionQualitativeInhibiting}-\eqref{PropositionQualitativeInhibitingTW}, we are here able to determine explicitely the limit of $U$ in $+\infty$.

\subsection{Large time behavior for the Cauchy problem}
In the previous section, we have identified the limits of traveling waves as $\xi\to+\infty$. 
We now show that the same limits hold 
for the solution of~\eqref{GeneralEquationMotivationFinal} as $t\to+\infty$.
\begin{theorem}\label{PropositionAsymptoticsEnhancing}
Assume that~\eqref{AssumptionTensionEnhancing} holds and that $v_b>v_\star$.
Then any solution 
of~\eqref{GeneralEquationMotivationFinal} with $v_0\equiv v_b$ and $u_0$ compatcly supported satisfies
\begin{equation}\label{LimitingStatesEnhancing}
\lim\limits_{t\to+\infty}u(t,x)= u_\star(1):=f^{-1}\left(\frac{\omega}{r(1)}\right),\qquad
 \lim\limits_{t\to+\infty}v(t,x)=1,
\end{equation}
locally uniformly in $x\in\R^n$.
\end{theorem}
This theorem states that the level of social unrest converges to a sustainable excited 
state $u_\star(1)$. From a modeling point of view, 
this corresponds to a persisting social movement, that we refer to as a \emph{lasting upheaval}.


\subsection{Numerical simulations}\label{sec:NumericsEnhancing}

In this section we provide some numerical illustrations of the dynamics of system~\eqref{GeneralEquationMotivationFinal} in the enhancing case~\eqref{AssumptionTensionEnhancing}. 

\subsubsection{Threshold between calm and lasting upheaval}
Let us consider the following particular instance of~\eqref{GeneralEquationMotivationFinal} satisfying the enhancing assumption~\eqref{AssumptionTensionEnhancing}:
\begin{equation}\label{ExampleEnhancing}
\left\{\begin{aligned}
&\D_t  u -\D_{xx} u=u\left[v(1-u)-\frac{1}{3}\right],\\
&\D_t  v-\D_{xx} v=uv(1-v),
\end{aligned}\right.
\end{equation}
with initial condition $u_0(x)=0.2(1-x^2)_+$ and $v_0\equiv v_b$ ranging in $(0,1)$. In this case we have
$v_\star= \nicefrac{1}{3}$ and $K_b:= v_b-\nicefrac{1}{3}$.

\autoref{fig:Enhanc_Calm} refers to the case $v_b<v_\star$.
We observe there that a triggering event is promptly 
followed by a return to calm, namely, $u$ rapidly vanishes. Next, 
$v$ converges in long time to its initial value $v_b$. This is in agreement with the discussion in Section~\ref{sec:ResumptionCalm}. Let us emphasize, however, that this is only true when considering a sufficiently small initial condition $u_0(x)$, see Section~\ref{sec:enhancing_MagnitudeTriggering} for more details.

\begin{figure}[p]
\center
\begin{subfigure}[p]{0.4\linewidth}
\includegraphics[scale=0.5]{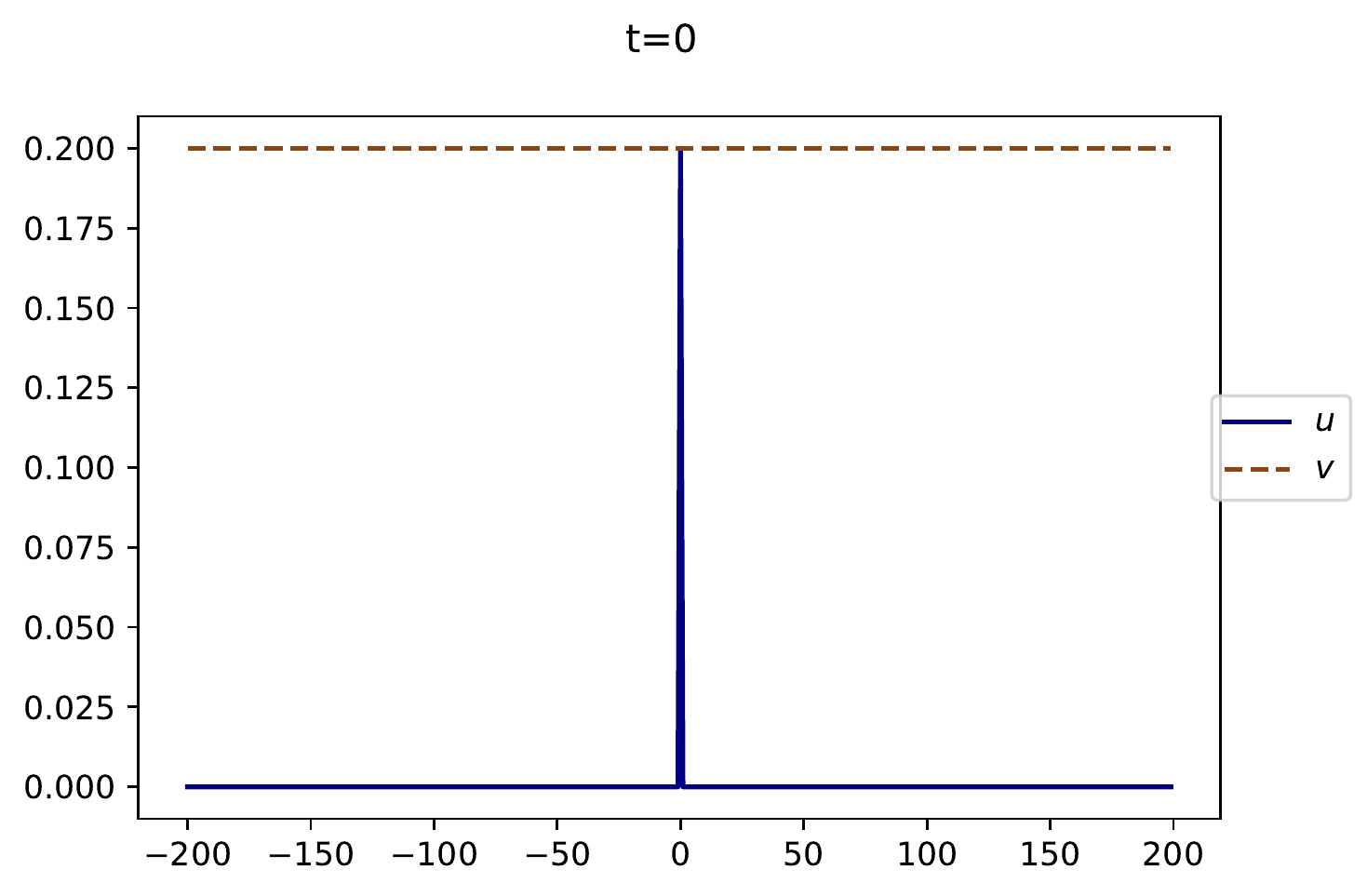}
  \end{subfigure}
  \hfill
\begin{subfigure}[p]{0.4\linewidth}
\includegraphics[scale=0.5]{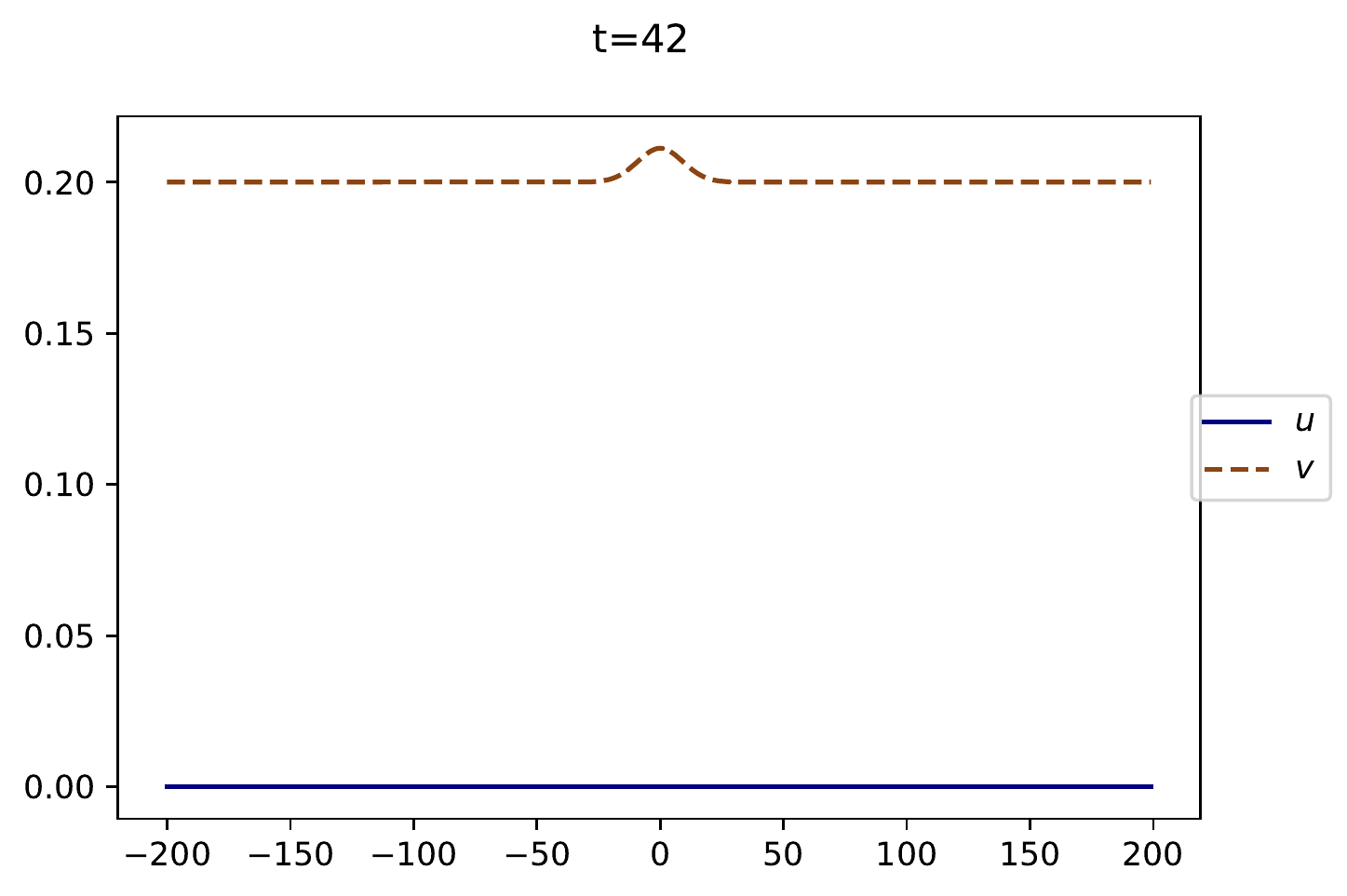}
  \end{subfigure}
  \\
  \begin{subfigure}[p]{0.4\linewidth}
\includegraphics[scale=0.5]{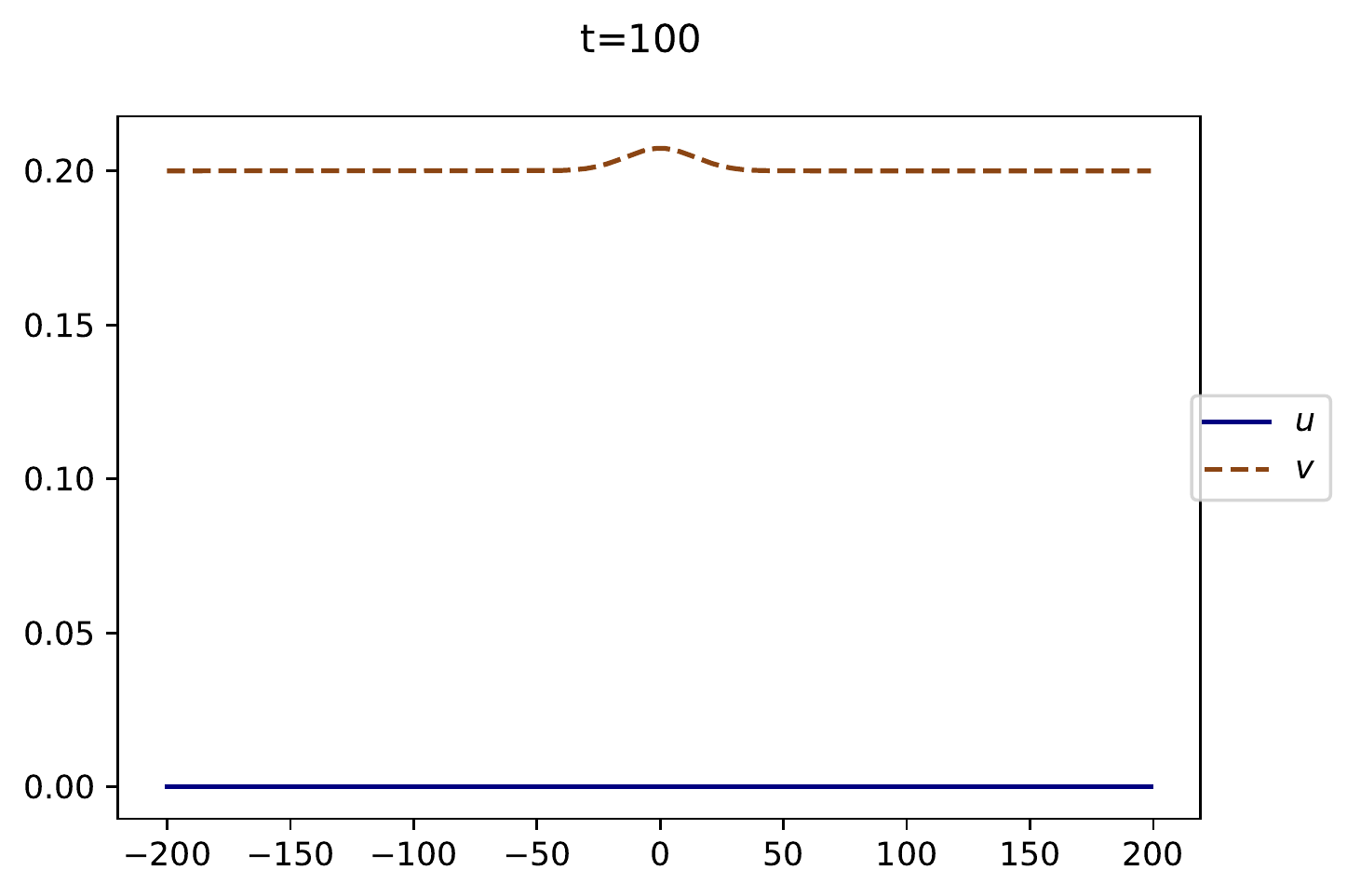}
  \end{subfigure}
  \caption{Enhancing case -- return to calm. Snapshots at different times of the space distribution of the solution of~\eqref{ExampleEnhancing} with $v_b=0.2<v_\star$. Horizontal axis: space. Blue solid line: $u(t,\cdot)$. Brown dashed line: $v(t,\cdot)$.
\href{https://drive.google.com/file/d/1h-5hvAqnKxpn5opKZbg4E6Ic1E5xVIXJ/view?usp=sharing}{\color{blue}\underline{Video: Enhancing\_v0=02.mp4}}    
  }
  \label{fig:Enhanc_Calm}
\end{figure}

On the contrary, if $v_b>v_\star$, a small triggering event ignites a burst of social unrest that spreads through space, as can be seen on \autoref{fig:Enhanc_Revolution}. This is in agreement with Sections~\ref{sec:Burst} and~\ref{sec:SpatialPropagation}. More precisely,
the simulation shows that the solution converges to two traveling waves, one leftwards and the other one rightwards, i.e., two fixed profiles moving at a constant speed.
The profiles for both $u$ and $v$ have the shape of 
monotone waves, as stated in \autoref{PropositionQualitativeInhibitingTW}. In addition, the solution 
converges pointwise to $(u_\star(1),1)$ as $t\to+\infty$, 
in agreement with~\autoref{PropositionAsymptoticsEnhancing}.

Since the asymptotic state as $t\to+\infty$ features a positive level of activity, the dynamics describe a persisting movement of social unrest, that we call a \emph{lasting upheaval}.

\begin{figure}[p]
\center
\begin{subfigure}[p]{0.4\linewidth}
\includegraphics[scale=0.5]{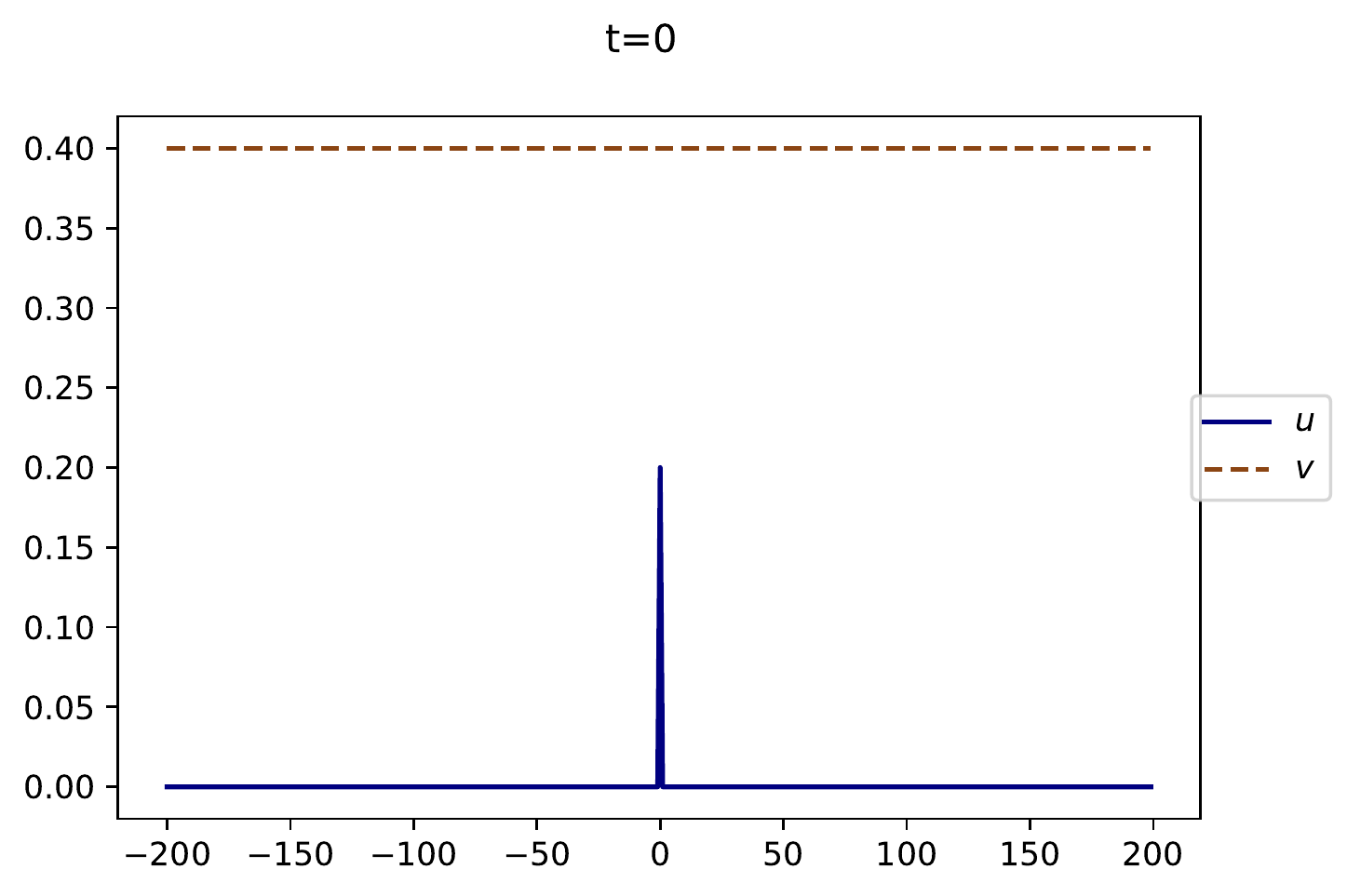}
  \end{subfigure}
  \hfill
\begin{subfigure}[p]{0.4\linewidth}
\includegraphics[scale=0.5]{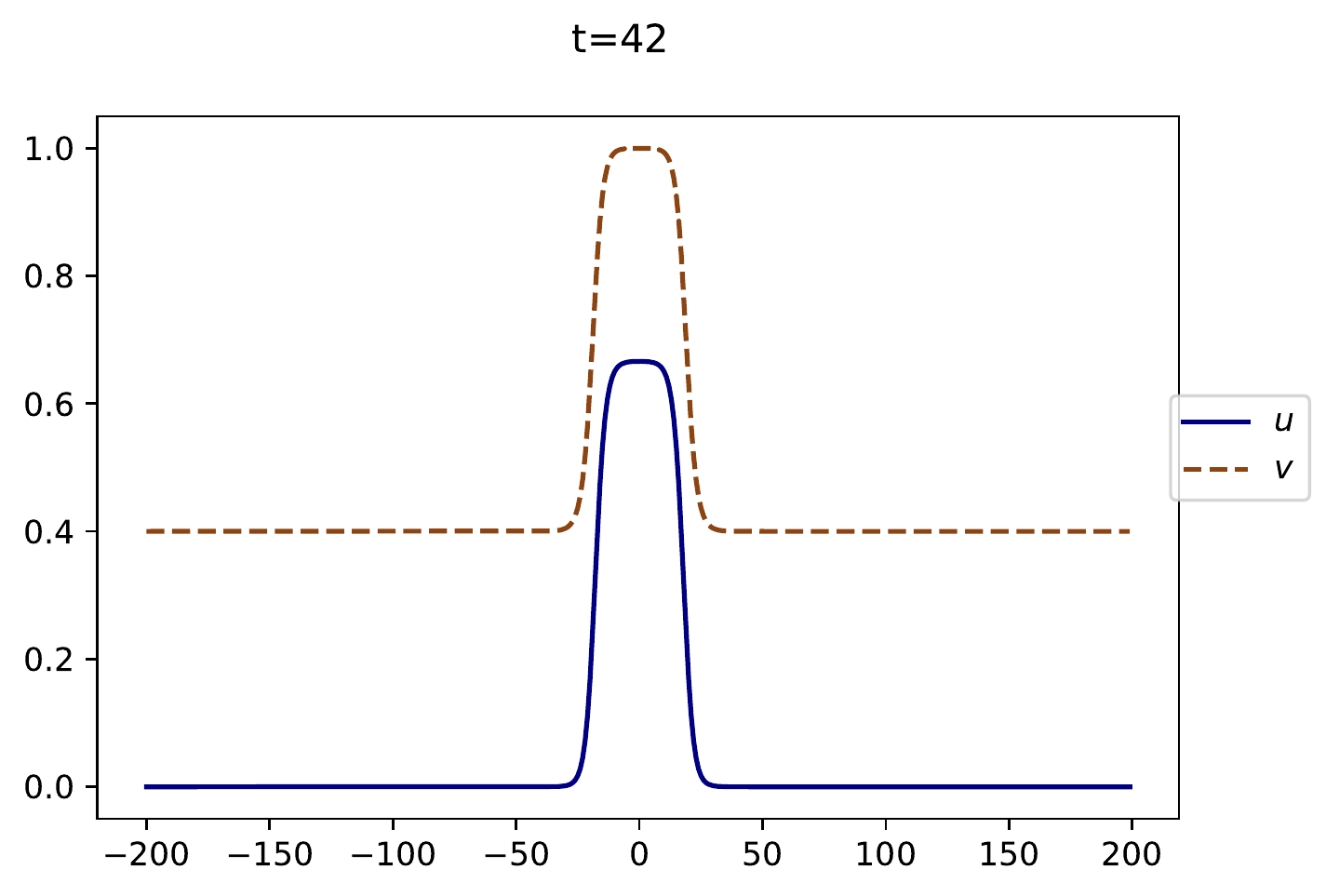}
  \end{subfigure}
  \\
  \begin{subfigure}[p]{0.4\linewidth}
\includegraphics[scale=0.5]{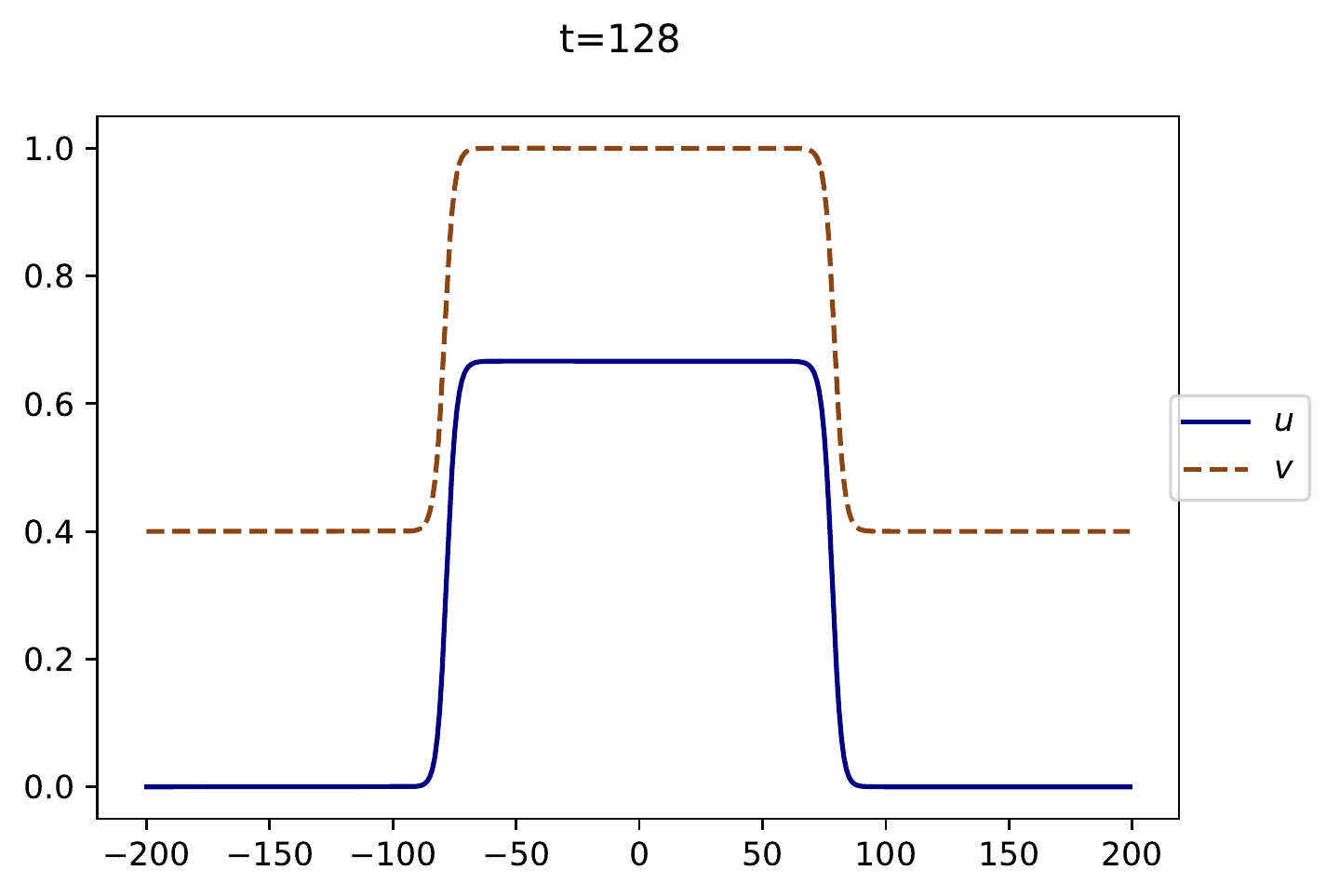}
  \end{subfigure}
    \hfill
\begin{subfigure}[p]{0.4\linewidth}
\includegraphics[scale=0.5]{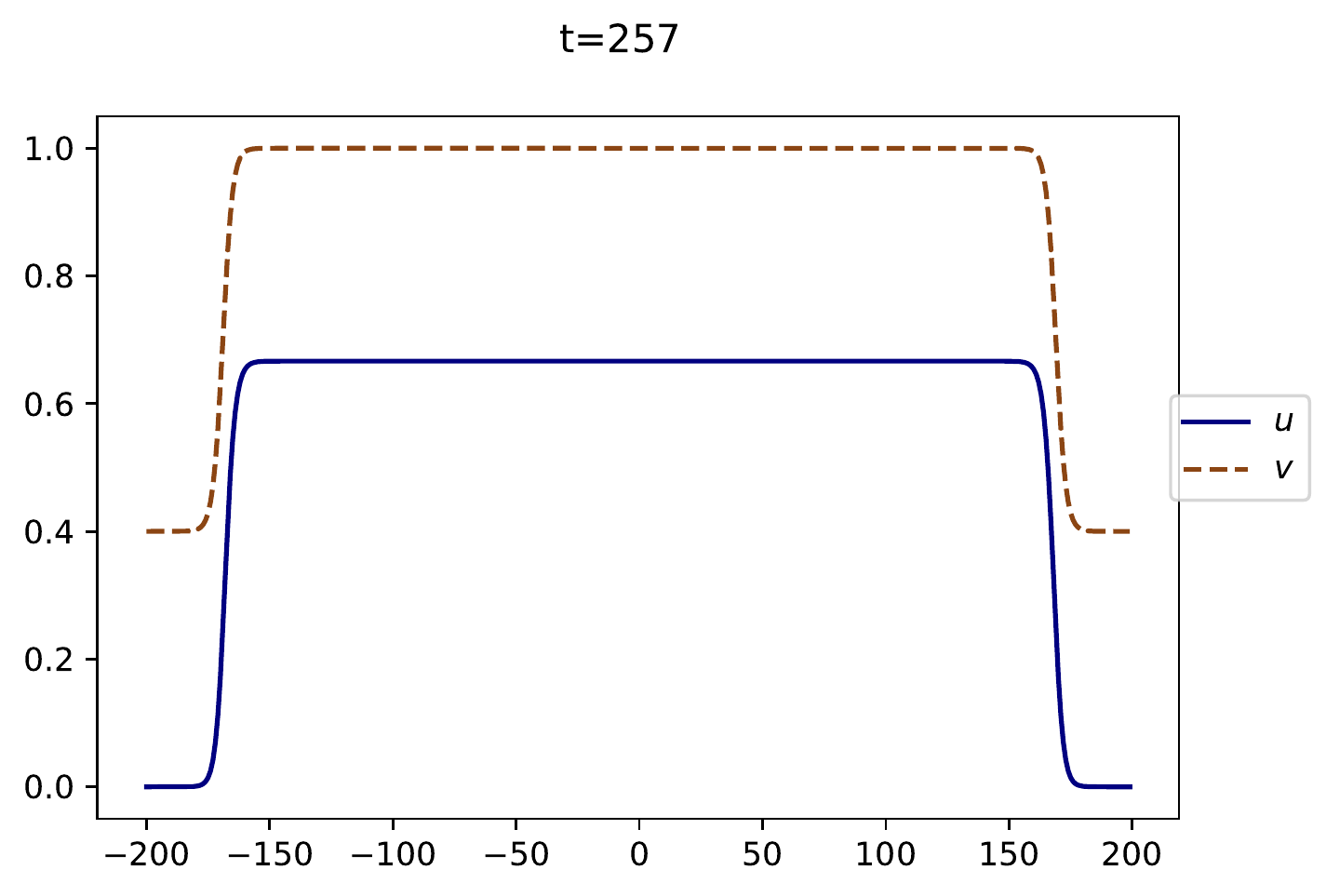}
  \end{subfigure}
  \caption{Enhancing case -- lasting upheaval. Snapshots at different times of the solution of~\eqref{ExampleEnhancing} with $v_b=0.4>v_\star$. Horizontal axis: space. Blue solid line: $u(t,\cdot)$. Brown dashed line: $v(t,\cdot)$. 
  \href{https://drive.google.com/file/d/1xuGtKmmlHj22paZp6hZZaha3ALdXvUHt/view?usp=sharing}{\color{blue}\underline{Video: Enhancing\_v0=04.mp4}}
}
  \label{fig:Enhanc_Revolution}
\end{figure}


\subsubsection{Magnitude of the triggering event}\label{sec:enhancing_MagnitudeTriggering}
In the \emph{tension enhancing} case, the magnitude of the triggering event (i.e., the size of $u_0(x)$) may be of crucial importance to determine the regime of the dynamics when $v_0\equiv v_b<v_\star$. This has to be put in contrast with \emph{tension inhibiting} case (described in Section~\ref{sec:inhibiting}) for which the regime of the dynamics do not depends on the size of $u_0(x)$.

Figure~\ref{fig:Triggering_eps} depicts two distinct dynamics for system~\eqref{ExampleEnhancing} 
corresponding to $v_b=0.3<v_\star=1/3$:
the initial condition $u_0(x)=0.1(1-x^2)_+$ exhibits a 
\emph{return to calm}, see~\autoref{fig:Triggering_eps=0.1},
whereas $u_0(x)=0.5(1-x^2)_+$ ignites a lasting upheaval, see~\autoref{fig:Triggering_eps=0.5}.

\begin{figure}[h]
  \begin{subfigure}[p]{\linewidth}
\includegraphics[width=0.45\textwidth]{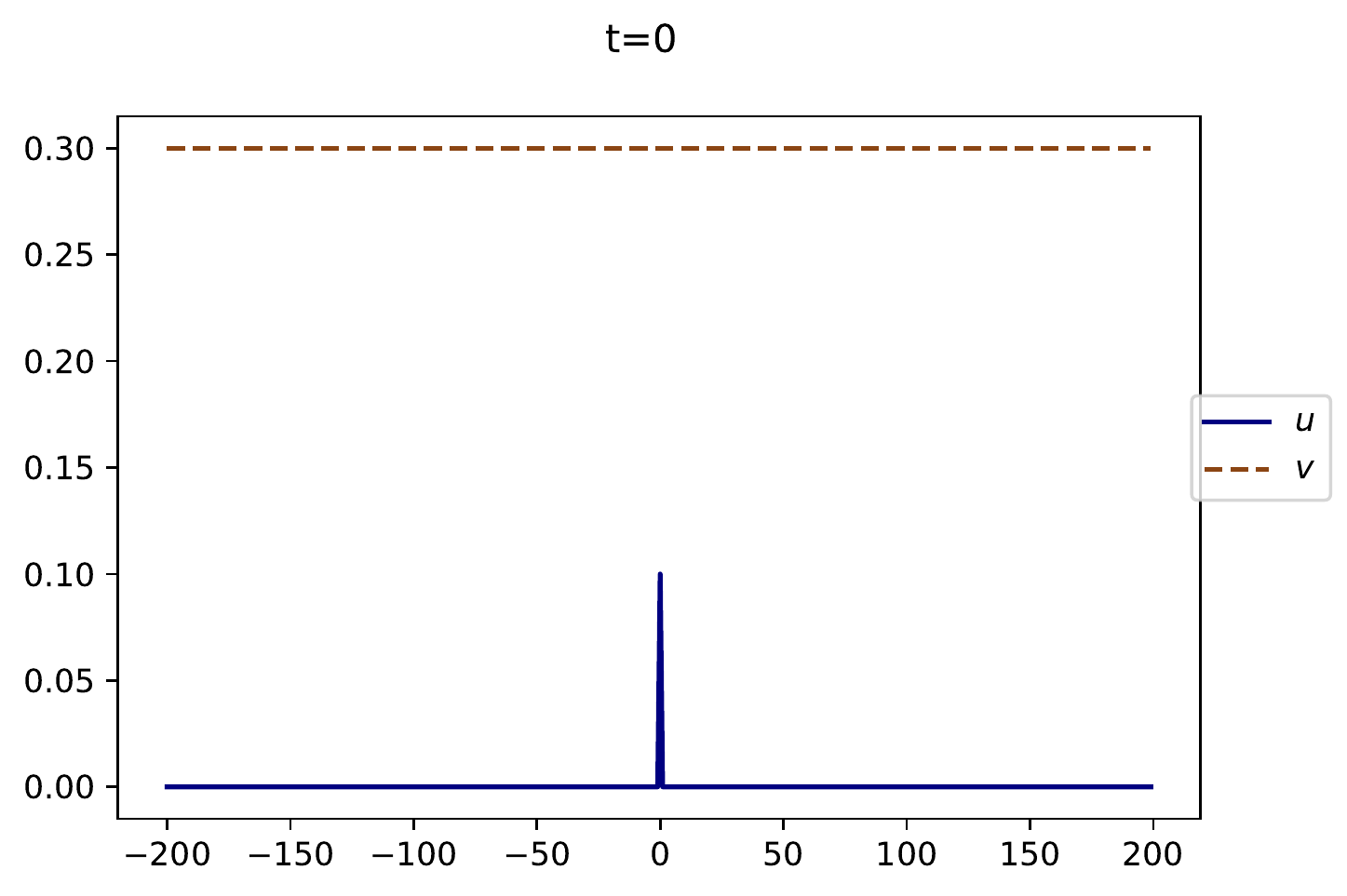}
\hfill
\includegraphics[width=0.45\textwidth]{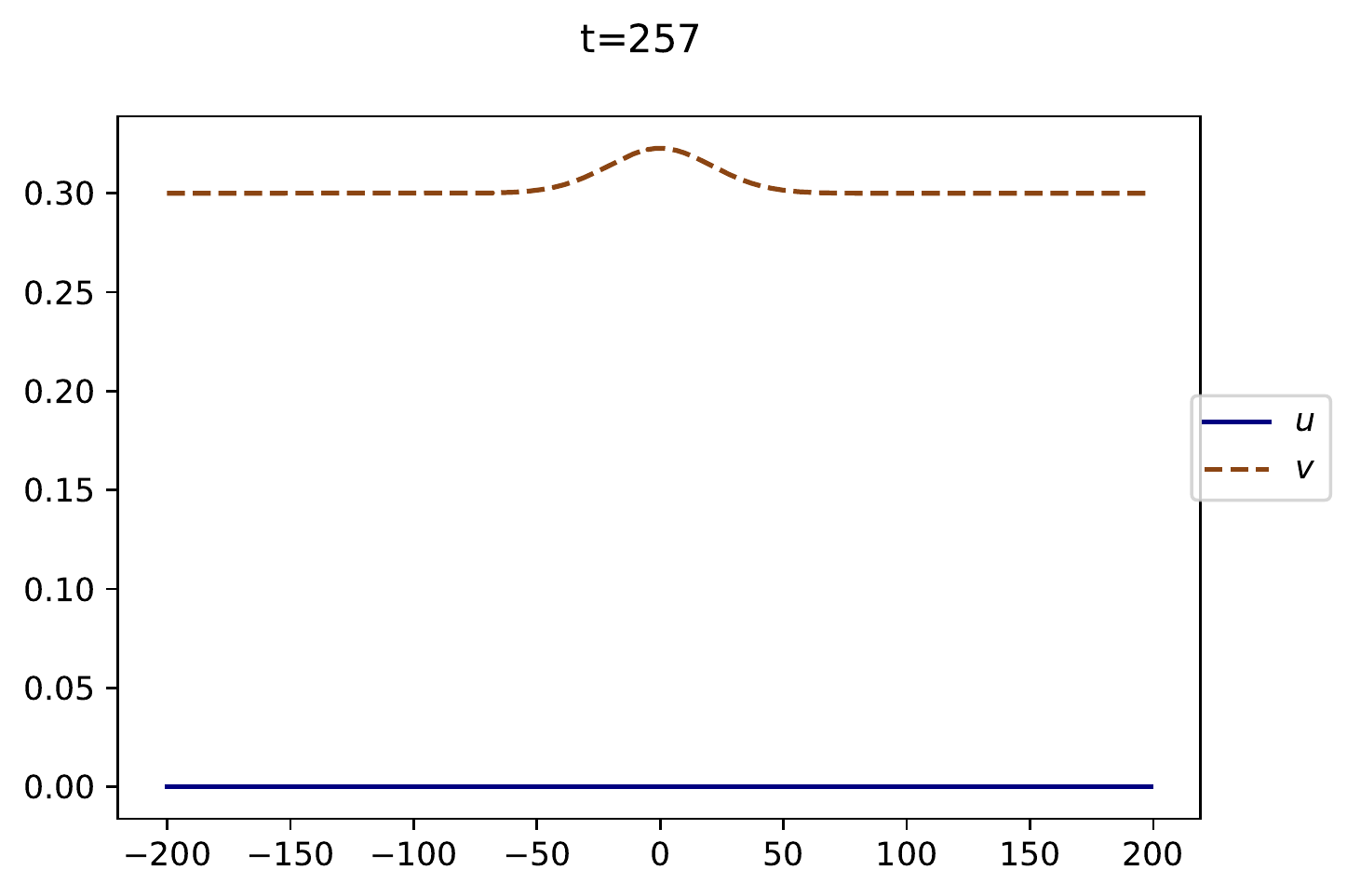}

\caption{Calm for $u_0(x)=0.1(1-x^2)_+$.
  \href{https://drive.google.com/file/d/1gJbgraLTfZbShXiY_74oxvQsvhByQOQr/view?usp=sharing}{\color{blue}\underline{Video: Triggering\_u0=01.mp4}}
} 
\label{fig:Triggering_eps=0.1}
  \end{subfigure}
 
 \begin{subfigure}[p]{\linewidth}
    \centering\includegraphics[width=0.45\textwidth]{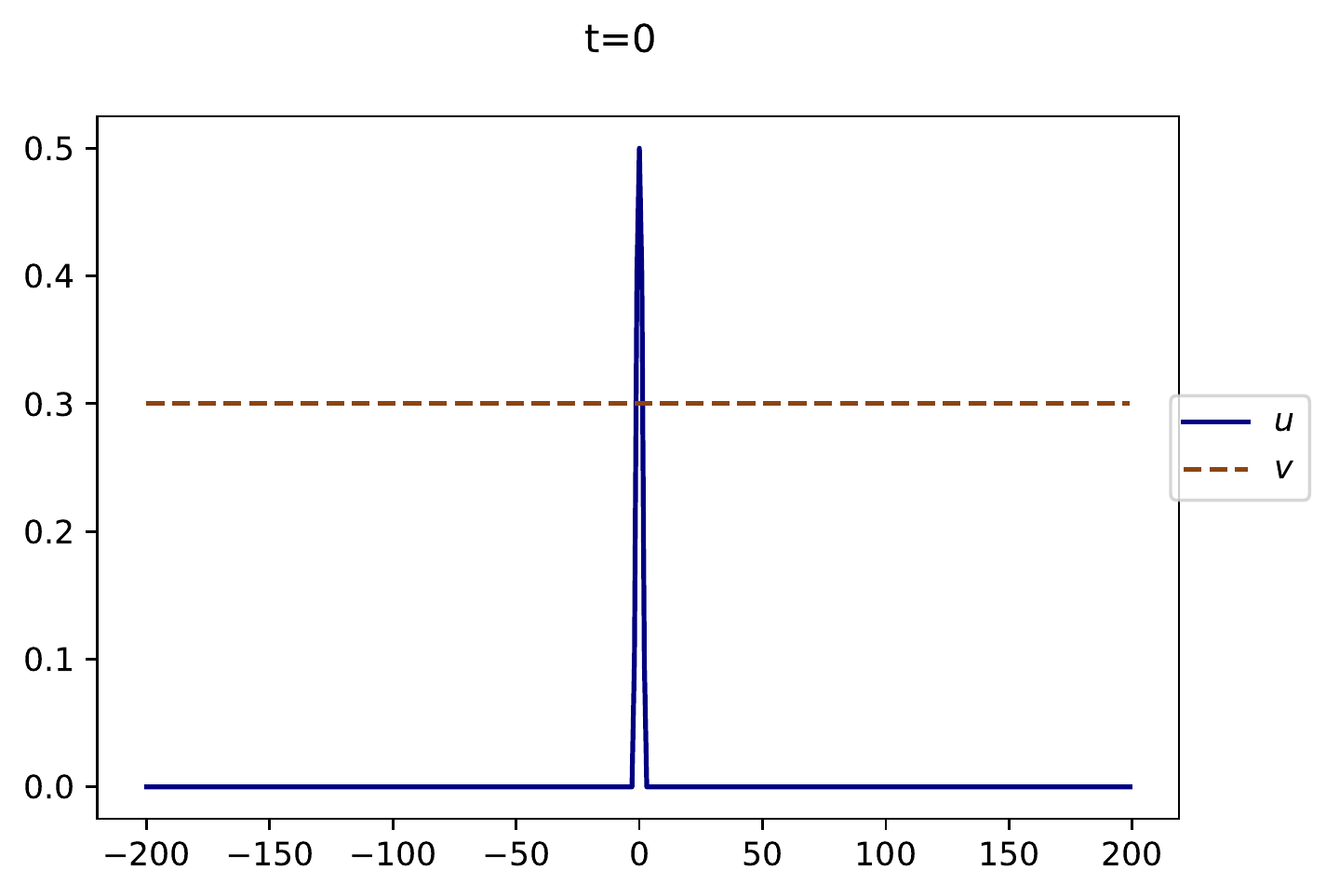}
  \hfill
    \centering\includegraphics[width=0.45\textwidth]{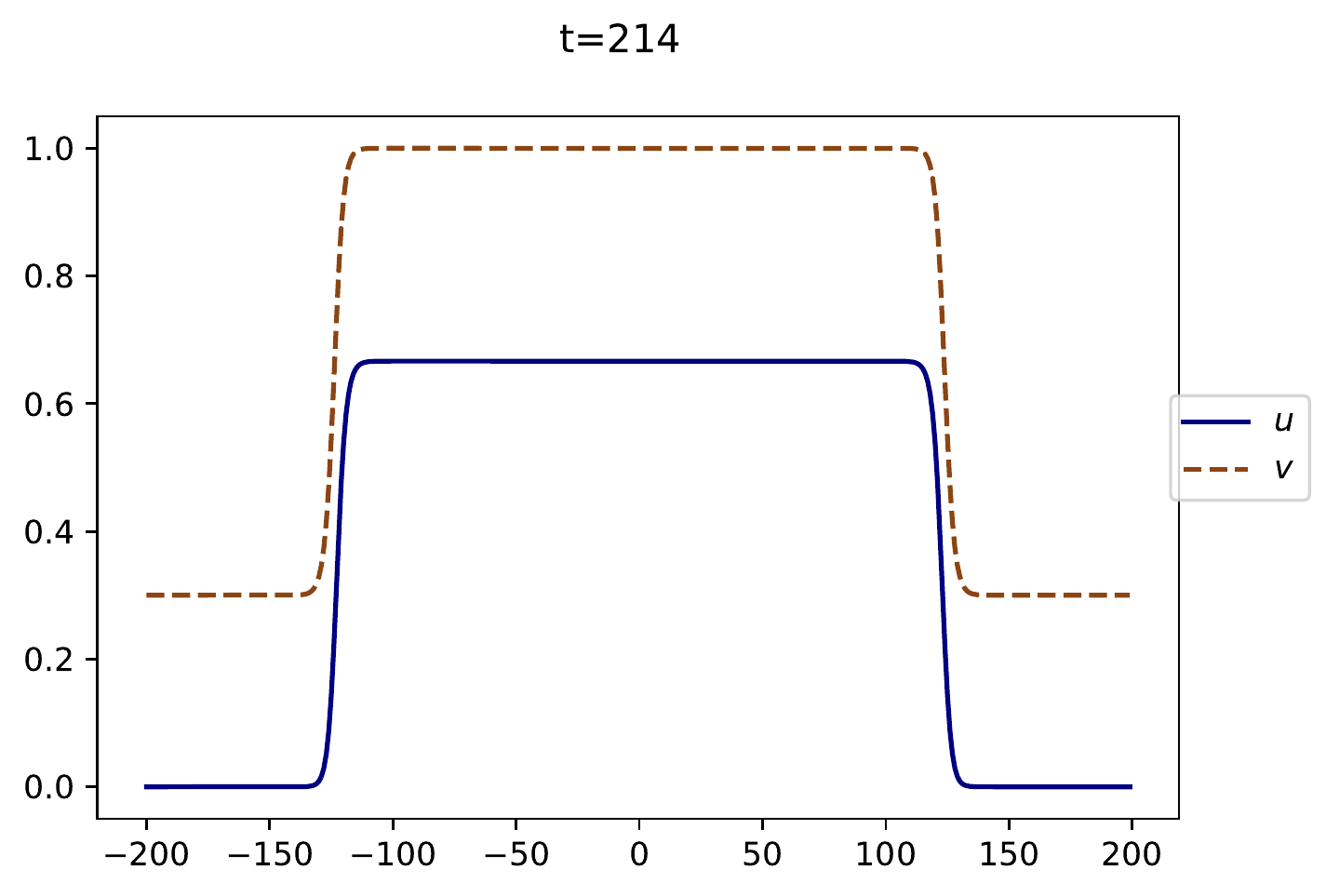}
\caption{lasting upheaval for $u_0(x)=0.5(1-x^2)_+$.
 \href{https://drive.google.com/file/d/1Ad1PmfaI22bDwc_sgcEVOg4If0-qL5ua/view}{\color{blue}\underline{Video: Triggering\_u0=05.mp4}}
} \label{fig:Triggering_eps=0.5}
    \end{subfigure}
  \caption{Enhancing case -- influence of the magnitude of the triggering event. Snapshots at different times of the solution of~\eqref{ExampleEnhancing} with $v_b=0.3<v_\star$. Horizontal axis: space. Blue solid line: $u(t,\cdot)$. Brown dashed line: $v(t,\cdot)$.}
 \label{fig:Triggering_eps}
\end{figure}

\begin{figure}[H]
\center
\begin{subfigure}[p]{0.4\linewidth}
\includegraphics[scale=0.5]{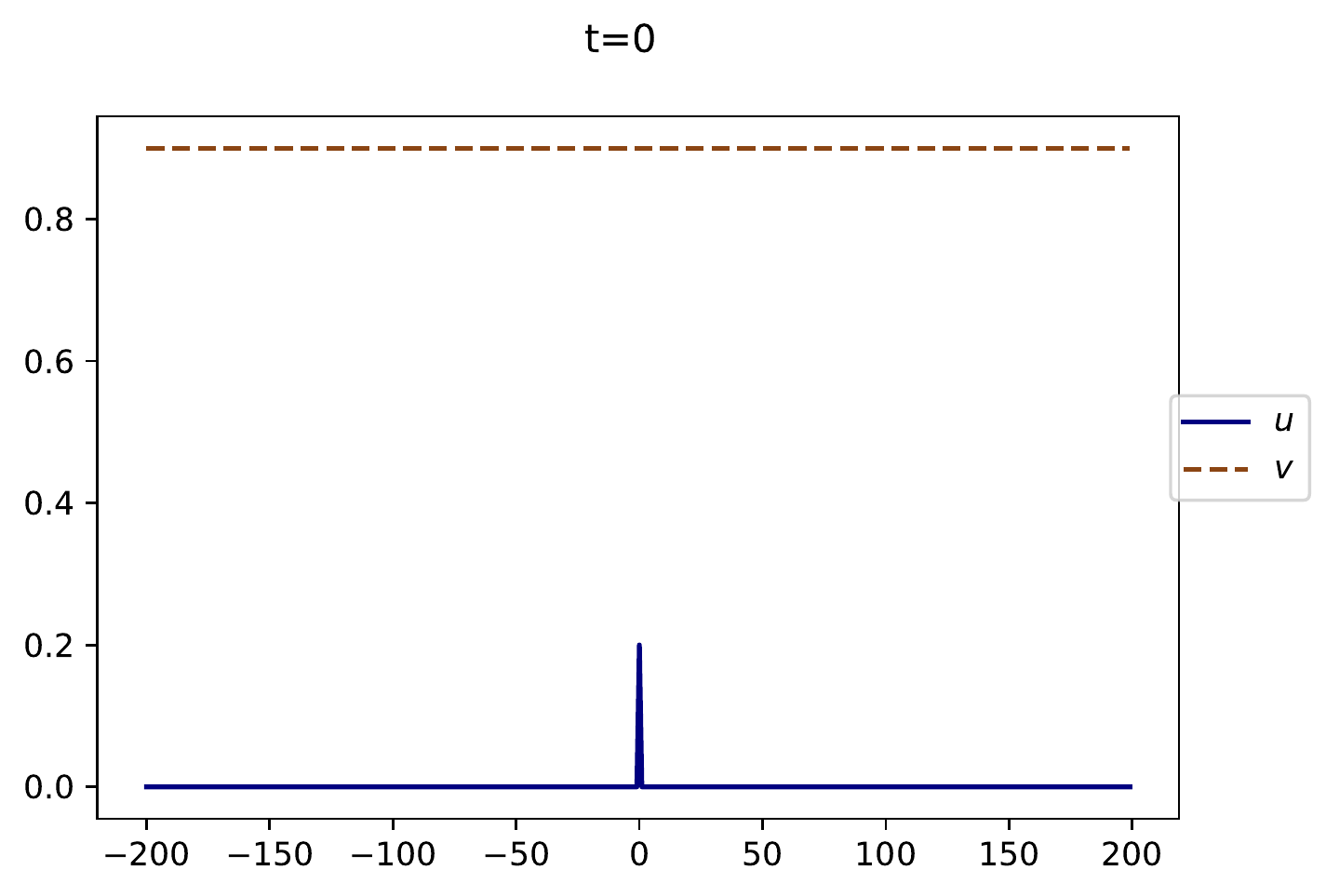}
  \end{subfigure}
  \hfill
\begin{subfigure}[p]{0.4\linewidth}
\includegraphics[scale=0.5]{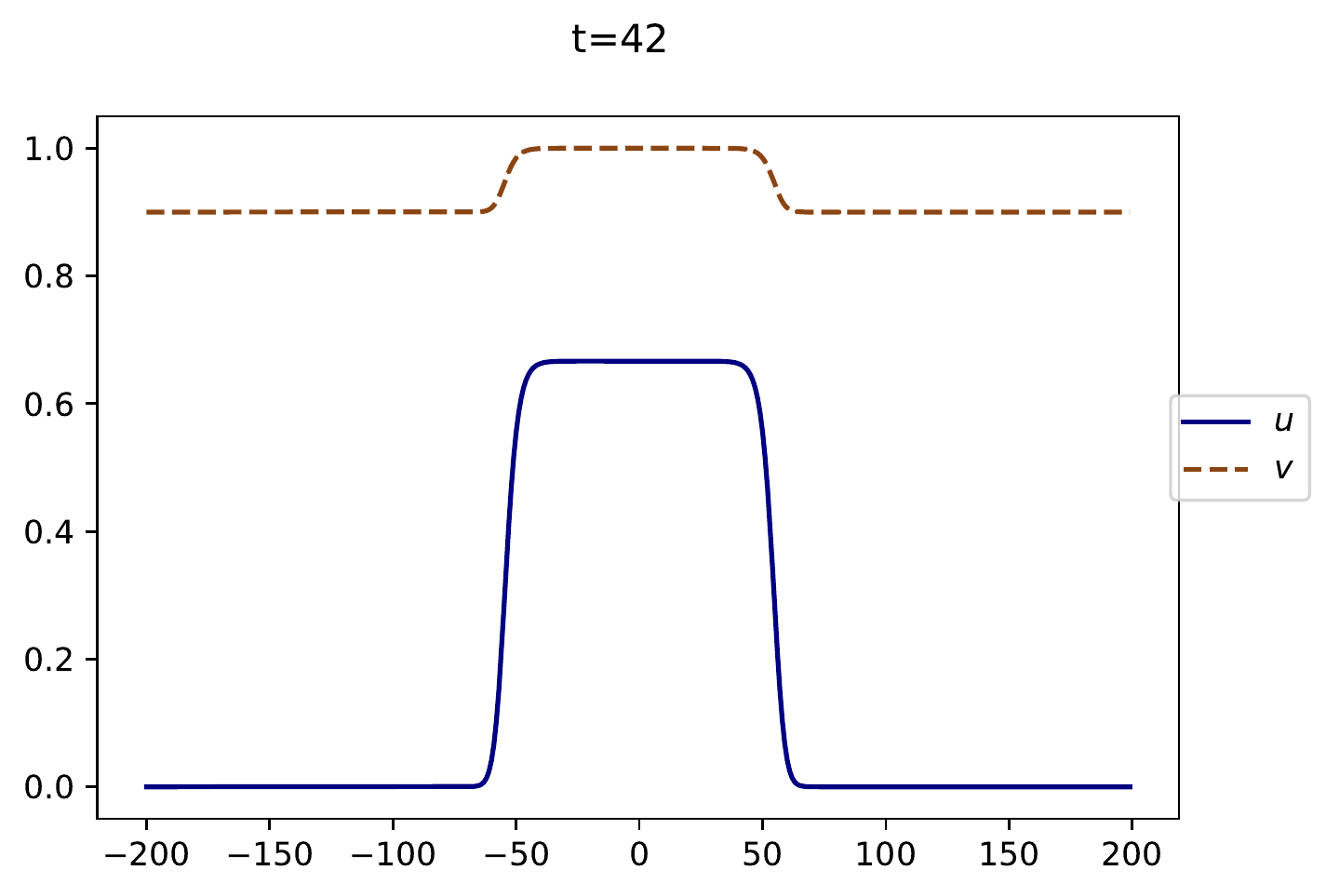}
  \end{subfigure}
  \\
  \begin{subfigure}[p]{0.4\linewidth}
\includegraphics[scale=0.5]{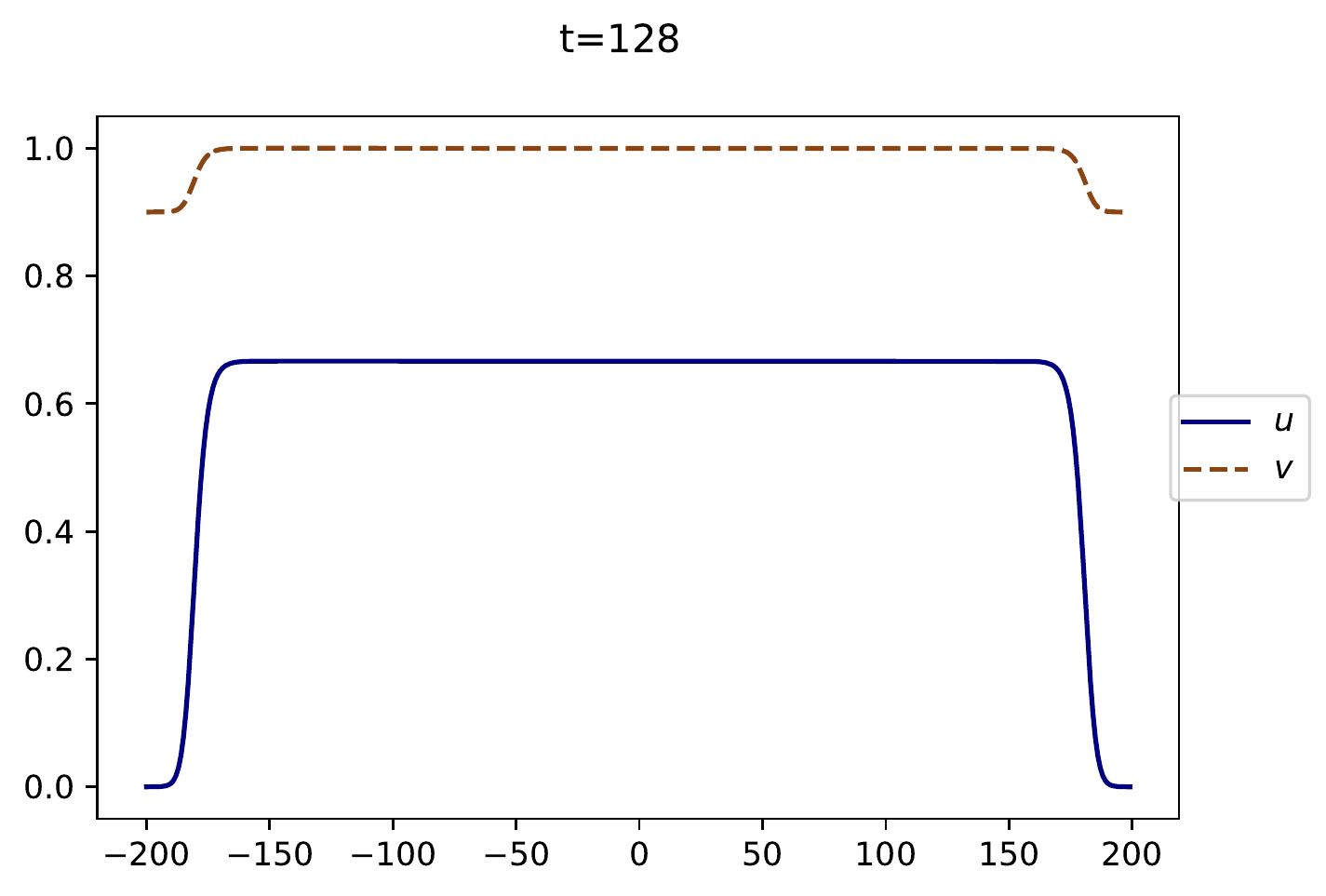}
  \end{subfigure}
  \caption{Enhancing case -- lasting upheaval. Simulation of~\eqref{ExampleEnhancing} with $v_b=0.9>v_\star$. Horizontal axis: space. Blue solid line: $u(t,\cdot)$. Brown dashed line: $v(t,\cdot)$. 
 \href{https://drive.google.com/file/d/1DXV5pa2zVHBNPI1MhXPM24wjmnNKoF1o/view?usp=drive_web}{\color{blue}\underline{Video: Enhancing\_v0=09.mp4}}  
  }
  \label{fig:Enhanc_RevolutionBIS}
\end{figure}

The phenomenon depicted by Figure~\ref{fig:Triggering_eps}
can be explained by the following heuristic: if $u_0$ is large enough, then $u(t,\cdot)$ remains at high values for a sufficiently long time so that $v$ increase and reach values above the threshold $v_\star$.

\subsubsection{Speed of propagation}\label{sec:Enhancing_Numerics_Speed}

We see in \autoref{fig:Enhanc_Revolution} that, if the initial level of social tension $v_b$ 
is above the threshold $v_\star$, the solution of~\eqref{ExampleEnhancing} propagates through space at a constant speed.
If we increase the initial social tension, we observe that the solution propagates faster through space. See \autoref{fig:Enhanc_RevolutionBIS} for a numerical simulation of~\eqref{ExampleEnhancing} with $v_b=0.9$.

To see this phenomenon more clearly, we plot the speed of propagation $c$ of the solution of~\eqref{ExampleEnhancing} as a  function of $v_b$ in \autoref{fig:Enhanc_Speed_Graph} (the speed is computed with the method presented in Section~\ref{sec:InhibSpeedNumerics}).
We see that $c$ is indeed an increasing function of $v_b$.

The theoretical result presented in Section~\ref{sec:SpatialPropagation} states that $c$ ranges in $(c_b,c_1)$ defined in~\eqref{Def_c}.
This is confirmed numerically in \autoref{fig:Enhanc_Speed_Graph}.
\begin{figure}[H]
\center
\begin{subfigure}[p]{0.45\linewidth}
\includegraphics[width=\linewidth]{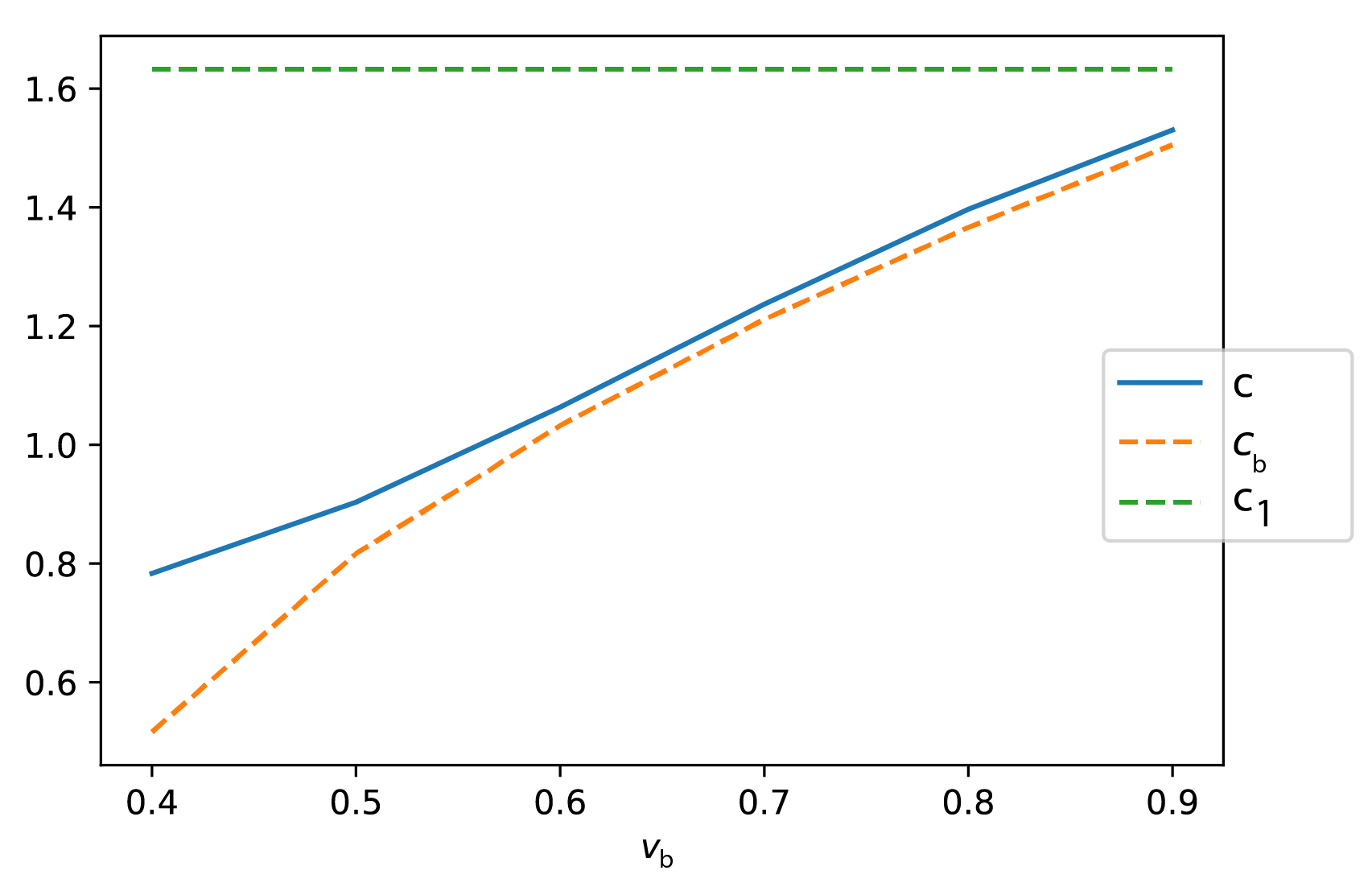}
\caption{On system~\eqref{ExampleEnhancing}.}
\label{fig:Enhanc_Speed_Graph}
  \end{subfigure}
  \hfill
\begin{subfigure}[p]{0.45\linewidth}
\includegraphics[width=\linewidth]{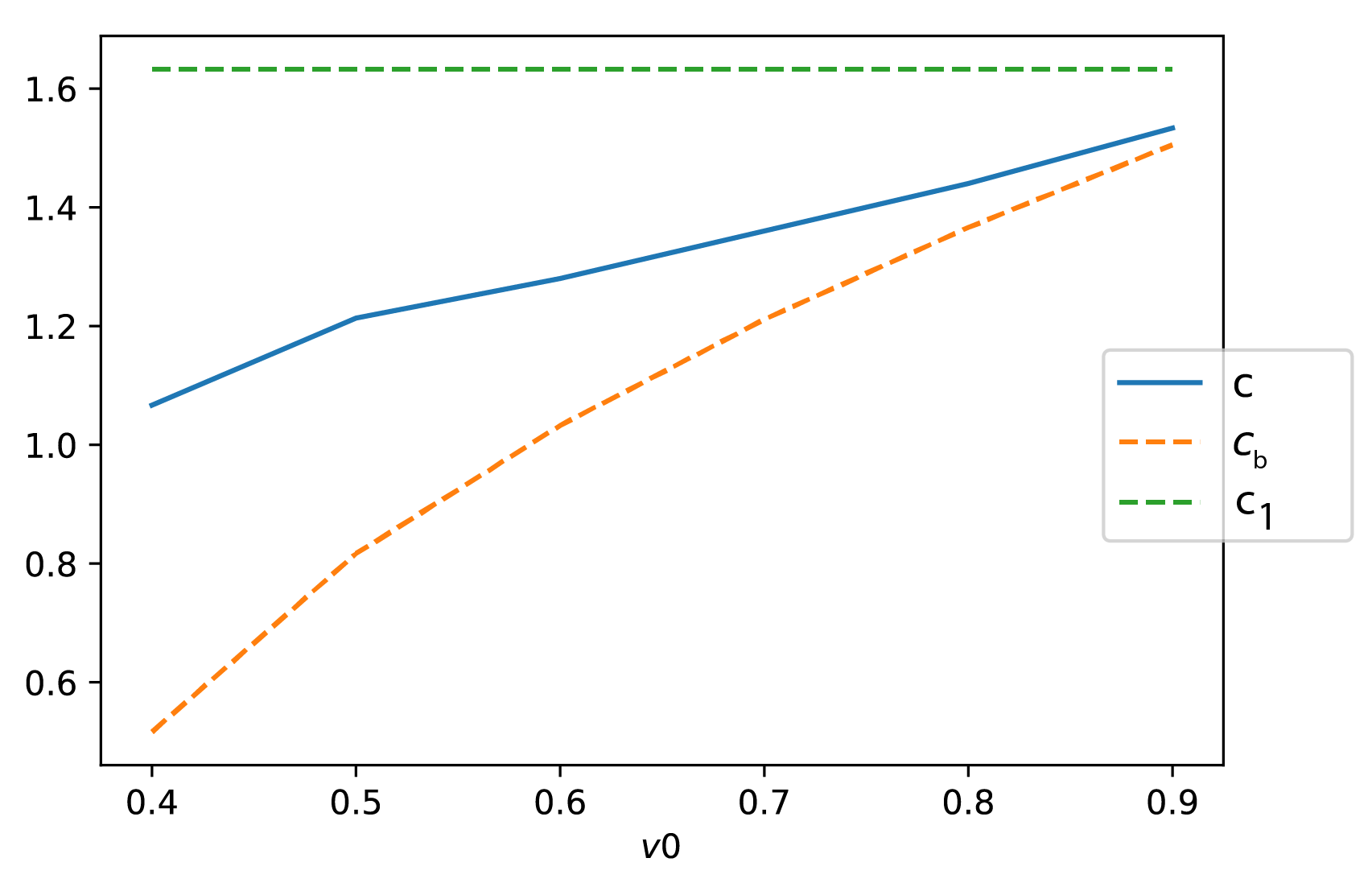}
\caption{On system~\eqref{ExampleEnhancing3} with $d_2=20$.}\label{fig:Speed_SubOpt}
  \end{subfigure}
  \caption{Enhancing case -- speed of propagation.  Blue solid line: empirical speed $c$. Dashed lines: theoretical bounds on the speed, $c_b=2\sqrt{v_b-\frac{1}{3}}$ and $c_1=2\sqrt{1-\frac{1}{3}}$.}
 \end{figure}

\autoref{fig:Enhanc_Speed_Graph} could suggest that $c\approx c_b$, or that $c$ does not depend on the equation on $v$ (i.e. $d_2$ and $\Psi$) like in the inhibiting case (Section~\ref{sec:InhibSpeedNumerics}). However, numerical experiment shows that it is not the case in general. This is a substential difference with the inhibiting case, for which $c=c_b$.
To see this, let us consider the same system~\eqref{ExampleEnhancing} with a diffusion coefficient $d_2$
on the second equation:
\begin{equation}\label{ExampleEnhancing3}
\left\{\begin{aligned}
&\D_t  u -\D_{xx} u=u\left[v(1-u)-\frac{1}{3}\right],\\
&\D_t  v-d_2\D_{xx} v=uv(1-v).
\end{aligned}\right.
\end{equation}
We plot in \autoref{fig:Speed_SubOpt} the speed of propagation $c$ of the solution of~\eqref{ExampleEnhancing3} with $d_2=20$ as a function of $v_b$. We see that $c$ and $c_b$ largely differ, especially for small values of $v_b$.
This is a numerical evidence that $v_b$ depends on $d_2$. 
Fixing $v_b=0.9$, we see on \autoref{fig:Speed_SubOpt_d2} that $c$ is indeed an increasing function of $d_2$. Again, this has to be put in contrast with the inhibiting case, for which $c=c_b$ does not depend on $d_2$.

Let us now investigate how the speed depends on the magnitude of $\Psi$. Consider the analogous of system~\eqref{ExampleEnhancing} with a variable magnitude for the reaction term in the second equation, namely,
\begin{equation}\label{ExampleEnhancing4}
\left\{\begin{aligned}
&\D_t  u -\D_{xx} u=u\left[v(1-u)-\frac{1}{3}\right],\\
&\D_t  v-\D_{xx} v=k uv(1-v),
\end{aligned}\right.
\end{equation}
where $k\geq0$ is a parameter, and $v_b=0.5$. We plot the speed $c$ as a function of $k$ in \autoref{fig:Enhanc_Speed_Graph_k}. We see that $c$ is an increasing function of $k$, and so that it indeed depends on the magnitude of $\Psi$.

\begin{figure}[h]
\center
\begin{subfigure}[p]{0.45\linewidth}
\includegraphics[scale=0.5]{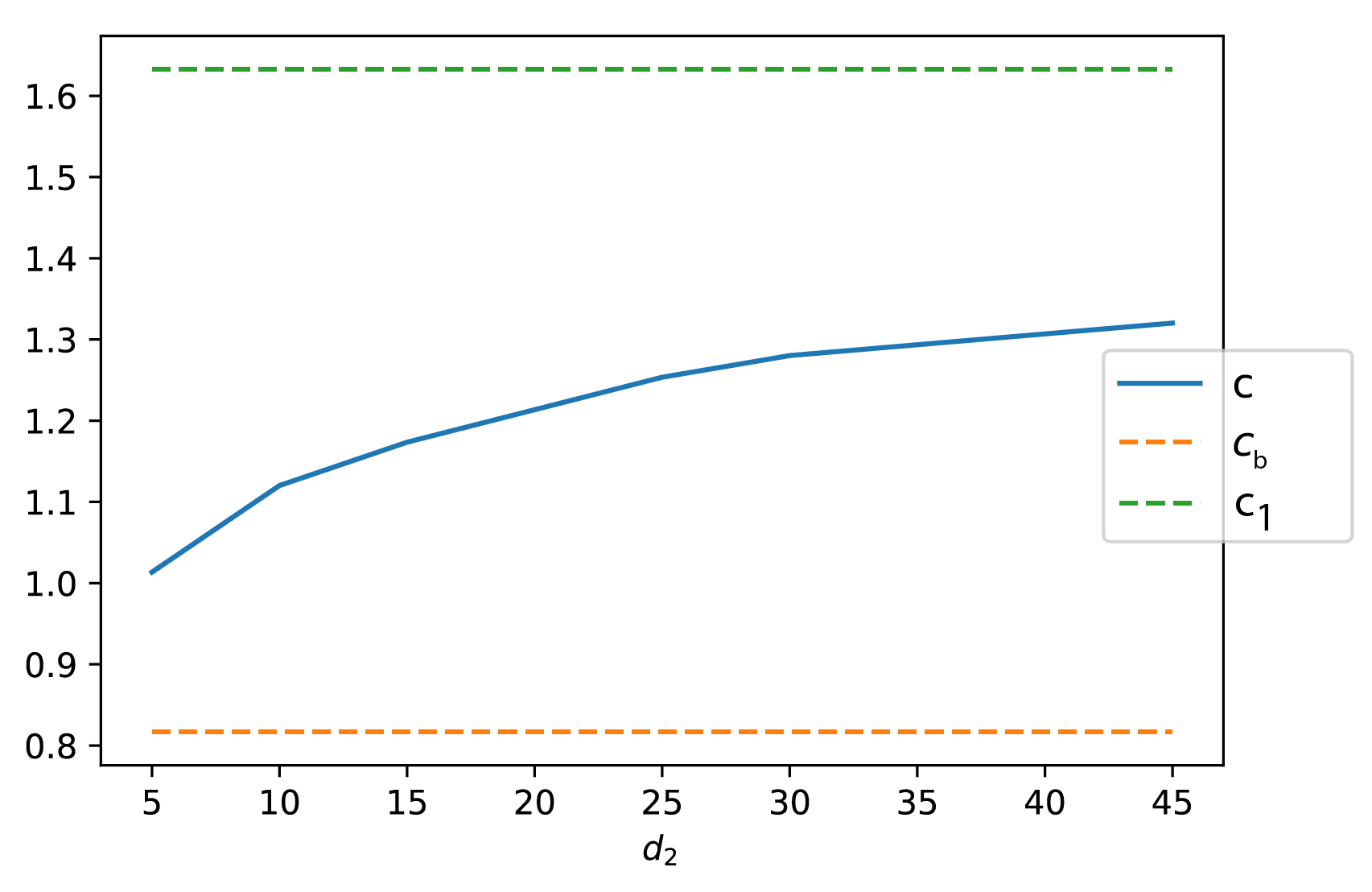}
\caption{Speed $c$ as a function of $d_2$ on the system~\eqref{ExampleEnhancing3} with $v_b=0.9$.}\label{fig:Speed_SubOpt_d2}
  \end{subfigure}
  \hfill
\begin{subfigure}[p]{0.45\linewidth}
\includegraphics[width=\linewidth]{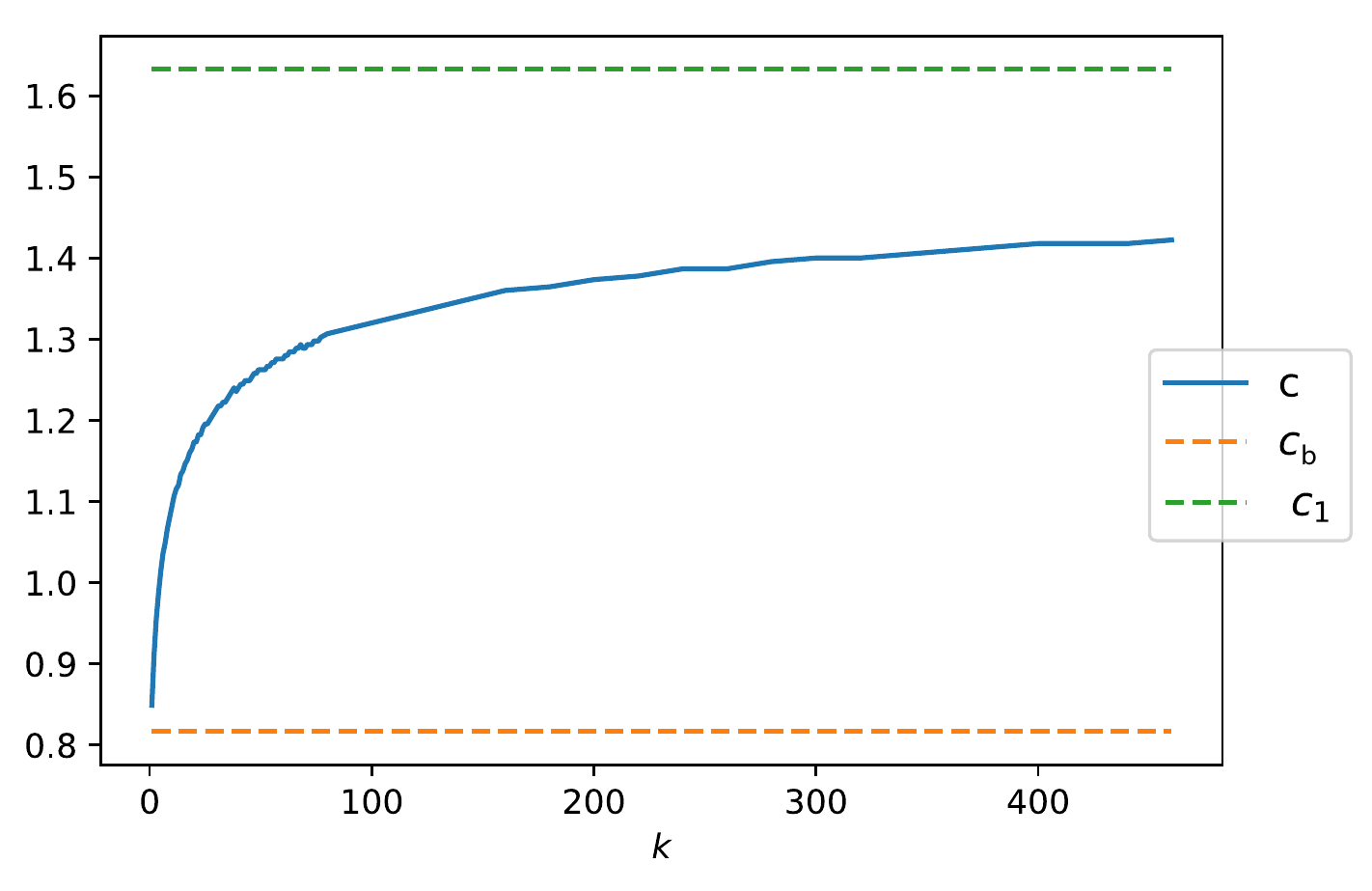}
\caption{Speed $c$ as a function of $k$ on the system~\eqref{ExampleEnhancing4} with $v_b=0.5$.}
\label{fig:Enhanc_Speed_Graph_k}
  \end{subfigure}
  \caption{Enhancing case -- dependence of the speed $c$ on the parameters $d_2$ and $\Psi$}
  \label{fig:Enhanc_Speed_d2k}
\end{figure}

These observations motivate the following questions.\\
\textbf{Open problem:} Do we have $\lim_{d_2\to+\infty} c= c_1$ in~\eqref{ExampleEnhancing3}? Do we have $\lim_{k\to0} c= c_b$ or $\lim_{k\to+\infty} c= c_1$ in~\eqref{ExampleEnhancing4}?

\section{Mixed cases}\label{sec:MixedCase}

We propose in this section to focus on some particular instances of~\eqref{GeneralEquationMotivationFinal} which exhibit more complex behaviors. We begin with a model featuring a double threshold phenomenon
which can give rise to both ephemeral riots and lasting upheaval. 
Next, we present other models where oscillating traveling waves appear.

\subsection{Double threshold: calm-riot-lasting upheaval}
Consider the system
\begin{equation}\label{ExampleMixed}
\left\{\begin{aligned}
&\D_t  u -\D_{xx} u=u\left[v(1-u)-\frac{1}{3}\right],\\
&\D_t  v-d_2\D_{xx} v=uv(1-v)(v-\frac{1}{2}),
\end{aligned}\right.
\end{equation}
and some initial conditions $u_0(x)\gneqq0$ and $v_b\in(0,1)$. We find $v_\star= 1/3$ from definition~\eqref{Def_v_star}.
Depending on the value of $v_b$, we have the following dichotomy:
\begin{itemize}
\item if $v_0\equiv v_b<1/2$: the system is tension inhibiting~\eqref{hyp:inhibiting},
\item if $v_0\equiv v_b>1/2$: the system is tension enhancing~\eqref{AssumptionTensionEnhancing}.
\end{itemize}
Indeed, note that $v_0<1/2$ implies that $v(t,x)\in (0,1/2)$ and $v_0>1/2$ that $v(t,x)$; one can thus restrict $\Psi$ to a suitable range where $\Psi(\cdot,v(t,x))$ is either positive or negative.

We can then apply the results of the previous sections to infer that a small \emph{triggering event} can lead to three different situations:
\begin{itemize}
\item For a small initial level of social tension $v_b\in(0,v_\star)$:  \emph{return to calm}  (Section~\ref{sec:Analysis}).
\item For an intermediate initial level of social tension $v\in(v_\star,1/2)$: \emph{burst of an ephemeral riot} (Section~\ref{sec:inhibiting}).
\item For a high initial level of social tension $v\in(1/2,1)$: \emph{burst of a lasting upheaval} (Section~\ref{sec:Analysis}).
\end{itemize}
From the modeling point of view, this double threshold phenomenon means that, in the model~\eqref{ExampleMixed}, the initial level of social tension determines whether a small triggering event ignites a social movement, and also whether the movement will be ephemeral or persisting.

We propose to illustrate this double threshold phenomenon with numerical simulations of~\eqref{ExampleMixed}. We take $v_0\equiv v_b\in(0,1)$ and $u_0(x)=0.1(1-x^2)_+$.
Figure~\ref{fig:v0_3} illustrates the \emph{return to calm} for $v_b=0.3$ ; Figure~\ref{fig:v0_4} the propagation of a \emph{riot} for $v_b=0.4$ ; Figure~\ref{fig:v0_6} the propagation of a \emph{lasting upheaval} for $v_b=0.6$. 

\begin{figure}[p]
  \centering
  \begin{subfigure}[p]{0.45\linewidth}
    \centering\includegraphics[width=\textwidth]{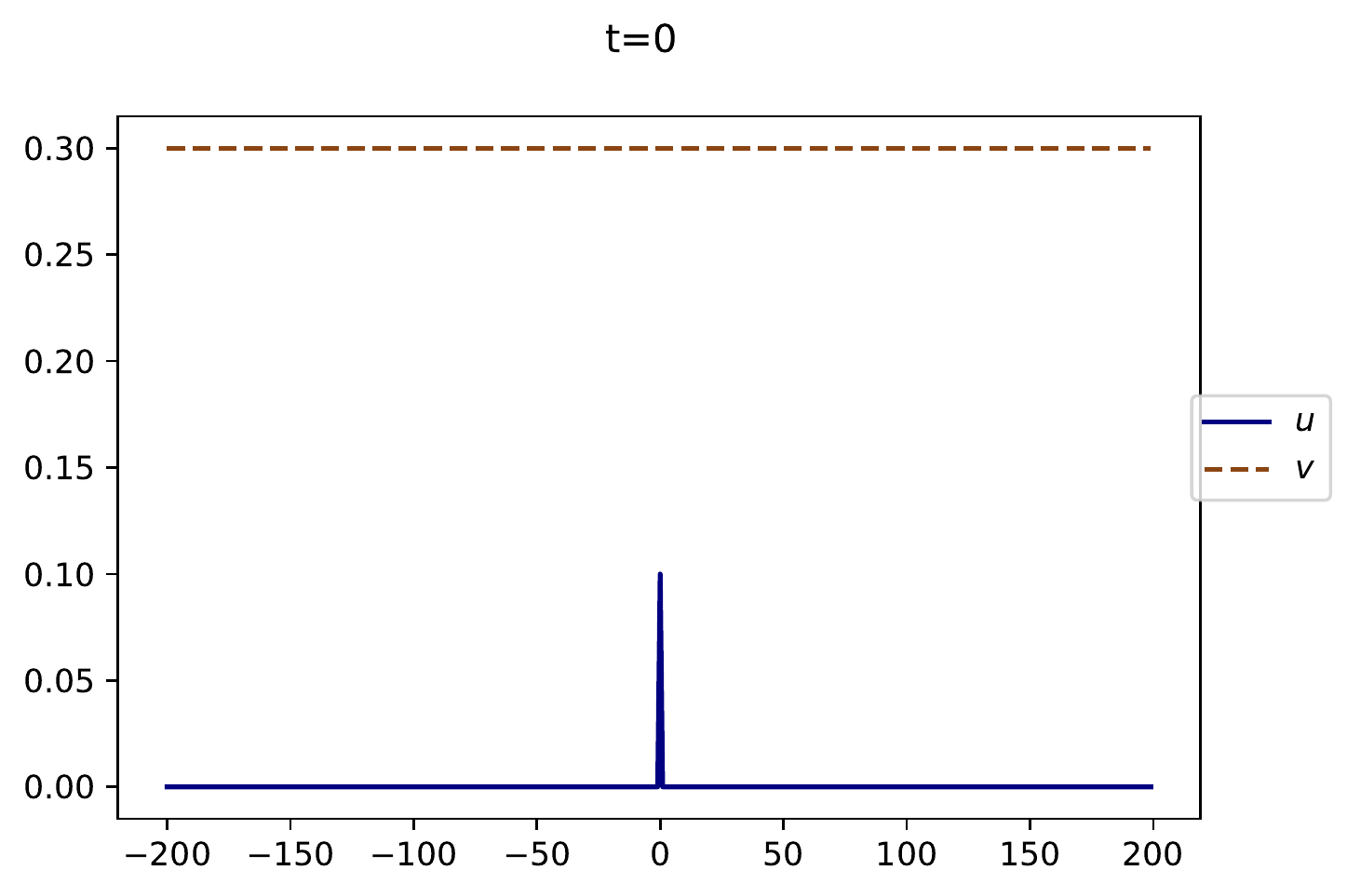}
  \end{subfigure}
  \hfill
  \begin{subfigure}[p]{0.45\linewidth}
    \centering\includegraphics[width=\textwidth]{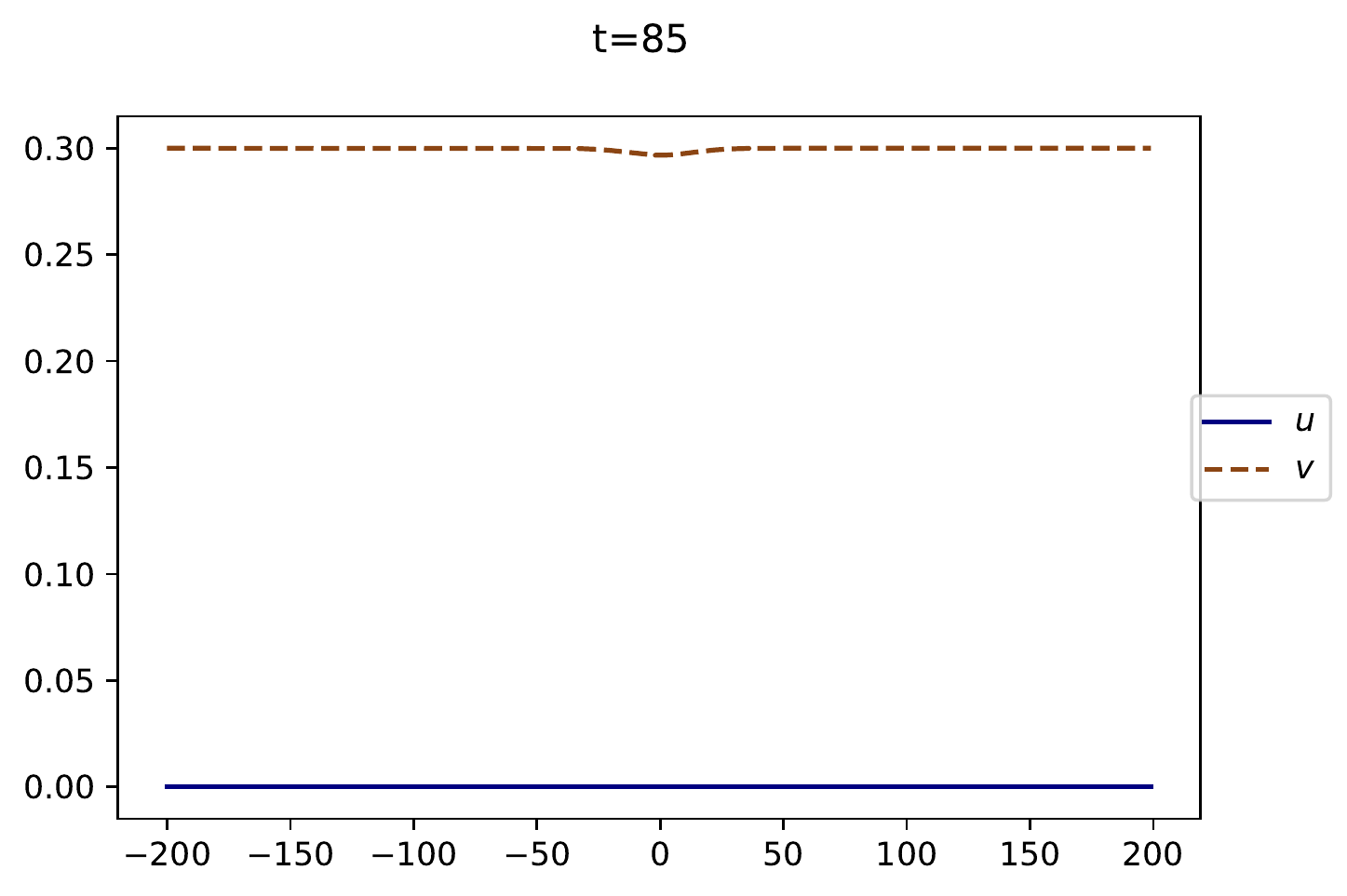}
  \end{subfigure}
\caption{Mixed Case -- Calm. Snapshots at different times of the solution of \eqref{ExampleMixed} with  $v_b=0.3<v_\star$. Horizontal axis: space. Blue solid line: $u(t,\cdot)$. Brown dashed line: $v(t,\cdot)$.
 \href{https://drive.google.com/file/d/1KVBRa7s9EXdSlj0jdgYwlpi5UMj9yj4-/view}{\color{blue}\underline{Video: Mixed\_V0=03.mp4}}  
} 
\label{fig:v0_3}
\end{figure}

\begin{figure}[p]
  \centering
  \begin{subfigure}[p]{0.45\linewidth}
    \centering\includegraphics[width=\textwidth]{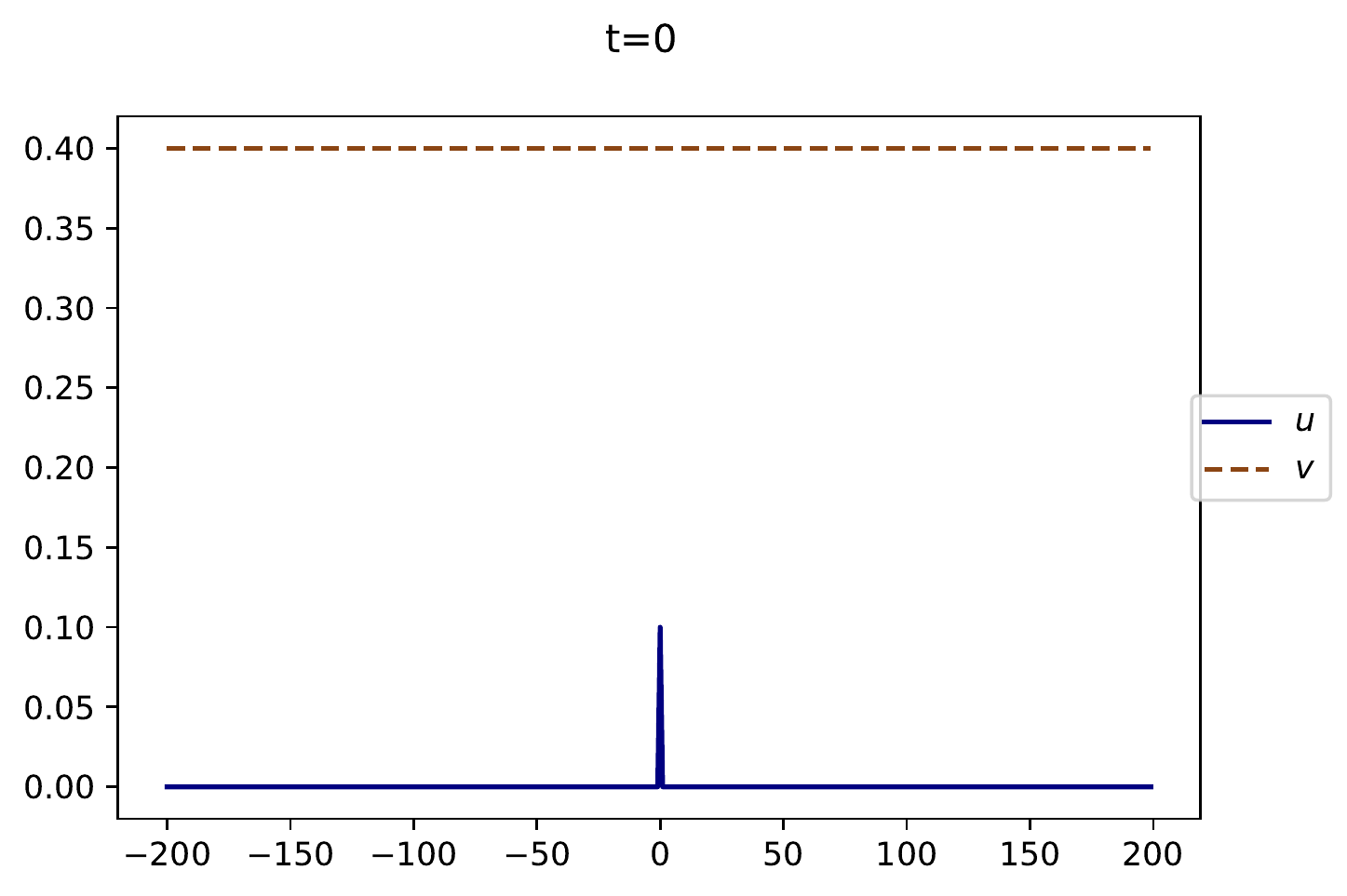}
  \end{subfigure}
  \hfill
  \begin{subfigure}[p]{0.45\linewidth}
    \centering\includegraphics[width=\textwidth]{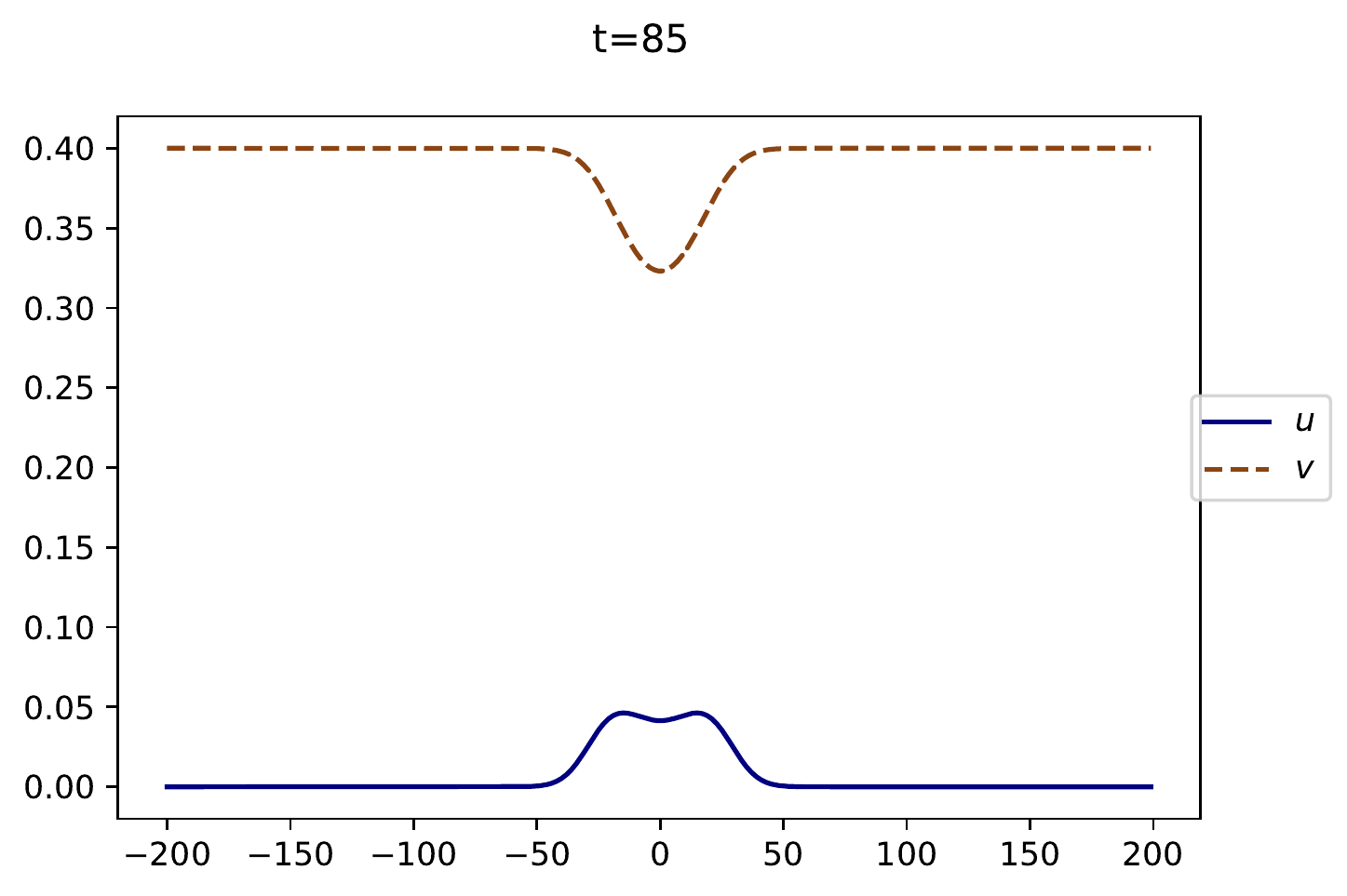}
  \end{subfigure}\\
    \begin{subfigure}[p]{0.45\linewidth}
    \centering\includegraphics[width=\textwidth]{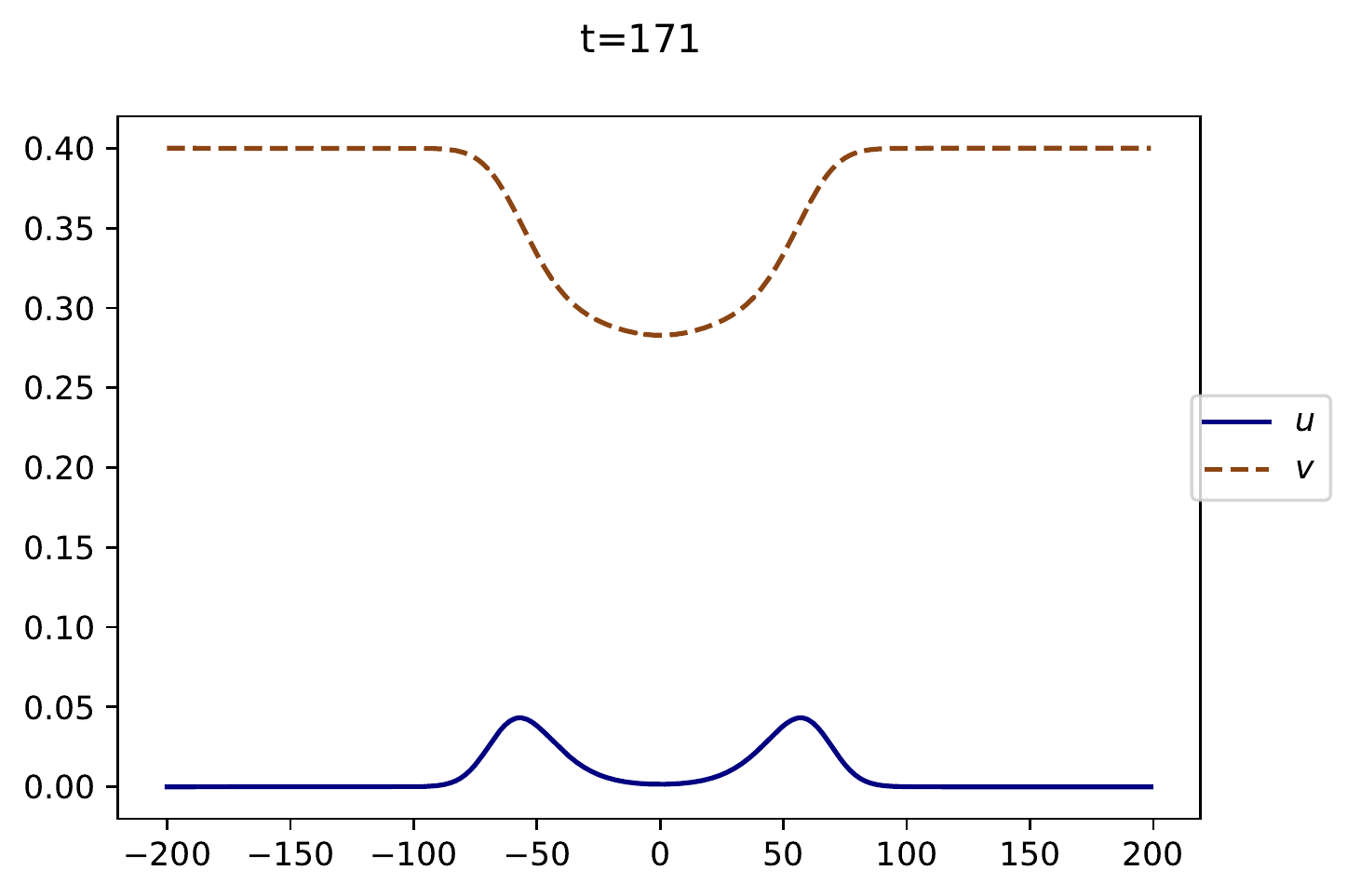}
  \end{subfigure}
  \hfill
  \begin{subfigure}[p]{0.45\linewidth}
    \centering\includegraphics[width=\textwidth]{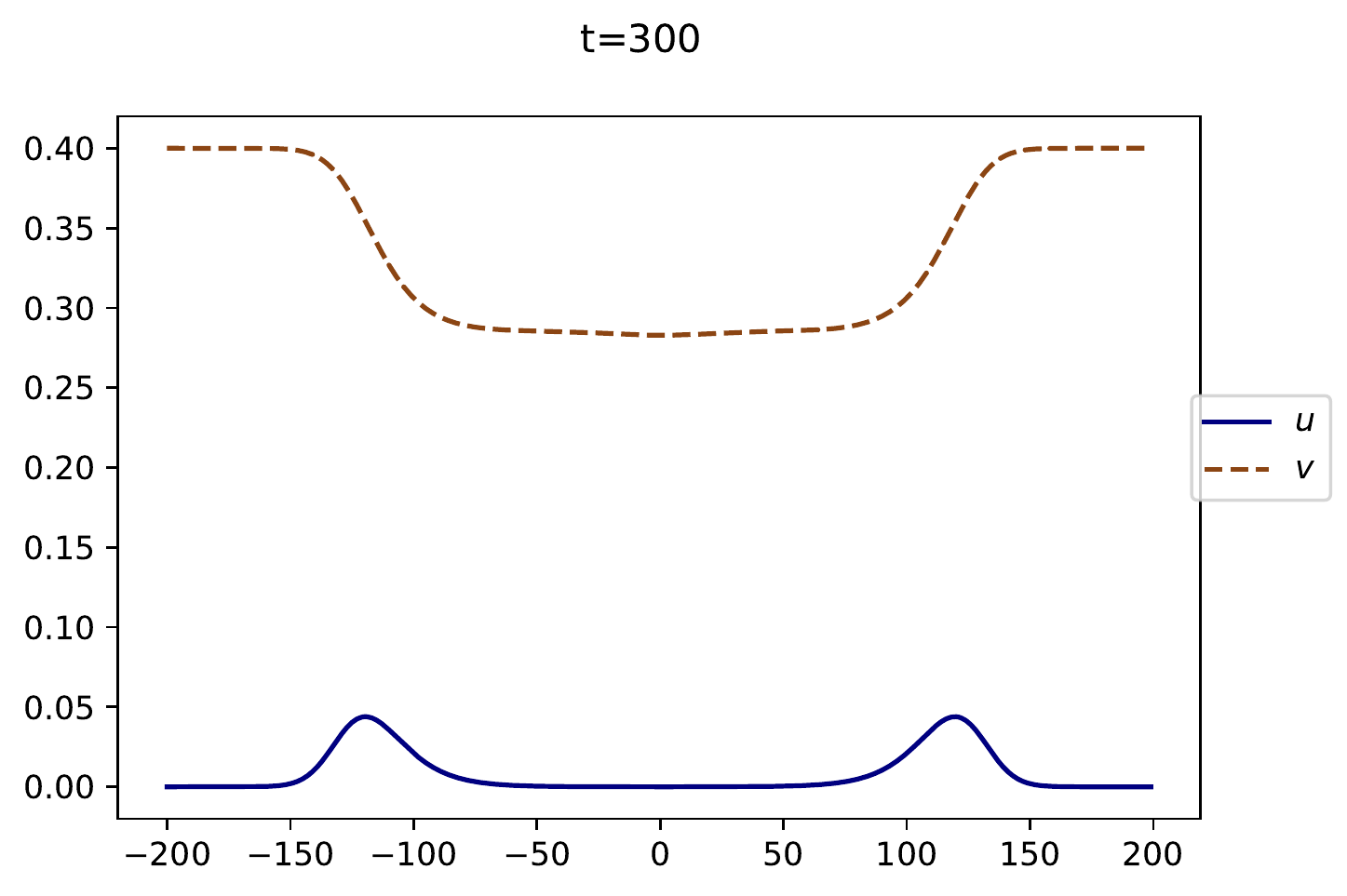}
  \end{subfigure}
\caption{Mixed Case -- riot. Snapshots at different times of the solution of \eqref{ExampleMixed} with $v_b=0.4\in(v_\star,1/2)$. Horizontal axis: space. Blue solid line: $u(t,\cdot)$. Brown dashed line: $v(t,\cdot)$.
 \href{https://drive.google.com/file/d/1yTrCbvqhhz4ww2QVMomwKawjVt4ZHVHb/view}{\color{blue}\underline{Video: Mixed\_V0=04.mp4}}  
} \label{fig:v0_4}
\end{figure}

\begin{figure}[p]
  \centering
  \begin{subfigure}[p]{0.45\linewidth}
    \centering\includegraphics[width=\textwidth]{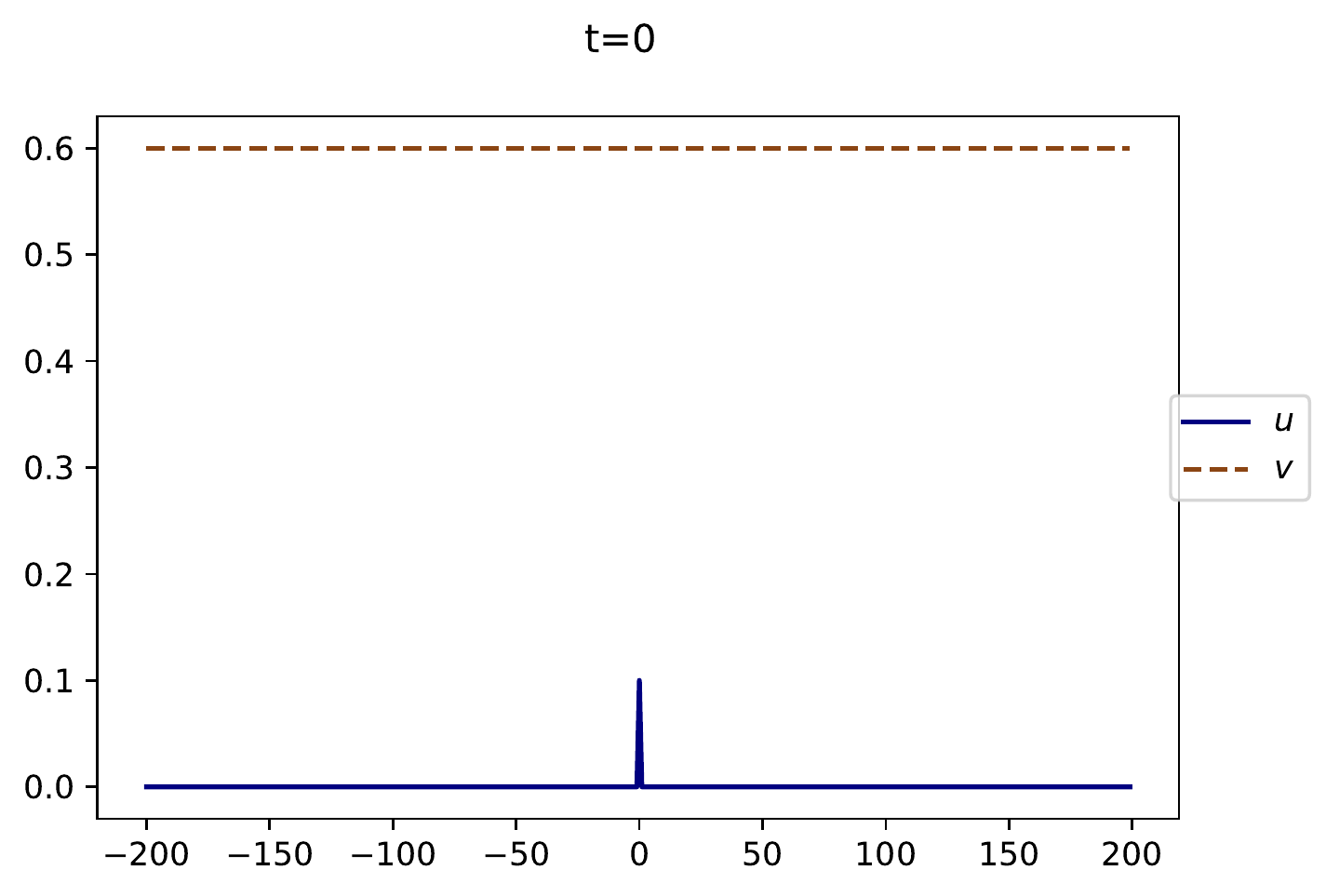}
  \end{subfigure}
  \hfill
  \begin{subfigure}[p]{0.45\linewidth}
    \centering\includegraphics[width=\textwidth]{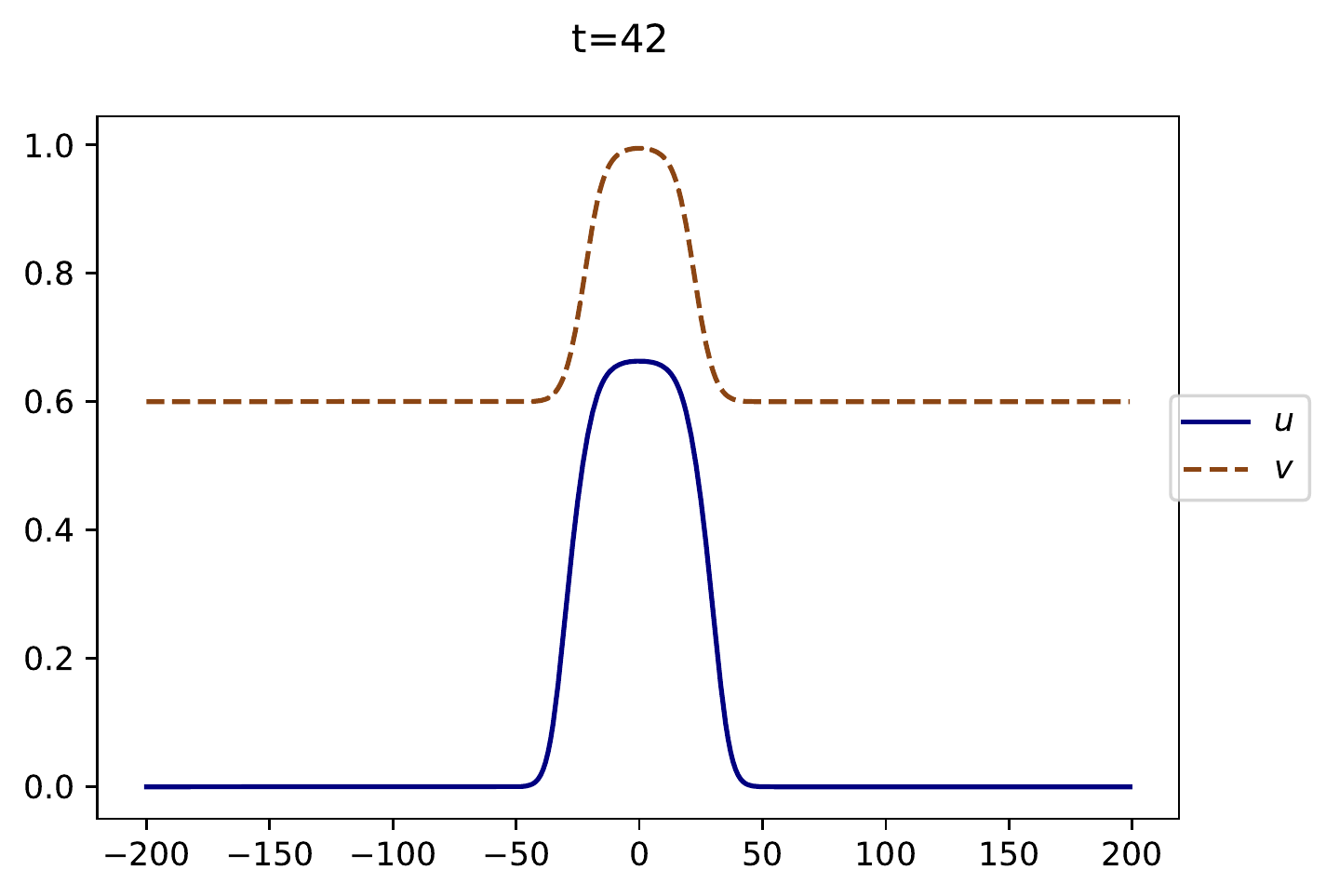}
  \end{subfigure}\\
    \begin{subfigure}[p]{0.45\linewidth}
    \centering\includegraphics[width=\textwidth]{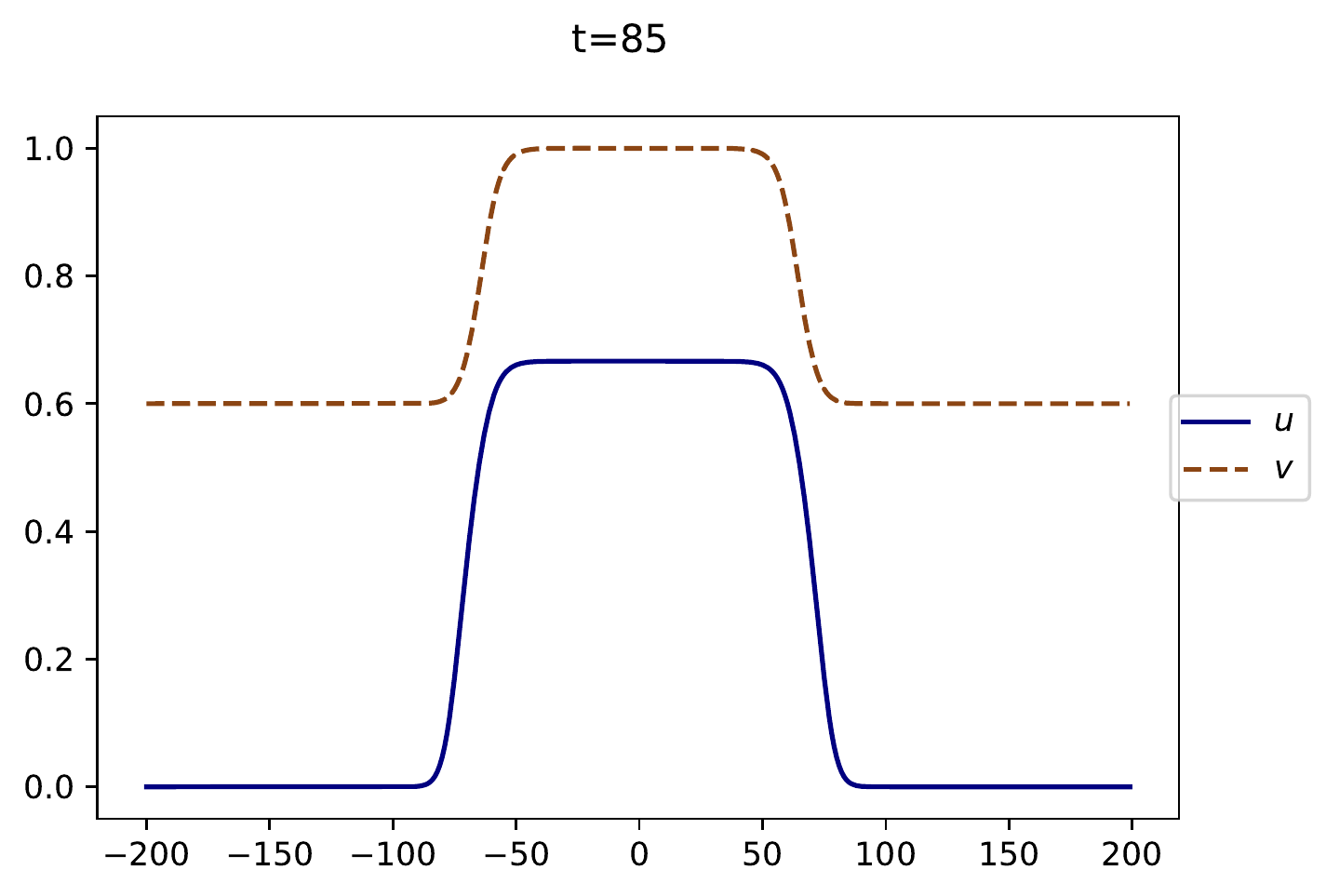}
  \end{subfigure}
  \hfill
  \begin{subfigure}[p]{0.45\linewidth}
    \centering\includegraphics[width=\textwidth]{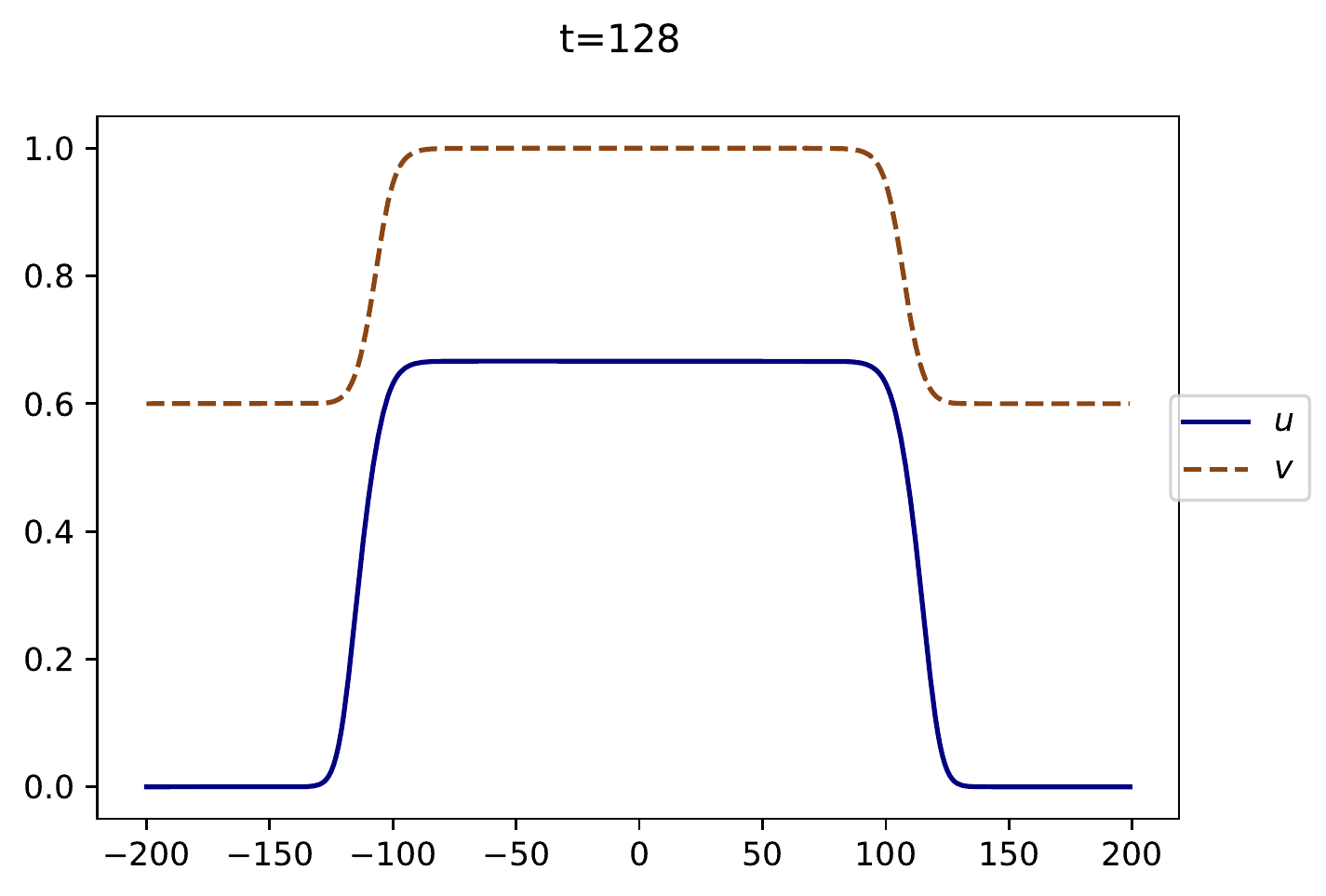}
  \end{subfigure}
\caption{Mixed Case -- lasting upheaval. Snapshots at different times of the solution of \eqref{ExampleMixed} for $v_b=0.6\in(1/2,1)$.
 \href{https://drive.google.com/file/d/1xfBlZgMDpt92-kHoMWyCT-n100rTVVSf/view}{\color{blue}\underline{Video: Mixed\_V0=06.mp4}}  
} \label{fig:v0_6}
\end{figure}

\subsection{Oscillating traveling waves}

We saw in Sections~\ref{sec:Inhibiting_TravelingWave},~\ref{sec:enhancing_TW} that the traveling wave can have the shape of a bump or a monotonic wave. Yet, some traveling waves may have a more complex shape. For example, consider the following particular instance of~\eqref{GeneralEquationMotivationFinal}
\begin{equation}\label{ExampleMixed3}
\left\{\begin{aligned}
&\D_t  u -\D_{xx} u=u\left[v(1-u)-\frac{1}{3}\right],\\
&\D_t  v-\D_{xx} v=uv(1-v)(v-10u).
\end{aligned}\right.
\end{equation}
We see in~\autoref{fig:Oscillation} that the solution converges to a traveling wave which features damped oscillations on its tail.

We suspect that some well-chosen parameters could generate travelings wave with undamped oscillation; yet, we are not able to produce such an example.

\begin{figure}[p]
  \centering
  \begin{subfigure}[p]{0.45\linewidth}
    \centering\includegraphics[width=\textwidth]{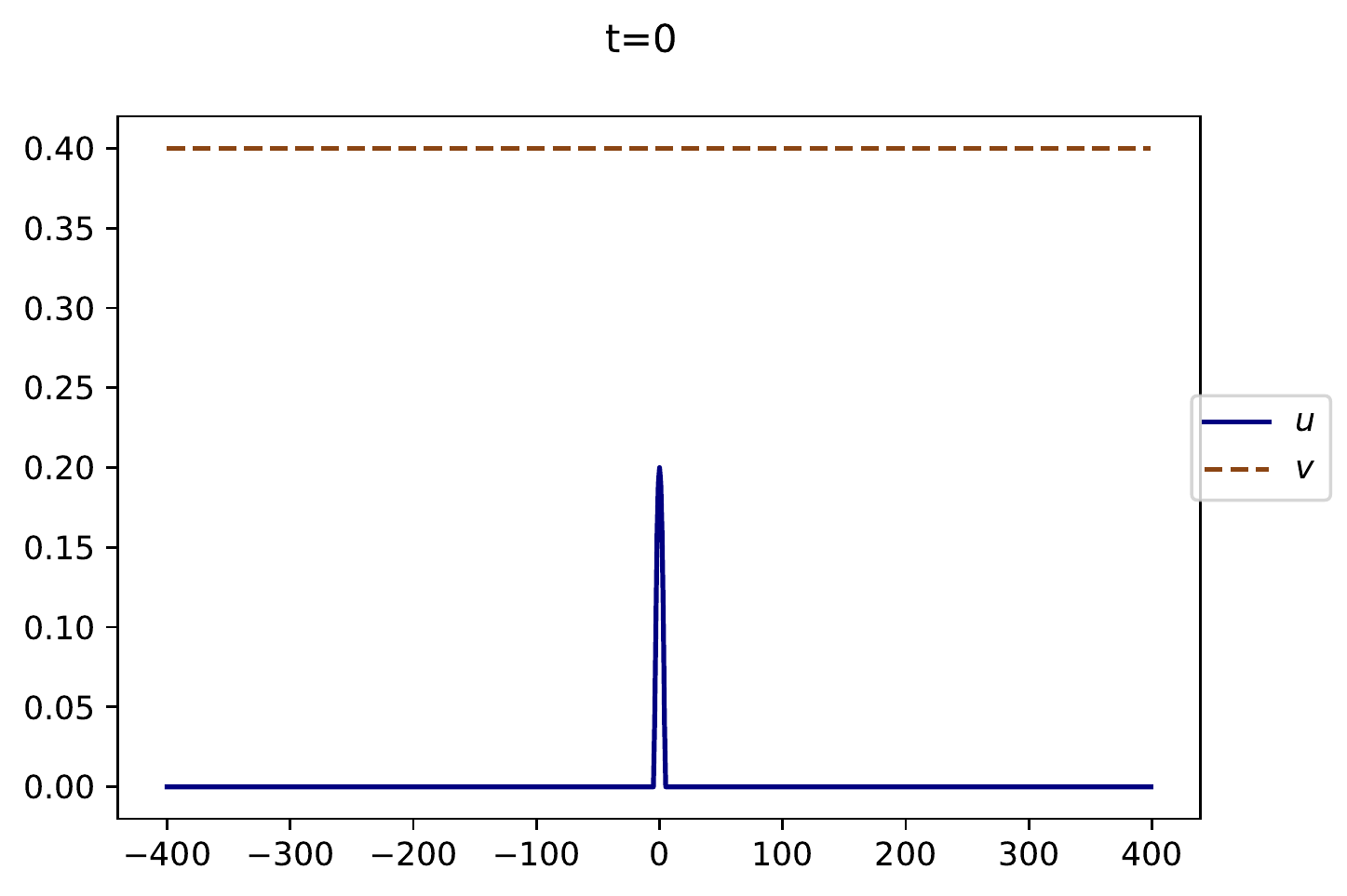}
  \end{subfigure}
  \hfill
  \begin{subfigure}[p]{0.45\linewidth}
    \centering\includegraphics[width=\textwidth]{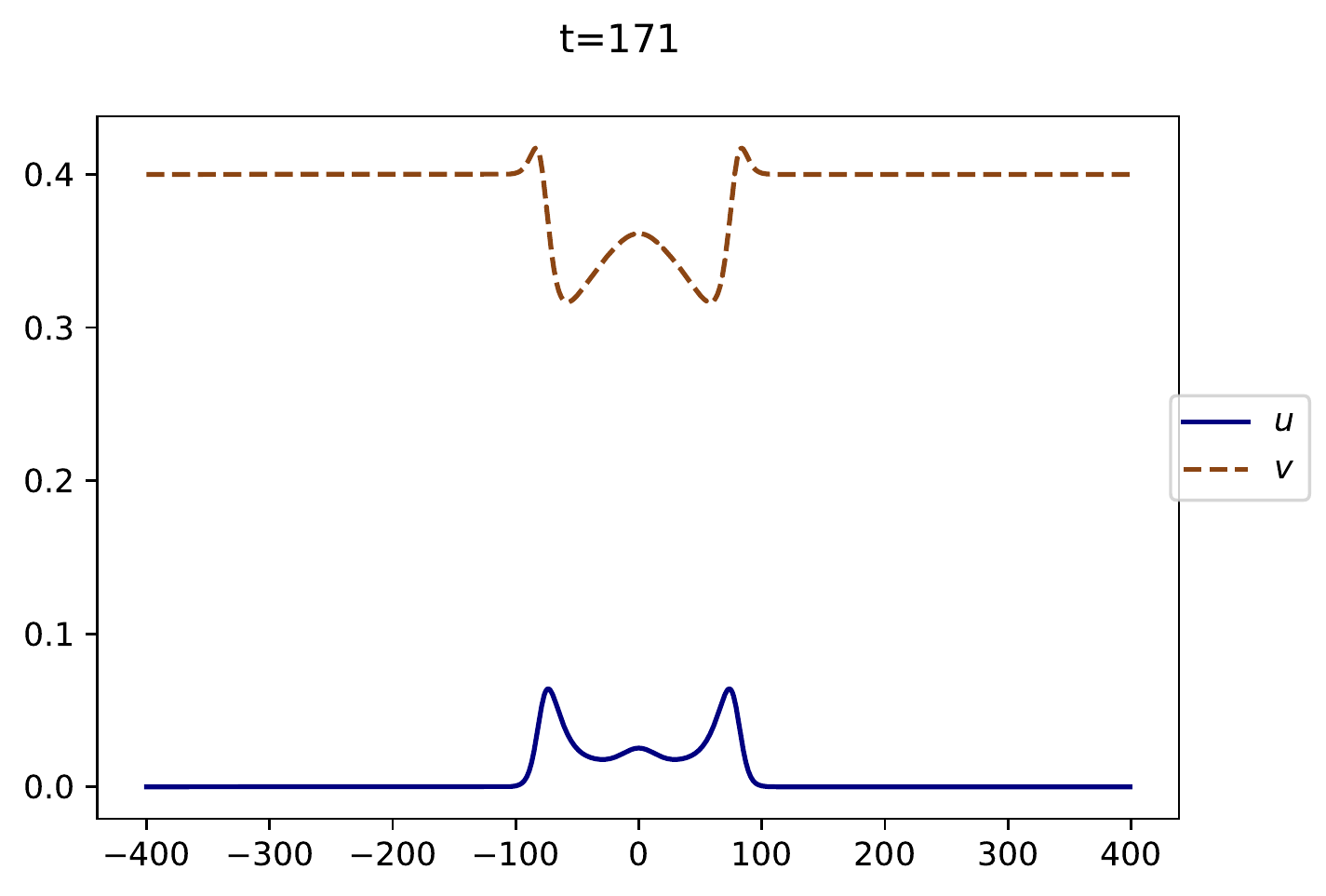}
  \end{subfigure}\\
    \begin{subfigure}[p]{0.45\linewidth}
    \centering\includegraphics[width=\textwidth]{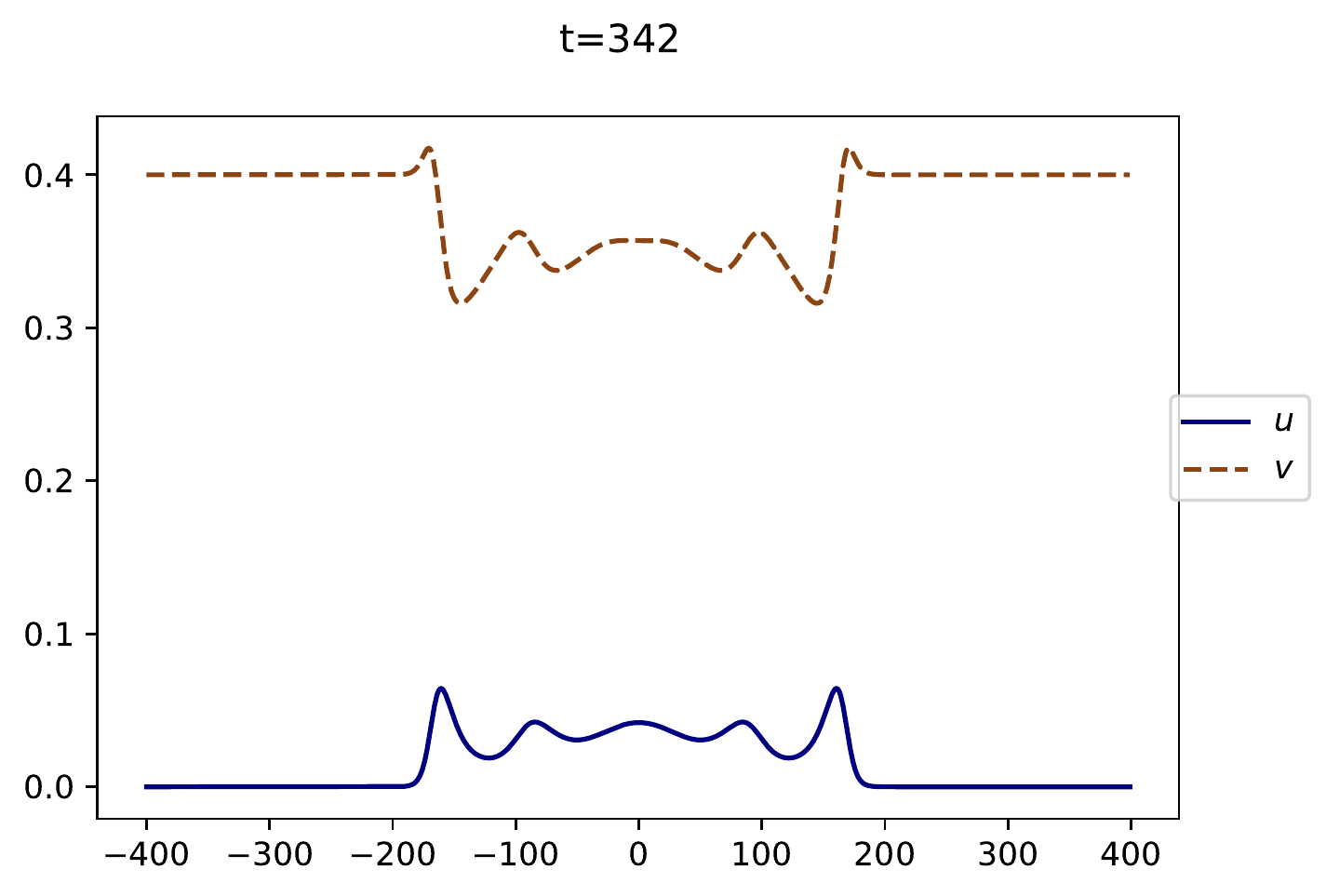}
  \end{subfigure}
  \hfill
  \begin{subfigure}[p]{0.45\linewidth}
    \centering\includegraphics[width=\textwidth]{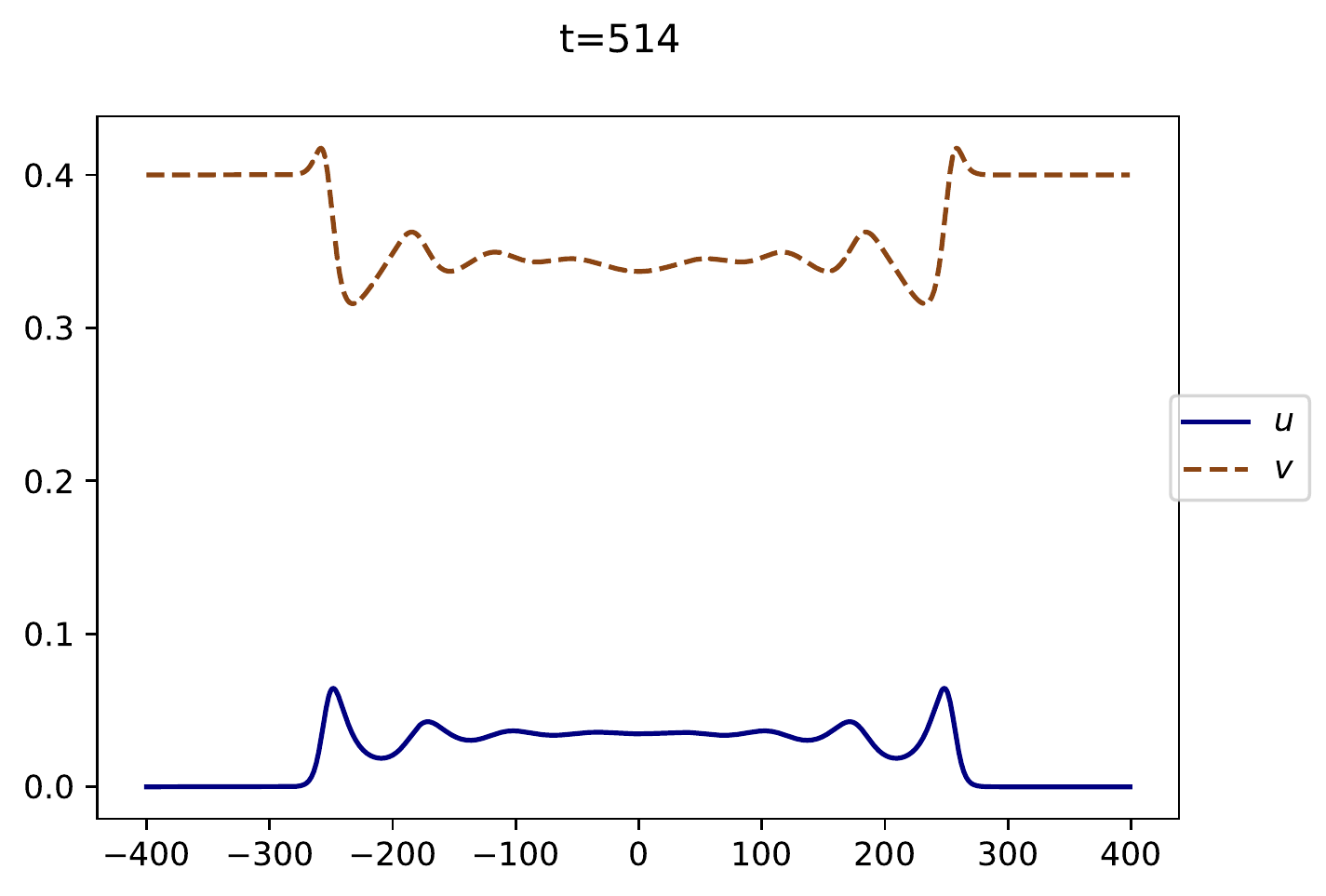}
  \end{subfigure}
\caption{Mixed Case -- Oscillating traveling wave. Snapshots at different times of the solution of \eqref{ExampleMixed3} for $v_b=0.4$ and $u_0(x)=0.2(1-\frac{x^2}{100})_+$.
\href{https://drive.google.com/file/d/1X07uf1xZsUDaQrcp2mPlTHzGGILawBp5/view}{\color{blue}\underline{Video: Oscillation\_v0=04.mp4}}  
} \label{fig:Oscillation}
\end{figure}

\subsection{Magnitude of the triggering event and terraces}\label{sec:Mixed_MagnitudeTriggering}

As described in Section~\ref{sec:enhancing_MagnitudeTriggering}, when the system is not \emph{tension inhibiting}, the magnitude of the triggering event (i.e. the size of $u_0$) may be of crucial importance to determine the regime of the dynamics when $v_0\equiv v_b<v_\star$. 

The same phenomenon may occur even in the case $v_0\equiv v_b>v_\star$.
Indeed, consider the system
\begin{equation}\label{ExampleMixed2}
\left\{\begin{aligned}
&\D_t  u -\D_{xx} u=u\left[v(1-u)-\frac{1}{3}\right],\\
&\D_t  v-\D_{xx} v=uv(1-v)(v+u-\frac{1}{2}).
\end{aligned}\right.
\end{equation}
We find $v_\star=1/3$ from definition~\eqref{Def_v_star}.
Fixing $v_0\equiv v_b>v_\star$ close to the threshold $1/2$, we expect that the magnitude of the triggering event determines whether the dynamics give rise to an ephemeral riot or a lasting upheaval. We see in \autoref{fig:Double_Threshold_eps=0.1} that, taking $v_b=0.4$ and $u_0(x)=0.1(1-\frac{x^2}{10})_+$, system \eqref{ExampleMixed2} gives rise to an ephemeral \emph{riot}. On the contrary, with the same initial level of social tension $v_b=0.4$, but with a larger triggering event $u_0(x)=0.5(1-\frac{x^2}{10})_+$, we observe a \emph{lasting upheaval} in \autoref{fig:Double_Threshold_eps=0.5}.

\begin{figure}[p]
  \begin{subfigure}[p]{\linewidth}
\includegraphics[width=0.45\textwidth]{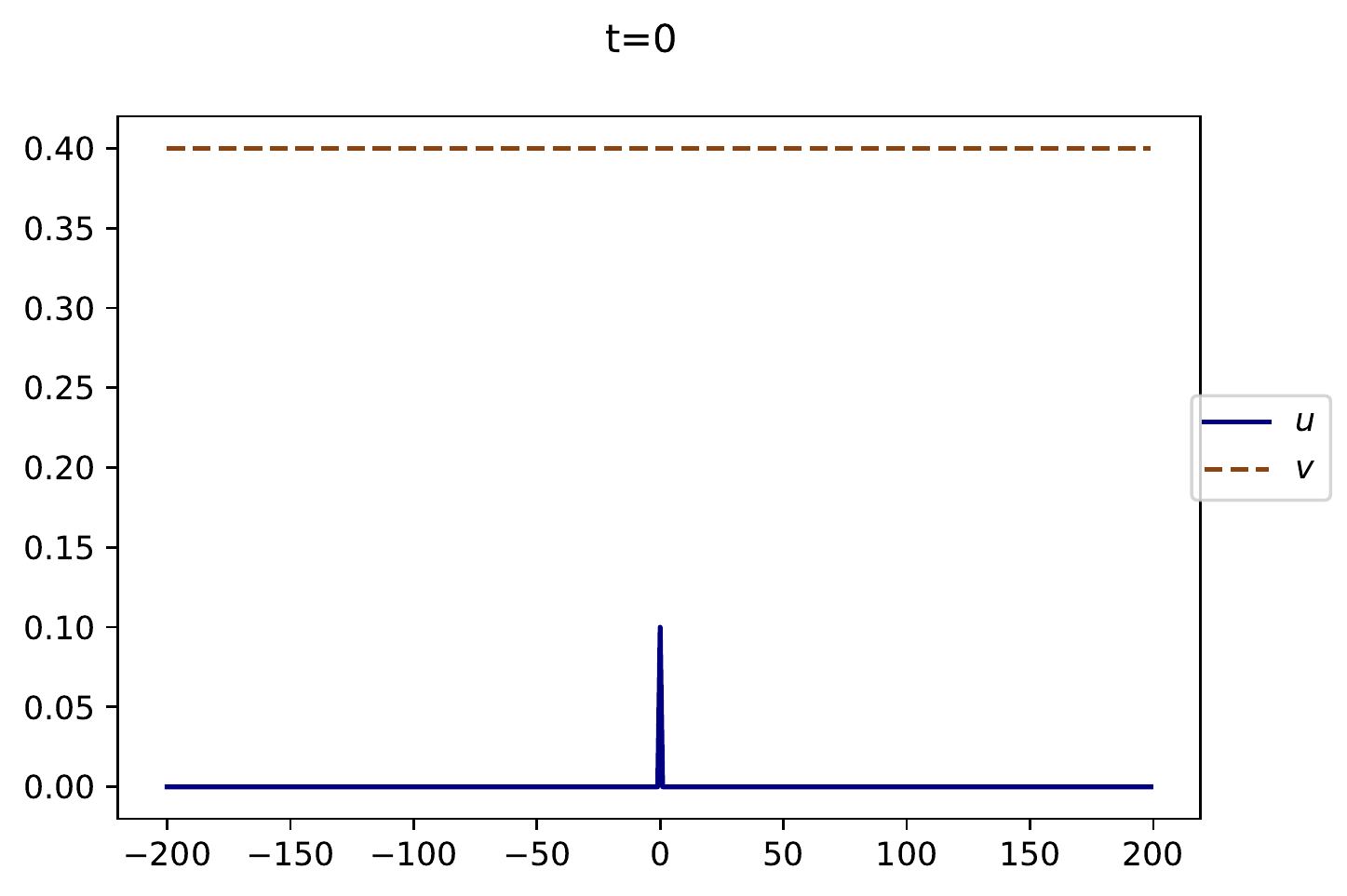}
\hfill
\includegraphics[width=0.45\textwidth]{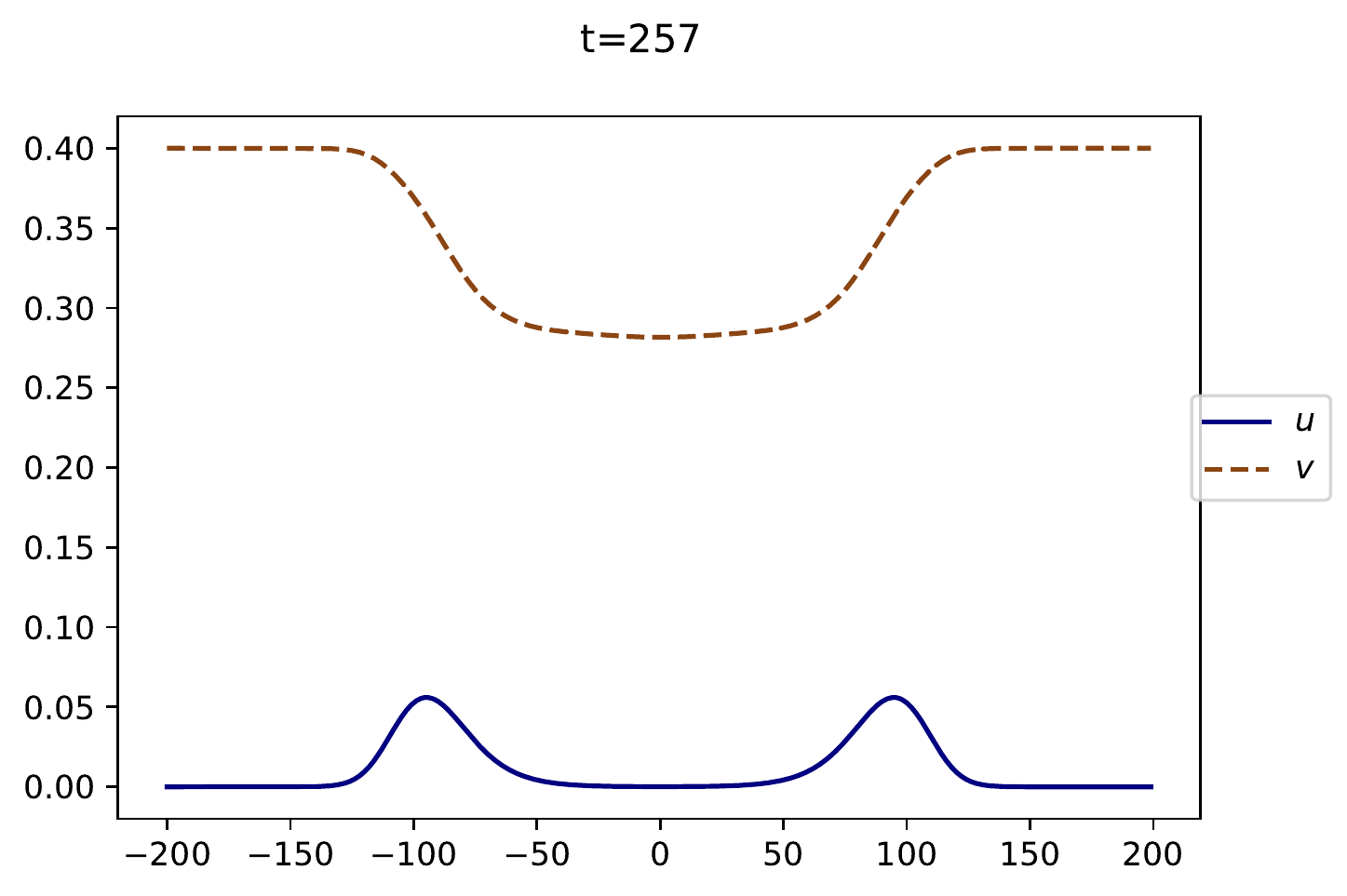}

\caption{Riot for $u_0(x)=0.1(1-\frac{x^2}{10})_+$.
\href{https://drive.google.com/file/d/1JakO5wnqsPXBQin5g_XNwPbGy_0C5JcG/view}{\color{blue}\underline{Video: Threshold\_Riot-Revolution\_u0=01.mp4}}  
} \label{fig:Double_Threshold_eps=0.1}
  \end{subfigure}
 
 \begin{subfigure}[p]{\linewidth}
    \centering\includegraphics[width=0.45\textwidth]{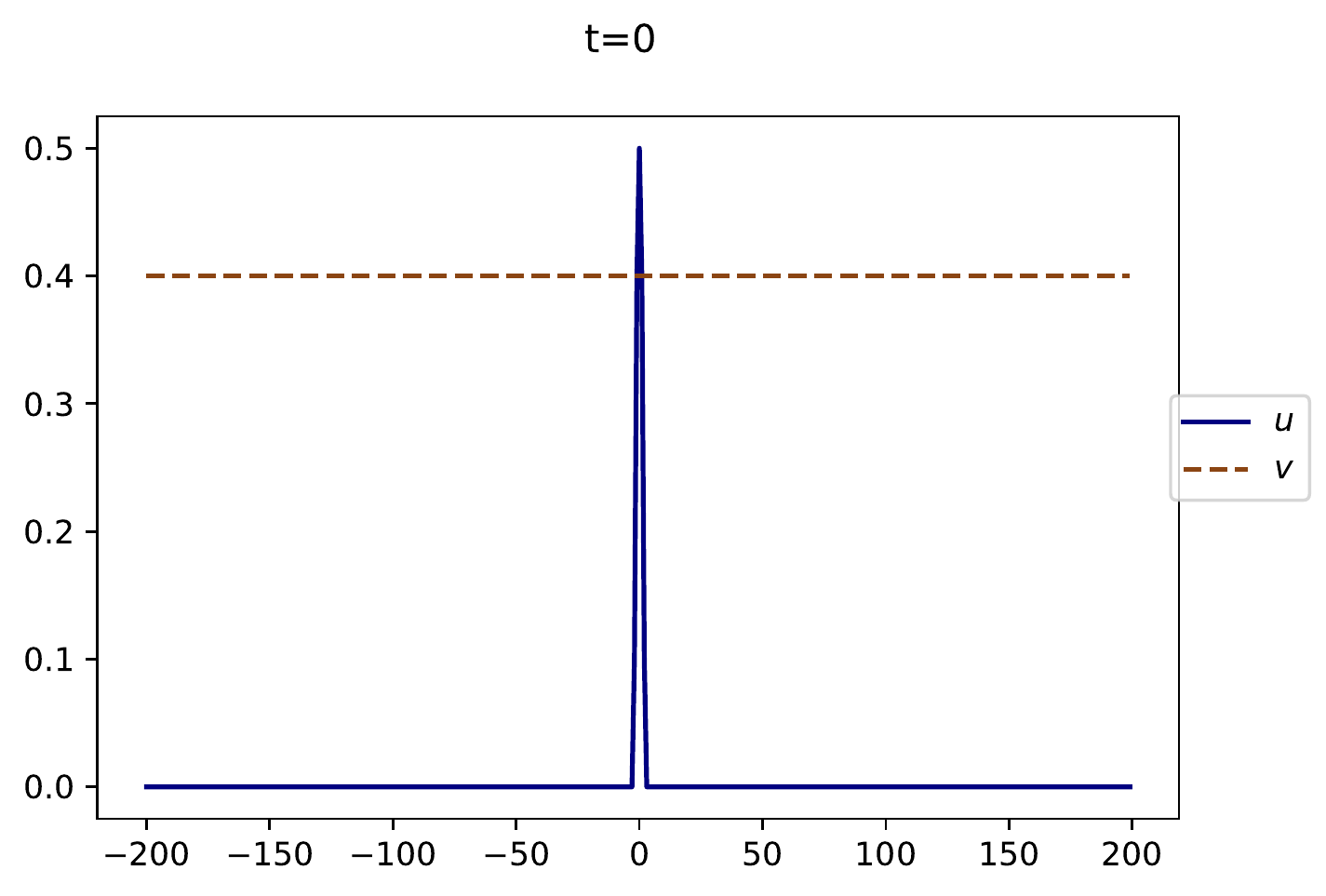}
  \hfill
    \centering\includegraphics[width=0.45\textwidth]{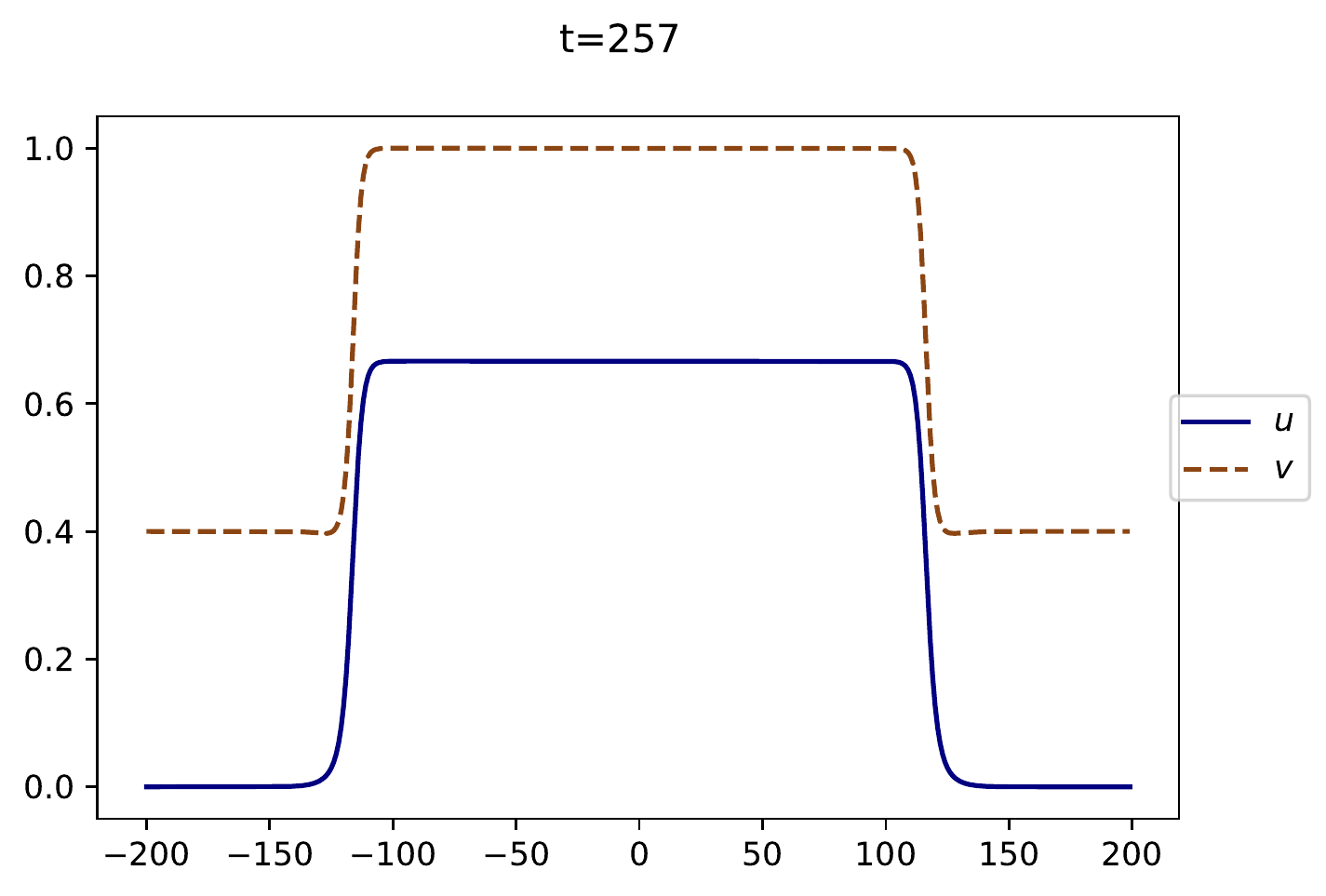}
\caption{lasting upheaval for $u_0(x)=0.5(1-\frac{x^2}{10})_+$.
\href{https://drive.google.com/file/d/1X18LuAg4awBXqzqem8Tu7wDK8VKQyDOy/view}{\color{blue}\underline{Video: Threshold\_Riot-Revolution\_u0=05.mp4}}  
} \label{fig:Double_Threshold_eps=0.5}
    \end{subfigure}
  \caption{Mixed case -- influence of the magnitude of the triggering event. Snapshots at different times of the solution of \eqref{ExampleMixed2} with $v_b=0.4$}
\end{figure}

%
%
%
%
%
%
%

There are also some examples where a sufficiently large triggering event generates a fast riot followed by a slower persisting upheaval. Mathematically speaking, this consists of a terrace, that is, the superposition of two traveling waves traveling at different speeds. Consider
\begin{equation}\label{ExampleMixed4}
\left\{\begin{aligned}
&\D_t  u -\D_{xx} u=u\left[v(1-u)-\frac{1}{3}\right],\\
&\D_t  v-\D_{xx} v=\frac{1}{2} uv(1-v)(v+u-\frac{2}{3}).
\end{aligned}\right.
\end{equation}
with $v_b=0.44$ and $u_0=\eps(1-x^2)_+$.
Taking $\eps=1.5$, we see in \autoref{fig:Bistable_eps=1.5} that the dynamics correspond to that of a riot. However, if we consider a triggering event with a larger magnitude by taking $\eps=1.6$, we observe a terrace in~\autoref{fig:Bistable_eps=1.6}, featuring an ephemeral riot followed by a lasting upheaval (the traveling speed of the upheaval equals approximately half the one of the riot).

\begin{figure}[p]
  \centering
  \begin{subfigure}[p]{0.45\linewidth}
    \centering\includegraphics[width=\textwidth]{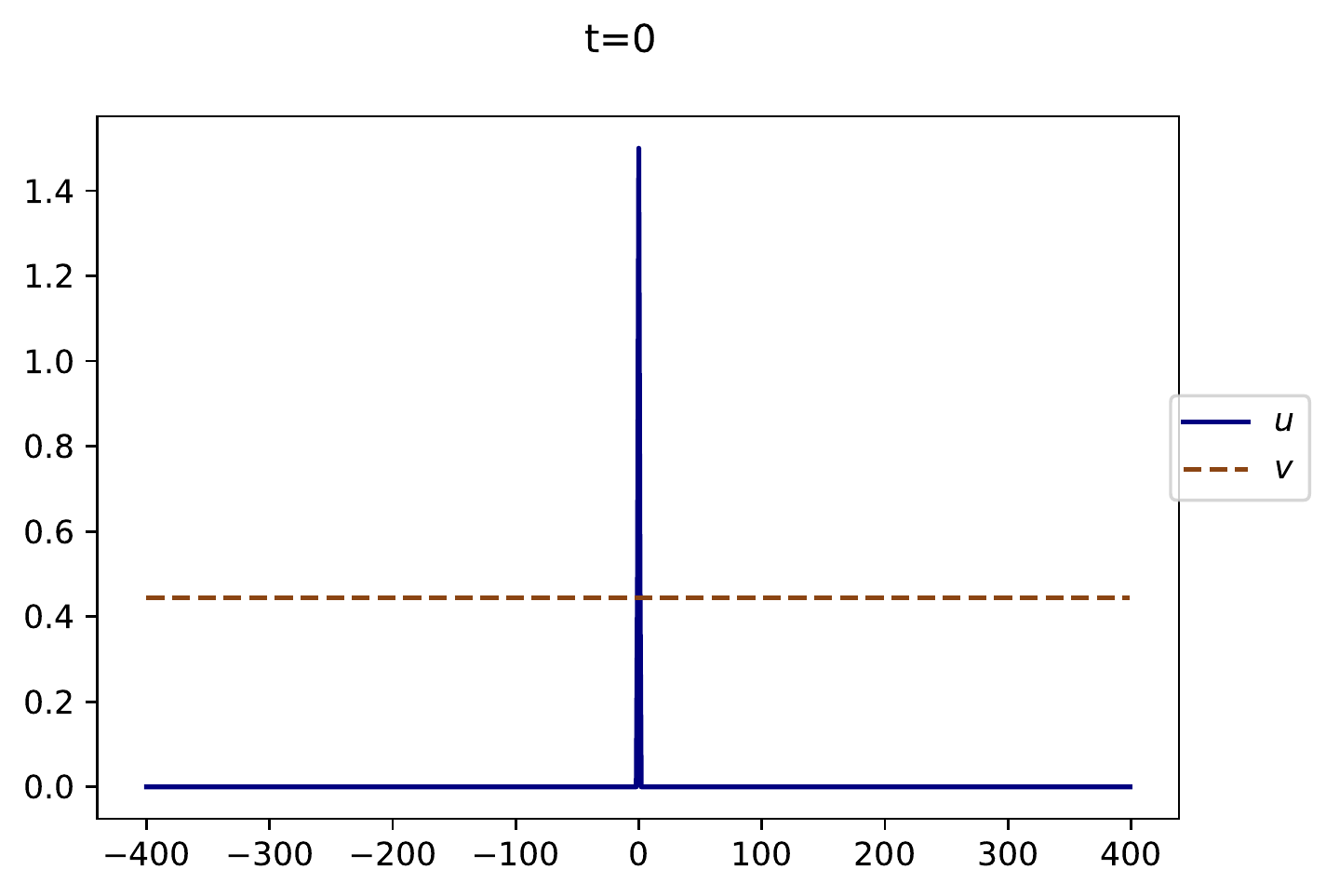}
  \end{subfigure}
  \hfill
  \begin{subfigure}[p]{0.45\linewidth}
    \centering\includegraphics[width=\textwidth]{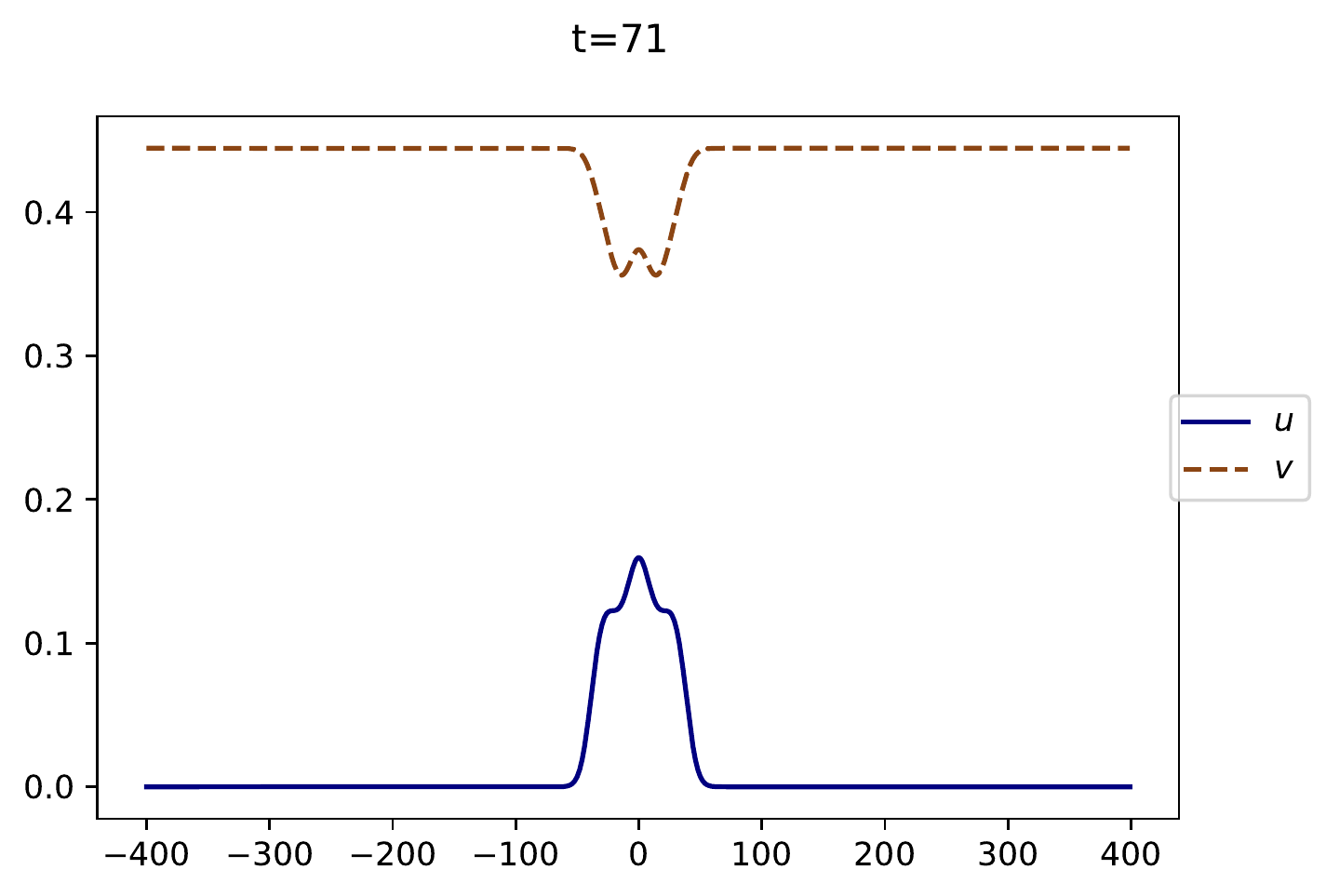}
  \end{subfigure}\\
    \begin{subfigure}[p]{0.45\linewidth}
    \centering\includegraphics[width=\textwidth]{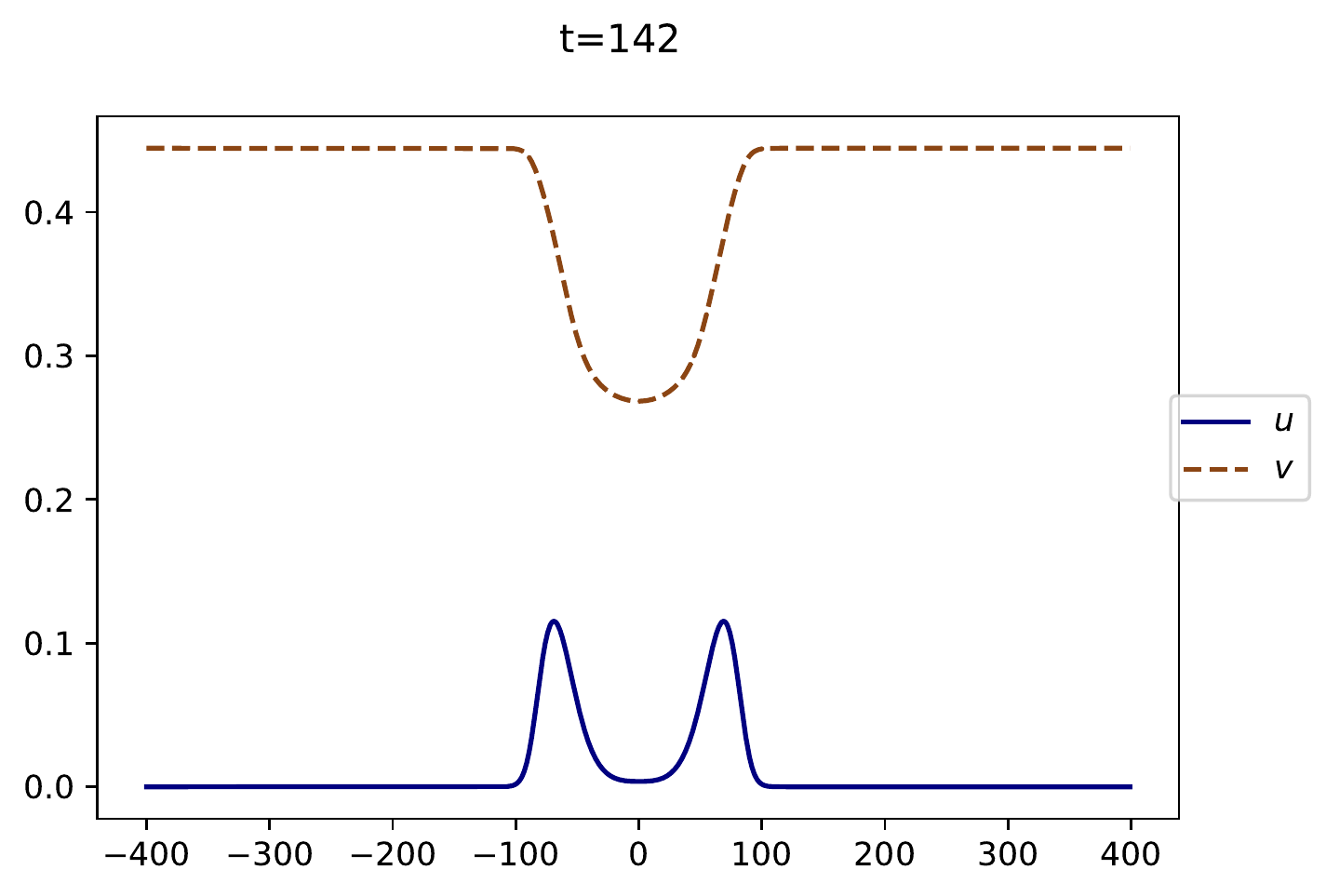}
  \end{subfigure}
  \hfill
  \begin{subfigure}[p]{0.45\linewidth}
    \centering\includegraphics[width=\textwidth]{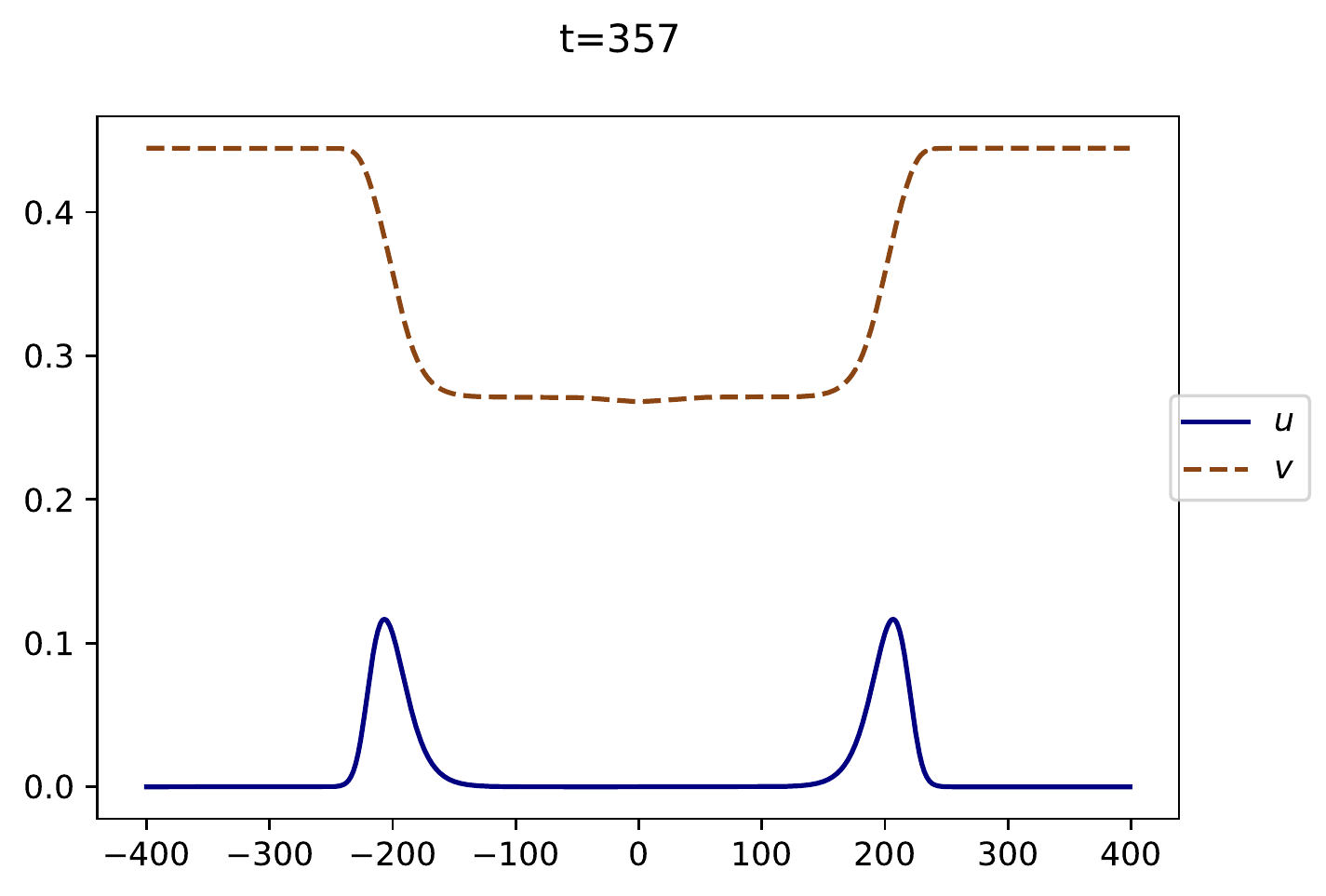}
  \end{subfigure}
\caption{No terrace: simple riot. Snapshots at different times of the solution of \eqref{ExampleMixed4} for $v_b=0.44$ and $u_0(x)=1.5(1-x^2)_+$.
\href{https://drive.google.com/file/d/1f4hYsMODEv5PqGHCzv18Z9_oKbRBlIkF/view}{\color{blue}\underline{Video: Terrace\_eps=15e-1.mp4}}  
} \label{fig:Bistable_eps=1.5}
\end{figure}

\begin{figure}[p]
  \centering
  \begin{subfigure}[p]{0.45\linewidth}
    \centering\includegraphics[width=\textwidth]{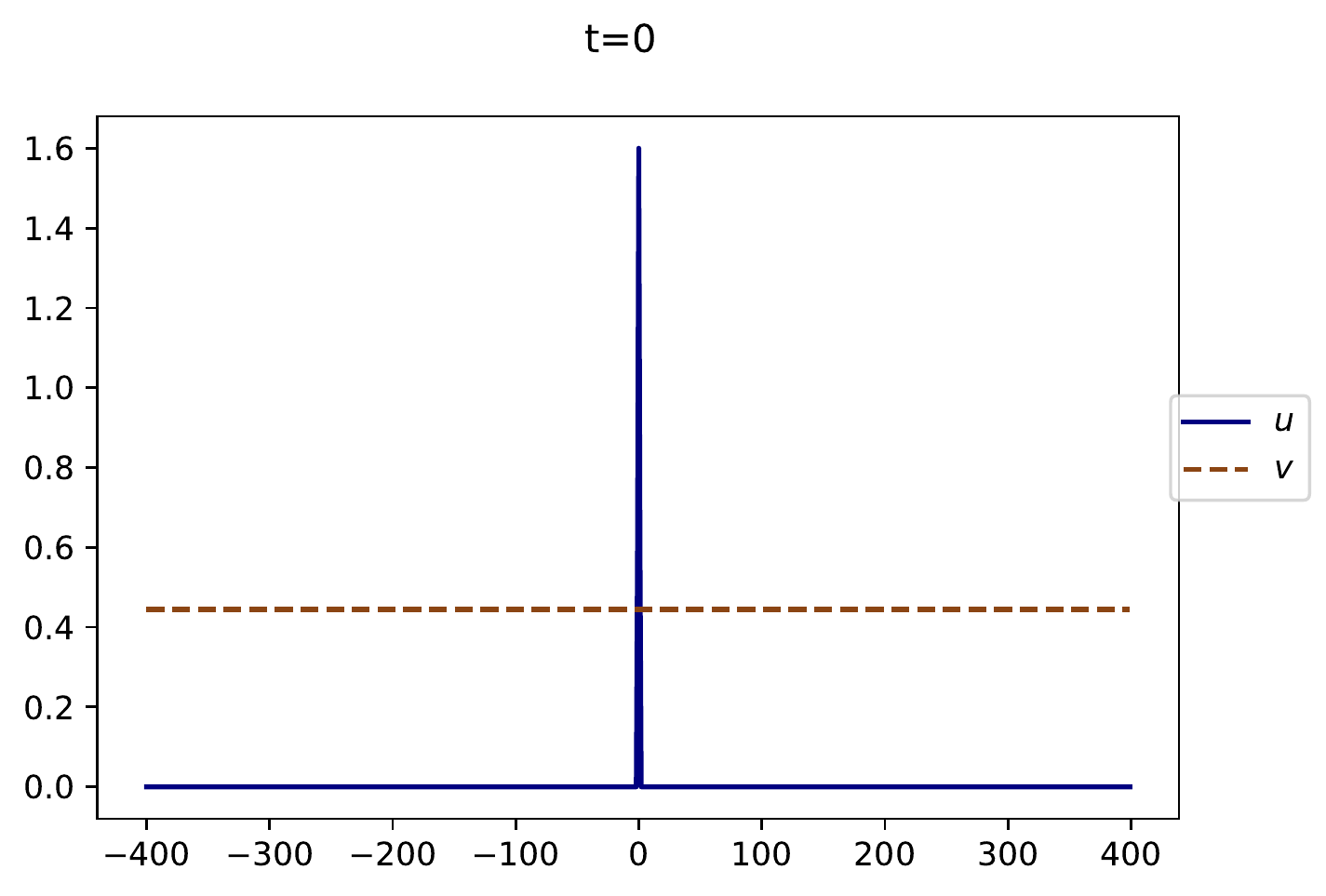}
  \end{subfigure}
  \hfill
  \begin{subfigure}[p]{0.45\linewidth}
    \centering\includegraphics[width=\textwidth]{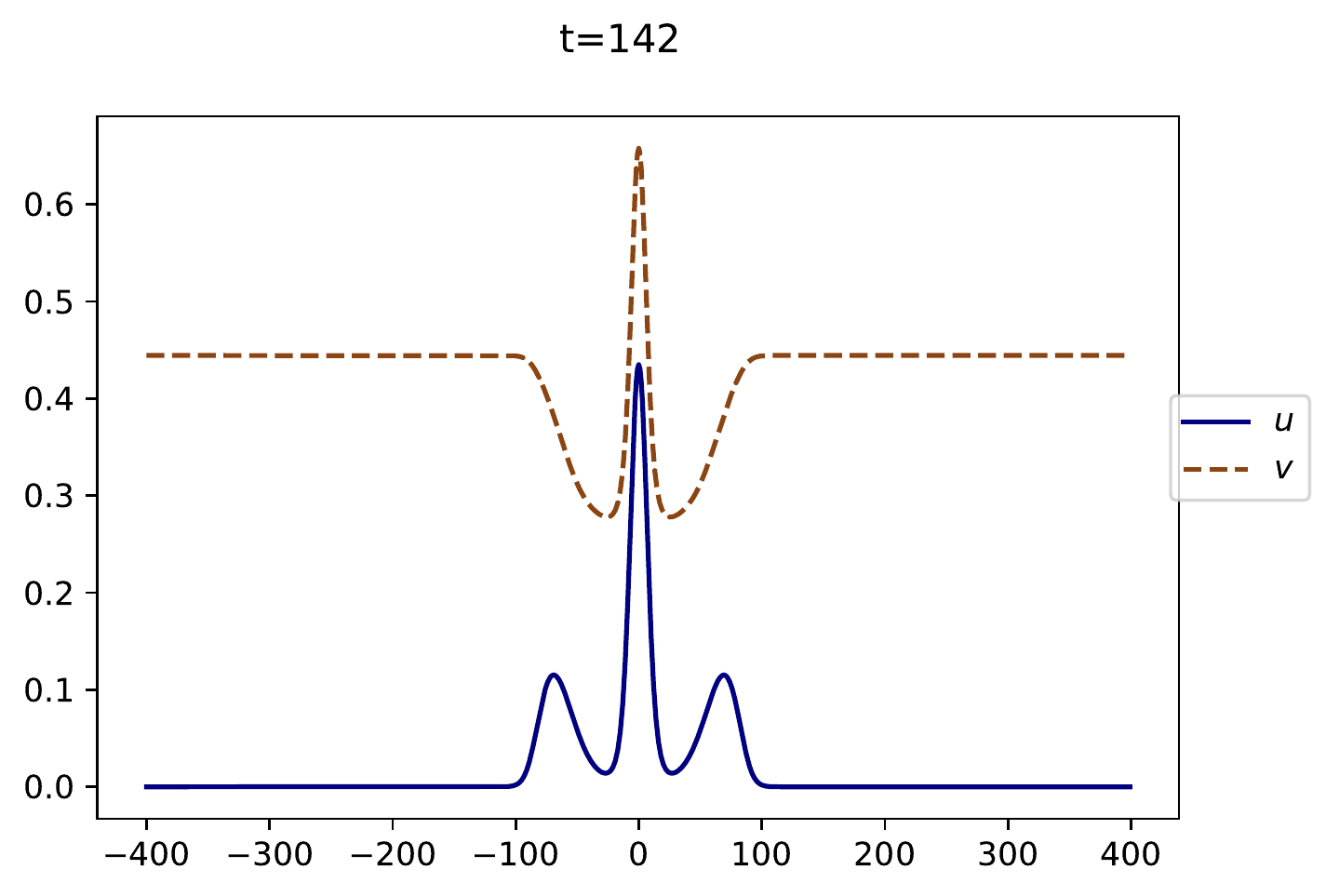}
  \end{subfigure}\\
    \begin{subfigure}[p]{0.45\linewidth}
    \centering\includegraphics[width=\textwidth]{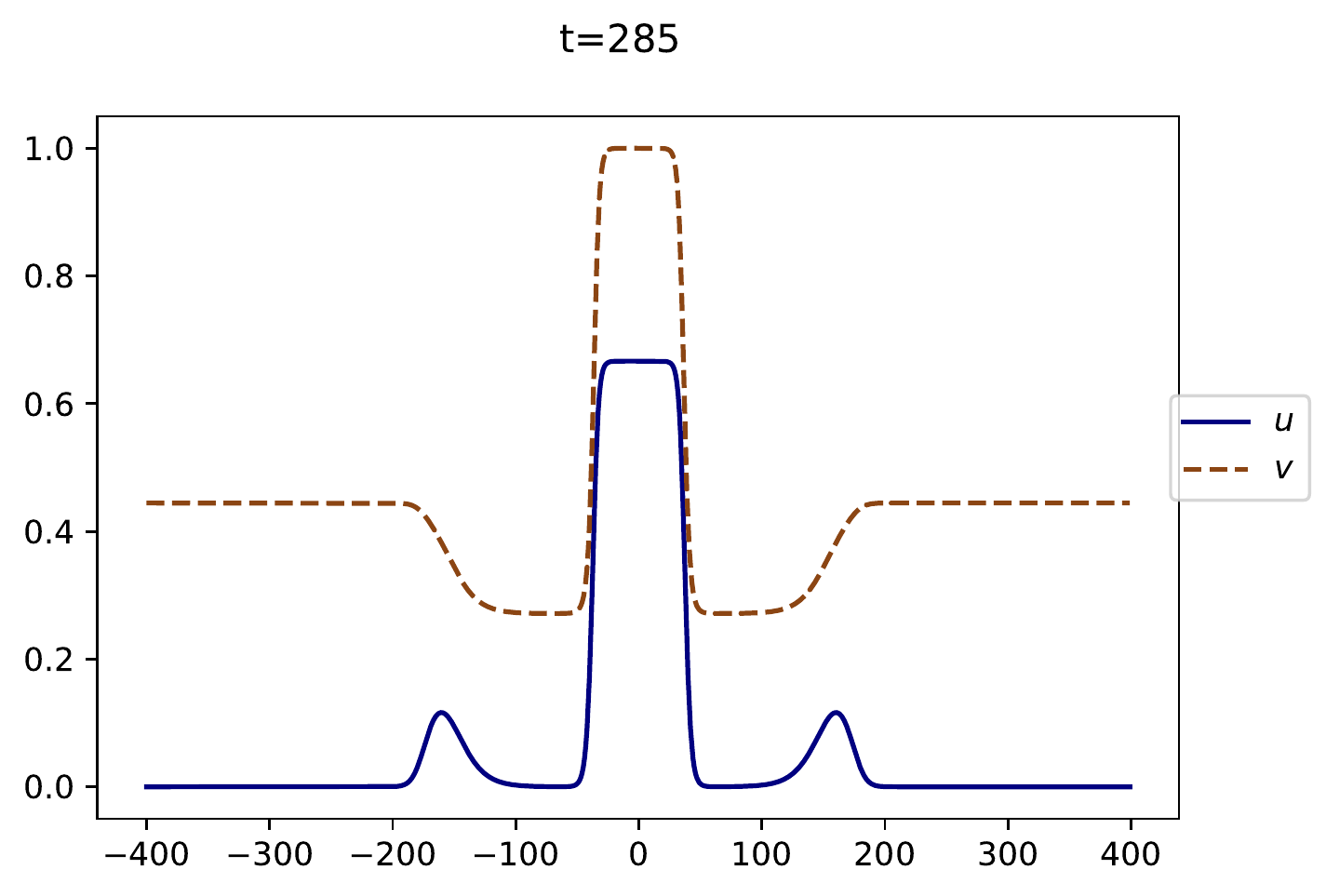}
  \end{subfigure}
  \hfill
  \begin{subfigure}[p]{0.45\linewidth}
    \centering\includegraphics[width=\textwidth]{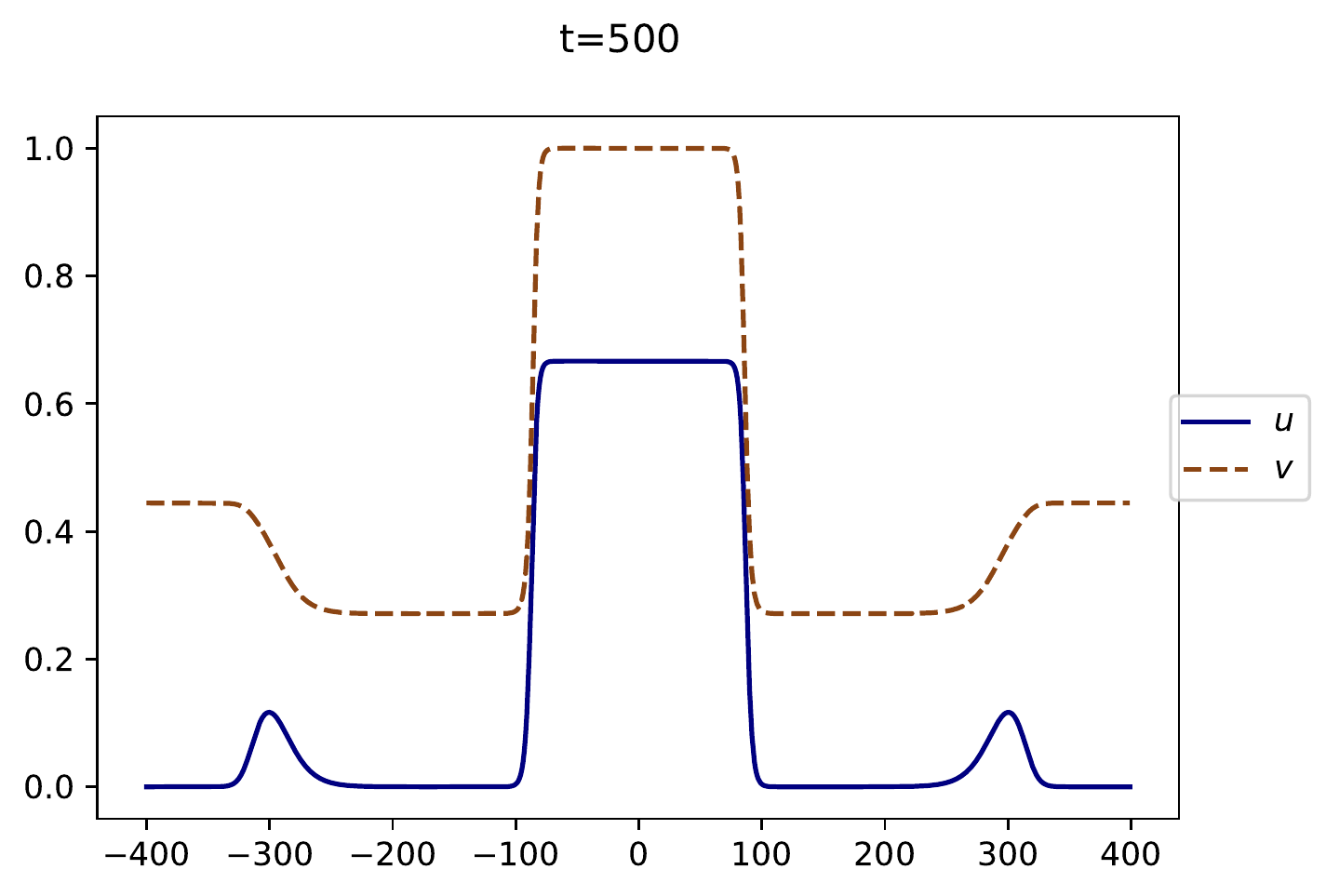}
  \end{subfigure}
\caption{Terrace: riot followed by a lasting upheaval. Snapshots at different times of the solution of \eqref{ExampleMixed4} for $v_b=0.44$ and $u_0(x)=1.6(1-x^2)_+$.
\href{https://drive.google.com/file/d/1cMK2xrqJsRP6LLxvFSTbcVaS_ppngAEA/view}{\color{blue}\underline{Video: Terrace\_eps=16e-1.mp4}}  
} \label{fig:Bistable_eps=1.6}
\end{figure}

The same phenomenon can also be observed on the following example
\begin{equation}\label{ExampleMixed5}
\left\{\begin{aligned}
&\D_t  u -\D_{xx} u=u\left[v(1-u)-\frac{1}{3}\right],\\
&\D_t  v-\D_{xx} v= u\left(u-\frac{1}{4}\right) v(1-v).
\end{aligned}\right.
\end{equation}
for $v_b=0.5$ and $u_0=\eps(1-x^2)_+$. Loosely speaking, this system can be interpreted as \emph{tension inhibiting} for $u<1/4$ and \emph{tension enhancing} for $u>1/4$. We observe a riot for $\eps=0.5$ (\href{https://drive.google.com/file/d/1CoHGi0pdrG4xlulolwmObx56mCDFnxUH/view}{\color{blue}\underline{video: Terrace2\_eps=05.mp4}}) and a terrace consisting of a riot followed by a lasting upheaval for $\eps=0.6$ (\href{https://drive.google.com/file/d/1VNRnLojb_L-XCGR3eeWLUpTmUwJHxPmh/view}{\color{blue}\underline{video: Terrace2\_eps=06.mp4}}).

\section{Spatially heterogeneous models}\label{sec:SpaceHeterogeneous}

Data show that the dynamics of social unrest is highly influenced by many factors that are not homogeneous in space, such as the density of population, poverty, etc. See, e.g.~\cite{Bonnasse-Gahot2018}.
An interesting extension of our model is to introduce spatial heterogeneity in the system~\eqref{GeneralEquationMotivationFinal}. 
In this section, we propose some possible ways to do this and provide some numerics.

\subsection{Gap problem}

The first approach to account for space heterogeneity is to assume that the function $\Phi(u,v)$ (which, we recall, stands for the growth of the level of social unrest $u$) depends explicitly on the space variable $x$. We may assume that some areas in space are not favorable to the growth of social unrest.
It is then an interesting problem to determine whether a propagating movement of social unrest can overcome those areas.

For some interval $\mathcal{K}:=(a, a+L)$, representing the \emph{gap} (or the \emph{obstacle}), we choose $\Phi$ of the form
\begin{equation*}
\Phi(x,u,v)=u\left[r(v)f(u)\mathds{1}_{x\not\in(a, a+L)}-\omega\right].
\end{equation*}
Namely, in dimension $n=1$, we consider the system
\begin{equation}\label{HeterogeneousGEquation}
\left\{\begin{aligned}
&\D_t  u -d_1 \D_{xx} u=u\left[r(v)f(u)\mathds{1}_{x\not\in(a, a+L)}-\omega.\right],\\
&\D_t  v-d_2\D_{xx} v=\Psi(u,v).
\end{aligned}\right.
\end{equation}

We ask the following question: is there a threshold on the length $L$ of the obstacle above which the propagation of the solution is blocked? 

First, let us remark that, if $v_0\equiv v_b> v_\star$, then our system enjoys a \emph{Hair-Trigger} effect: it means that arbitrarily small initial conditions ignites a social movement (see Sections~\ref{sec:ResumptionCalm} and \ref{sec:Burst}). Since, for all $t>0$, we have that $u(t,\cdot)>0$, we deduce that there exists no gap that can block the propagation. 
We illustrate this remark with a numerical simulation on \autoref{fig:GapHairTrigger} for the system~\eqref{ExampleInhibiting2}, that is
\begin{equation}\label{GapExample1}
\left\{\begin{aligned}
&\D_t  u -\D_{xx} u=u\left[v (1-u)\mathds{1}_{x\not\in(60,60+L)} -\frac{1}{3} \right],\\
&\D_t  v-\D_{xx} v=-uv,
\end{aligned}\right.
\end{equation}
 with $v_b=0.6$ (we recall that $v_\star= 1/3$) and $u_0(x)=0.2(1-x^2)_+$. 
We see on \autoref{fig:GapHairTrigger} that a riot begins to propagate, but then fades as it reaches the obstacle. However, after some times, the riot grows again and continue to propagate beyond the obstacle.

\begin{figure}[p]
  \centering
  \begin{subfigure}[p]{0.45\linewidth}
    \centering\includegraphics[width=\textwidth]{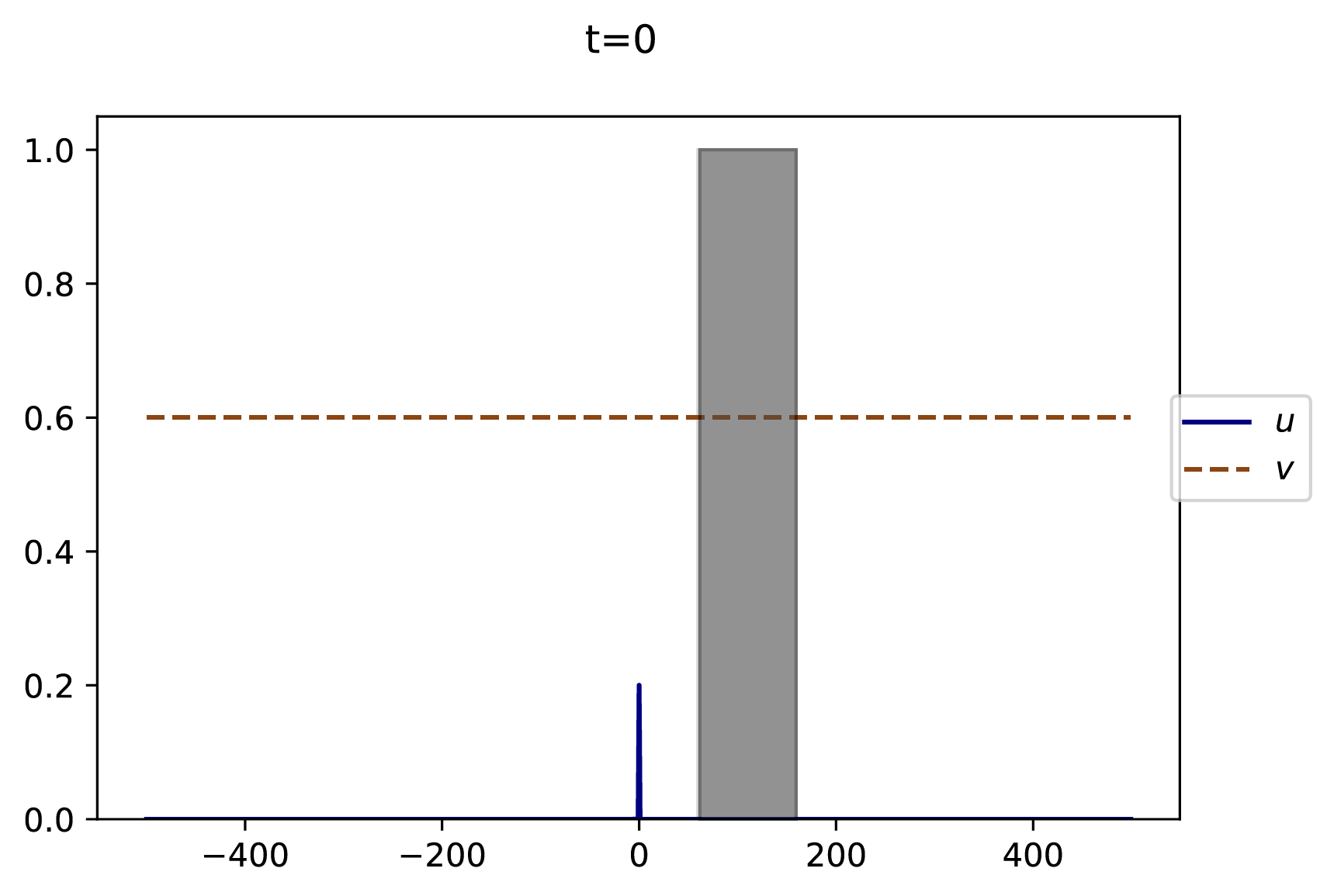}
  \end{subfigure}
  \hfill
  \begin{subfigure}[p]{0.45\linewidth}
    \centering\includegraphics[width=\textwidth]{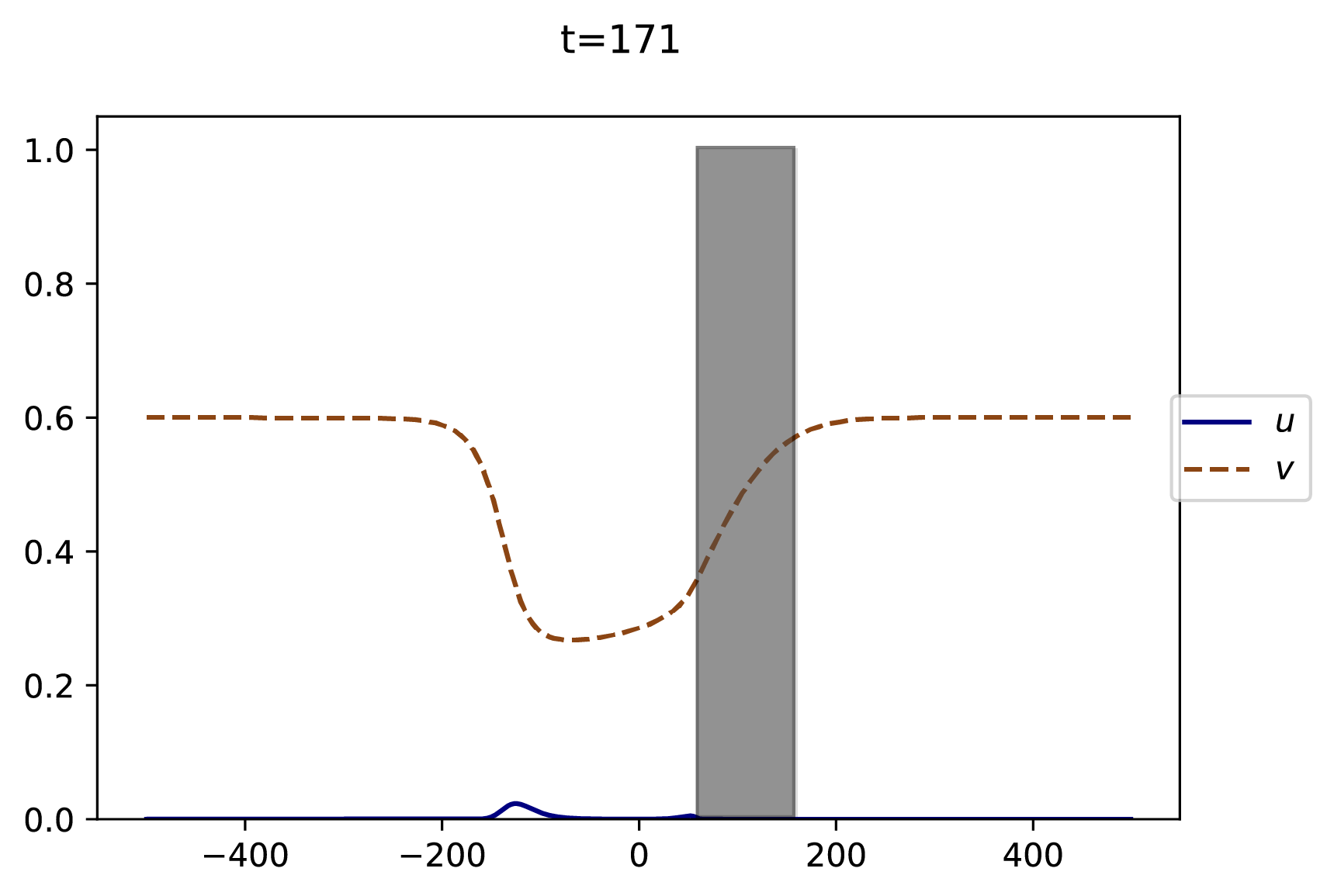}
  \end{subfigure}\\
    \begin{subfigure}[p]{0.45\linewidth}
    \centering\includegraphics[width=\textwidth]{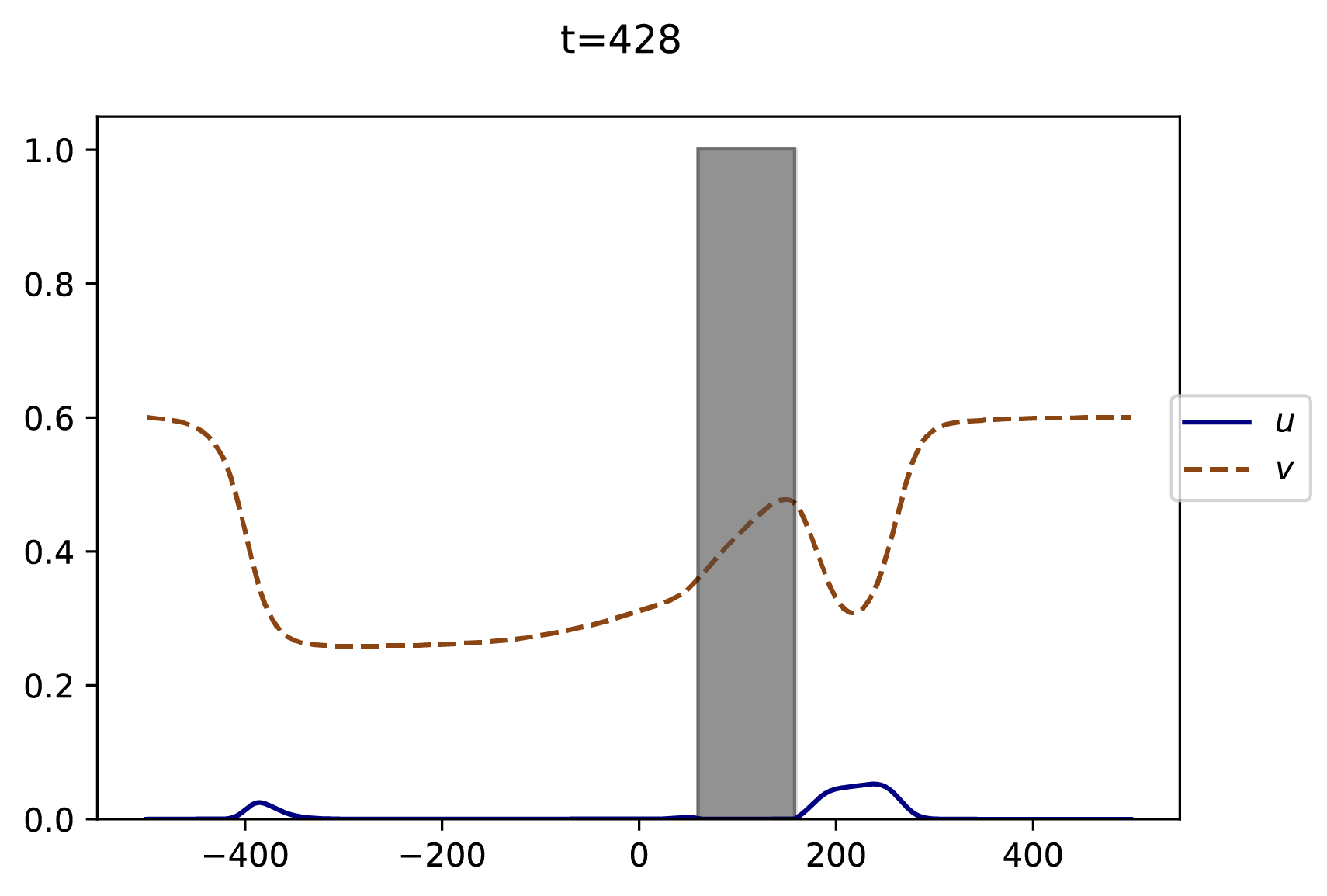}
  \end{subfigure}
  \hfill
  \begin{subfigure}[p]{0.45\linewidth}
    \centering\includegraphics[width=\textwidth]{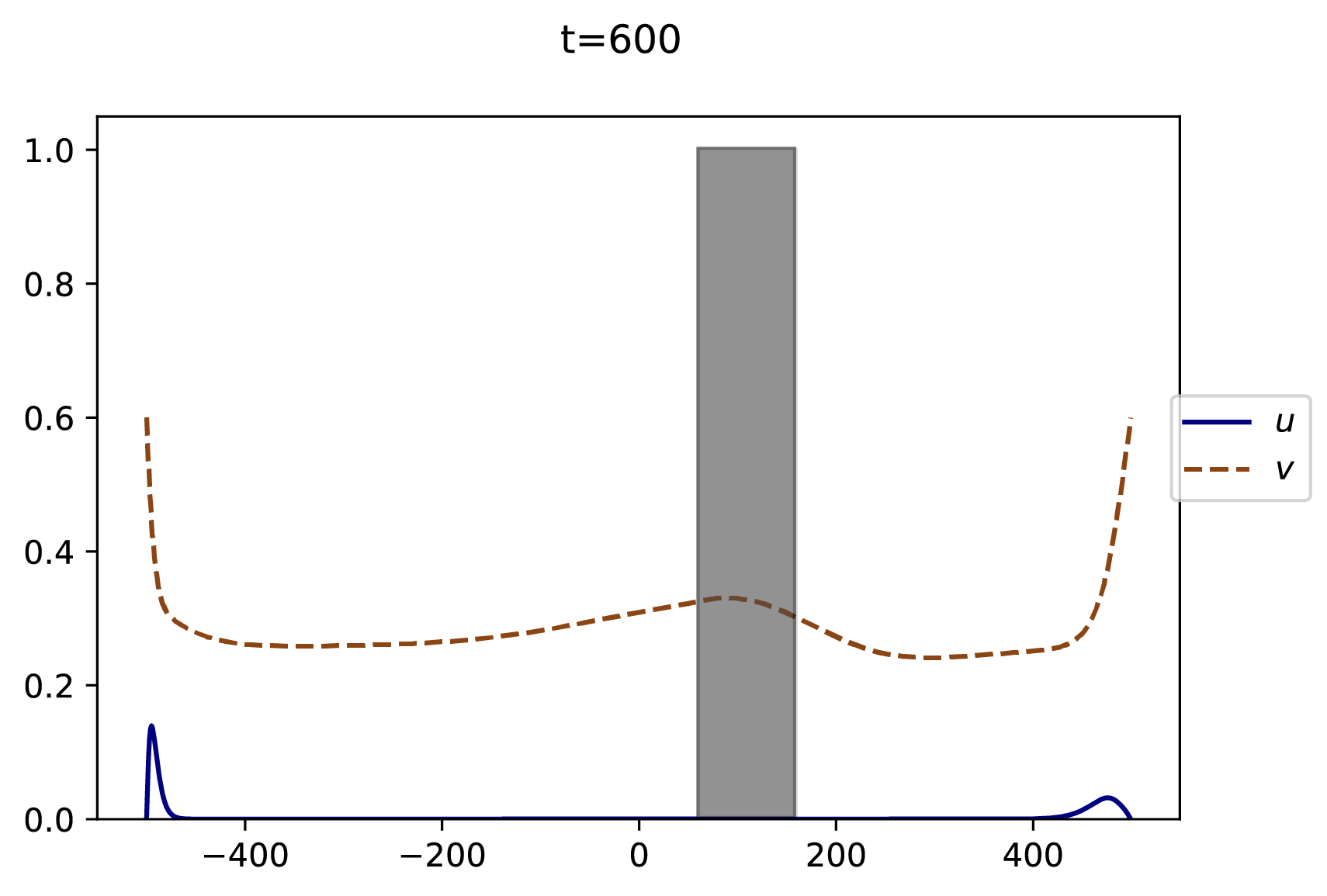}
  \end{subfigure}
\caption{Gap problem -- propagation. Snapshots at different times of the solution of~
\eqref{GapExample1} with $L=100$, $v_b=0.6$ and $u_(x)=0.2(1-x^2)_+$. The shaded area represents the obstacle $\mathcal{K}=(60,160)$.
\href{https://drive.google.com/file/d/1KS0nIULX_LVzPA73O4t6KxUTqcku3Co2/view}{\color{blue}\underline{Video: Gap\_Hair-Trigger\_inhibiting.mp4}}
} \label{fig:GapHairTrigger}
\end{figure}

However, if we are in a situation where $v_0\equiv v_b<v_\star$ but $u_0$ is large enough to triggers a social movement (as described in Sections~\ref{sec:enhancing_MagnitudeTriggering},~\ref{sec:Mixed_MagnitudeTriggering}), there might exists a length $L$ above which the gap blocks the propagation. Let us consider the following system
\begin{equation}\label{GapExample2}
\left\{\begin{aligned}
&\D_t  u -\D_{xx} u=u\left[v(1-u)\mathds{1}_{x\not\in(60,60+L)} -\frac{1}{3} \right],\\
&\D_t  v-\D_{xx} v=6uv(1-v)(v+10u-\frac{2}{3}),
\end{aligned}\right.
\end{equation}
with $v_b=0.3$ and $u_0(x)=(1-\frac{x^2}{10})_+$.
We plot in \autoref{fig:GapBistable_passage} a numerical simulation of the above system for $L=40$. We observe that an upheaval propagates until it reaches the obstacle. Then, after some time, the upheaval manages to overpass the obstacle and continue to propagate. If now we increase the length of the gap to $L=60$, we see in \autoref{fig:GapBistable_blockage} that the spreading of the upheaval is blocked by the gap (even after a long time).

\begin{figure}[p]
  \centering
  \begin{subfigure}[p]{0.45\linewidth}
    \centering\includegraphics[width=\textwidth]{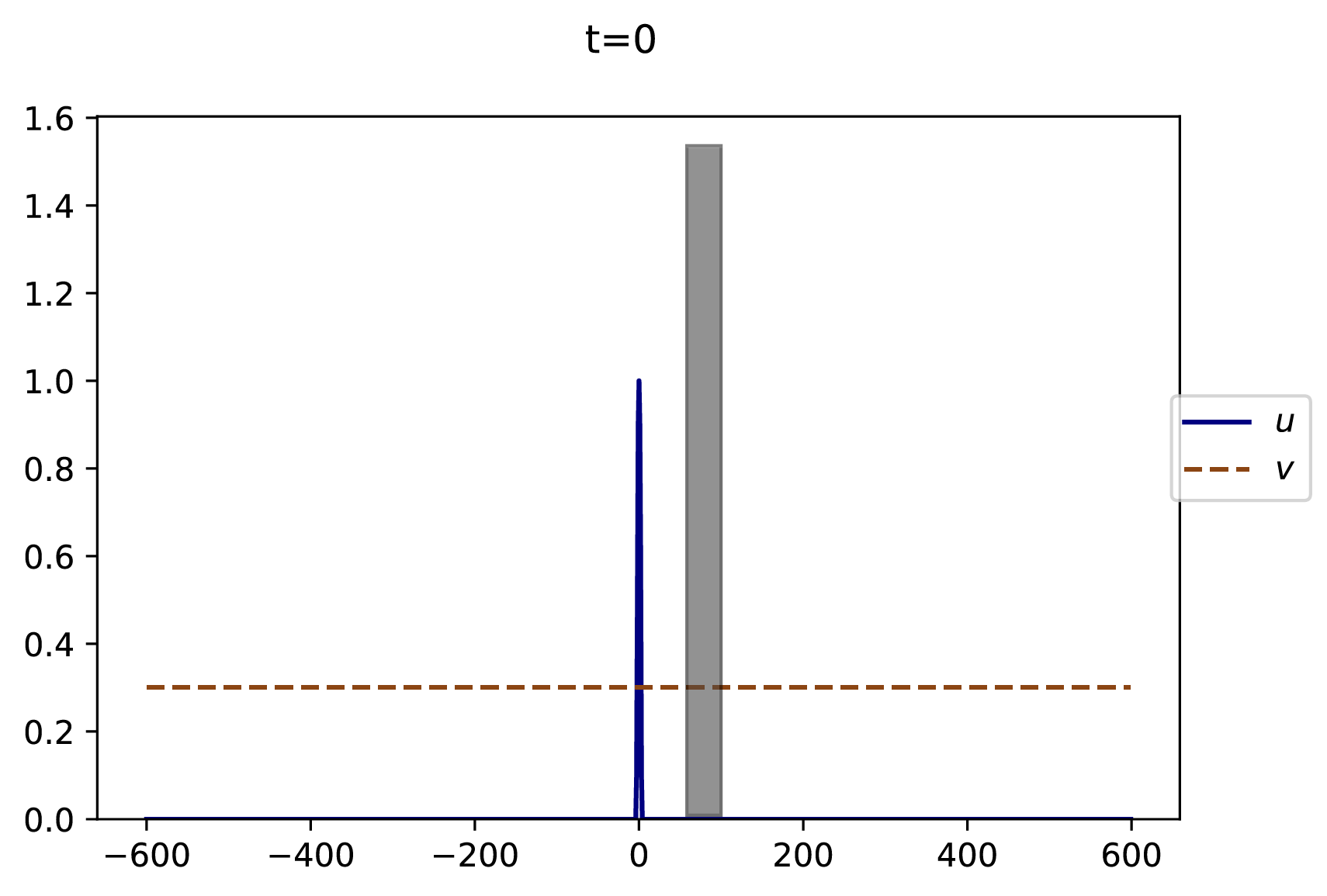}
  \end{subfigure}
  \hfill
  \begin{subfigure}[p]{0.45\linewidth}
    \centering\includegraphics[width=\textwidth]{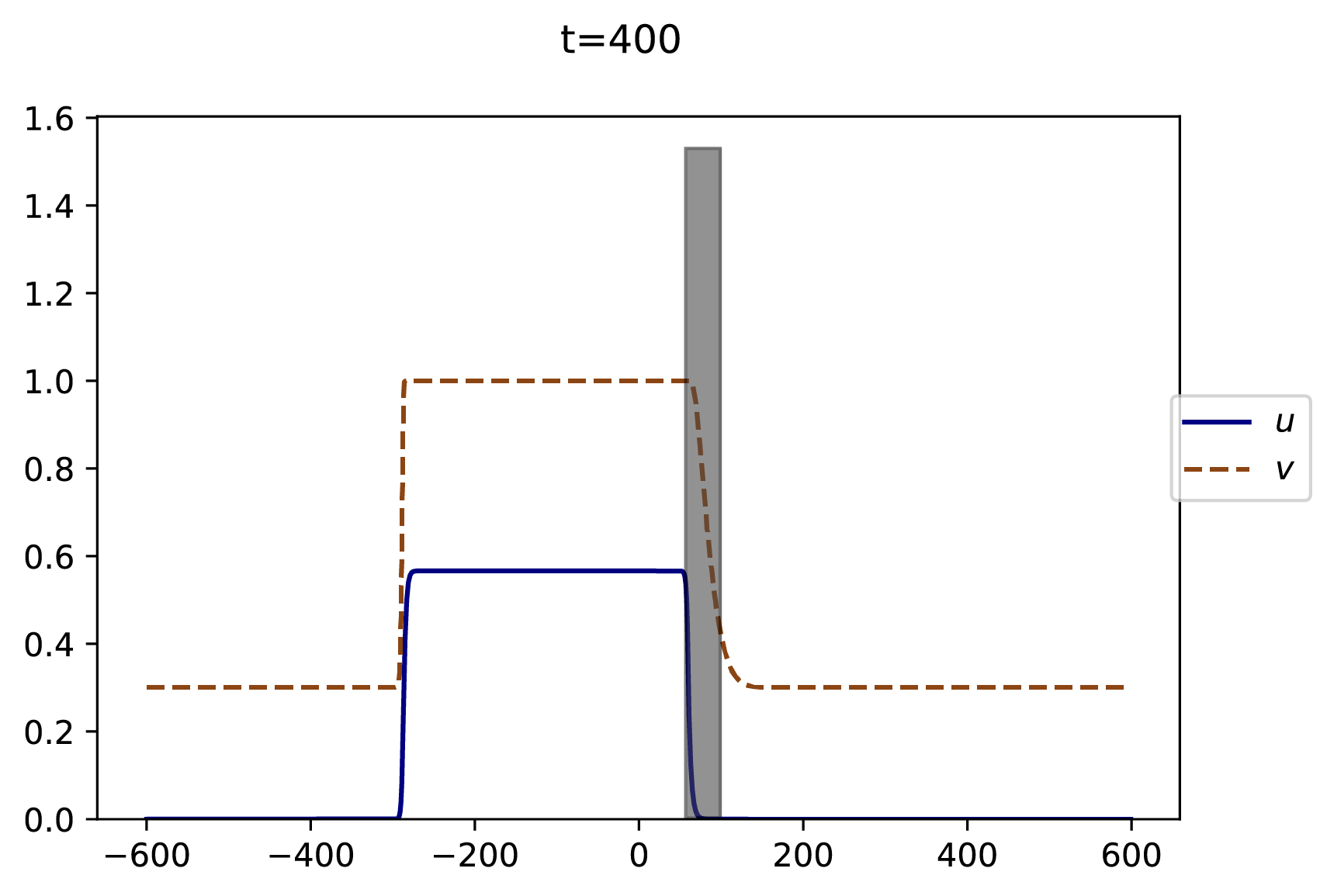}
  \end{subfigure}\\
    \begin{subfigure}[p]{0.45\linewidth}
    \centering\includegraphics[width=\textwidth]{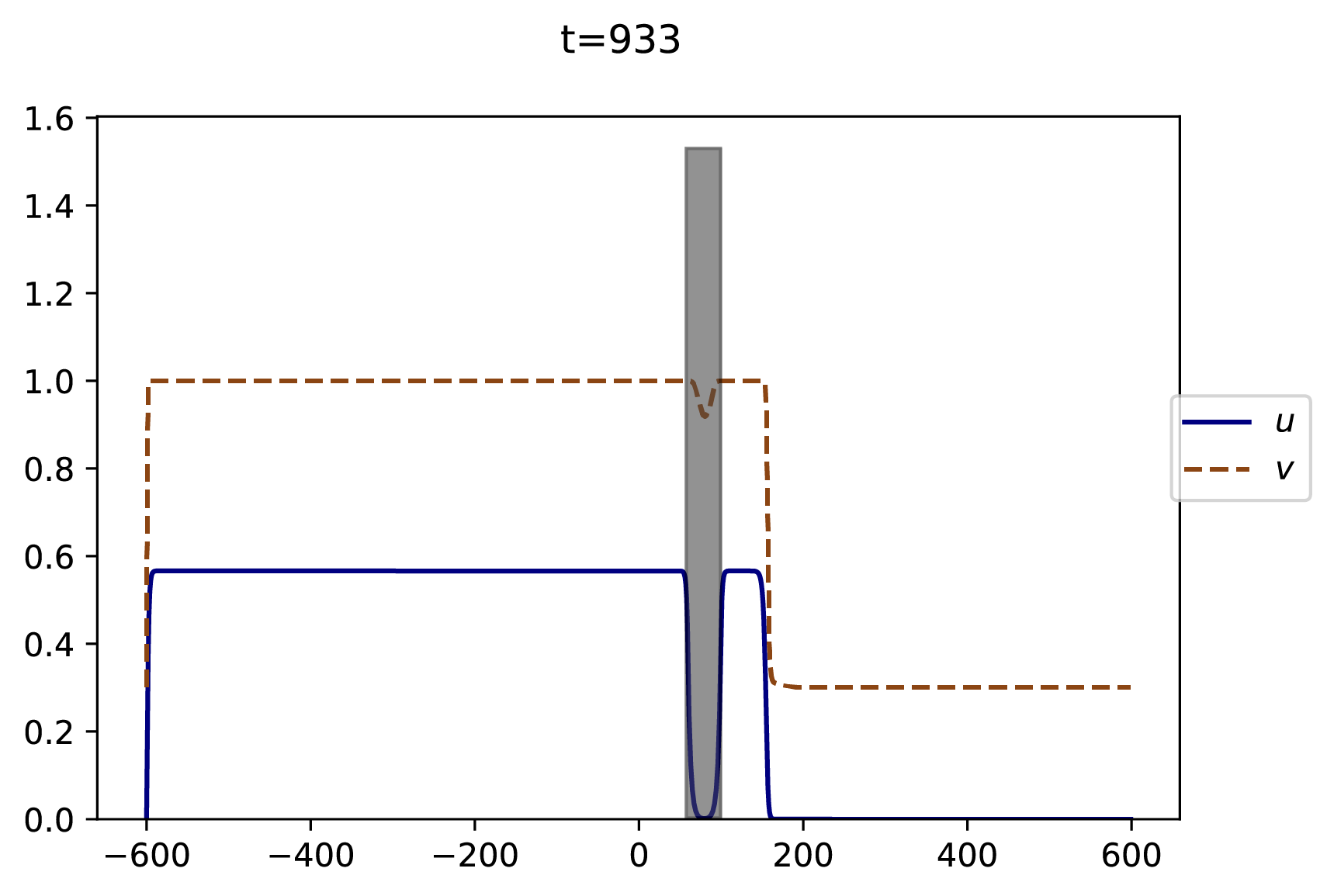}
  \end{subfigure}
  \hfill
  \begin{subfigure}[p]{0.45\linewidth}
    \centering\includegraphics[width=\textwidth]{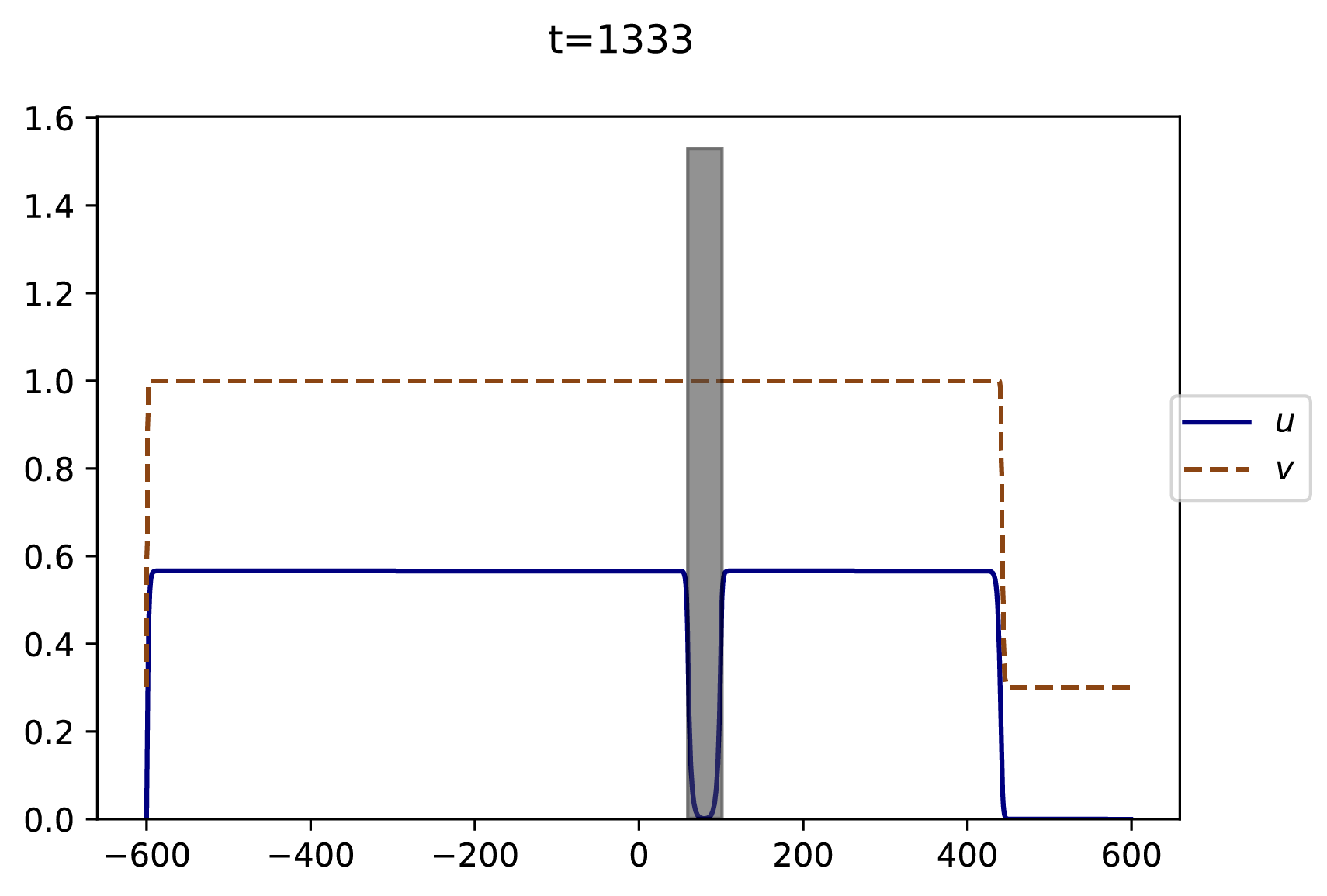}
  \end{subfigure}
\caption{Gap problem -- propagation. Snapshots at different times of the solution of~\eqref{GapExample2} with $L=40$, $v_b=0.3$ and $u_0(x)=(1-\frac{x^2}{10})_+$. The shaded area represents the obstacle $\mathcal{K}=(60,100)$.
\href{https://drive.google.com/file/d/1qzwGkFAt883exY6KwEo6b4OAYcj1L3JA/view}{\color{blue}\underline{Video: Gap=40.mp4}}
} \label{fig:GapBistable_passage}
\end{figure}

\begin{figure}[p]
  \centering
  \begin{subfigure}[p]{0.45\linewidth}
    \centering\includegraphics[width=\textwidth]{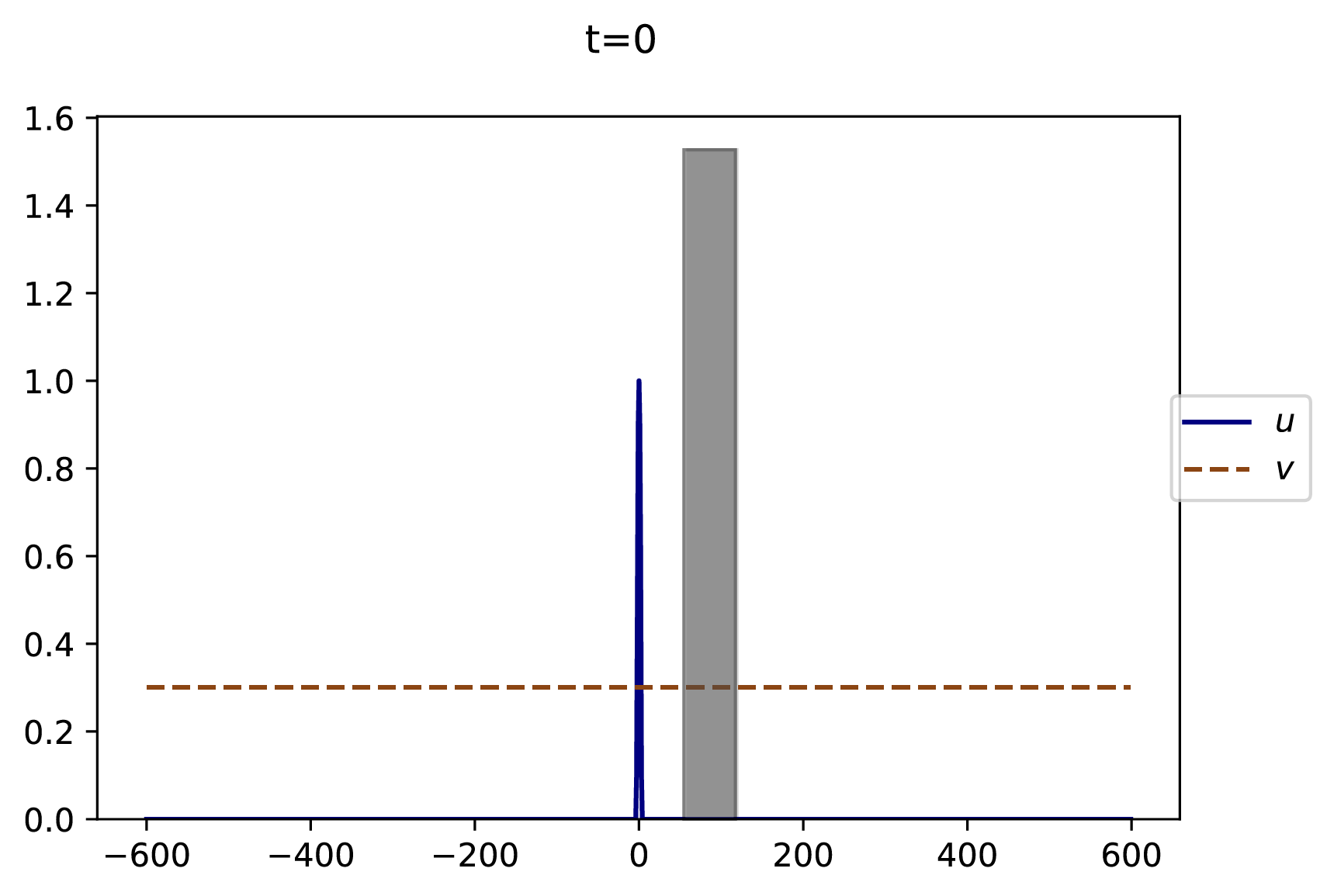}
  \end{subfigure}
  \hfill
  \begin{subfigure}[p]{0.45\linewidth}
    \centering\includegraphics[width=\textwidth]{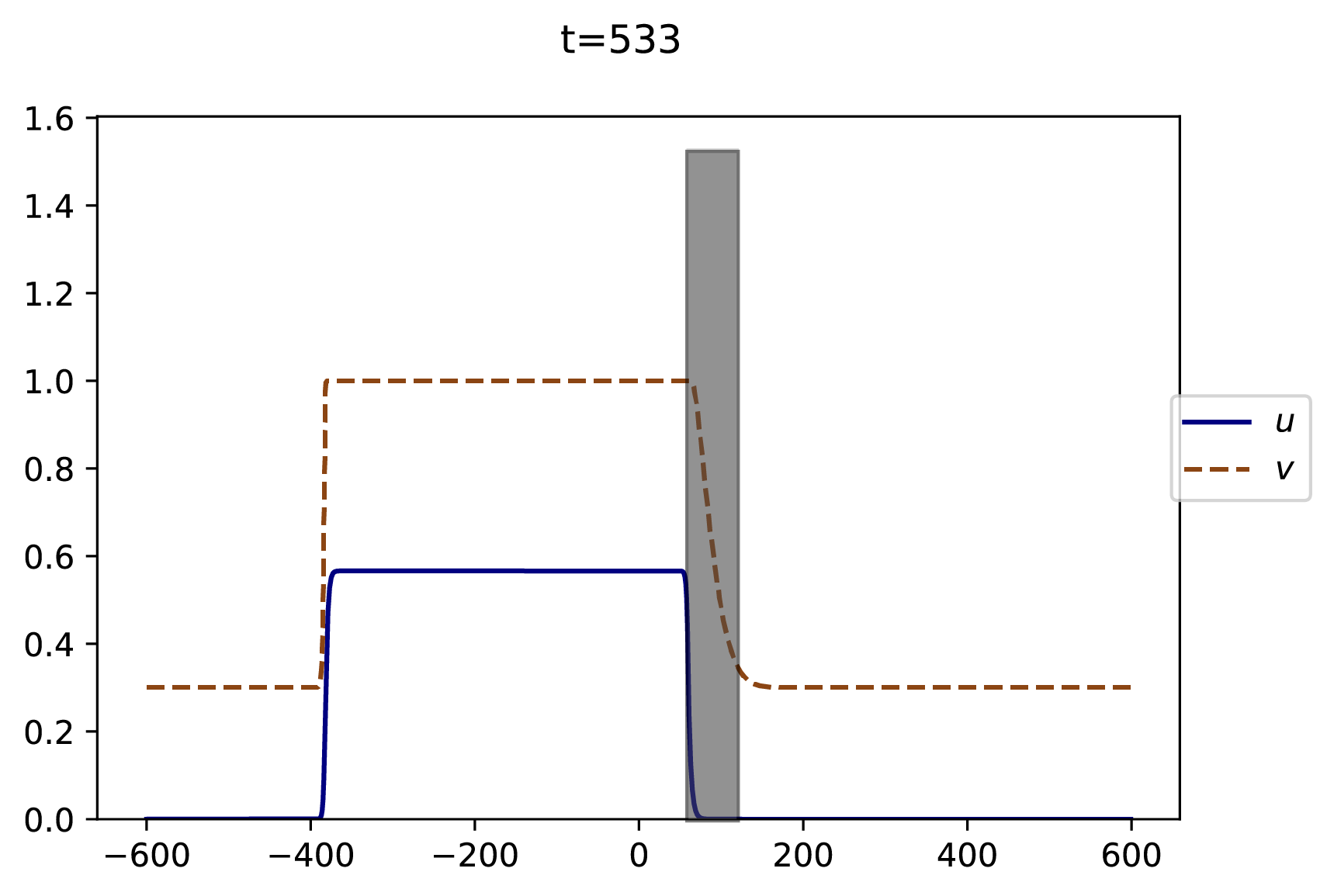}
  \end{subfigure}\\
    \centering
    \begin{subfigure}[p]{0.45\linewidth}
    \centering\includegraphics[width=\textwidth]{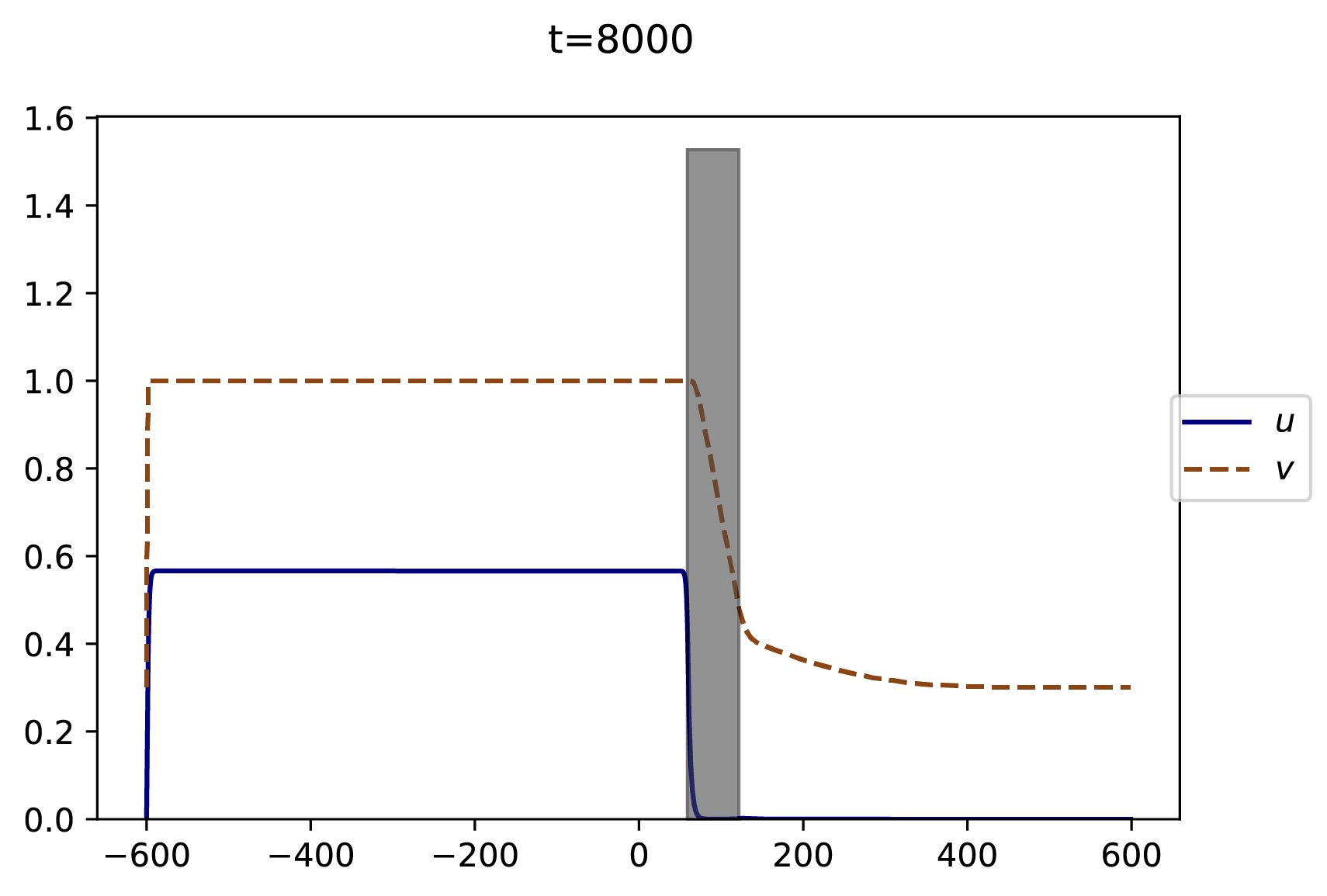}
  \end{subfigure}
\caption{Gap problem -- blockage. Snapshots at different times of the solution of~\eqref{GapExample2} with $L=60$, $v_b=0.3$ and $u_0(x)=(1-\frac{x^2}{10})_+$. Space is represented horizontally, and the shaded area represents the obstacle $\mathcal{K}=(60,120)$. Blue solid line: $u(t,\cdot)$. Brown dashed line: $v(t,\cdot)$.
\href{https://drive.google.com/file/d/1TQq1kXQeuXsvud2YK6SXdBmxNrIAz0Nw/view}{\color{blue}\underline{Video: Gap=60.mp4}}
} \label{fig:GapBistable_blockage}
\end{figure}

Of course, one could also consider that $\Psi$ depends on the space variable $x$, but we omit this case for conciseness.

\subsection{Non-uniform initial social tension}
So far, we have only dealt with a constant initial level of social tension $v_0\equiv v_b$ (although more general results are available in~\cite{Berestycki2019b} when $v_0-v_b$ is compactly supported). Another way to encode space heterogeneity in our model is to consider the case of a non-constant initial level of social tension $v_0$. This case is relevant from a modeling perspective, since the social tension may indeed vary between, for examples, cities and suburbs, or poor and rich areas (see~\cite{Bonnasse-Gahot2018}). 
To account for this situation with a model case, we may assume that $v_0(x):=v(t=0,x)$ is periodic  and oscillates around the threshold $v_\star$.

Consider the inhibiting system~\eqref{ExampleInhibiting2}, that we recall here,
\begin{equation}\label{HeteroV0}
\left\{\begin{aligned}
&\D_t  u -\D_{xx} u=u\left[v (1-u)-\frac{1}{3} \right],\\
&\D_t  v-\D_{xx} v=-uv,
\end{aligned}\right.
\end{equation}
with $u_0(x)=0.2(1-x^2)_+$ and
\begin{equation}\label{HeteroV0_def}
v_b(x)=\left\{\begin{aligned}
&0.2 &&\text{if }\vert x\vert \in[100k,100k+L],\\
&0.6 &&\text{if }\vert x\vert \in(100k-L,100k),
\end{aligned}\right. \qquad \forall k=1,2,\dots,
\end{equation}
for some parameter $L\in(0,100)$.
We recall that, for equation~\eqref{HeteroV0}, we have $v_\star=\frac{1}{3}$ (where $v_\star$ is defined in~\eqref{Def_v_star}). Therefore, $v_0(x)$ oscillates periodically around $v_\star$, creating zones of length $100-L$ that are favorable for propagation ($v_0(x)=0.6>v_\star$), and zones of length $L$ that are unfavorable ($v_0(x)=0.2<v_\star$). If the favorable zone is sufficiently large, then the solution manages to propagate, as can be seen on \autoref{fig:V0GAP_propagation} for $L=20$.
On the contrary, if the favorable zone is too thin, then the propagation is blocked, as can be seen on \autoref{fig:V0GAP_blockage} for $L=60$.
(We point out that, for the numerical simulation, we impose $v_0(x)=0.6$ on $(-60,60)$ to ensure that the movement of social unrest gets properly ignited at $x=0$ before being affected by the oscillations of $v_b$).

\begin{figure}[p]
  \centering
  \begin{subfigure}[p]{0.45\linewidth}
    \centering\includegraphics[width=\textwidth]{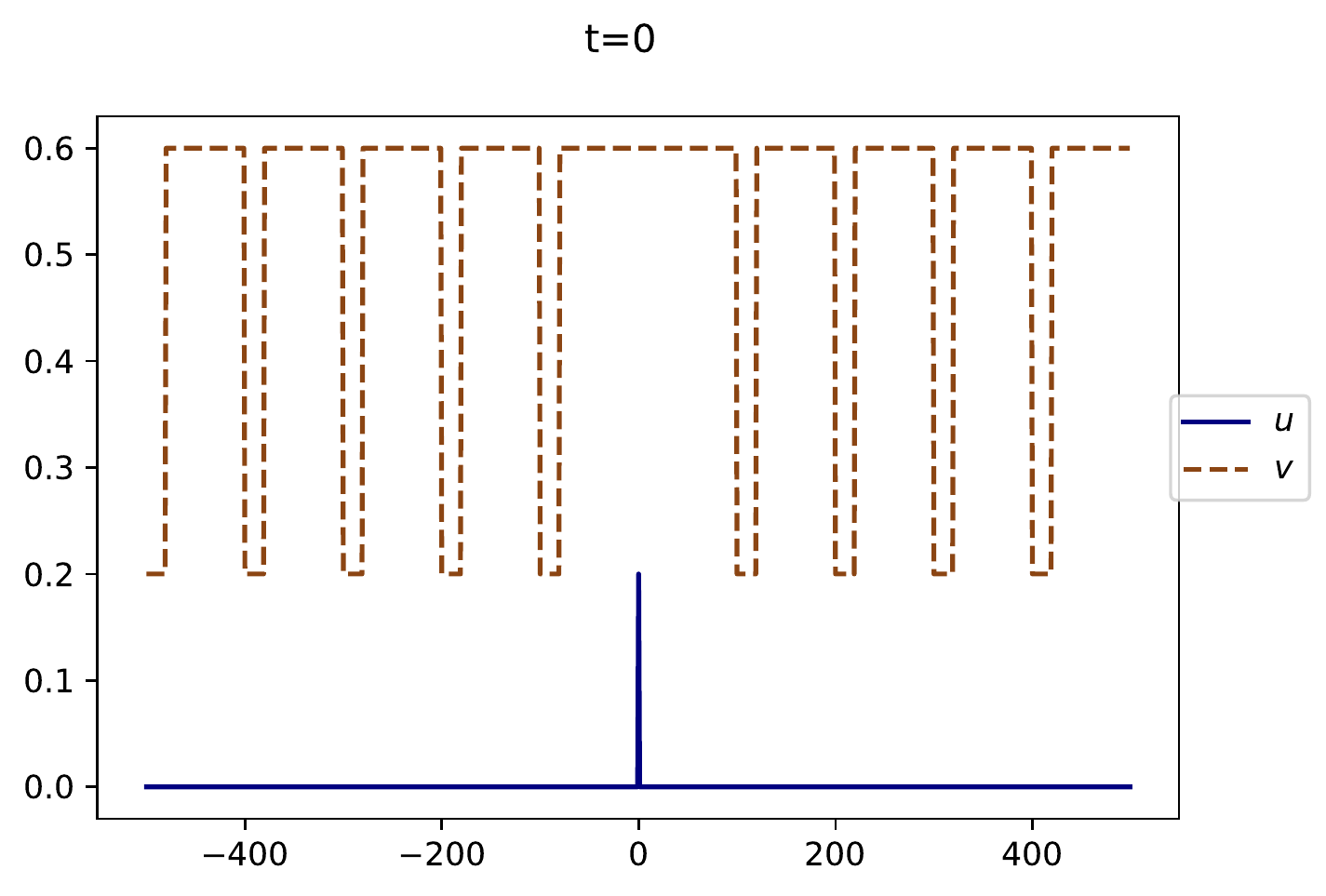}
  \end{subfigure}
  \hfill
  \begin{subfigure}[p]{0.45\linewidth}
    \centering\includegraphics[width=\textwidth]{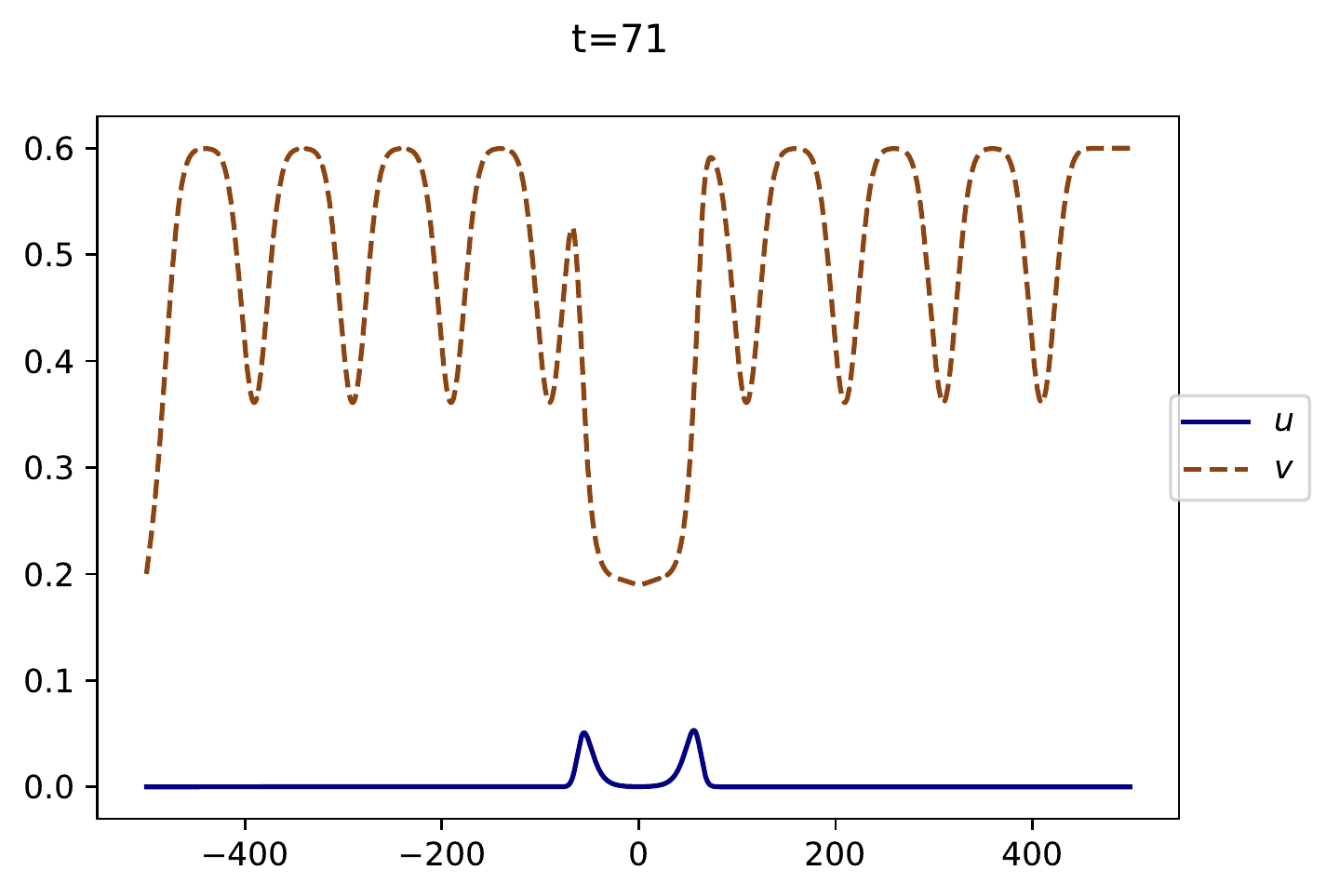}
  \end{subfigure}\\
    \centering
    \begin{subfigure}[p]{0.45\linewidth}
    \centering\includegraphics[width=\textwidth]{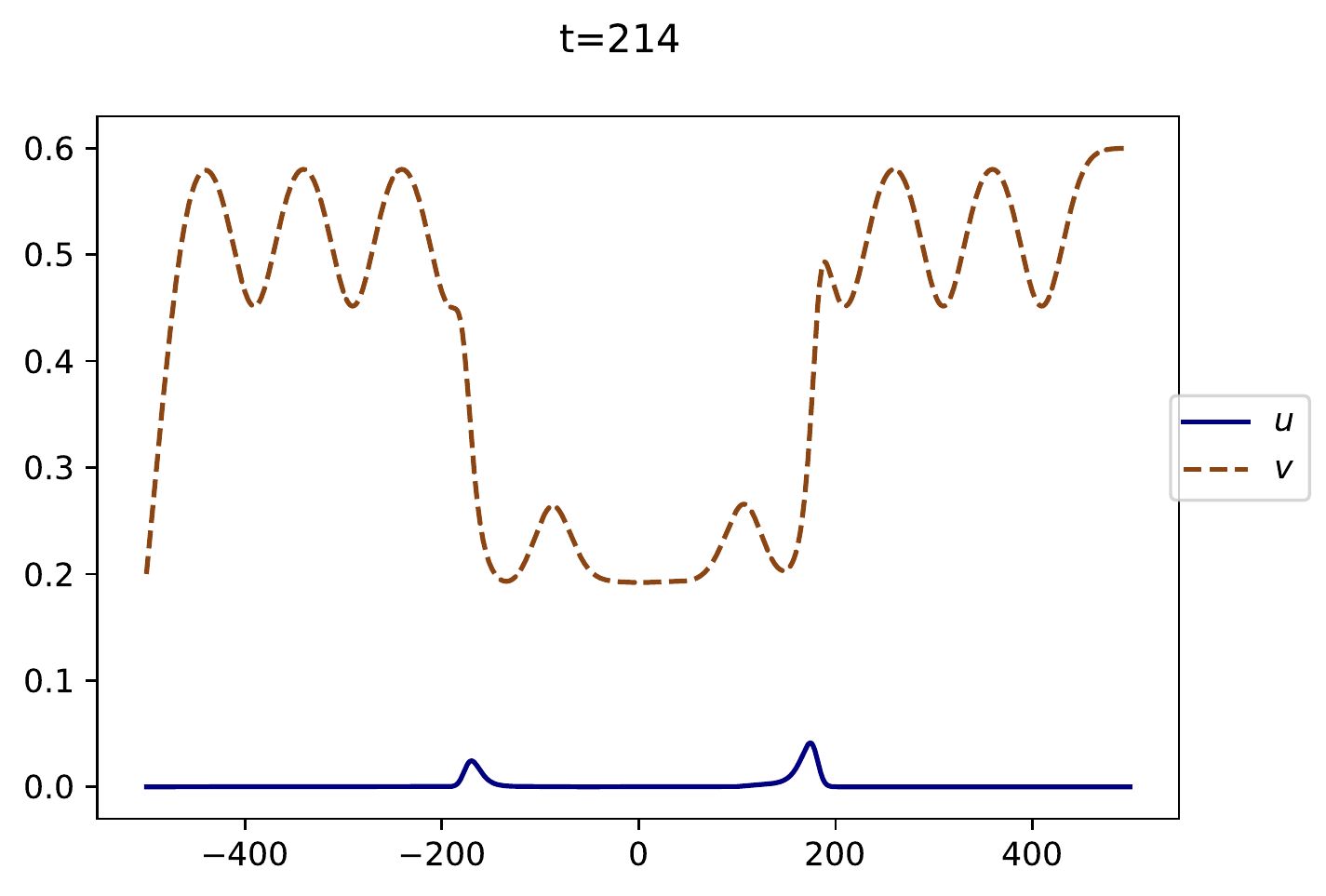}
  \end{subfigure}
   \hfill
  \begin{subfigure}[p]{0.45\linewidth}
    \centering\includegraphics[width=\textwidth]{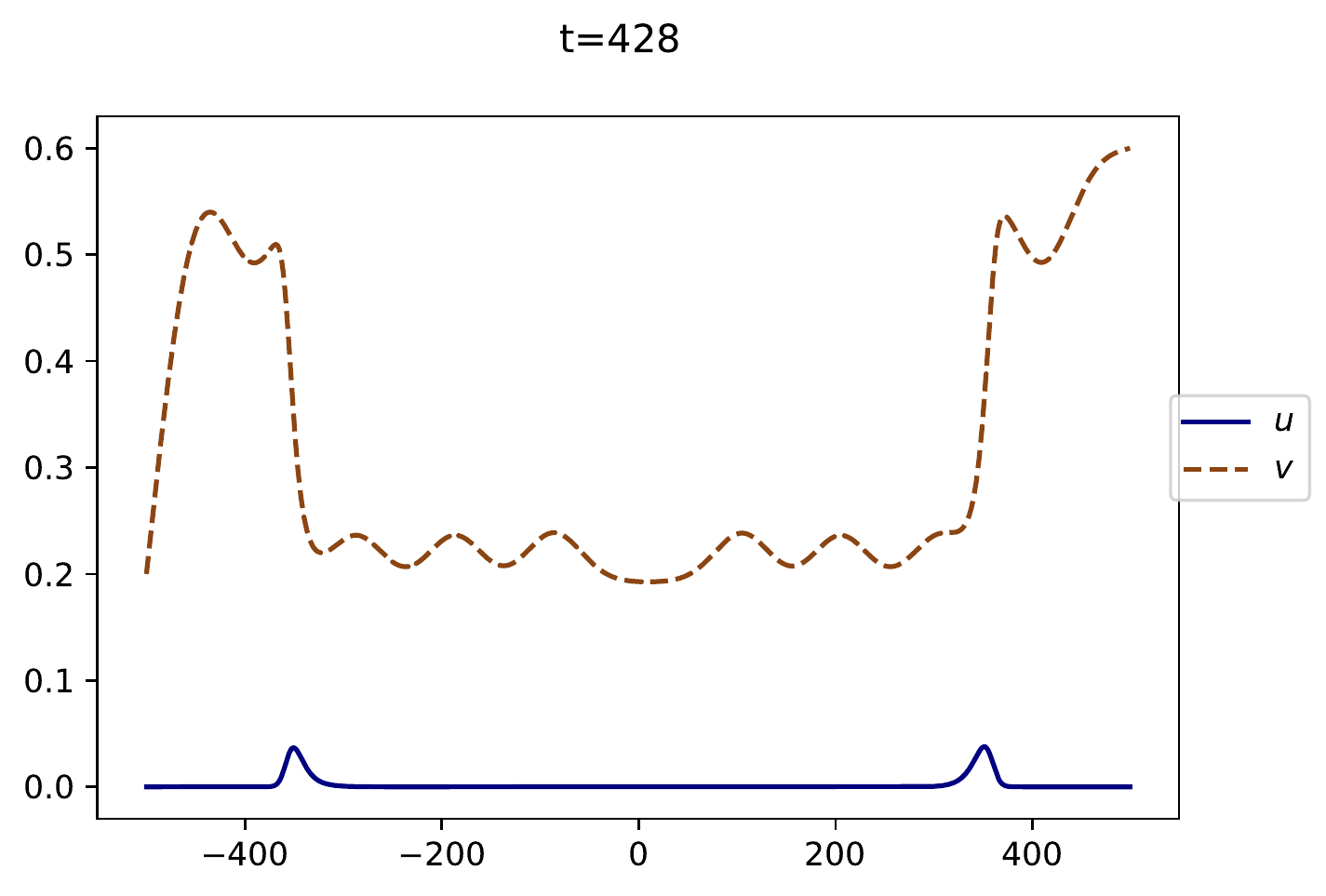}
  \end{subfigure}
\caption{Non-uniform initial social tension $v_0$ -- propagation. Snapshots at different times of the solution of \eqref{HeteroV0}-\eqref{HeteroV0_def} with $L=20$, $u_0(x)=0.2(1-x^2)_+$. Horizontal axis: space. Blue solid line: $u(t,\cdot)$. Brown dashed line: $v(t,\cdot)$.
\href{https://drive.google.com/file/d/1d74a9btmcuVt9SWbMXs24uzT8KJKCj2o/view}{\color{blue}\underline{Video: Periodic\_V0\_L=20.mp4}}
} \label{fig:V0GAP_propagation}
\end{figure}

\begin{figure}[p]
  \centering
  \begin{subfigure}[p]{0.45\linewidth}
    \centering\includegraphics[width=\textwidth]{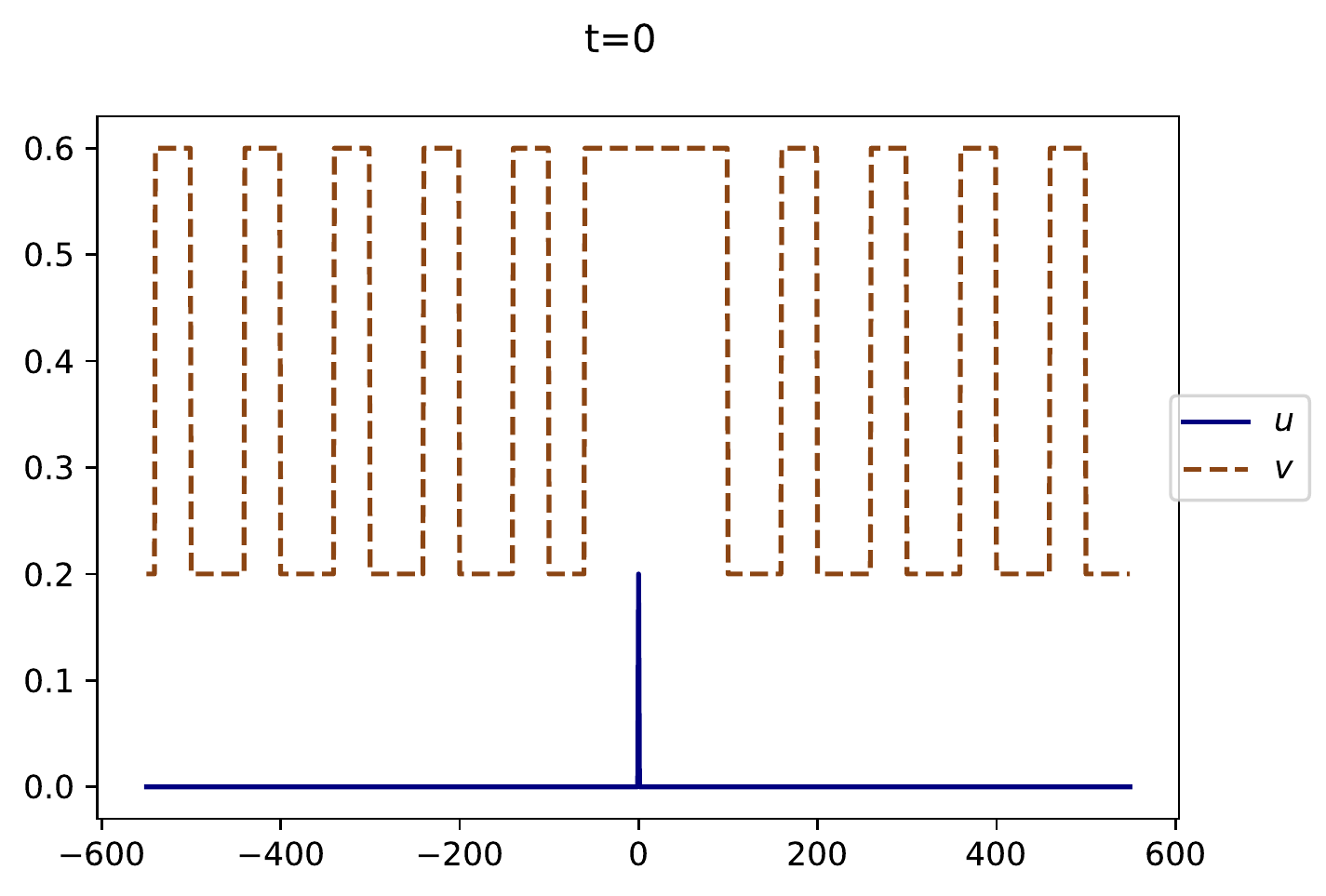}
  \end{subfigure}
  \hfill
  \begin{subfigure}[p]{0.45\linewidth}
    \centering\includegraphics[width=\textwidth]{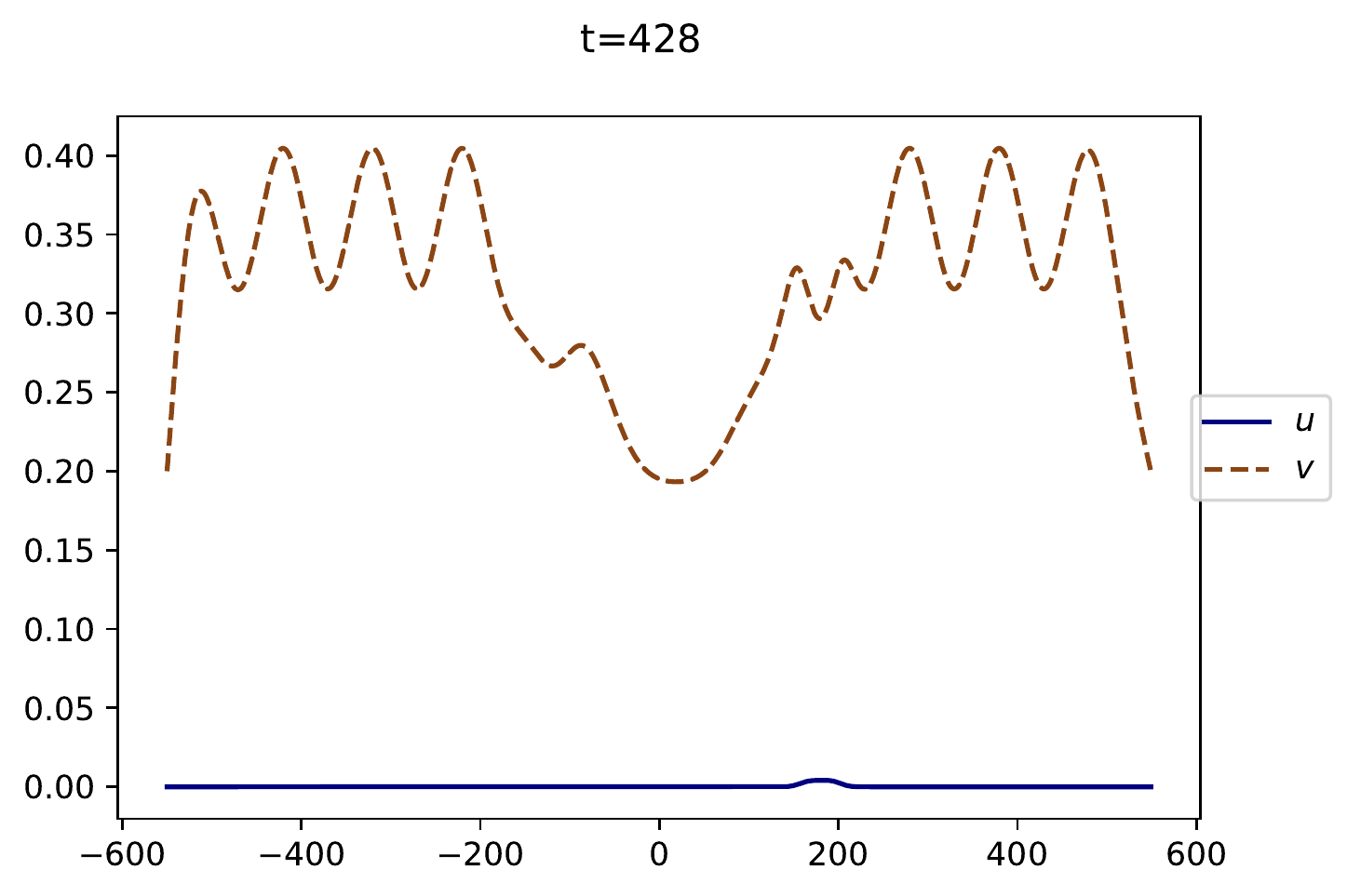}
  \end{subfigure}\\
    \centering
    \begin{subfigure}[p]{0.45\linewidth}
    \centering\includegraphics[width=\textwidth]{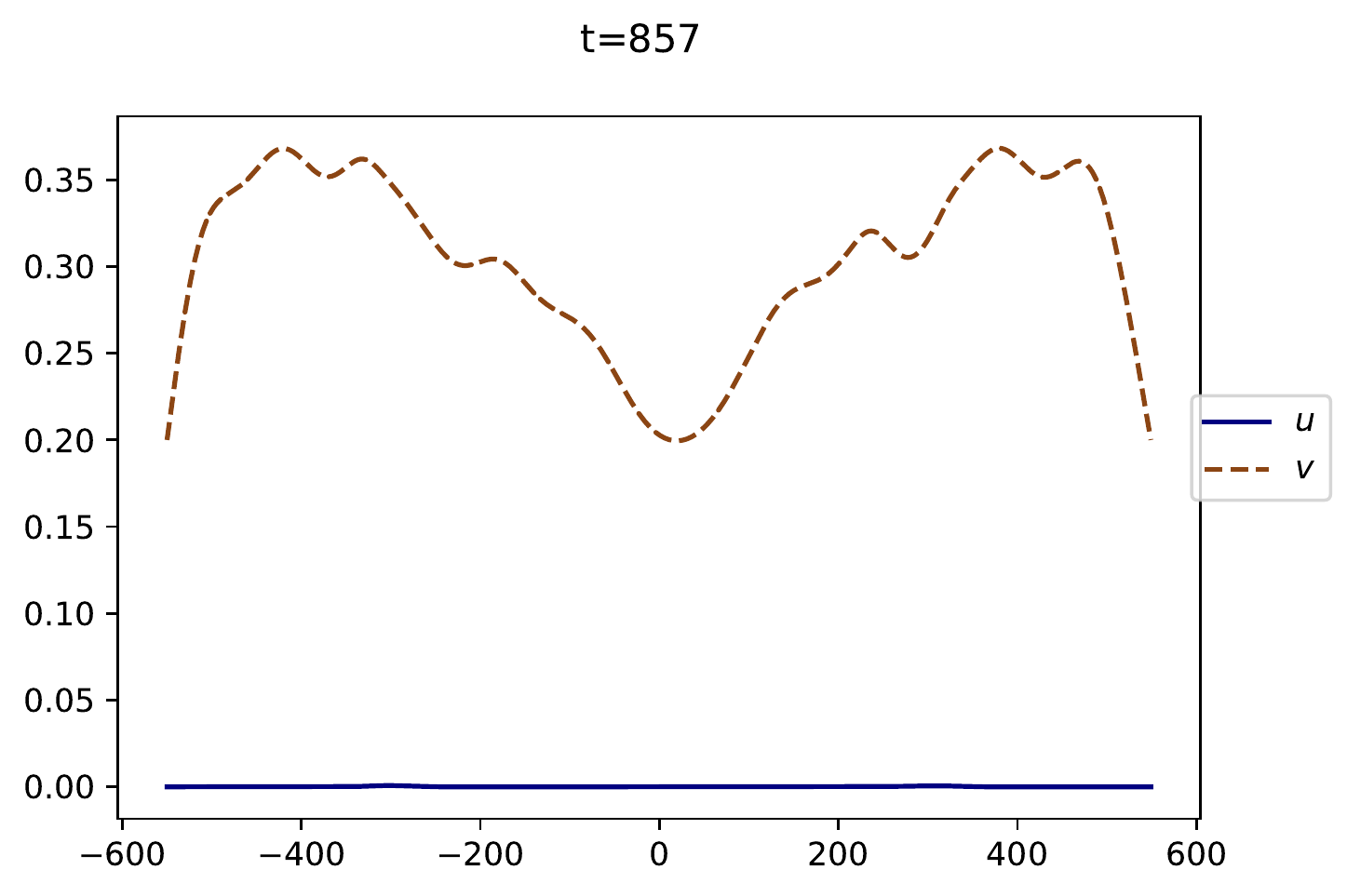}
  \end{subfigure}
   \hfill
  \begin{subfigure}[p]{0.45\linewidth}
    \centering\includegraphics[width=\textwidth]{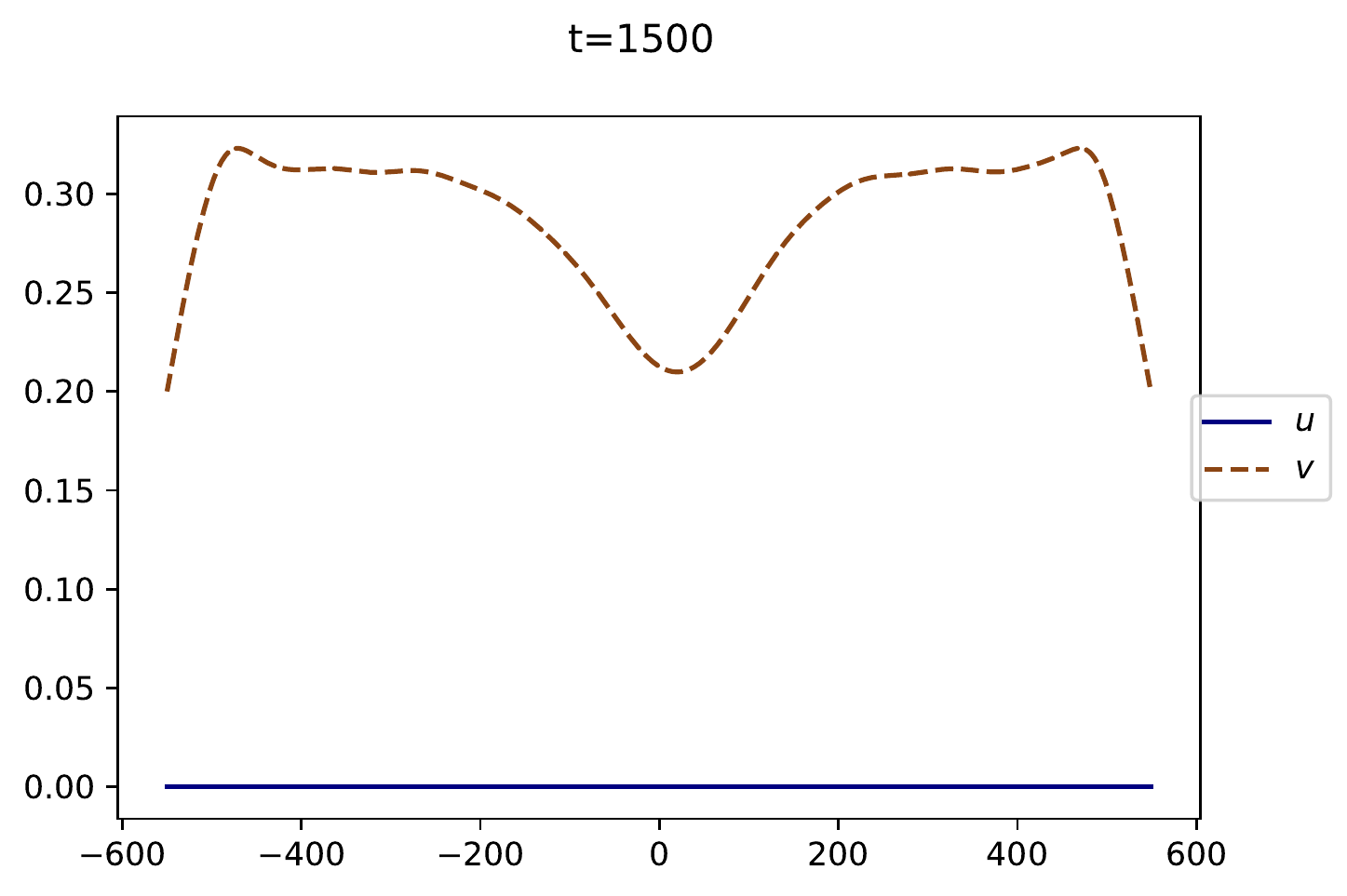}
  \end{subfigure}
\caption{Non-uniform initial social tension $v_0$ -- blockage. Snapshots at different times of the solution of \eqref{HeteroV0}-\eqref{HeteroV0_def} with $L=60$, $u_0(x)=0.2(1-x^2)_+$. Horizontal axis: space. Blue solid line: $u(t,\cdot)$. Brown dashed line: $v(t,\cdot)$.
\href{https://drive.google.com/file/d/1Ll0nUeYplzLKK96els3CVqkgs-ZffIUM/view}{\color{blue}\underline{Video: Periodic\_V0\_L=60.mp4}}
} \label{fig:V0GAP_blockage}
\end{figure}

We leave for future works the rigorous analysis of the case of a non-constant initial level of social tension $v_0(\cdot)\equiv v_b(\cdot)$. In this case, however, let us indicate that the threshold phenomenon on the initial level of social tension may not involve the sign of $v_b-v_\star$, but rather the sign of $\lambda_b$ defined as the lowest eigenvalue of the operator, $\forall \varphi\in C^2(\Omega)$,
\begin{equation*}
-d_1\Delta \varphi-\D_u\Phi(0,v_b(x))\varphi,
\end{equation*}
and given by the expression
\begin{equation}\label{eigenvalueExtensions}
\lambda_b:= \inf\left\{\int_{\R^n} d_1\vert\nabla\varphi\vert^2-\left[r(v_b(x))f(0)-\omega\right]\varphi^2: \varphi\in H^1(\R^n),\ \Vert\varphi\Vert_{L^2}=1\right\}.
\end{equation}

%

\subsection{Including geometry}

Another possible way to include spatial heterogeneity in our model is to consider the system~\eqref{GeneralEquationMotivationFinal} on a domain $\Omega\subset\R^n$ rather than on the entire space. We may impose Neumann boundary condition
\begin{equation*}
\D_{\nu}u=0,\qquad \text{on }\D\Omega,
\end{equation*}
with $\D_{\nu}$ the outer normal derivative, accounting for the fact that there is no flux of individuals across the boundary.

From the modeling point of view, the boundary $\D\Omega$ could stand for various structural spatial obstructions for rioters, such as streets, highways, fences, rivers, mountains etc. These features play sometimes a significant role; for example, the ring around Paris in the 2005 riots~\cite{Bonnasse-Gahot2018}.

Mathematically speaking, it is known that heterogeneous geometry can largely affect the propagation properties of Reaction-Diffusion equations (e.g., see~\cite{Ducasse2018,Berestyckib}). However, fewer papers deal with systems of Reaction-Diffusion such as our system~\eqref{GeneralEquationMotivationFinal}. We leave further investigations on this topic for future works.

\section{Proofs}\label{sec:Proofs}

\subsection{The tension inhibiting case}\label{sec:ProofsInhibiting}

We begin with the proof of Propositions~\ref{PropVInftyLeqVStar} and~\ref{Propo_Moyenne_Inhib}.

\begin{proof}[Proof of Proposition~\ref{PropVInftyLeqVStar}]
	First of all, we know from~\eqref{PropositionQualitativeInhibiting} that $U(\pm\infty)=0$,
	whence $U'(\pm\infty)=0$ by elliptic estimates.
	From the equation on $V$, we get 
	\begin{equation*}
	U=\frac{cV'}{G(V)}.
	\end{equation*}
	Injecting this formula into the equation on $U$, and integrating over $(-\infty,+\infty)$ we find
	\begin{equation*}
	c\left(U(+\infty)-U(-\infty)\right)=d_1(U'(+\infty)-U'(-\infty)) +c\int_{-\infty}^{+\infty} V'\left[\frac{r(V)-\omega}{G(V)}\right],
	\end{equation*}
	that is,
	\begin{equation*}
	\int_{-\infty}^{+\infty} V'\left[\frac{r(V)-\omega}{G(V)}\right]=0,
	\end{equation*}
	and so $Q(V_\infty)-Q(v_b)=0$
\end{proof}




\begin{proof}[Proof of Proposition~\ref{Propo_Moyenne_Inhib}]
	The smallness of $u_0$ is determined by the following condition:
	\begin{equation}\label{segment}
	\forall v\in[\tilde v_b, v_b],\quad
	r(v)f(u_0)>\omega,
	\end{equation}
	which can be achieved since $v_b\geq\tilde v_b> v_\star$ and so $r(v)f(0)>\omega$ for all $v\in[\tilde v_b,\tilde v]$.
	For all time $t\geq0$, both $u(t,\cdot)$ and $v(t,\cdot)$ are constant. We thus omit to write the dependence on $x$.
	The proof relies on a simple argument in the phase plane. 
	Denote $\gamma(t)$ and $\tilde \gamma(t)$ the trajectories in the plane $(u,v)$, namely $\gamma(t):=(u(t),v(t))$ and $\tilde \gamma(t):=(\tilde u(t),\tilde v(t))$. 
	By the inhibiting assumption~\eqref{hyp:inhibiting}, these trajectories are nonincreasing in the $v$ component and, by
	Theorem~\ref{PropositionAsymptoticsInhibiting}, they have two endpoints $(0,v_\infty)$, 
	$(0,\tilde v_\infty)$ respectively. 
	%
	In addition, by uniqueness of solutions of the Cauchy problem, 
	they cannot cross each other, nor can $\gamma$ cross
	the segment $\{u_0\}\times( \tilde v_b,  v_b)$, 
	because there $u'>0$ thanks to~\eqref{segment}.
	The result then follows.
\end{proof}

Let us now turn to the proof of \autoref{PropositionAsymptoticsInhibiting}. The first statement of the theorem is a consequence of~\cite[Proposition~7]{Berestycki2019b}.
In order to prove the second statement, we make 
use of the following classical lemma, whose proof is inspired from that of 
\cite[Theorem~1.8]{Berestycki1997b}.
\begin{lemma}\label{LemmaLiouvilleRiots}
	Let $w\in C^{2}(\R^n)$ be a bounded function satisfying $w\Delta w\geq 0$ in $\R^n$. 
	If $n\leq2$, then $w$ is constant.
\end{lemma}

\begin{proof}
	For $R>0$, we define a cut-off function
	\begin{equation*}
	\chi_R(x):=\chi\left(\frac{\vert x\vert}{R}\right), \quad \forall x\in\R^n,
	\end{equation*} 
	for $\chi$ a smooth nonnegative function such that $${\chi(z)=
		\left\{\begin{aligned}
		&1 &&\text{if }0\leq z\leq 1,\\
		&0 &&\text{if } z\geq 2,
		\end{aligned}\right.}\quad \vert \chi'\vert\leq 2.$$
	Multiplying the equation on $w$ by $\chi_{R}^2$, integrating on $\R^n$ and using the divergence theorem, we find
	\begin{align*}
	0&\leq -\int_{\R^n}\nabla\left(\chi_{R}^2 w\right)\cdot\nabla\ w\\
	&=-\int_{\R^n}\chi_{R}^2\vert\nabla w\vert^2-2\int_{\R^n}\chi_{R} w\nabla\chi_{R}\cdot\nabla w.
	\end{align*}
	Using the Cauchy-Schwarz inequality, we deduce
	\begin{equation}\label{inequalityRIOT}
	\int_{\R^n}\chi_{R}^2\vert\nabla w\vert^2
	\leq 2\sqrt{\int_{\substack{B_{2R}\backslash B_R}}\chi_{R}^2\vert\nabla w\vert^2}\sqrt{\int_{\R^n} w^2\vert\nabla\chi_{R}\vert^2},
	\end{equation}
	where $B_R$ is the ball of radius $R$ and center $0$. Recall that $w$ is bounded and that $\vert\nabla \chi_R\vert^2\leq R^{-2}\Vert \nabla \chi\Vert_\infty.$
	As $n\leq 2$, we deduce that
	\begin{equation*}
	\int_{\R^n} w^2\vert\nabla\chi_{R}\vert^2\text{ is bounded, uniformly in }R\geq1.
	\end{equation*}
	From \eqref{inequalityRIOT}, we have that
	$\int_{\R^n} \chi_{R}^2\vert\nabla w\vert^2$ is uniformly bounded. Therefore, $\int_{\substack{B_{2R}\backslash B_R}}\chi_{R}^2\vert\nabla w\vert^2$ converges to $0$ as $R\to+\infty$, and so the term on the left-hand side of~\eqref{inequalityRIOT} also converges to $0$. At the limit, we find
	${
		\int_{\R^n}\vert\nabla w\vert^2\leq0.
	}$
	Hence $\nabla w\equiv0$, which ends the proof.
\end{proof}

\begin{proof}[Proof of the second statement of \autoref{PropositionAsymptoticsInhibiting}]
	Assume that $v(t,x)$ converges pointwise to some $v_\infty(x)$ as $t\to+\infty$.
	From classical parabolic estimates, the convergence actually occurs in $C^2_{loc}$. Since $\D_t v -\Delta v\leq0$, when $t\to+\infty$ we find
	\begin{equation}\label{1stLemmaEquOnV}
	-\Delta v_\infty\leq0\quad\text{on $\R^n$.}
	\end{equation} 
	Lemma~\ref{LemmaLiouvilleRiots} implies that $v_\infty$ is constant.
	
	We claim that $v_\infty\leq v_\star$. By contradiction, assume $v_\infty>v_\star$, then, for $t$ large enough, $u$ satisfies
	\begin{equation*}
	\D_t u\geq d_1\Delta u+u\big(r(\alpha)f(u)-\omega \big),
	\end{equation*}
	with $\alpha:=\frac{v_\infty+v_\star}{2}>v_\star$. Thus, $u$ is a supersolution of a classical KPP equation, we deduce
	\begin{equation*}
	\liminf\limits_{t\to+\infty}u(t,x)> 0.
	\end{equation*} 
	We reach a contradiction with the first assertion of~\autoref{PropositionAsymptoticsEnhancing}, namely~\eqref{limifuto0}, which we proved previsouly.
	
	Finally, let us show that $u(t,\cdot)$ converges locally uniformly to $0$ when $t\to+\infty$. Let us first consider the case $v_\infty\neq0$. By contradiction, assume that there exists a diverging sequence of times $t_k>0$ and a ball $B\subset\R^n$ such that
	\begin{equation}\label{ContradictionUto0}
	\inf_{\substack{k\geq0\\ x\in B}}u(t_k,x)>0.
	\end{equation}
	We use the notation
	$$
	\tilde \Psi(t,x):=\frac{\Psi(u,v)}{u}.
	$$
	Since $v_\infty\in(0,1)$, the inhibiting assumption~\eqref{hyp:inhibiting} implies
	\begin{equation}\label{BorneInfPsiTilde}
	\sup_{\substack{k\geq0\\ x\in B}}\tilde \Psi(t_k,x)<0.
	\end{equation}
	From the equation on $v$, we can write 
	\begin{equation*}
	\D_t v-\Delta v= u\tilde \Psi.
	\end{equation*}
	From the convergence of $v(t,\cdot)$, we deduce that $u(t,\cdot)\tilde \Psi(t,\cdot)$ converges locally uniformly to $0$ as $t\to+\infty$. From~\eqref{BorneInfPsiTilde}, we thus have that $u(t_k,x)$ converges to $0$ as $k\to+\infty$ and $x\in B$: we reach a contradiction with~\eqref{ContradictionUto0}.

	Let us now consider the case $v_\infty= 0$. Fixing $\eps\in(0,v_\star)$, for $t$ large enough $u$ satisfies
	\begin{equation*}
	\D_t u -d_1\Delta u\leq u\left(r(\eps)f(u)-\omega \right)\leq -Cu,
	\end{equation*}
	where $C:=\omega-r(\eps)f(0)>0$. It proves that $u(t,\cdot)$ converges uniformly to $0$ as $t\to+\infty$.
\end{proof}

\subsection{The tension enhancing case}\label{sec:ProofsEnhancing}

Let us prove \autoref{PropositionAsymptoticsEnhancing} which deal with the asymptotic behavior of the solution in the tension enhancing case. The proof follows the same line as the proof of Proposition~10 in~\cite {Berestycki2019b} but, here, we derive a more precise result on the limit of $U$ in $+\infty$, and so we give a full proof for completeness.

\begin{proof}[Proof of~\autoref{PropositionAsymptoticsEnhancing} in the case $d_2>0$]
 	First of all, we know from Lemma~\ref{Lemma_uv} that
 	$u>0$ and $0<v<1$ for all $t>0$, $x\in\R^n$.
	The proof is achieved in four steps: we first derive an upper bound for $u$, next a lower bound for 
	$v$, then for $u$, and we finally conclude.
	
	\smallskip
	{\em Upper bound for $u$.}\\
	Let us derive the upper bound for $u$ as $t\to+\infty$.
	We see that, by the monotonicity of $r$,
	$$\D_t  u -d_1 \Delta u\leq u\big[r(1)f_+(u)-\omega\big],$$
	where $f_+:=\max\{f,0\}$ is the positive part of $f$.
	Let $U$ be the solution of the ODE $U'=U\big(r(1)f_+(U)-\omega\big)$ with initial datum
	$U(0)=\max\{M,\sup u_0\}$. There holds that $U\searrow u_\star(1)$
	given by~\eqref{def:u_star}, i.e., the zero of 
	$u\mapsto r(1)f(u)-\omega$, whose existence and uniqueness is guaranteed
	by~\eqref{Def_v_star} and~\eqref{AssumptionTensionEnhancing}. Then, by comparison,
	we get the desired upper bound:
	\begin{equation*}
	\limsup_{t\to+\infty}\left(\sup_{x\in\R^n}u(t,x)\right)\leq u_\star(1)<M.
	\end{equation*}
	It follows that there exist $M'<M$ and $T>0$ such that 
	\begin{equation}\label{u<u*}
	u(t,x)\leq M'<M,\quad\forall t\geq T,\ x\in\R^n.
	\end{equation}	
	
	\smallskip
	{\em Lower bound for $v$.}\\
	Consider a sequence $(x_k)_{k\in\N}$ in 
	$\R^n$ such that $|x_k|\to+\infty$. By parabolic estimates,
	the functions $u(\cdot,x_k+\cdot)$, $v(\cdot,x_k+\cdot)$ converge 
	(up to subsequences) as $k\to+\infty$, locally uniformly in $[0,+\infty)\times\R^n$,
	to some functions $\tilde u$, $\tilde v$ which are still 
	solutions of~\eqref{GeneralEquationMotivationFinal}, with initial datum
	$(0,v_b)$. Hence $(\tilde u,\tilde v)\equiv(0,v_b)$ because
	$\Phi(0,v_b)=\Psi(0,v_b)=0$. This means that $(u(t,x),v(t,x))\to
	(0,v_b)$ as $|x|\to+\infty$, for any given $t\geq0$. In particular, for any fixed $t\geq0$ and
	$\underline v\in(v_\star,v_b)$, we have that $v(t,x)>\underline v$ for $|x|$ sufficiently large.
	Moreover, \eqref{AssumptionTensionEnhancing} and~\eqref{u<u*} imply that 
	$\Psi(u,v)>0$ for $t\geq T$, $x\in\R^n$,
%
	that is, $v$ is a supersolution of the heat equation, 
	and we know that at time $T$ it is larger than $\underline v$ outside a large ball. 
	By comparison with the heat equation, 
	one readily deduces that, for given $\underline v'\in(v_\star,\underline v)$, 
	there exists $T'>T$ such that
	\begin{equation}\label{v>v'}
	v(t,x)\geq \underline v'>v_\star,\quad\forall t\geq T',\ x\in\R^n.
	\end{equation}

%
%
%
%
%
%
%

	\smallskip
	{\em Lower bound for $u$.}\\
	Consider the equation 
	\begin{equation*}
	\D_t \hat u -d_1\Delta \hat u = \hat u\big[r(\underline v')f(\hat u)-\omega\big],
	\end{equation*}
	which is a standard scalar KPP equation. 	
	Observe indeed that $r(\underline v')f(0)-\omega>0>r(\underline v')f(M)-\omega$ by the 
	definition~\eqref{Def_v_star} of~$v_\star$ and~\eqref{AssumptionTensionEnhancing}.
	We consider the solution of this KPP equation starting at time $T'$ with the datum 
	$\hat u(T',x)=\min\{u(T',x),u_\star(\underline v')\}$, where
	$u_\star(\underline v')>0$ is given by~\eqref{def:u_star}, i.e., 
	$r(\underline v')f(u_\star(\underline v'))=\omega$.
 	It follows from the classical result of~\cite{Aronson1978a} that $\hat u(t,x)\nearrow
	u_\star(\underline v')$ as $t\to+\infty$, locally uniformly in $x\in\R^n$.
	For $t\geq T'$, using that
	$\hat u(t,\cdot)\leq u_\star(\underline v')$ and $v(t,\cdot)\geq \underline v'$ by~\eqref{v>v'},
	we see that $\hat u$ is a subsolution of the first equation 
	in~\eqref{GeneralEquationMotivationFinal}, whence, by comparison,
	\begin{equation}\label{u>u*}
	\liminf_{t\to+\infty}u(t,x)\geq u_\star(\underline v'),\quad
	\forall x \in \R^n.
	\end{equation}
	
	\smallskip
	{\em Conclusion.}\\
	Let $(t_k)_{k\in\N}$ be an arbitrary sequence diverging to $+\infty$.
	The functions $u(t_k+\cdot,\cdot)$, $v(t_k+\cdot,\cdot)$ converge 
	(up to subsequences) as $k\to+\infty$, locally uniformly in $\R\times\R^n$,
	to some functions $u_\infty$, $v_\infty$ which are entire 
	solutions  (i.e., for $t\in\R$, $x\in\R^n$) of the equations 
	in~\eqref{GeneralEquationMotivationFinal}.
	Moreover,~\eqref{u<u*}, \eqref{v>v'}, \eqref{u>u*} yield 
	$u_\star(\underline v')\leq u_\infty\leq u_\star(1)$ and $v_\infty\geq\underline v'$.
	From~\eqref{AssumptionTensionEnhancing} we deduce that necessarily 
	$v_\infty\equiv1$ and then that $u_\infty\equiv u_\star(1)$.
\end{proof}

\begin{proof}[Proof of~\autoref{PropositionAsymptoticsEnhancing} in the case $d_2=0$]
	We immediately see that $0<u<\max\{\sup u_0,M\}$, owing
	to~\eqref{AssumptionTensionEnhancing} and the parabolic strong maximum principle,
	and that $0<v<1$, by elementary ODE considerations, for all $t>0$, $x\in\R^n$.
	Moreover, we 
	observe that the upper bound~\eqref{u<u*} for $u$ is derived in the above
	proof for the case $d_2>0$ only by arguing on 
	the first equation in~\eqref{GeneralEquationMotivationFinal}, hence it holds true when~$d_2=0$. 
	We want to derive now the upper bound 
	\begin{equation}\label{u<MR}
	u(t,x)< M,\quad\forall t\geq0,\ |x|\geq R,
	\end{equation}
	for some possibly very large $R$. 
	For this we consider, for any given direction $e\in\mathbb{S}^{n-1}$,
	the function 
	$$u_e(t,x)=e^{\sigma (t+1)-x\cdot e}.$$
	It is readily seen that there exists $\sigma$ sufficiently large so that
	this is a supersolution of the first equation 
	in~\eqref{GeneralEquationMotivationFinal}.
	Moreover, since $u_0$ is compactly supported, we can choose $\sigma$, possibly even larger
	and independent of $e$, so that in addition $u_e(0,x)>u_0(x)$ for all $x\in\R^n$.
	Therefore, by comparison, $u\leq u_e$ for all $t\geq0$ and $x\in\R^n$, 
	which, being true for any $e\in\mathbb{S}^{n-1}$, yields $u(t,x)\leq e^{\sigma (t+1)-|x|}$.	
	We deduce in particular that $u(t,x)<M$ for all $t\in[0,T)$ an $|x|\geq \sigma(T+1)-\log M$.
	Combining this with~\eqref{u<u*} eventually gives~\eqref{u<MR} with $R=\sigma(T+1)-\log M$.
		
	We now use the upper bound~\eqref{u<MR} for $u$ in the equation for $v$.
	Owing to~\eqref{AssumptionTensionEnhancing},
	it implies that, for $t\geq 0$ and $|x|\geq R$, 
	$\partial_t v=\Psi(u,v)>0$, hence 
	\begin{equation}\label{v>vb}
	v(t,x)\geq v_b,\quad\forall t\geq0,\ |x|\geq R.
	\end{equation}
	
	Let us derive a lower bound on $u$.
	Let $\lambda_\rho$ be the Dirichlet principal eigenvalue of $-\Delta$ in $B_\rho$, 
	and $\varphi_\rho$ be the associated (positive) eigenfunction.
	It is well known that $\lambda_\rho\searrow0$ as $\rho\to+\infty$, hence in particular
	$d_1\lambda_\rho<r(v_b)f(0)-\omega$ for $\rho$ large enough, 
	because $v_b>v_\star$ defined by~\eqref{Def_v_star} and therefore $r(v_b)f(0)>\omega$.
	It follows that, for such a $\rho$ and for $\eps>0$ small enough, 
	$$-d_1\Delta(\eps\varphi_\rho)=\lambda_\rho\eps\varphi_r<
	\eps\varphi_\rho\big[r(v_b)f(\eps\varphi_\rho)-\omega\big],$$
	whence, by \eqref{v>vb},
	$\eps\varphi_\rho(x-x_0)$ is a subsolution to the first equation 
	in~\eqref{GeneralEquationMotivationFinal} for $t>0$ and $x\in B_\rho(x_0)$ whenever
	$|x_0|>R+\rho$. Take $x_0$ satisfying $|x_0|>R+\rho$ and $\eps>0$ small enough so that
	the above property holds and moreover
	$\eps\varphi_\rho(x-x_0)<u(1,x)$ for all $x\in B_\rho(x_0)$.
	The comparison principle yields
	$u(t,x)>\eps\varphi_\rho(x-x_0)$ for all $t\geq 1$, $x\in B_\rho(x_0)$,
	thus, using the parabolic Harnack inequality, we find,
	for any compact set~$\mathcal{K}\subset\R^n$,
	\begin{equation}\label{u>0}
	m:=\inf_{\substack{t\geq 2\\ x\in\mathcal{K}}}u(t,x)>0.
	\end{equation}
	
	We are now in a position to conclude. For $s\geq0$ call
	$$g(s):=\min_{z\in[m,M']}\Psi(z,s).$$
	This function satisfies $g(s)>0$ for $s\in(0,1)$ by~\eqref{AssumptionTensionEnhancing} 
	and $g(1)=0$ by~\eqref{AssumptionSaturationV}.
	For $t\geq \max\{T,2\}$ and $x\in\mathcal{K}$, since $m\leq u(t,x)\leq M'$ by
	\eqref{u<u*} and~\eqref{u>0}, we see that 
	$$\partial_t v(t,x)=\Psi(u,v)\geq g(v).$$ 
	As a consequence, because $\min_{x\in\mathcal{K}}v(\max\{T,2\},x)>0$, by the continuity of $v$, 
	we infer that $v(t,x)\to1$ as $t\to+\infty$ uniformly in $x\in\mathcal{K}$.
	We have thereby shown that $v(t,x)\to1$ as $t\to+\infty$ locally uniformly in $x\in\R^n$.
	Finally, for any sequence $(t_k)_{k\in\N}$ diverging to $+\infty$,
	the function $u(t_k+\cdot,\cdot)$ converges 
	(up to subsequences) as $k\to+\infty$, locally uniformly in $\R\times\R^n$, 
	to a nonnegative, bounded solution $\tilde u$ of 
	\begin{equation*}
	\D_t \tilde u -d_1\Delta  \tilde u =  \tilde u\big[r(1)f(\tilde u)-\omega\big],
	\end{equation*}
	which satisfies $\tilde u(t,x)\geq m$ for all $t\in\R$, $x\in\mathcal K$ thanks to~\eqref{u>0}.
	It is a straightforward consequence of~\cite{Aronson1978a} that the only entire solution of this
	standard KPP equation satisfying such a property
	is $\tilde u\equiv u_\star(1)$.
	This concludes the~proof.
\end{proof}

\section{Conclusion}\label{sec:conclusion}
\subsection{Main findings}

An increasing number of papers consider systems of Reaction-Diffusion equations to model the dynamics of epidemics or collective behaviors such as riots. 
However, most of the work study particular and different cases. 
In this paper, we try to propose a unified mathematical approach, based on the theoretical framework developed in~\cite{Berestycki2019b}.  Although we focus on the problem of modeling social unrest, our goal is to keep a rather general mathematical approach that can be transposed to other topics in social dynamics.  

Our model involves two quantities, the level social unrest $u$ and the level of social tension~$v$, which play asymmetric roles.  We examine the problem of a system initially at equilibrium $u=0$, $v=v_b$, for which a triggering event $u_0(\cdot)\gneqq0$ occurs at $t=0$. 
After stating our modeling assumptions, we derive the Reaction-Diffusion
 system~\eqref{GeneralEquationMotivationFinal}.  

In Section~\ref{sec:Analysis}, we highlight a threshold phenomenon on the initial level of social tension
$v_0\equiv v_b$. On the one hand, if $v_b$ is below a threshold value $v_\star$ and the triggering event is small enough, the system returns to equilibrium quickly, and we speak of a return to calm. On the other hand, if $v_b$ is above $v_\star$, an arbitrarily small triggering event causes an eruption of social unrest.  Then, the movement of social unrest spreads through space with an asymptotically
constant speed. 

We are able to derive
more complete theoretical and numerical results on two subclasses of models.
The first one, called \emph{tension inhibiting}, 
is such that the movement of social unrest dissipates social tension.  Once the level of social tension falls below the threshold value $v_\star$, in turn, the level of social unrest fades until it is extinguished
as $t\to+\infty$. 
This behavior is exhibited by both traveling wave solutions, 
c.f.~\eqref{PropositionQualitativeInhibiting},
as well as by solutions of the Cauchy problem, c.f.~\autoref{PropositionAsymptoticsInhibiting}.
Tension inhibiting models thus 
give rise to limited duration movement of social agitation, that we call ``riots''.  An interesting property is that the intensity of the triggering event has no influence on the qualitative dynamic of the system. 
We numerically observe that the solution converges to two opposite 
traveling waves moving with the speed
$c_b:=2\sqrt{r(v_b)f(0)-\omega}$ (which does not depend on the parameters of the equation on $v$)
and link the steady state $(0,v_b)$ to another one $(0,v_\infty)$,
the profile of $u$ having the shape of a bump, and that of $v$ a monotonous decreasing wave, linking .  
We also invesigate theoretically and numerically 
the question of estimating the final level of social tension $v_\infty$, revealing the non-monotonic structure
underlying the inhibiting system, see~\autoref{Propo_Moyenne_Inhib}.

The second specific class of models we examine is the \emph{tension enhancing}. 
For such systems, if 
the initial level of social tension $v_b$ is higher than the threshold value $v_\star$, the dynamics gives rise to a movement of social agitation that converges in a long time to a sustainable excited state. This case typically accounts for time-persisting social movements, which we call here lasting upheaval. 
We numerically observe that if $v_b<v_\star$,
the solution converges towards two opposite 
traveling waves, whose speed can take intermediate values between
$c_b$ and $c_1:=2\sqrt{r(1)f(1)-\omega}$ 
(depending on the parameters of the equation on $v$, c.f.~Figure~\ref{fig:Enhanc_Speed_d2k}).
These waves connect $(0,v_b)$ to $(u_\star(1),1)$ (defined in~\eqref{def:u_star}),
the profiles of $u$ and $v$ having the shape of increasing waves,
c.f.~Theorem~\ref{PropositionQualitativeInhibitingTW}.
If $v_b<v_\star$, contrarily to the \emph{tension inhibiting} case, we observe that a sufficiently strong triggering event can still ignite a lasting upheaval, see Figure~\ref{fig:Triggering_eps}.

The \emph{tension inhibiting} and \emph{tension enhancing} classes of models give a good idea of the variety of behaviors that our model can generate.
In Section~\ref{sec:MixedCase}, we examine mixed cases that exhibit more complex behaviors; some models feature a double threshold effect between \emph{return to calm}, \emph{riot} and \emph{lasting upheaval}, others generate oscillating traveling waves or terraces (consisting of a riot followed by a lasting upheaval).

In Section~\ref{sec:SpaceHeterogeneous}, we propose several ways to include spatial heterogeneity in our model. We first consider the case of heterogeneous coefficients and study how an obstacle (i.e. an area of depressed growth for social unrest) affects the propagation of a social movement. On the one hand, if $v_b>v_\star$, the propagation of the social movement is guaranteed in any case. On the other hand, if $v_b<v_\star$, propagation is only possible if the triggering event is sufficiently large and the gap is sufficiently small. We then consider the case of an initial level of social tension $v_0$ that is not constant. This case accounts for the variability of populations according to the neighborhood (for example, between a city and its suburbs) which may have a significant impact on social movement according to data.
Finally, we mention that our framework allows to include geometrical heterogeneity through the domain on which we pose the system of equations~\eqref{GeneralEquationMotivationFinal}.

\subsection{Possible extensions and perspectives}

We conclude by mentionning several other extensions which are relevant regarding the modeling of social unrest.
\subsubsection{Non-local diffusion}\label{sec:Ext_NonLocal}
A possible extension of our model is two replace the Laplace operator in~\eqref{GeneralEquationMotivationFinal} with some non-local diffusion operator. One can consider, for example, that the classical diffusion is replaced by the convolution with an integrable kernel $K(\cdot)$
\begin{equation*}
K\ast u(x)= \int_{\Omega} K(x-y)u(y)dy.
\end{equation*}
Another interesting example is the fractional Laplacian, for $s\in(0,1)$,
\begin{equation*}
\Delta^s u (x)= c_{n,s}\int_{\R^n}\frac{u(x)-u(y)}{\vert x-y\vert^{n+2s}}dy, \qquad \forall x\in\R^n,
\end{equation*}
with $c_{n,s}$ a normalization constant.

On the one hand, a non-local diffusion on the level of activity $u$ could account for the fact that rioters can travel to another location. On the other hand, a nonlocal on the level of social tension could account for the global spreading of information through media.

Non-local diffusion is increasingly used in various modeling situations (e.g.,~\cite{Berestycki2016a} deals with the modeling of riots, and~\cite{Mendez2010} contains many other topics), and often leads to some anomalous behaviors. We let the reader refer to~\cite{Mendez2010,Lewis2016} and references therein for more details.

\subsubsection{Compartmental models}
An underlying hypothesis of our modeling approach is that all individuals are identical. Yet, the variability of individuals sometimes plays an important role in collective behaviors~\cite{Granovetter1978,Gavrilets2015}. It is often admitted that certain social and economic classes are more prone to trigger or drive a social movement, such as students~\cite{Altbach2014}, rural population~\cite{Past2020}, activists~\cite{Mistry2015}, etc.

One way to include individuals variability in our model is to consider two different levels of activity $u_1$ and $u_2$, the first accounting for the rioting activity of activist and leaders, the other accounting for the rioting activity of more reluctant individuals. It remains unclear how our conclusions would be affected by this additional feature, and we leave this question as an open problem.


\section*{Aknowledgments}

The research leading to these results has received funding from the European Research Council under the European Union’s Seventh Framework Programme (FP/2007-2013) / ERC Grant Agreement n. 321186 - ReaDi - “Reaction-Diffusion Equations, Propagation and modeling” held by Henri Berestycki. This work was also partially supported by the French National Research Agency (ANR), within project NONLOCAL ANR-14-CE25-0013.

\bibliographystyle{abbrv}
\bibliography{library}
\end{document}